\documentclass[11pt, a4paper]{article}

\usepackage{graphicx}%
\usepackage{multirow}%
\usepackage{amsmath,amssymb,amsfonts}%
\usepackage{amsthm}%
\usepackage{mathrsfs}%
\usepackage[title]{appendix}%
\usepackage{xcolor}%
\usepackage{textcomp}%
\usepackage{manyfoot}%
\usepackage{booktabs}%
\usepackage{algorithm}%
\usepackage{algorithmicx}%
\usepackage{algpseudocode}%
\usepackage{listings}%

\usepackage[margin=3cm]{geometry}
\usepackage{cite}
\usepackage{amscd}
\usepackage{dsfont}
\usepackage{float}
\usepackage{hyperref}
\hypersetup{colorlinks=true, urlcolor= black, linkcolor=black, citecolor=black}

\usepackage{bbm}
\usepackage{stmaryrd}
\usepackage{comment}
\usepackage{times}
\numberwithin{equation}{section}
\usepackage{tikz}
\usepackage[font=small]{caption}

\newtheorem{theorem}{Theorem}[section]

\newtheorem{lemma}[theorem]{Lemma}
\newtheorem{proposition}[theorem]{Proposition}
\newtheorem{corollary}[theorem]{Corollary}

\theoremstyle{definition}

\newtheorem{definition}[theorem]{Definition}
\newtheorem{remark}[theorem]{Remark}

\usepackage{mathtools}
\mathtoolsset{showonlyrefs}

\renewcommand{\epsilon}{\varepsilon}
\renewcommand{\phi}{\varphi}

\newcommand{\mathd}{{\mathrm d}}
\newcommand{\Iota}{{\mathrm I}}

\begin{document}

\title{The Markov property for $\varphi^4_3$ on the cylinder}

\author{Nikolay Barashkov
\footnote{Max Planck Institute for Mathematics in the Sciences,
Leipzig; \url{nikolay.barashkov@mis.mpg.de}} \:  and Trishen S. Gunaratnam\footnote{Tata Institute of Fundamental Research, Mumbai; \url{trishen@math.tifr.res.in}}\footnote{International Centre for Theoretical Sciences, Bengaluru; \url{trishen.gunaratnam@icts.res.in}}}

\maketitle
\begin{abstract}
We prove that the $\varphi^4_3$ model satisfies a version of Segal's axioms in the special case of three-dimensional tori and cylinders. As a consequence, we give the first proof that this model satisfies a Markov property and we characterize its boundary law up to absolutely continuous perturbations. In addition, we use Segal's axioms to give an alternative construction of the $\varphi^4_3$ Hamiltonian on two-dimensional tori as compared with Glimm (Comm. Math. Phys., 1968). We exploit this probabilistic approach to prove novel fundamental spectral properties of the Hamiltonian, such as discrete spectrum and a Perron-Froebenius type result on its ground state. The key technical contributions of this article are the development of tools to analyze $\varphi^4_3$ models with rough boundary conditions. We heavily use the variational approach to $\varphi^4_3$ models introduced in Barashkov and Gubinelli (Duke, 2020) that is based on the Bou\'e-Dupuis formula and dual to Polchinski's continuous renormalization group.
\end{abstract}


\tableofcontents

\section{Introduction}

The study of relativistic quantum fields through the lens of statistical physics was initiated in the '50s and '60s by Schwinger \cite{S58}, Nakano \cite{N59}, and Symanzik \cite{S66}. Mathematically, quantum field theory concerns the study of operator valued distributions acting on $d$-dimensional Minkowski space $\mathbb R^{1,d-1}$ (see, for example, Wightman's axioms \cite{W56}). By analytically continuing time $t$ into imaginary time $it$, $d$-dimensional Minkowski space $\mathbb R^{1,d-1}$ is mapped to the Euclidean space $\mathbb R^d$. Under this mapping, the Feynman path integral approach to quantum field theory \cite{F48} is transformed to the study of \emph{Gibbs probability measures} on Euclidean fields $\varphi:\mathbb R^d \rightarrow \mathbb R$. Informally, these are defined by their averages on observables $F(\varphi)$ via the path integral
\begin{equation} \label{eq: intro: path integral}
	\int_{\varphi: \mathbb R^d \rightarrow \mathbb R} F(\varphi) \exp(-S(\varphi)) D\varphi.
\end{equation}
Above, $D\varphi$ is the (non-existent) uniform measure on the space of fields $\varphi:\mathbb R^d\rightarrow \mathbb R$ and the functional $S(\varphi)$ is the called the action of the theory. The study of these probability measures that in addition allow one to reconstruct quantum fields in Minkowski space is called \emph{Euclidean field theory}.

The rigorous construction of path integrals of the form \eqref{eq: intro: path integral} is delicate. The notion of a uniform or Lebesgue measure on the infinite dimensional space of fields $\varphi: \mathbb R^d\rightarrow \mathbb R$ is not well-defined. We will exclusively be interested in the case where the action takes the form:
\begin{equation}
	S(\varphi) = \int_{\mathbb R^d} V(\varphi(x)) + \frac 12 |\nabla \varphi(x)|^2 \, dx,
\end{equation}
where $V$ is a potential function and $\nabla$ is the gradient on $\mathbb R^d$. In the case that $V\equiv 0$, the path integrals can be rigorously defined by using Gaussian measure theory \cite{J97,S07} and the resulting object is called the \emph{Gaussian Free Field} (GFF). It is natural to then construct the case when $V$ is non-zero by using the GFF as a base measure.  When $d=1$, the "GFF" reduces to Brownian motion and path integrals are well-defined for quite general potentials $V$ and are known as \emph{Feynman-Kac} formulae \cite{S13}. The construction of path integrals for general potentials $V$ is notoriously difficult in dimensions $d\geq 2$. This is because the GFF is almost surely \emph{not} a function, and thus one has to evaluate the potential at a \emph{random Schwartz distribution}. This is a serious issue when $V$ is nonlinear as e.g.\ multiplication of Schwartz distributions is ill-defined. Whilst quadratic potentials still largely fall in the framework of Gaussian measure theory, the next simplest example of a nonlinear potential ---  the \emph{quartic} potential ---  does not.

We are interested in Euclidean field theories with quartic potentials, known as $\varphi^4_d$ models ---  particularly the three-dimensional model, $\varphi^4_3$. The issue of the construction/destruction of $\varphi^4_d$ models on tori and in infinite volume is now more or less settled. Due to the fact that fourth powers of distributions are ill-defined, it is natural to first put an ultraviolet/small-scale cutoff into the model ---  for example, by mollifying the fields or using a lattice discretization ---  and then use a \emph{renormalization procedure} to cancel divergences that are produced when the cutoff is removed. Concretely, this amounts to putting \emph{counter-terms} in the action $S(\varphi)$ that are diverging as the cutoff is removed. The rigorous implementation of a renormalization procedure to construct $\varphi^4_2$ \cite{N66, GJ12, S15} and $\varphi^4_3$ \cite{GJ73, FO76} path integrals were major achievements of the constructive field theory programme in the '60s and '70s. We refer to the introduction of \cite{GH21} for further background and references. Let us also mention that in higher dimensions, there are triviality (i.e.\ destruction) results: see \cite{A82, F82} for the case of $d\geq 5$ and \cite{ADC21} for $d=4$.

 In his seminal paper \cite{N73a}, Nelson proved that the \emph{Markov property}, which is intimately tied with the study of Gibbs measures, is sufficient to recover a relativistic quantum theory in Minkowski space. Although straightforward to establish for lattice models, the Markov property in continuum is a much more delicate property. Examples where it is known include free fields \cite{N73b} and $\varphi^4_2$ fields \cite{N66, S15}. In the Markovian framework, the Hamiltonian can be described probabilistically by considering continuous versions of transfer matrices, that is semigroups generated by path integrals emanating from this boundary (technically speaking, path integrals on \emph{cobordisms}). Furthermore, it is related to Segal's elegant axiomatization of conformal field theories and other quantum fields \cite{S88} (see also \cite{A88}). For our purposes, this can be summarized as a \emph{gluing property} of certain objects associated to the (bulk) theory called \emph{amplitudes}, which in our setting will be weighted Laplace transforms, along boundaries via integration with respect to a \emph{boundary measure}. Segal's axioms have been proved rigorously for relatively few models. The examples we know of include free fields \cite{N73b}, $\varphi^4_2$ fields \cite{P07,L24}, and more recently Liouville conformal field theory \cite{GKRV21}, where it played an important role in establishing the conformal bootstrap. In all of these cases, the boundary measure corresponds to the natural boundary GFF.

 So far, establishing the Markov property for $\varphi^4_3$ has remained an open problem. Indeed, it was shown \cite{AL08} that any field describing the conditional boundary law of the $\varphi^4_3$ measure must be singular with respect to the boundary GFF. This is a significant obstacle to obtaining Nelson's reconstruction theorem and we do not know of any other cases of the Markov property for quantum fields being shown where the boundary law is not the boundary GFF, or variants thereof. This being said, the weaker property of \emph{reflection positivity} has been shown for $\varphi^4_3$ and by the Osterwalder-Schrader theorem \cite{OS73, OS75} it is sufficient to reconstruct a quantum theory. However, the key deficiency of this approach is that many of fundamental objects on the Minkowski side are defined implicitly and are hard to access (e.g.\ using probabilistic techniques). This is a significant obstacle to study quantum fields beyond their construction. Indeed, let us emphasize that the latter objects include the \emph{Hamiltonian} of the theory.

In this article, we show the Markov property for $\varphi^4_3$ models in the special case of (three-dimensional) cylinders and tori. We show that appropriately defined amplitudes associated to these models obey a gluing property in the spirit of Segal's framework. This is the first instance we know of where the boundary measure involved in the gluing is non-Gaussian and singular with respect to the boundary GFF. We use the gluing property to give a probabilistic construction the Hamiltonian of the $\varphi^4_3$ model on two-dimensional tori, giving an alternative approach to that of Glimm \cite{G68}. To the best of our knowledge, beyond lower boundedness of the spectrum \cite{GJ73}, no other properties of the $\varphi^4_3$ Hamiltonian have been established. We exploit the probabilistic approach and establish the following fundamental spectral properties of this infinite dimensional, highly singular, renormalized Schr\"odinger operator: discreteness of spectrum, the simplicity of the lowest eigenvalue, and strict positivity of the associated normalized eigenvector. These three themes of results are respectively formalized in Theorems \ref{theorem: markov}, \ref{thm: main}, and \ref{thm: hamiltonian} below. 

At the heart of our proofs is the development of techniques to understand the \emph{conditional boundary law} of $\varphi^4_3$ fields, up to absolutely continuous perturbations. Although the boundary law is singular with respect to the boundary GFF, we are able to obtain a precise characterization of this singular perturbation, including regularity information. We then develop techniques that allow us to construct $\varphi^4_3$ models with \emph{rough boundary conditions} drawn according to this conditional law. Let us mention that the construction of $\varphi^4_3$ models (on cylinders) with boundary conditions, not necessarily rough, appears to the best of our knowledge to be novel and we believe of independent interest. We build on the approach developed by \cite{BG20}, based on the Bou\'e-Dupuis variational representation of Laplace transforms \cite{BD98}, which is a dual representation of Polchinski's continuous renormalization group transformations \cite{P84} (see \cite{BGH23} for an exact description of the duality). Furthermore, we heavily use paracontrolled calculus tools developed in the context of analyzing singular stochastic partial differential equations \cite{GIP15}.

We now give precise statements of results.

\subsection{Construction of $\varphi^4_3$ models with rough boundary conditions}

Our first result, Theorem \ref{theorem: rough boundary conditions} below, concerns the construction of $\varphi^4_3$ models on cylinders with \emph{rough boundary conditions}. Since the uniform does not exist in infinite dimensions, we will begin by describing a more useful (Gaussian) reference measure. 

Fix $L > 0$ and let $M=[-L,L]\times \mathbb T^2$. Let $S'(M)$ denote the space of Schwartz distributions on $M$. The massive Gaussian free field (GFF) with Dirichlet boundary conditions is the random element of $S'(M)$ with law given by the centred Gaussian measure $\mu$ with covariance $(-\Delta + m^2)^{-1}$, where $\Delta$ is the Dirichlet Laplacian on $M$ and $m^2>0$ is the mass. We now introduce boundary conditions for the GFF. Let $S'(\mathbb T^2)$ denote the space of Schwartz distributions on $\mathbb T^2$. Let $H(\varphi_-,\varphi_+)$ be the unique solution to $(-\Delta + m^2)H(\varphi_-,\varphi_+) \equiv 0$ on $(-L,L)\times \mathbb T^2$ and with boundary conditions $\varphi_-$ on $\partial^- M:= \{-L\}\times\mathbb T^2$ and $\varphi_+$ on $\partial^+M := \{L\}\times\mathbb T^2$. The GFF with boundary conditions $(\varphi_-,\varphi_+)$ on $\partial M$ is the random variable $\varphi+H(\varphi_-,\varphi_+)$, where $\varphi \sim \mu$, i.e.\ we shift the mean of $\mu$ to be $H(\varphi_-,\varphi_+)$. We will denote the resulting measure by $\mu( \,  \cdot \mid \varphi_-,\varphi_+)$.

The $\varphi^4_3$ model on $M$ with boundary conditions will be obtained as a weak limit of regularized and \emph{renormalized} approximate measures, which we now describe. Fix $T\geq 0$. Given $\varphi \in S'(M)$ and $T \geq 0$, we write $\varphi_T:= \varphi \ast \rho_T$, where $(\rho_T)_{T \geq 0}$ is a family of mollifiers on $\mathbb T^2$  and the convolution is performed in the periodic direction of $M$. If $\varphi \sim \mu$, we may define the Wick powers of $\varphi_T$ as follows. For every $x,y \in M$, let $C^M_T(x,y):= \mathbb E_{\mu}[\varphi_T(x) \varphi_T(y)]$. Define
\begin{align}
:\varphi_T(x)^2: \, &:=\varphi_T(x)^2-C_T^M(x,x), \\  :\varphi_T(x)^4: \, &:= \varphi_T(x)^4 - 6C_T^M(x,x)\varphi_T(x)^2	+3C_T^M(x,x)^2.
\end{align}
The approximate potential is the measurable map $V_T: S'( M) \rightarrow \mathbb R$ defined by
\begin{equation}
V_{T}(\varphi):= \int_M : \varphi_T(x)^4: - \gamma_T^ M(x) :\varphi_T(x)^2:\, dx - \delta_T^ M,
\end{equation}
where $\gamma_T^ M: M \rightarrow \mathbb R$ and $\delta_T^M \in \mathbb R$ are the (second order) \emph{mass} and \emph{energy renormalizations}, respectively . Given $\varphi_-,\varphi_+ \in S'(\mathbb T^2)$, the approximate model with these boundary conditions is the probability measure $\nu_T(\, \cdot \mid \varphi_-,\varphi_+)$  on $S'(\mathcal M)$ defined by the density
\begin{equation}
\nu_T(d\varphi \mid \varphi_-,\varphi_+) = \frac{1}{z_T(\varphi_-,\varphi_+)} \exp(-V_T(\varphi_T)) \mu(d\varphi \mid \varphi_-,\varphi_+),
\end{equation}
where $z_T(\varphi_-,\varphi_+)$ is a normalization constant. 

In order to obtain tightness of the measures, we will need to impose regularity restrictions on the boundary conditions. For $s \in \mathbb R$, let $H^s(\mathbb T^2)$ denote the usual fractional Sobolev space of regularity $s$ on $\mathbb T^2$. For deterministic regularity boundary conditions, it is sufficient\footnote{We remark that we did not optimize the regularity in the deterministic boundary condition case.} to take  $s>1/2-\kappa$ for $\kappa>0$ sufficiently small. However, in order to establish the Markov property for the limiting measure, we will need to consider \emph{random rough} boundary conditions sampled according to the natural conditional law of $\varphi^4_3$ restricted to (translations of) $\mathbb T^2$. To this end, let us write $\mu^0$ to denote the centred Gaussian measure on $S'(\mathbb T^2)$ with covariance $(-\Delta_{\mathbb T^2}+m^2)^{-1/2}$. By the domain Markov property of the GFF \cite{S07}, $\mu^0$ is (equivalent to) the conditional law of the GFF restricted to translates of $\mathbb T^2$. We will say that a measure $\tilde \nu^0$ is an \emph{admissible boundary law} if $\tilde\nu^0$ is the law of random variables $W^0+Z^0$, where $W^0$ is either $0$ or distributed according to $\mu^0$, and where $Z^0$ is a random variable that is almost surely in $H^{1/2-\kappa}(\mathbb T^2)$ for every $\kappa >0$ sufficiently small (and with sufficient moments). Let us note that the regularity of $Z^0$ is not sufficiently high to apply the Girsanov theorem with respect to $\mu^0$, and thus $\tilde\nu^0$ is not necessarily absolutely continuous to $\mu^0$ ---  in fact, the interacting boundary measure considered below is indeed singular. See Definition \ref{defn: admissible law} for a more precise statement.  

We now state our first main result concerning the construction of $\varphi^4_3$ models on $M$ with boundary conditions sampled according to admissible boundary laws.

\begin{theorem} \label{theorem: rough boundary conditions}
Let $\varphi_-,\varphi_+$ be i.i.d sampled according to an admissible boundary law $\tilde \nu^0$. Then almost surely there exists a choice of sequence of renormalizations $(\gamma^M_T,\delta^M_T)_{T \geq 0}$ such the sequence $(\nu_T(\, \cdot \mid \varphi_-,\varphi_+))_{T \geq 0}$ converges weakly to a non-trivial, non-Gaussian probability measure $\nu(\, \cdot \mid \varphi-,\varphi_+)$ on $S'(M)$. Moreover, $\nu$ admits a non-trivial Laplace transform $\mathcal L(\, \cdot \mid \varphi_-,\varphi_+): C^\infty(M) \rightarrow \mathbb R$ which can be described by an explicit variational problem. 
\end{theorem}

The limiting measure $\nu(\, \cdot \mid \varphi_-,\varphi_+)$ is called the $\varphi^4_3$ model on $M$ with boundary conditions $(\varphi_-,\varphi_+)$ sampled according to $\tilde\nu^0$. We point out that full convergence, rather than just tightness, is established with the help of the explicit description of the limiting (unnormalized) Laplace transform, see Theorem \ref{thm: bulk}.  Furthermore, let us remark that in order to obtain tightness, it is necessary to choose the renormalizations $\gamma_T^M$ and $\delta_T^M$ to diverge as $T \rightarrow \infty$. In principle, we may then obtain a family of $\varphi^4_3$ models by shifting $\gamma_T^M$ by any bounded function. Moreover, due to loss of translation invariance in the non-periodic direction, $\gamma_T^M$ may indeed be a function rather than a constant, and this would cause issues when considering e.g. the Markov property. We show that one can choose $\gamma_T^M$ to be constant and canonically chosen to agree with the mass renormalization for the model on the torus, at least up to a finite correction, and thus in agreement with perturbation theory. On the other hand, shifts of $\delta_T^M$ will just affect the normalization constant and in fact $\delta^M_T$ is not necessary for the construction of $\nu$. It is required to ensure the convergence of the \emph{unnormalized Laplace transforms}, which are related to the amplitudes described below. Finally, let us remark that we could have introduced a coupling constant $\lambda > 0$ in front of the quartic term ---  all of the results of this article carry through in this case and therefore we set $\lambda =1$ without loss of generality.

The proof of Theorem \ref{theorem: rough boundary conditions} is closely related to the proof of Theorem \ref{thm: main} below, see Remark \ref{remark: phi43 rough bc proof strat}. Below, we will restate this theorem in a more precise way after introducing the relevant approximate models, see Theorem \ref{thm: phi43 random bc}.

\subsection{Gluing of amplitudes and the Markov property}

We now turn to our second main result, Theorem \ref{thm: main} below, concerning the gluing of $\varphi^4_3$ amplitudes. As a preliminary, let us observe that the $\varphi^4_3$ models constructed in Theorem \ref{thm: main} can easily be generalized to any three-dimensional cylinder. Intuitively, the gluing property tells us how to integrate out the boundary conditions of two $\varphi^4_3$ models on cylinders and obtain a model on the glued cylinder. The relevant objects are called amplitudes and they roughly correspond to \emph{unnormalized} Laplace transforms under the $\varphi^4_3$ measures. As we shall prove below (see the discussion after Theorem \ref{thm: hamiltonian}), the amplitudes are intimately related to transition probabilities of a natural $\varphi^4_3$ Markov process and, as such, generate a semigroup.

In order to define the amplitudes, a natural starting point is to consider the analogue of the gluing property of unnormalized Laplace transforms of $\varphi^4_3$ models in the prelimit. Let us write $M_- = [-L,0]\times\mathbb T^2$ and $M_+=[0,L]\times \mathbb T^2$, so that $M$ is obtained by gluing $M_-$ and $M_+$ along the boundary $B=\{0\}\times\mathbb T^2$. For notational convenience, we write $M_\sigma$ with $\sigma \in \{\emptyset, \pm \}$ to denote $M$ and $M_\pm$, respectively. Fix $T\geq 0$. Let $\mathcal L_T^\sigma( \, \cdot \mid \varphi_-,\varphi_+)$ denote the Laplace transform of $\nu_T^\sigma$. The unnormalized Laplace transform is the map $z_T^\sigma(\, \cdot \mid \varphi_-,\varphi_+):= z_T^\sigma(\varphi_-,\varphi_+)\mathcal L^\sigma_T(\,\cdot \mid \varphi_-,\varphi_+)$. Note that $z^\sigma_T(0\mid \varphi_-,\varphi_+)=z^\sigma_T(\varphi_-,\varphi_+)$. By the domain Markov property for the underlying GFF,
\begin{equation} \label{eq: intro: prelimit glue}
z_T(\varphi_-,\varphi_+) = C_M \mathbb E_{\mu^0}[ z_T^-(\varphi_-, {\Large \boldsymbol{\cdot} }\, ) z^+_T( {\Large \boldsymbol{\cdot} }\, , \varphi_+)],
\end{equation}
where $C_M>0$ is a constant related to the Dirichlet-to-Neumann map on $M$. Unfortunately, for generic fields sampled according to $\mu^0$, the normalization constants inside the expectation do not converge almost surely (even up to subsequence). This is due to a fourth (Wick) power of the (harmonic extension of the) boundary field, which is not well-defined in the limit $T\rightarrow \infty$.

A key step to obtaining convergence is to exponentially tilt the boundary Gaussian by the singular term. At finite $T\geq 0$, this yields a probability measure $\nu^0_T$ that is equivalent to $\mu^0$. We will thus end up with a new set of normalization constants $\mathcal Z_T^\sigma(\cdot,\cdot)$ and the gluing property reads
\begin{equation}
	\mathcal Z_T(\varphi_-,\varphi_+) = C_T(M,\varphi_-,\varphi_+) \mathbb E_{\nu^0_T}[\mathcal Z_T^-(\varphi_-,{\Large \boldsymbol{\cdot} }\,)\mathcal Z_T^+({\Large \boldsymbol{\cdot} }\,, \varphi_+)],
\end{equation}
where $C_M(\varphi_-,\varphi_+)$ is another constant that is induced\footnote{Let us also mention that in this constant is a term that allows us to remove the dependency of $\nu^0_T$ on the lengths of the cylinders considered.} by the tilt applied to $z_T(\varphi_-,\varphi_+)$.   We will be able to show that these new normalization constants ---  which are, up to a renormalization so that $C_T(M,\varphi_-,\varphi_+)=1$, the natural candidates for the $\varphi^4_3$ amplitudes ---  converge almost surely with respect to any admissible boundary law as $T \rightarrow \infty$. Furthermore, we will be able to construct the weak limit of $(\nu^0_T)_{T \geq 0}$ that we will denote by $\nu^0$. However, the limiting boundary law $\nu^0$ is \emph{singular} with respect to the boundary Gaussian $\mu^0$. Let us stress that this is a key difficulty that permeates throughout the whole paper and is at the heart of why the Markov property for $\varphi^4_3$ had remained unproven. As we shall describe in Section \ref{sec: intro: proof strat}, the construction of $\nu^0$ and the subsequent techniques we develop to control the divergences induced by the boundary measure $\nu^0$ in the amplitudes are the key innovations of the paper.

We will now turn to the statement of the gluing property of the $\varphi^4_3$ amplitudes. 
\begin{theorem}\label{thm: main}
There exists a probability measure $\nu^0$ on $S'(\mathbb T^2)$ and a family of amplitudes corresponding to measurable maps 
\begin{equation}
\mathcal A^\sigma({\Large \boldsymbol{\cdot} }\,\mid {\Large \boldsymbol{\cdot} }\,,{\Large \boldsymbol{\cdot} }\,): C^\infty(M^\sigma) \times S'(\mathbb T^2)\times S'(\mathbb T^2)\rightarrow \mathbb R
\end{equation}
such that the following are satisifed.
\begin{enumerate}
\item[(i)] The boundary measure $\nu^0$	is an admissible boundary law.
\item[(ii)] Let $\varphi_-,\varphi_+$ be i.i.d. sampled according to $\nu^0$. For every $f \in C^\infty(M^\sigma)$, $\mathcal A^\sigma(f \mid \varphi_-,\varphi_+)$ is, up to a multiplicative constant, the limit of unnormalized Laplace transforms $\mathcal Z_T(f \mid \varphi_-\ast \rho_T,\varphi_+\ast \rho_T)$ of the prelimiting $\varphi^4_3$ measures as $T \rightarrow \infty$. 
\item[(iii)] The amplitudes satisfy the gluing property. Almost surely for every $\varphi_-,\varphi_+$ independently sampled according to any admissible boundary law, and for every $f \in C^\infty(M)$,
\begin{equation}
\mathcal A(f\mid \varphi_-,\varphi_+) =  \mathbb E_{\nu^0} \left[ \mathcal A^{-}(f|_{M_-} \mid \varphi_-,\varphi^0) \,  \mathcal A^{+}(f|_{ M_+} \mid \varphi^0,\varphi_+) \right],
\end{equation}
where the field $\varphi^0\sim \nu^0$ is independent from $\varphi_-$ and $\varphi_+$, and $f|_{M_\pm}$ is the restriction of $f$ to $M_\pm$.
\end{enumerate}
\end{theorem}

\begin{figure}[htbp]

\tikzset{every picture/.style={line width=0.75pt}} 

\begin{tikzpicture}[x=0.68pt,y=0.68pt,yscale=-1,xscale=1]

\draw  [fill={rgb, 255:red, 155; green, 155; blue, 155 }  ,fill opacity=1 ] (99.14,203.33) .. controls (99.14,188.61) and (107.82,176.67) .. (118.53,176.67) .. controls (129.23,176.67) and (137.91,188.61) .. (137.91,203.33) .. controls (137.91,218.06) and (129.23,230) .. (118.53,230) .. controls (107.82,230) and (99.14,218.06) .. (99.14,203.33) -- cycle ;
\draw    (49.39,176.67) -- (118.53,176.67) ;
\draw    (49.39,230) -- (118.53,230) ;
\draw  [draw opacity=0] (49.39,230) .. controls (49.39,230) and (49.39,230) .. (49.39,230) .. controls (38.68,230) and (30,218.06) .. (30,203.33) .. controls (30,188.61) and (38.68,176.67) .. (49.39,176.67) -- (49.39,203.33) -- cycle ; \draw   (49.39,230) .. controls (49.39,230) and (49.39,230) .. (49.39,230) .. controls (38.68,230) and (30,218.06) .. (30,203.33) .. controls (30,188.61) and (38.68,176.67) .. (49.39,176.67) ;  

\draw  [fill={rgb, 255:red, 155; green, 155; blue, 155 }  ,fill opacity=1 ] (253.76,203.33) .. controls (253.76,218.06) and (245.08,230) .. (234.37,230) .. controls (223.67,230) and (214.99,218.06) .. (214.99,203.33) .. controls (214.99,188.61) and (223.67,176.67) .. (234.37,176.67) .. controls (245.08,176.67) and (253.76,188.61) .. (253.76,203.33) -- cycle ;
\draw    (303.51,230) -- (234.37,230) ;
\draw    (303.51,176.67) -- (234.37,176.67) ;
\draw  [draw opacity=0] (303.51,176.67) .. controls (303.51,176.67) and (303.51,176.67) .. (303.51,176.67) .. controls (314.22,176.67) and (322.9,188.61) .. (322.9,203.33) .. controls (322.9,218.06) and (314.22,230) .. (303.51,230) .. controls (303.51,230) and (303.51,230) .. (303.51,230) -- (303.51,203.33) -- cycle ; \draw   (303.51,176.67) .. controls (303.51,176.67) and (303.51,176.67) .. (303.51,176.67) .. controls (314.22,176.67) and (322.9,188.61) .. (322.9,203.33) .. controls (322.9,218.06) and (314.22,230) .. (303.51,230) .. controls (303.51,230) and (303.51,230) .. (303.51,230) ;  

\draw    (170,180) .. controls (168.63,188.23) and (170.91,199.79) .. (141.82,200) ;
\draw [shift={(140,200)}, rotate = 0.49] [color={rgb, 255:red, 0; green, 0; blue, 0 }  ][line width=0.75]    (10.93,-3.29) .. controls (6.95,-1.4) and (3.31,-0.3) .. (0,0) .. controls (3.31,0.3) and (6.95,1.4) .. (10.93,3.29)   ;
\draw   (530.28,203.33) .. controls (530.28,188.61) and (538.96,176.67) .. (549.67,176.67) .. controls (560.37,176.67) and (569.05,188.61) .. (569.05,203.33) .. controls (569.05,218.06) and (560.37,230) .. (549.67,230) .. controls (538.96,230) and (530.28,218.06) .. (530.28,203.33) -- cycle ;
\draw    (449.39,176.67) -- (549.67,176.67) ;
\draw    (449.39,230) -- (549.67,230) ;
\draw  [draw opacity=0] (449.39,230) .. controls (449.39,230) and (449.39,230) .. (449.39,230) .. controls (438.68,230) and (430,218.06) .. (430,203.33) .. controls (430,188.61) and (438.68,176.67) .. (449.39,176.67) .. controls (449.39,176.67) and (449.39,176.67) .. (449.39,176.67) -- (449.39,203.33) -- cycle ; \draw   (449.39,230) .. controls (449.39,230) and (449.39,230) .. (449.39,230) .. controls (438.68,230) and (430,218.06) .. (430,203.33) .. controls (430,188.61) and (438.68,176.67) .. (449.39,176.67) .. controls (449.39,176.67) and (449.39,176.67) .. (449.39,176.67) ;  
\draw    (181,180) .. controls (180.8,188.75) and (178.75,199.94) .. (209.1,200.01) ;
\draw [shift={(211,200)}, rotate = 179.29] [color={rgb, 255:red, 0; green, 0; blue, 0 }  ][line width=0.75]    (10.93,-3.29) .. controls (6.95,-1.4) and (3.31,-0.3) .. (0,0) .. controls (3.31,0.3) and (6.95,1.4) .. (10.93,3.29)   ;
\draw    (330,200) .. controls (367.76,173.73) and (388.38,174.63) .. (418.61,198.88) ;
\draw [shift={(420,200)}, rotate = 219.26] [color={rgb, 255:red, 0; green, 0; blue, 0 }  ][line width=0.75]    (10.93,-3.29) .. controls (6.95,-1.4) and (3.31,-0.3) .. (0,0) .. controls (3.31,0.3) and (6.95,1.4) .. (10.93,3.29)   ;

\draw (171,153.4) node [anchor=north west][inner sep=0.75pt]    {$\varphi ^{0}$};
\draw (25,232.4) node [anchor=north west][inner sep=0.75pt]    {$\mathcal{A}^{-}\left( f\mid \varphi _{-} ,\varphi ^{0} \ \right)$};
\draw (214,232.4) node [anchor=north west][inner sep=0.75pt]    {${\textstyle \mathcal{A}^{+}\left( f\mid \varphi ^{0\ } ,\varphi _{+} \ \right)}$};
\draw (441,232.4) node [anchor=north west][inner sep=0.75pt]    {${\textstyle \mathcal{A}( f\mid \varphi _{-} ,\varphi _{+} \ )}$};
\draw (41,192.4) node [anchor=north west][inner sep=0.75pt]    {$M^{-}$};
\draw (292,192.4) node [anchor=north west][inner sep=0.75pt]    {$M^{+}$};
\draw (452,192.4) node [anchor=north west][inner sep=0.75pt]    {$M$};
\draw (371,152.4) node [anchor=north west][inner sep=0.75pt]    {$\nu ^{0}$};

\end{tikzpicture}
  
    \caption{Pictorial representation of Segal's gluing property as stated in Theorem \ref{thm: main}. Amplitudes on $M^-$ and $M^+$ are glued by integrating the boundary field $\varphi^0$ according to the interacting boundary measure $\nu^0$ to produce the amplitude on $M$.}
    \label{fig:segal-gluing}
\end{figure}

We defer a discussion of the proof of Theorem \ref{thm: main} to Section \ref{sec: intro: proof strat}, where it is explained in detail along with precise section references. Let us however remark that condition (ii) above, although not directly relevant for the gluing (which involves the convergence of the amplitudes with boundary values sampled according to $\nu^0_T$), implies that the limiting $\varphi^4_3$ amplitudes are non-trivial.

We now turn to some implications of Theorem \ref{thm: main} and the techniques we use to prove it. First of all, let us mention that the restriction to the cylinders $M,M_-,M_+$ is artificial and clearly we can consider any three-dimensional cylinder. The proof techniques also extend to the case of three-dimensional tori. Let us write $\mathcal M$ to denote the three-dimensional torus obtained from gluing the boundaries of $M$ together. The amplitude $\mathcal M$ is a measurable map $\mathcal A^{\rm per}:C^\infty(\mathcal M) \rightarrow \mathbb R$ that is proportional to the Laplace transform of the $\varphi^4_3$ model constructed in \cite{BG20}. The constant of proportionality is related to the change of covariance in the underlying free fields and is inessential to the rest of the paper, so we will not make it explicit.  

\begin{corollary}\label{cor: tori gluing}
For every $f \in C^\infty(\mathcal M) \subset C^\infty(M)$,
\begin{equation}
\mathcal A^{\rm per}(f) = \mathbb E_{\nu^0}[ \mathcal A(f \mid \varphi^0,\varphi^0)].
\end{equation}
\end{corollary}

We now turn our attention towards establishing the Markov property for $\varphi^4_3$ models in the special case of cylinders and tori. For simplicity we will just state it for the cylinder. Given $O \subset M$, we say that a function $f:S'(M) \rightarrow \mathbb R$ is $O$-measurable if it is measurable with respect to the $\sigma$-algebra generated by $\{ \varphi(g) : \varphi \in S'(M), g \in C^\infty(M),  {\rm supp}(g) \subset O \}$. Let now $\ell \in (0,L)$. For clarity, we will write $\nu^L$ and $\nu^{\ell}$ to denote the $\varphi^4_3$ measures on $M$ and $M_\ell$, respectively. 

\begin{theorem}\label{theorem: markov}
For every $f: S'(M) \rightarrow \mathbb R$ bounded and $M_\ell$-measurable,  
\begin{equation}
\mathbb E_{\nu}[f \mid \mathcal F_{M_\ell^c}](\varphi^c) = \mathbb E_{\nu^\ell({\Large \boldsymbol{\cdot} }\,\mid \varphi_-,\varphi_+)}[f],
\end{equation}
where $\varphi_-,\varphi_+$ are the restrictions of $\varphi^c$ to the components of the boundary $\partial M_\ell$ which exists $\nu$-almost surely and are i.i.d. according to $\nu^0$ .
\end{theorem}

\subsection{The $\varphi^4_3$ Hamiltonian on two-dimensional tori}

We now turn to our final main result concerning the construction and spectral properties of the $\varphi^4_3$ Hamiltonian on $\mathbb T^2$, stated below in Theorem \ref{thm: hamiltonian}. The Hilbert space of the Hamiltonian will be a dense subset of $\mathcal H:= L^2(\nu^0)$. As hinted to above, this is the reason why we chose $(\nu^0_T)_{T \geq 0}$, and thus $\nu^0$, to not depend on the lengths of the ambient cylinders. 

We will define the $\varphi^4_3$ Hamiltonian as the infinitesimal generator of $(\mathbf P_\tau)_{\tau \geq 0}$. This semigroup is built from the amplitudes, as we now explain. Let us write $M_\tau:= [0,\tau]\times \mathbb T^2$ and denote the corresponding $\varphi^4_3$ amplitudes by $\mathcal A_\tau$. In this subsection, we will only consider the amplitudes evaluated at $f=0$ and thus we will simply write e.g.\ $\mathcal A_\tau(\varphi_-,\varphi_+)$. For every $\tau > 0$, we may define a linear operator $\mathbf P_\tau $ with kernel $\mathcal A_\tau$. Explicitly its action on sufficiently regular functions $F:\mathcal H\rightarrow \mathbb R$ is given by
\begin{equation}
\mathbf P_\tau F(\varphi):= \mathbb E_{\nu^0}[F({\Large \boldsymbol{\cdot} })\mathcal A_\tau ({\Large \boldsymbol{\cdot} }, \varphi)], \qquad \forall \varphi \in \mathcal H.
\end{equation}
The fact that the right-hand side is well-defined for e.g.\ bounded functions is established in Section \ref{sec:spectral} and uses the gluing property of Theorem \ref{thm: main} and Corollary \ref{cor: tori gluing}, as well as bounds on the torus partition functions from \cite{BG20}. Furthermore, we also show that $\mathbf P_\tau$ extends to a bounded linear operator on $\mathcal H$.

The following theorem is our main result on the $\varphi^4_3$ Hamiltonian on $\mathbb T^2$.
\begin{theorem} \label{thm: hamiltonian}
The semigroup $(\mathbf P_\tau)_{\tau}$ admits a non-trivial infinitesimal generator given by an unbounded, self-adjoint, linear map $\mathbf H$ densely defined on some domain $D(\mathbf H) \subset \mathcal H$. The spectrum of $\mathbf H$, denoted by $\sigma(\mathbf H)$, is discrete and bounded from below. The lowest eigenvalue $E_0 \in \mathbb R$ is simple and its associated normalized eigenvector is strictly positive almost surely with respect to $\nu^0$. Furthermore, every eigenvector is in the Orlicz\footnote{Recall that the Orlicz space $L^2 (\log_+L)^\alpha$ is the Banach space of functions with integral $\int |f|^2 (\log_+|f|)^\alpha d\nu^0 < \infty$. See Section \ref{subsec: tightness ground state} for further detail.} space $L^2 (\log_+ L)^\alpha(\nu^0)$ for every $\alpha \geq 1$.  
\end{theorem}

The proof of Theorem \ref{thm: hamiltonian} is contained in Section \ref{sec:spectral} and splits into two parts: i) construction, and ii) spectral properties. See Section \ref{subsec: proof of Theorem Hamiltonian}. Given the construction, we are able to establish discrete spectrum and the Perron-Froebenius type result on the lowest eigenvalue/eigenvector by similar arguments as in \cite{GGV24}. The $L^2 (\log_+ L)^\alpha$ bounds on the eigenvectors follow by bounds on the (periodic) amplitudes, see Lemma \ref{lemma: eigenvector bounds} and Proposition \ref{prop: eigenvector bounds limiting}. We remark that our techniques could prove  bounds stronger bounds, replacing $L^2(\log_+L)^\alpha$ by any Orlicz space $L^2 {\psi}$, with $\psi(t) = O( \exp( (\log_+ t)^{\alpha'} )$ where $\alpha'$ is sufficiently small, thus just falling short of $L^p$ bounds for $p>2$. 

We now turn to the construction. Our ultimate goal here is to apply the Hille-Yosida theorem to $(\mathbf P_\tau)_{t\geq 0}$ and thereby construct its generator. In order to do this, we need to verify that $(\mathbf P_\tau)_{\tau \geq 0}$ is a strongly continuous semigroup. Whilst the semigroup property follows from the gluing property of Theorem \ref{thm: main}, the strong continuity is the key difficulty of the construction ---  this is a manifestation of the fact that the boundary measure is singular with respect to the boundary Gaussian.  For every $\tau \geq 0$, we define approximate semigroups $(\mathbf P_\tau^T)_{t \geq 0}$ which act on \emph{different} Hilbert spaces $\mathcal H_T:= L^2(\nu^0_T)$ by using the approximate amplitudes as kernels. By standard Schr\"odinger operator theory, these semigroups are infinitesimally generated by approximate Hamiltonians $\mathbf H^T$. By a compact resolvent argument as in \cite{GGV24}, these have discrete spectrum and one can further show that the lowest eigenvalue $E_0^T$ is simple and its associated eigenvector $e_0^T$ can be taken strictly positive almost surely with respect to $\nu^0_T$. For every $T \geq 0$, we now perform a \emph{ground state transform} of the Hilbert space and the semigroup. (See \cite{LHB11} for applications of the ground state transform in the context of $P(\varphi)_1$ measures and processes.) For the Hilbert space $\mathcal H_T$, this amounts to replacing the probability measure $\nu^0_T$ by the probability measure with density $e_0^T(\varphi)^2 \nu^0_T(d\varphi)$. For the semigroup, it amounts to dividing the kernel by $e_0^T(\varphi)e_0^T(\varphi')$ and multiplying by a constant involving $E_0^T$. The new semigroups $(\overline {\mathbf P}_\tau^T)_{\tau \geq 0}$ have the advantage of being \emph{Markovian} and as such give rise to a family of Markov processes, for which we are able to prove tightness in a suitable space.

Let us now explain how to go from tightness of the Markov processes to obtainining strong continuity of the original semigroup $(\mathbf P_\tau)_{\tau \geq 0}$. We work in the grand coupling of all the boundary measures $(\nu^0_T)_{T \geq 0}$ and $\nu^0$ ---  this is explained more in Section \ref{sec: intro: proof strat} and rigorously constructed in Section \ref{sec: boundary}. We prove that the eigenvalues $(E_0^T)_{T \geq 0}$ are uniformly bounded and that the eigenvectors $(e_0^T)_{T \geq 0}$ are precompact in $L^2$ of the coupling. Given \emph{any}\footnote{Although we do not prove it, it is reasonable to think that the limit point is the ground state eigenvector of the limiting Hamiltonian.} limit point $\tilde e_0$, we show that on a suitable dense subset, $\mathbf P_\tau$ is approximated well by $\mathbf P_\tau^T$, which in turn is related by an isometric transformation to $\overline{\mathbf P}_\tau^T$. We then use uniform bounds on $\overline{\mathbf P}_\tau^T$ to deduce strong continuity of $(\mathbf P_\tau)_{\tau \geq 0}$. Let us emphasize that the construction of the dense set, which is defined in terms of limit point $\tilde e_0$, requires the $\nu^0$-almost sure positivity of $\tilde e_0$. 

\begin{remark}
Let us mention that, by adapting techniques used in \cite{GGV24} (in turn inspired by those in the context of Gibbs measures relative to Brownian motion, see \cite{LHB11} and references therein), we believe it to be possible to use Theorem \ref{thm: hamiltonian} to construct $\varphi^4_3$ measures on the infinite cylinder $\mathbb R_+\times \mathbb T^2$.	Furthermore, one can also show exponential decay of correlations for this infinite cylinder measure (for \emph{any} coupling constant $\lambda > 0$).
\end{remark}

\subsection{Strategy of proof of Theorem \ref{thm: main}} \label{sec: intro: proof strat}

Let $T>0$ and let $\varphi_-,\varphi_+ \in S'(\mathbb T^2)$. For simplicity, we will explain the argument for the amplitudes evaluated at $f=0$, i.e.\ the normalization constants. Our starting point is the gluing formula in the prelimit given by \eqref{eq: intro: prelimit glue} for the normalisation constant $z_T(\varphi_-,\varphi_+)$. We will now consider the factors inside the $\mu^0$-expectation on the right-hand side of \eqref{eq: intro: prelimit glue}. Without loss of generality, consider $z_T^-(\varphi_-,\varphi^0)$ for $\varphi^0 \sim \mu^0$. Letting $V_T^-$ denote the potential on $M_-$ and $\mu^-$ the GFF on $M_-$, we have that
\begin{equation}
z_T^-(\varphi_-,\varphi^0)= \mathbb E_{\mu^-}\left[\exp\left(-V_T^-(\varphi_T+  H(\varphi_-,\varphi^0)_T)\right)\right],
\end{equation}
where we write $H(\varphi_-,\varphi^0)_T = \rho_T \ast H(\varphi_-,\varphi^0)$. The natural thing that we would want to do is to control the right-hand side almost surely in $\varphi^0$. However, expanding the potential (and ignoring Wick renormalization for now), we will see a term corresponding to
\begin{equation}
	\int_{M_-} (H(\varphi_-,\varphi^0)_T(x))^4 \, dx. 
\end{equation}
The variance of this term diverges logarithmically as $T \rightarrow \infty$, even when Wick renormalized (with respect to the boundary measure $\mu^0$). Therefore, even with further renormalization, it is not clear how to control it. 

The strategy that we adopt to handle this divergence is to \emph{not} view this term as a random variable, but rather to absorb it into the density of the boundary measure $\mu^0$. This formally gives an \emph{interacting boundary measure} $\nu^0_T$ defined on $S'(\mathbb T^2)$ by the density
\begin{equation} \label{eq: intro: density bdry}
\nu^0_T(d\varphi^0) = \frac{1}{\mathcal Z_T^0} \exp\left( - \int_{M_-} :H(\varphi_-,\varphi^0)_T(x)^4:\, dx - \frac 12 \delta^0_T  \right) \mu^0(d\varphi^0),
\end{equation}
where $::$ corresponds to Wick ordering with respect to $\mu^0$, $\delta^0_T \in \mathbb R$ is a boundary energy renormalization\footnote{The factor $1/2$ comes from splitting this term over $M_-$ and $M_+$.}, and $\mathcal Z^0_T$ is a normalization constant. This definition allows us to redefine the amplitudes without this (divergent) fourth order Harmonic extension term in the potential. They correspond to (generalized) Laplace transforms of $\varphi^4_3$ measures with additional bulk-boundary terms, i.e.\ integrals involving both the bulk fields $\varphi_T$ and the harmonic extension terms.

The proof of Theorem \ref{thm: main} now splits into two major parts.
\begin{enumerate}
\item[(I)] The construction and properties of the limiting boundary theory $\nu^0$. This is the content of Theorem \ref{thm: boundary} and is proven in Section \ref{sec: boundary}. 
\item[(II)] The construction of the amplitudes ---  in particular, handling the boundary-bulk contributions ---  and almost sure bounds in the boundary field. This is the content of Theorems \ref{thm: bulk} and \ref{thm: enhancement converge}. Their proofs span Sections \ref{sec: renormalization of bulk amplitudes} - \ref{sec: smallfield reduction}. 
\end{enumerate} 
Using parts (I) and (II), the proof of Theorem \ref{thm: main} then follows from considering the gluing formula in the prelimit and using a truncation argument to justify taking limits. See Section \ref{subsec: proof of main}.

\begin{remark} \label{remark: phi43 rough bc proof strat}
Since the amplitudes of Theorem \ref{thm: main} are intimately related to Laplace transforms of the $\varphi^4_3$ field with boundary conditions, the proof of Theorem \ref{theorem: rough boundary conditions} follows from the construction of the amplitudes in Theorem \ref{thm: bulk}, i.e.\ part (II) above. See also Theorem \ref{thm: phi43 random bc}.	
\end{remark}

\subsubsection{Part I: Boundary theory}

We now briefly describe the strategy to show convergence of the interacting boundary measures $\nu^0_T$ to a unique limiting boundary measure $\nu^0$ and that $\nu^0$ is an admissible boundary law. We prove that the unnormalized Laplace transforms of the boundary field $\mathcal Z^0_T(f)$ converge to a limiting object $\mathcal Z^0_\infty(f)$ and that the latter is sufficient to ensure weak convergence to a limiting probability measure $\nu^0$.

Following the variational approach to constructing $\varphi^4$ models initiated in \cite{BG20}, we express $\mathcal Z^0_T(f)$ as an expectation of a well-behaved random variable of a sequence of i.i.d. complex Brownian motions. This just follows by representing the boundary Gaussian $\mu^0$ in terms of a random Fourier series (also known as a Karhunen-Lo\`eve expansion) and representing these Gaussian random variables as stochastic integrals. The upshot of representing the expectation like this is that it allows us to apply tools used to analyze continuous-time martingales to our setting ---  in particular, Girsanov's theorem. By the Bou\'e-Dupuis theorem, we are thus able to represent
\begin{equation}
	-\log \mathcal Z^0_T(f) = \inf_{v \in \mathbb H^0} \mathbb F^0_T(v).
\end{equation}
Above, $\mathbb H^0$ is some Banach space of stochastic processes that are called \emph{drifts}. The cost function $\mathbb F^0_T: \mathbb H^0 \rightarrow \mathbb R$ is an explicit functional which is of the form
\begin{equation}
\mathbb F^0_T(v) = \mathbb E[\Phi^0_T(W^0,v)+ G^0_T(v)],
\end{equation}
where $\mathbb E$ is expectation with respect to Wiener measure $\mathbb P$, ${\rm Law}_{\mathbb P}(W^0)=\mu^0$, and $G^0_T(v)$ is a \emph{good}/\emph{coercive} term which is positive and is the sum of two terms corresponding to a \emph{potential term} and a \emph{drift entropy} term.
 
In order to prove convergence of the $\mathcal Z^0_T$, we therefore are required to prove that the infima of the cost functions $\mathbb F^0_T$ converge (e.g.\ to infima of a limiting cost function).
Due to the divergence of the quartic term in \eqref{eq: intro: density bdry}, the resulting variational problem contains a singular term in $\Phi^0_T(W^0,v)$ in the limit $T\rightarrow \infty$. As in \cite{BG20}, we overcome this divergence by postulating an ansatz on the drifts $v$. This ansatz implies that, to leading order, $v$ is described on small scales by a term that is singular in the limit $T \rightarrow \infty$. This induces a divergence in the drift entropy term which exactly cancels the singular integral up to a divergent constant, which can then be renormalized by $\delta^0_T$. This results in a \emph{renormalized} variational problem
\begin{equation}
	-\log\mathcal Z^0_T(f) = \inf_{v \in \mathbb L^0} \mathbb F^{0, {\rm ren}}_T(v),
\end{equation}
where $\mathbb L^0$ is the space of drifts that are compatible with the ansatz and $\mathbb F^{0, {\rm ren}}_T$ is a renormalized cost function of the same form as $\mathbb F^0_T$ but with $\Phi^0_T(W^0,v)$ replaced by a renormalized term, $\Phi^{0,{\rm ren}}_T(W^0,v)$, and with a modified entropy term in $G^0_T(v)$. We remark here that, unlike in the case of $\varphi^4_3$ which requires a \emph{paracontrolled ansatz}, a relatively simple shift suffices for the ansatz ---  this is due to the regularizing properties of the harmonic extension and explains why mass renormalization is not needed to construct the measure $\nu^0$. In this sense, the measure sits somewhere in between $\varphi^4_2$ and $\varphi^4_3$ in terms of constructability.

We then show that there is a limiting functional $\mathbb F^0_\infty$ on this space such that the infima $\mathbb F^{0, {\rm ren}}_T$ converge to the infima $\mathbb F^0_\infty$, which is sufficient to prove weak convergence of $(\nu^0_T)_{T \geq 0}$ to a limiting probability measure $\nu^0$. We do this by showing that the functionals $\Gamma$-converge to the limit and use the fundamental theorem of $\Gamma$-convergence to ensure convergence of the infima ---  this requires some uniform equicoercivity of the cost functions. In order to obtain the uniform equicoercivity, we need to be able to control the term $\Phi^{0,{\rm ren}}_T(W^0,v)$ by the modified good (coercive) terms uniformly in $T$. In order to do this, we have to develop suitable regularity estimates on Wick powers of harmonic extensions of $W^0$. See Proposition \ref{prop: stochastic boundary v2}.  

Finally, we are left to show that $\nu^0$ is an admissible boundary law. Heuristically, this follows by establishing that there is in fact a minimizer of the limiting functional $\mathbb F^0_\infty$, denoted by $v^*$, and that $\nu^0 = {\rm Law}_{\mathbb P}(W^0+I(v^*))$, where $I$ is some deterministic regularizing map. We then show that $I(v^*)$ satisfies appropriate regularity and moment estimates. The arguments for this part follow closely those in \cite{BG23}.
\subsubsection{Part II: Bulk amplitudes}

We now turn to the convergence and estimates on the bulk amplitudes. Since the amplitudes will be proportional to unnormalized Laplace transforms associated to $\varphi^4_3$ with rough boundary conditions, we will discuss the latter. For simplicity, we will consider $f=0$. The boundary conditions will move as $T \rightarrow \infty$ since we have to sample them according\footnote{We could of course replace $\nu^0_T$ by some tight sequence of admissible boundary laws.} to $\nu^0_T$ in order to establish gluing. For every $T \geq 0$, let $\varphi_{-,T}, \varphi_{+,T} \sim \nu^0_T$. The unnormalized Laplace transforms we consider are denoted $\mathcal Z_T(\varphi_{-,T},\varphi_{+,T})$. Concretely, they are given by the expectation
\begin{equation}
\mathcal Z_T(\varphi_-,\varphi_+) = \mathbb E_{\mu}[\exp(-V(\varphi_T \mid \gamma_T, \delta_T-\delta^0_T) + U_T(\varphi_T \mid \varphi_{-,T},\varphi_{+,T}))],
\end{equation} 
where $V(\varphi_T \mid \gamma_T, \delta_T-\delta^0_T)$ is the $\varphi^4$ (bulk) potential with mass renormalization $\gamma_T$ and energy renormalization $\delta_T-\delta^0_T$, and $U_T(\varphi_T \mid \varphi_-,\varphi_+)$ is a potential term that involves boundary-bulk contributions (but not the quartic term in the harmonic extension since that was absorbed into the boundary field).

As before, the starting point is the Bou\'e-Dupuis representation of the amplitudes/rough boundary normalization constants. For every $T\geq 0$, almost surely with respect to $\nu^0_T\otimes\nu^0_T$, we have that
\begin{equation}
	-\log \mathcal Z_T(\varphi_{-,T},\varphi_{+,T}) = \inf_{v \in \mathbb H} \mathbb F_T(v),
\end{equation}
where $\mathbb H$ is some Banach space of drifts (different from the space considered in the boundary theory) and $\mathbb F_T(v)$ is a cost function that takes the form
\begin{equation}
	\mathbb F_T(v)= \mathbb E[\Phi_T(v \mid \varphi_{-,T}, \varphi_{+,T}) + G_T(v)],
\end{equation}
and as before $\Phi_T(v\mid \varphi_{-,T}, \varphi_{+,T})$ is an error term and $G_T(v)$ is a good term (potential plus drift entropy). The global picture of the argument is the same as in the boundary case, namely establish convergence of infima to the infima of a limiting functional. The estimates on the amplitudes come from uniform estimates on the amplitudes in the prelimit, obtained by estimating the error term in terms of the good term.

The significant difference here as compared with the boundary case or $\varphi^4$ models on the torus is that the error term splits into a \emph{bulk error}, $\Phi_T^{\rm bulk}(v)$, and a \emph{boundary-bulk error}, $\Phi_T^{\partial}(v \mid \varphi_{-,T}, \varphi_{+,T})$. There are divergences in both the bulk and the boundary-bulk error terms. In order to handle the bulk divergences, we perform a paracontrolled ansatz on the drift as in the case of the $\varphi^4_3$ model on the torus $\mathbb T^3$, see \cite{BG20}. One catch here is to show that the \emph{residual} energy renormalization $\delta_T -\delta^0_T$, where $\delta^0_T$ is the boundary energy renormalization used to construct $\nu^0$, suffices to cancel any divergent constants coming from \emph{both} the bulk and the boundary-bulk error terms. 

The main obstacles for us are in the boundary-bulk error terms. The boundary terms are too rough to be treated by purely deterministic arguments. First of all, there are divergent terms that need to be renormalized by a further paracontrolled ansatz involving boundary field terms. In order to control the error terms induced by the ansatz, we cannot argue purely analytically. Therefore we have to establish more refined stochastic regularity estimates on the resulting renormalized products of (harmonic extensions) of the boundary fields. As a result, we get that some of these contributions can be bounded almost surely by random variables that are measurable with respect to $\nu^0\otimes\nu^0$. Any remaining terms are then bounded by analytic arguments that leveraging both the regularizing properties of the harmonic extension and also the underlying Dirichlet boundary conditions for the bulk free field, and in turn the drifts.

\subsection{Open problems}

\begin{enumerate}
	\item[(i)] \textbf{Segal's axioms.} It would be interesting to extend Segal's gluing formula for amplitudes to the case of general manifolds and thereby establish (a version of) the full set of Segal's axioms for $\varphi^4_3$. This would require developing tools to handle rough boundary conditions on more general geometries ---  see \cite{Ba23, HS23} for results on $\varphi^4$ models without boundary conditions. Furthermore, it would be interesting to see whether the techniques we have developed can be applied to establish Segal's axioms for other quantum fields where the boundary measure may be singular with respect to the underlying Gaussian. Let us mention the Sine-Gordon model for certain ranges of parameters (see \cite{GM24} for the massive case).
	\item[(ii)] \textbf{Full space Markov and DLR conditions.}  It is natural to ask whether the Markov property can be extended to more general domains for finite volume measures or infinite volume measures. This is a step towards characterizing infinite volume $\varphi^4_3$ measures by DLR conditions \cite{D70, LR69}. In this article, our methods can be applied to analyze finite volume specifications on cylinders and we expect it to extend to arbitrary domains. Let us also mention that for Euclidean fields, a stronger Markov property where one can additionally condition on infinite volume subsets also holds in some cases. Such a property would be interesting to prove for $\varphi^4_3$ and would allow Nelson's reconstruction theorem to apply in the full space directly. See \cite{AK79a, AK79b, Z84, AK88}. Finally, in this spirit, it would be interesting to analyze the effect of boundary conditions on Gibbs measures. For example, the construction of non-translation invariant measures analogous to Dobrushin states for the 3d Ising model \cite{D73}. 
	\item[(iii)] \textbf{Further analysis of the Hamiltonian.} 
	First, a natural follow-up is to show the equivalence of our construction of the $\varphi^4_3$ Hamiltonian on $\mathbb T^2$ and Glimm's construction \cite{G68}. Moreover, we construct a heat semigroup associated to this model and a second follow-up would be to investigate smoothness properties of its kernel. Finally, it would be nice to analyze the spectral properties of the operator in the infinite volume limit and its interplay with the presence of a phase transition as in the 2d case, see \cite{GJS07, GJ71, SZ76, BDW25}. 
	\item[(iv)] \textbf{Extension to the fermionic setting.} It would be interesting to study markovianity and gluing formulae for fermionic models. See \cite{D24, ABDVG22, DVFG22} for fermionic models that have been recently treated using probabilistic approaches and see \cite{S75, CHP25} for coupled boson-fermion models. Let us also mention \cite{T17} which establishes Segal's axioms for free fermions, though not from a probabilistic perspective.
\end{enumerate}

\subsection{Paper organization}
\begin{itemize}
\item[-] In Section \ref{sec: gaussian}, we will introduce the Gaussian fields that we work with in this paper. 

\item[-] In Section \ref{sec: amplitudes and gluing}, we will define $\varphi^4_3$ models, the interacting boundary measure, and amplitudes, and state the key technical theorems of the paper. We will then prove Theorems \ref{thm: main} and \ref{theorem: markov} assuming these results. 

\item[-] In Section \ref{sec: boundary}, we will construct and analyze the boundary theory. 

\item[-] In Section \ref{sec: renormalization of bulk amplitudes}, we will derive a renormalized variational representation of the bulk amplitudes.

\item[-] In Section \ref{sec:stochastic}, we will obtain moment bounds on various stochastic terms that appear in the renormalized variational problem.

\item[-] In Section \ref{sec:equicoercive}, we will bound remainder terms that appear in the renormalized variational problem.

\item[-] In Section \ref{sec: convergence of bulk amplitudes}, we will prove convergence of the bulk amplitudes.

\item[-] In Section \ref{sec: smallfield reduction}, we obtain estimates on large field contributions of the boundary field.

\item[-] In Section \ref{sec:spectral}, we prove Theorem \ref{thm: hamiltonian}.

\item[-] In Appendix \ref{appendix:besov}, we collect basic facts about Besov spaces and paraproducts.

\item[-] In Appendix \ref{sec: Dirichlet covariance estimates}, we state covariance estimates on the Dirichlet Laplacian and related operators, as well as state their regularizing properties.

\item[-] In Appendix \ref{appendix: harmonic extension}, we state covariance estimates on the boundary covariance, as well as state their regularizing properties. 
\end{itemize}

\paragraph{Acknowledgements.} We thank Francesco De Vecchi, Yankl Mazor, and Romain Panis for helpful remarks, and Nicolai Reshetikhin for an inspiring discussion. We thank the University of Geneva and Max Planck Institute, Leipzig for hosting us at various points of this project. NB gratefully acknowledges financial support from ERC Advanced Grant 74148 “Quantum Fields and Probability”. TSG is supported by the Department of Atomic Energy, Government of India, under project no.12-R\&D-
TFR-5.01-0500.

\section{Preliminaries}
\label{sec: gaussian}

\subsection{Notation}

\textit{Geometry.} We fix $L>0$ unless stated otherwise. Denote by $M=[-L,L]\times \mathbb{T}^2$ the three-dimensional cylinder of length $2L$, where $\mathbb{T}^2:= (\mathbb{R}/\mathbb{Z})^2$ is the two-dimensional unit flat torus. We write $x = (\tau,z)$ for a generic element of $M$, where $\tau \in [-L,L]$ is the longitudinal coordinate and $z \in \mathbb T^2$ is the axial coordinate. The boundary of $M$ is given by $\partial M= \partial^-M \sqcup \partial^+M$, where $\partial^-M = \{-L\}\times \mathbb T^2$ and $\partial^+ M = \{L\} \times \mathbb T^2$. We write $d(\cdot,\partial M)$ to denote the distance to the boundary and note that it depends only on the $\tau$ coordinate. We write $M^{\rm per}$ to denote the periodization of $M$, where we identify $-L\times \mathbb{T}^2$ with $L\times \mathbb{T}^2$. We will fix a distinguished cut $B:=\{0\}\times\mathbb T^2$ and write $M \setminus B:=M^- \sqcup M^+$, where $M^- := [-L,0)\times\mathbb{T}^2$ and $M^+:= (0,L]\times\mathbb{T}^2$ are the backward and forward parts of the cylinder, respectively. Let us observe that $\partial M^- = \partial^-M \sqcup B$ and $\partial M^+ = B \sqcup \partial^+M$. Finally, let us note that $\partial^- M$, $B$, and $\partial^+ M$ are all isomorphic to $\mathbb T^2$ and we shall use this identification implicitly throughout.

\textit{Function spaces on $\mathbb{T}^2$.} Let $\mathbb{N}^*:=\{1,\dots\}$ denote the natural numbers and $\mathbb{N}:=\mathbb{N}\cup\{0\}$. For $k\in\mathbb{N}\cup\{\infty\}$, let $C^k(\mathbb{T})$ denote the space of $k$-times continuously differentiable functions (with the convention that $k=\infty$ are smooth functions and $k=0$ are continuous functions). Write $S'(\mathbb T^2)$ to denote the space of Schwartz distributions on $\mathbb T^2$. For $p \in [1,\infty]$, we write $L^p(\mathbb{T}^2)$ to denote the usual Lebesgue space on $\mathbb{T}^2$, defined with respect to the (unnormalised) Lebesgue measure. Let $\{ \mathsf{e}_n : n \in \mathbb{Z} \}$ denote the Fourier basis of $L^2(\mathbb{T}^2)$, where $\mathsf{e}_n(z) = e^{2\pi i n\cdot z}$. For any $f \in L^2(\mathbb{T}^2)$, let $\widehat f: \mathbb Z^2 \rightarrow \mathbb C$, $n \mapsto \langle f, \mathsf e_n\rangle$ denote the Fourier transform of $f$, where $\langle \cdot \rangle$ is the duality pairing. Sometimes we write $\widehat f = \mathcal{F}\{f\}$. Let $\mathcal{F}^{-1}$ denote the inverse Fourier transform. For $s \in \mathbb R$ and $p,q \in [1,\infty]$, let $H^s(\mathbb T^2)$, $W^{s,p}(\mathbb T^2)$, and $B^s_{p,q}(\mathbb T^2)$ denote the usual Sobolev, Bessel potential, and Besov spaces on $\mathbb T^2$. Note the equality $H^s(\mathbb T^2) = W^{s,2}(\mathbb T^2) = B^s_{2,2}(\mathbb T^2)$. Furthermore, we write $\mathcal C^s(\mathbb T^2) = B^s_{\infty,\infty}(\mathbb T^2)$ to denote the Besov-H\"older space. We refer to \cite{BCD11} for definitions of these standard spaces. Finally, let us note that given any function space $X (\mathbb T^2)$, we will sometimes write $X_z:= X(\mathbb T^2)$.  

\textit{Bulk function spaces.} We define function/distribution spaces on $M$, but everything here equally applies to $M^-$, $M^+$, or any $I\times \mathbb T^2$, where $I \subset \mathbb R$ is a finite union of connected intervals. Similarly as above, we let $C^k(M)$ denote the $k$-times continuously differentiable space of functions on $X$, $S'(M)$ denote the space of Schwartz distributions on $M$, and $L^p(M)$ denote the Lebesgue space on $M$.  For $s \in \mathbb R$, we define $H^s(M)$ to denote the subspace of $H^s(M^{\rm per})$ obtained as the completion of $\{ f \in C^\infty(X) : f|_{\partial X}=0 \}$. Note that in the case $s \leq 0$, these spaces are the same. We can also define the Bessel potential spaces $W^{s,p}(M)$ and Besov spaces $B^{s}_{p,q}(M)$ similarly. Recall the coordinates $(\tau,z)$ imposed on $M$. Given a function space $Y$ on $\mathbb T^2$, we write $L^p_\tau Y_z$ to denote the space of maps that are $L^p$ in the $\tau$-variable and $Y$ in the $z$-variable. Similarly for $C^k_\tau Y_z$ and also for the $Y$-valued H\"older spaces $C^\alpha_\tau Y_z$, where $\alpha \in (0,1)$.

\subsection{Gaussian Free Fields in the bulk}

We will now introduce the massive GFF on $M$ of mass $m^2>0$ and with Dirichlet (zero) boundary conditions through a random series representation. We recall that $m^2$ is arbitrary and we will consider it fixed throughout. Let $\mathcal G_M:=\{g_n : n \in \mathbb{N}^*\times \mathbb{Z}^2 \}$ denote a set of standard complex normal random variables defined on some probability space $(\Omega, \mathcal F, \mathbb P)$ that are independent modulo the constraint $g_{n_1,-n_2,-n_3}=\overline{g_{n_1,n_2,n_3}}$. Furthermore, let $\mathcal I_M = \{ \mathsf f_{n}: n \in \mathbb N^* \times \mathbb Z^2\}$, where
\begin{equation}
	\mathsf f_{(n_1,n_2,n_3)}(\tau,z):=\sin(\pi n_1\tau/L)\mathsf e_{(n_2,n_3)}(z), \qquad \forall (\tau,z) \in M,
\end{equation}
and where we recall that $\{ \mathsf e_{n} : n \in \mathbb Z^2 \}$ is the usual Fourier basis on $\mathbb T^2$. 

The GFF on $M $ is the random variable $\varphi$ defined by the random series
\begin{equation}
\varphi:=\sum_{n \in \mathbb{N}^*\times \mathbb{Z}^2} \frac{g_n}{\sqrt{4\pi^2(n_1/2L)^2+4\pi^2(n_2^2+n_3^2)+m^2}} \mathsf f_n. 
\end{equation}
Let us remark that, by an explicit covariance computation and independence, the above series converges $\mathbb P$-almost surely in $H^{-1/2-\kappa}(M)$ for $\kappa>0$. By a standard hypercontractivity argument, it can be realized on $\mathcal C^{-1/2-\kappa}(M)$. (See \cite{BG20}.) The law of $\varphi$ under $\mathbb P$ is a probability measure $\mu$ corresponding to a centred Gaussian measure on $\mathcal C^{-1/2-\kappa}(M)$ with covariance $C^M:=(-\Delta+m^2)^{-1}$, where $\Delta$ is Laplacian on $M$ with Dirichlet boundary conditions on $\partial M$. We identify $C^M$ with its kernel $(x,y)\mapsto C^M(x,y)$. 

We now define the GFF with boundary conditions. Given $\varphi^0_-,\varphi^0_+ \in S'(\mathbb T^2)$, the $m$-harmonic extension (henceforth harmonic extension) $H(\varphi_-^0,\varphi_+^0)$ is the smooth function on ${\rm Int}(M):=(-L,L)\times\mathbb T^2$ that solves the PDE
\begin{align}
	(-\Delta+m^2) H(\varphi_-^0,\varphi_+^0)(x)&=0, \quad x\in {\rm Int}(M),
	\\
	H(\varphi_-^0,\varphi_+^0)|_{\partial^- M} &= \varphi_-^0,
	\\
	H(\varphi_-^0,\varphi_+^0)|_{\partial^+M} &= \varphi_+^0.
\end{align}
Given $\varphi_-^0,\varphi_+^0 \in S'(\mathbb T^2)$, the GFF on $M$ with these boundary conditions on $\partial M$ is the random variable 
\begin{equation}
\varphi+H(\varphi_-^0,\varphi_+^0),	
\end{equation}
where $\varphi \sim \mu$. We denote its law by $\mu_{\varphi_-^0,\varphi_+^0}$.

The GFFs  on $M^-$ and $M^+$ are defined similarly as above by symmetry. When we wish to speak about a generic property/object related to the GFF on $M$, $M_-$, or $M_+$ in a unified way, we will often refer to $\mu^\sigma$ for $\sigma \in \{\emptyset, -, +\}$, respectively. For example, we write $C^{M_\sigma}$ for the corresponding covariance and  $\mu^\sigma_{\varphi_-^{0,\sigma},\varphi_+^{0,\sigma}}$ for the corresponding law with boundary conditions $\varphi_-^{0,\sigma}, \varphi_+^{0,\sigma} \in S'(\mathbb T^2)$ on $\partial M^\sigma$. Other objects will be decorated similarly.  

\subsection{Regularization and Wick powers in the bulk}

We begin by introducing our regularization. Let $\rho \in C^\infty(\mathbb R;[0,1])$ be compactly supported on $B(0,2)$ and such that $\rho \equiv 1$ on $B(0,1)$. We will first consider a family of mollifiers $(\rho_T)_{T>0}$ acting on $\mathbb T^2$ by convolution. For every $T>0$, let $\rho_T$ be defined by its Fourier transform
\begin{equation}
	\hat \rho_T(n) = \rho(\langle n \rangle/T), \qquad \forall n \in \mathbb Z^2,
\end{equation}
where $\langle \cdot \rangle = \sqrt{4\pi^2|\cdot|^2+m^2}$.
We will extend the action of convolution with $\rho_T$ to functions and distributions on $M$ by convolving only in the periodic $z$-variable. Given $f \in C^\infty(M)$, let $\rho_T \ast f \in C^\infty(M)$ be the smooth function defined by
\begin{equation}
\rho_T\ast f(\tau,z)
:=
\int_{\mathbb T^2} f(\tau,z'-z) \rho_T(z') dz',
\quad \forall (\tau,z) \in M.
\end{equation}
We in turn extend the action of this convolution to distributions on $M$ via duality.  Furthermore, by symmetry considerations, we define analogous convolutions as acting on distributions on $M^-$ and $M^+$. When clear from context, we simply write $\varphi_T:= \rho_T \ast \phi$.

Let us now consider regularizations of the GFF on $M^\sigma$. Without loss of generality, we consider $M^\sigma = M$. For every $T>0$, $\mu$ almost surely the random variable $\varphi_T$ is a function of positive regularity. In order to see this, note that
\begin{equation}
	\varphi_T = \sum_{n_1 \in \mathbb N} \sum_{n_2,n_3 \in \mathbb Z} \frac{g_{(n_1,n_2,n_3)} \hat\rho(\langle (n_2,n_3)\rangle/T) }{\sqrt{4\pi^2(n_1/2L)^2+\langle (n_2,n_3)\rangle^2} }\mathsf f_{(n_1,n_2,n_3)}.
\end{equation}
The term $\hat\rho(\langle (n_2,n_3)\rangle/T)$ implies that there exists $C>0$ such that $|n_2|, |n_3| \leq CT$. A covariance calculation together with hypercontractivity estimates yields that $\mathbb E[\|\varphi_T\|_{H^{s,p}}^2] < \infty$ for every $s<1/2$ and $p< \infty$. By Sobolev embedding, we have that $\varphi_T \in \mathcal C^{1/2-\kappa}(M)$ for some $\kappa > 0$ sufficiently small. We stress that this estimate is valid only when $T<\infty$ (otherwise the limiting object is distribution-valued, as we shall see later in this article).

\begin{remark}
The choice of regularization is motivated by the multiscale analysis later in this article and is an adaptation of that considered in \cite{BG20} to the case of the cylinder. The compact support of $\rho$ is for obtaining function-valued samples $\varphi_T$. The fact that $\rho \equiv 1$ on $B(0,1)$ is technical and related to the control of large fields. See the $\flat$ operator in Section \ref{sec: renormalization of bulk amplitudes}.
\end{remark}

We now define Wick powers of $\varphi_T$. Since this is a function, we may define $\varphi_T^p$ for every $p \in \mathbb N$. Let us write the regularized covariance in the bulk as
\begin{equation}
	C_T^M(x,y) = \mathbb E_{\mu}[\varphi_T(x)\varphi_T(y)], \qquad 
	\forall x,y \in M.
\end{equation}
The first four Wick powers are defined as follows. For every $x \in M$,
\begin{align}
	\llbracket \varphi_T \rrbracket(x) &= \varphi_T(x), &&\llbracket \varphi_T^2 \rrbracket(x) = \varphi_T^2 - C_T^M(x,x),&&
	\\ \llbracket \varphi_T^3\rrbracket (x) &= \varphi_T^3(x) - 3C_T^M(x,x) \varphi_T(x), &&\llbracket \varphi_T^4 \rrbracket = \varphi_T^4(x) - 6C_T^M(x,x)\varphi_T^2(x) + 3C_T^M(x,x)^2.
\end{align}
Note that $\varphi_T^p$ is a multinomial of degree $p$ in the Gaussian random variables $\mathcal G_M$. Although we will not use this explicitly, let us point out that the Wick ordering of this multinomial corresponds to an orthogonal projection onto so-called homogeneous Wiener chaos spaces generated by these Gaussian random variables, see \cite{J97}. By Wick's theorem, it is easy to see that $\varphi_T$ and $\llbracket \varphi_T^2 \rrbracket$ converge to well-defined random variables as $T \rightarrow \infty$. On the other hand, the third and fourth Wick powers have divergent variances. See e.g. \cite{BG20}. 

\begin{remark}
By convention, we also set $C^M_\infty = C^M$.	
\end{remark}

\subsection{Domain Markov property and boundary GFF}

The conditional law of the GFF on $B=\{0\}\times\mathbb T^2$ can be determined from its domain Markov property. This states that, conditional on the value of the field on $B$, the Dirichlet GFF on $M$ is equal to the sum of independent Dirichlet GFFs on $M_+$ and $M_-$, plus the harmonic extension of $\varphi|_B$ to the bulk. A similar statement also exists for the GFFs with boundary conditions. Below we will identify the law of the field $\varphi|_B$ and use it to define a \emph{boundary GFF} on $\mathbb T^2$.

For the moment, let us fix $M$ with cut $B$, i.e.\ $M\setminus B$. Let $\tilde \mu^0$ be a centred Gaussian measure on $S'(B)$ with covariance $(-\mathcal N)^{-1}$, where $\mathcal N$ is the Dirichlet-to-Neumann map $\mathcal N$ on $M$. Recall that $\mathcal N$ is a negative-definite linear map acting on a dense domain $D(\mathcal N)\subset L^2(B)$ that acts on smooth functions $\varphi^0 \in C^\infty(B)$ by
\begin{equation}
\mathcal N \varphi^0 (z)= \lim_{\tau \uparrow  0} \partial_\tau H^-(0, \varphi^0)(\tau,z)+\lim_{\tau \downarrow 0} \partial_\tau H^+(\varphi^0,0)(\tau,z), \qquad \forall z \in B,
\end{equation}
where $H^\sigma$ is the harmonic extension on $M^\sigma$. (See \cite[Chapter 12, Section C]{T96}.) We stress that $\mathcal N$ depends on the geometry: in this case the length of the cylinder. 

Let us recall the notion of measurability of observables. Given an open subset $X \subset M$, we say that $F:S'(M) \rightarrow \mathbb R$ is $X$-measurable if it is measurable with respect to the $\sigma$-algebra generated by $\{ \varphi(f) : \varphi \in S'(M), f \in C^\infty(M), {\rm supp}(f) \subset X \}$. 

We now state the domain Markov property of the GFF. We omit the proof since, in the context of Dirichlet boundary conditions, it is a standard consequence of the Helmholtz decomposition on the Sobolev space $H^1(M)$, see \cite[Section 2.6]{S07}, and the generic case with boundary conditions follows by straightforward modifications.

\begin{proposition} \label{prop: domain markov}
  Let $\varphi_-,\varphi_+ \in S'(\mathbb T^2)$. Let $F^+,F^-:S'(M)\rightarrow \mathbb R$ be bounded functions such that $F^\pm$ is ${\rm Int}(M^\pm)$-measurable, and let $F=F^+ F^-$ be their pointwise product. Then
 \begin{equation}
  \mathbb E_{\mu_{\varphi_-,\varphi_+}}[ F]
  =
\mathbb E_{\tilde\mu^0}\left[  \mathbb E_{\mu^-}[F^-(\varphi^- + H^-(\varphi_-,\varphi^0)) ]\mathbb E_{\mu^+}[F^+(\varphi^+ + H^+(\varphi_+,\varphi^0)) ] \right], 
  \end{equation}
  where $\varphi^{\sigma} \sim \mu^\sigma$ is the Dirichlet GFF on $M^\sigma$ and $\varphi^0 \sim \tilde\mu^0$. 
\end{proposition}

We would like a notion of a boundary GFF that applies to any subset isomorphic to $\mathbb T^2$. The dependence of $\mathcal N$ on the geometry makes $\tilde \mu^0$ unsuitable. However, we shall see that $\mathcal N$ is related to the Laplacian $\Delta_{\mathbb T^2}$ on $\mathbb T^2$. In the case of the infinite half cylinder $\mathbb R_+ \times \mathbb T^2$, the analogous Dirichlet-to-Neumann map can be computed explicitly using Fourier series and one can show it coincides with the linear map 
\begin{equation}
\mathcal N_0:=-\sqrt{-\Delta_{\mathbb T^2}+m^2}.	
\end{equation}
We will see that the same relation holds up to a geometry dependent smoothing term in the case of finite cylinders. 

Let us denote the infinite cylinder harmonic extension by $\overline H$, defined for $\varphi^0 \in S'(\mathbb T^2)$ as the smooth function on $\mathbb R_+\times \mathbb T^2$ that solves the PDE
\begin{equation}
(m^2-\Delta) \overline H\varphi^0=0 \qquad \text{on} \quad (\mathbb{R}\times \mathbb{T}^2) \setminus B,	
\end{equation}
and such that $\overline H \varphi^0|_B = \varphi^0$. By symmetry, one can consider the infinite cylinder harmonic extension on $\mathbb R_-\times \mathbb T^2$  and we shall abuse notation and also denote this by $\overline H$. Observe that
\begin{equation}
H^-(0,\varphi^0) = \overline H\varphi^0|_{M^-}+\Sigma^- (\varphi^0),
\end{equation}
where $\Sigma^-(\varphi^0)$ is the harmonic function on $M^-$ with boundary conditions
\begin{equation}
	\Sigma^-(\varphi^0)|_{\partial^- M^-}=-\overline H\varphi^0|_{\partial^- M^-}, \quad \Sigma^-(\varphi^0)|_{B}=0. 
\end{equation}
Note that $\Sigma^-(\varphi^0)$ is smooth up to both boundaries in this instance. An identical statement holds also for $H^+(\varphi^0,0)$, where now the smooth correction term is denoted $\Sigma^+(\varphi^0)$. 

We now establish a relation between $\mathcal N$ and $\mathcal N_0$. By the above, we have that
\begin{equation}
	\mathcal N\varphi^0=\mathcal N_0\varphi^0 + \mathcal N_\infty \varphi^0,
\end{equation}
where the smoothing term is defined by
\begin{equation}
\mathcal N_\infty \varphi^0 = \lim_{\tau \uparrow 0} \partial_\tau \Sigma^- (\varphi^0) + \lim_{\tau \downarrow 0} \partial_\tau \Sigma^+(\varphi^0).
\end{equation}
Here,  by smoothing we mean that $\mathcal N_\infty$ is an integral operator with smooth Schwartz kernel on $B\times B$. In particular, $\mathcal N_\infty: \mathcal C^\alpha(B) \rightarrow \mathcal C^{\alpha'}(B)$ is bounded for every $\alpha'\geq \alpha$. Let us emphasize that $\mathcal N_0$ is independent of the geometry of $M$, whereas $\mathcal N_\infty$ heavily depends on the geometry. 

We now define the boundary GFF. Let $\mathcal G_{\mathbb T^2}:= \{ g_n^0 : n \in \mathbb N^* \times \mathbb Z^2 \}$ denote a set of standard complex normal random variables, again defined on the probability space $(\Omega, \mathcal F, \mathbb P)$, which are independent modulo the constraint $g_{-n}^0:= \overline{g_n^0}$. The boundary GFF is the random variable $\varphi^0$ defined by the random series:
\begin{equation}
	\varphi^0:= \sum_{n \in \mathbb Z^2} \frac{g_n^0}{\langle n \rangle}\mathsf e_n.
\end{equation}
Again by a covariance computation, independence, and hypercontractivity, it is standard to show that the series converges almost surely in $\mathcal C^{-1/2-\kappa}(\mathbb T^2)$ for every $\kappa >0$. The law of $\varphi^0$ is denoted by $\mu^0$ and is the centred Gaussian probability measure on $\mathcal C^{-1/2-\kappa}(\mathbb T^2)$ with covariance $(-\mathcal N_0)^{-1}$. 

We will now show that $\tilde\mu^0$ is absolutely continuous with respect to $\mu^0$, so that $\mu^0$ describes the conditional boundary law of the GFF up to absolutely continuous perturbations. Later on we will use this fact whenever we apply the domain Markov poperty. Our proof follows closely the 2d proofs in \cite[Lemmas 4.1,4.5, and 5.3]{GKRV21}, but we give a proof adapted to our 3d setting (although with simpler geometry).  

\begin{lemma} \label{lemma: tilde mu and mu abs cont}
The measure $\tilde \mu^0$ is absolutely continuous with respect to $\mu^0$ and its Radon-Nikodym derivative is given by
\begin{equation} \label{eq: RN free}
\tilde \mu^0(d\varphi):= C(\mathcal N) e^{-\frac 12 \int_B \varphi \cdot (-\mathcal N_\infty) \varphi \, dz} \mu^0(d\varphi),
\end{equation}	
where the constant $C(\mathcal N)$ is given by a square root of a Fredholm determinant\footnote{Recall that for a trace-class operator $K$ acting on a Hilbert space $H$, the Fredholm determinant is defined by
\begin{align*}
\det (I+K)
&=
\sum_{k=0}^\infty {\rm Tr} \Lambda^k(K),
\end{align*}
where $\Lambda^k H$ is the $k$-th exterior power of $H$ and $\Lambda^k(K)$ is the natural extension of $K$ onto $\Lambda^k H$. },
\begin{equation}
	C(\mathcal N):= \sqrt{\det \left( I - (-\mathcal N_0)^{-1}\mathcal N_\infty \right)},
\end{equation}
\end{lemma}

\begin{proof}
We first claim that the Fredholm determinant in the constant $C(\mathcal N)$ exists because the linear map
\begin{equation}
	K:=-(-\mathcal N_0)^{-1}\mathcal N_\infty
\end{equation}
is trace-class. In order to show that, it suffices to show that $K$ is smoothing.  This in turn follows since $\mathcal N_\infty$ is (infinitely) smoothing and $(-\mathcal N_0)^{-1}$ has a kernel with a finite singularity at the origin, hence the convolution of their kernels is smooth.   

We now turn to the expression for the Radon-Nikodym derivative. Let $\Pi_N$ denote orthogonal projection on to fourier modes $|n|\leq N$. For $\# \in \{ \emptyset, 0,\infty\}$, let
\begin{equation}
	\mathcal N^N_\# = \Pi_N\circ \mathcal N_\#\circ \Pi_N.
\end{equation}
Furthermore, for this proof only, let $\mu^0_N:=(\Pi_N)_*\mu^0$ and $\tilde\mu^0_N :=(\Pi_N)_*\tilde\mu^0$ be the pushforward measures under the map $\varphi \mapsto \Pi_N \varphi$. Then a standard computation yields 
\begin{equation} \label{eq: N-proj free RN}
	\mu^0_N(d\varphi) = C_N(\mathcal N) e^{-\frac 12 \langle \varphi, (-\mathcal N^N_\infty) \varphi \rangle_B } \mu^0_N(d\varphi),
\end{equation}
where we recall that $\langle \cdot, \cdot \rangle_B$ denotes the inner product on $L^2(B)$, the constant term is defined by 
\begin{equation}
	C_N(\mathcal N):= \sqrt{\det (I-(-\mathcal N^N_0)^{-1}\mathcal N_\infty^N)},
\end{equation}
and the expression exists provided that the exponential term is integrable. Let us first point out that $C_N(\mathcal N) \rightarrow C(\mathcal N)$ because of the continuity of the map $K'\mapsto \det(I+K')$ in the space of trace-class linear maps. Indeed,
\begin{equation}
K_N:= -(-\mathcal N^N_0)^{-1}\mathcal N_\infty^N =\Pi_N K \Pi_N.
\end{equation}
Thus $K_N$ is a finite rank approximation of $K$. The fact that these linear maps now converge in trace-class norm follows immediately.  

We now turn to integrability and convergence of the exponential term in \eqref{eq: N-proj free RN}. Following \cite[Lemma 4.5]{GKRV21}, let us write
\begin{equation}
	(-\mathcal N)=(-\mathcal N_0)+(-\mathcal N_\infty) =(-\mathcal N_0)^{1/2}(I+\tilde K) (-\mathcal N_0)^{1/2},
\end{equation}
where 
\begin{equation}
	\tilde K:=(-\mathcal N_0)^{-1/2} (-\mathcal N_\infty)(-\mathcal N_0)^{-1/2}.
\end{equation}
Since $\mathcal N$ is positive definite, we therefore have that there exists $a \in (0,1)$ such that 
\begin{equation}
	\langle \varphi, (I+\tilde K )\varphi \rangle_B \geq a\langle \varphi, \varphi\rangle, \qquad \forall \varphi \in L^2(B).
\end{equation}
Thus
\begin{equation}
	e^{-\langle \varphi, (-\mathcal N_\infty)\varphi \rangle_B} \leq e^{(1-a)\langle \varphi, (-\mathcal N_0)\varphi \rangle_B}.
\end{equation}
The same is true for truncated maps. Hence the exponential terms in \eqref{eq: RN free} and \eqref{eq: N-proj free RN} are integrable, and thus the densities are well-defined.

It remains to verify that the function on the righthand side of \eqref{eq: RN free} is indeed the Radon-Nikodym derivative corresponding to $C(\mathcal N)^{-1} \frac{d\tilde\mu^0}{d\mu^0}$. It is sufficient to show that the exponential terms in \eqref{eq: N-proj free RN} converge to \eqref{eq: RN free} in expectation under $\mu^0$. Almost sure convergence is clear and convergence in expectation follows by a dominated convergence argument using smoothing properties of $\mathcal N_\infty$. For details in an almost exactly identical case see \cite[Lemma 5.3]{GKRV21}.
\end{proof}

\section{Amplitudes and gluing for $\varphi^4_3$}
\label{sec: amplitudes and gluing}

In this section, we will define $\varphi^4_3$ models and their amplitudes, and proceed to outline the proofs of our main theorems. We will begin by introducing the models with admissible boundary conditions and restate Theorem \ref{theorem: rough boundary conditions} more precisely, see Theorem \ref{thm: phi43 random bc} below. Next, we will introduce the interacting boundary measure and state the main theorem on its construction, see Theorem \ref{thm: boundary}. We will then turn to the issue of constructing limiting amplitudes and the key technical results of this paper, see Theorems \ref{thm: bulk} and \ref{thm: enhancement converge}.  Conditional on these three aforementioned results, we will then proofs of the main theorem concerning the gluing property, Theorem \ref{thm: main}, and the Markov property, Theorem \ref{theorem: markov}. The remaining sections will then concern the proofs of Theorems \ref{thm: boundary}, \ref{thm: bulk} and \ref{thm: enhancement converge}, as well as the construction and properties of the Hamiltonian, Theorem \ref{thm: hamiltonian}. 

\subsection{Admissible boundary conditions and their Wick powers}

We begin by introducing the notion of \emph{admissible boundary condition}.

\begin{definition}
\label{defn: admissible law}
An $S'(\mathbb T^2)$-valued random variable $\varphi^0$ is an admissible boundary condition if there exists a coupling $(\Omega^0, \mathcal F^0, \mathbb P^0)$ of $(W^0,Z^0) \in S'(\mathbb T^2)\times S'(\mathbb T^2)$ such that:
\begin{enumerate}
\item[(i)]  $\varphi^0$ is equal in law to $W^0+Z^0$.

\item[(ii)] $W^0 = 0$ almost surely or ${\rm Law}_{\mathbb P^0}(W^0) = \mu^0$. 
\item[(iii)] There exists $\kappa > 0$ such that 
\begin{equation}
\mathbb E^0\left[\|Z^0\|_{H^{1/2-\kappa}(\mathbb T^2)}^2\right] < \infty,
\end{equation}
where $\mathbb E^0$ denotes expectation with respect to $\mathbb P^0$.
\end{enumerate}
We call the law of $\varphi^0$ an \emph{admissible boundary law}.
\end{definition}

Let $\sigma \in \{\emptyset, -, + \}$. In order to define $\varphi^4_3$ potentials on $M^\sigma$, we will need to extend the notion of Wick powers to harmonic extensions of admissible boundary conditions, suitably regularized. Recall the family of mollifiers $(\rho_T)_{T \geq 0}$ introduced in Section \ref{sec: gaussian}. We first define Wick powers of 
\begin{equation}
H^\sigma \tilde W^0_T := H^\sigma(\rho_T\ast \tilde W^0,0) =\rho_T \ast (H^\sigma(\tilde W^0,0)) =: H^\sigma (\tilde W^0)_T ,	
\end{equation}
where $\tilde W^0$ is a random variable such that ${\rm Law}_{\mathbb P^0}(\tilde W^0)=\tilde\mu^0$. The reason why we consider $\tilde W^0$ first is because it is the natural measure that arises in the domain Markov property (see Lemma \ref{lemma: covariance DMP decomp} below). We will then explain how to extend this to $H^\sigma(W^0)$ (and hence $H^\sigma(0,W^0)$), and then to incorporate the $Z^0$ term. 

Let us define the kernel $C^{B,\sigma}_T:M^\sigma \times M^\sigma \rightarrow \mathbb R$
\begin{equation} \label{eqdef: Cb cov}
C_{T}^{B,\sigma}(x,y):= \mathbb E^0[ H^\sigma \tilde W^0_T(x) \, H^\sigma \tilde W^0_T (y)], \qquad \forall x,y \in M^\sigma	. 
\end{equation}
Note that the associated operator $C^{B,\sigma}$ is the covariance of the centred Gaussian measure corresponding to the law of the pushforward of $\varphi \sim \tilde \mu^0$ under the map $\cdot \mapsto \rho_T \ast H^\sigma(\cdot)$. As with many other quantities, when the ambient cylinder is clear, we shall drop the dependency of the notation on $\sigma$. 

The domain Markov property yields the following identity.
\begin{lemma} \label{lemma: covariance DMP decomp}
Let $T \geq 0$. Then
\begin{equation}
C^M_T(x,y)= C^B_T(x,y) + C^{M^-}_T(x,y) + C^{M^+}_T(x,y), \qquad \forall x,y \in M. 	
\end{equation}
Above, we have extended the kernels $C^{M^-}_T$ and $C^{M^+}_T$ to be zero if either $x$ or $y$ entry is in $M^+$ and $M^-$, respectively. Furthermore, $C^B_T(x,y) = C^{B,-}_T(x,y)$ if $x,y \in M^-$, $C^B_T(x,y) = C^{B,+}_T(x,y)$ if $x,y \in M^+$.
\end{lemma}

For ease of notation, let us consider Wick powers on $M$ (the case of $M^-$ and $M^+$ can be adapted trivially). The first four Wick powers of $H\tilde W^0_T$ can be easily computed pointwise for $x \in M$ and are given by:
\begin{align}
\llbracket H \tilde W^0_T \rrbracket(x) &:= H \tilde W^0_T(x),
\\\llbracket (H \tilde W^0_T)^2 \rrbracket(x) &:= (H\tilde W^0_T)^2(x) - C^B_T(x,x),
\\
\llbracket (H \tilde W^0_T)^3 \rrbracket (x) &:= (H\tilde W^0_T)^3(x) - 3C^B_T(x,x) \, H\tilde W^0_T(x), 
\\
\llbracket (H \tilde W^0_T)^4 \rrbracket (x) &:=(H \tilde W^0_T)^4(x) - 6C^B_T(x,x) (H \tilde W^0_T)^2(x) + 3 (C_T^B(x,x))^2.
\end{align}
We now define the Wick powers of $HW^0_T$ exactly the same as above and \emph{with the same Wick renormalization constants $C^B_T(x,x)$}.

\begin{remark}
Let us stress that the use of $C^B_T$ in the definition of Wick powers of $HW^0_T$ does \emph{not} correspond to orthogonal projections onto Wiener chaoses, whereas for $H\tilde W^0_T$ it does. However, as we have seen, $C^B_T$ can be compared with the covariance of $HW^0_T$ up to an infinitely smoothing term. Thus, this will not be an obstruction in our argument.	
\end{remark}

Given $\varphi^0:= W^0+Z^0$ an admissible boundary condition and $T \geq 0$, we can extend the notion of Wick powers to the field $\varphi^0_T = \rho_T \ast \varphi^0 $ via:
\begin{equation}
\llbracket (H \varphi^0_T)^p \rrbracket(x) := \sum_{k=0}^p {p \choose k} \llbracket (H W^0_T)^k \rrbracket (x) (H Z^0_T)^{p-k}, \qquad \forall x \in M,\,  \forall p \in \{1,2,3,4\}. 
\end{equation}
Above, we use the convention that the zeroeth Wick power is identically $1$. Furthermore, in the case of two independent admissible boundary conditions, $\varphi^0_-, \varphi^0_+$, we extend the Wick powers in the double harmonic extension as follows. Note that
\begin{equation}
H(\varphi_-^0, \varphi_+^0) = H(\varphi_-^0,0)+H(0,\varphi_+^0) = H(\varphi_-^0)+H(\varphi_+^0).	
\end{equation}
Writing $H(\varphi^0_-,\varphi^0_+)_T := \rho_T \ast H(\varphi^0_-,\varphi^0_+)$, we may define
\begin{equation}
\llbracket H(\varphi^0_-, \varphi^0_+)_T^p \rrbracket (x):=	\sum_{k=0}^p {p \choose k} \llbracket H(\varphi^0_-)_T^p\rrbracket (x) \llbracket H(\varphi^0_+)^{p-k}\rrbracket (x), \qquad \forall x \in M,\, \forall p \in \{1,2,3,4\}.
\end{equation}

\subsection{$\varphi^4_3$ models with admissible boundary conditions}

In this subsection, we define $\varphi^4_3$ models with admissible boundary conditions. We will construct them as weak limits of regularized measures that have been appropriately renormalized. We will first therefore define the approximate $\varphi^4_3$ models on $M^\sigma$ for $\sigma \in \{\emptyset, -, + \}$.

We begin by introducing the renormalizations. The \emph{second-order mass renormalization} is the map $(\gamma_T^\sigma)_{T \geq 0}$ defined by
\begin{equation}
\gamma_T^{\sigma} = -3\cdot 4^2 \int_{M^{\sigma,{\rm per}}}  \int_{M^{\sigma,{\rm per}}}C^{M^{\sigma,{\rm per}}}_T(x,y)^3 dx dy, \qquad \forall T \geq 0,
\end{equation}
where we recall that $M^{\sigma,{\rm per}}$ is the periodization of $M^\sigma$. By standard covariance estimates, we have that $\gamma_T^\sigma = O(\log T)$ as $T \rightarrow \infty$.  Our choice is such that $\gamma_T^\sigma$ is \emph{independent} of the choice of boundary conditions and also so that the gluing formula holds.

Now we turn to the energy renormalization $(\delta^\sigma_T(\varphi_-^{0,\sigma},\varphi_+^{0,\sigma}))_{T \geq 0 }$. It will depend on the choice of boundary condition -- more precisely, it will depend on whether $\varphi^{0,\sigma}_-$ or $\varphi^{0,\sigma}_+$ in the coupling have $W^0$ component distributed according to $\mu^0$. Let us write $W^0_-$ and $W^0_+$ for these components, respectively. We define
\begin{equation}
C^B_T(x,y \mid \varphi^{0,\sigma}_{-}, \varphi^{0,\sigma}_+) = C^B_T(x,y) ( \mathbbm 1_{\{ W^0_- \sim \mu^0  \}} + \mathbbm 1_{\{W^0_+ \sim \mu^0 \}}). 	
\end{equation}
Let us write
\begin{equation}
	\delta^\sigma_T(\varphi_-^{0,\sigma},\varphi^{0,\sigma}_+) = \delta^\sigma_T + \delta^{\sigma, \partial }_T(\varphi^{0,\sigma}_-,\varphi^{0,\sigma}_+),
\end{equation}
where
\begin{align}
\delta^\sigma_T&:= -3\cdot 4\int_{M^\sigma}  \int_{M^\sigma}C_T^{M^\sigma}(x,y)^4 dx dy
\\&\qquad + 2^5 \cdot 3^2 \int_{M^\sigma}	\int_{M^\sigma}\int_{M^\sigma} C_T^{M^{\sigma,{\rm per}}}(x,y)^2 C_T^{M^{\sigma,{\rm per}}}(x,u)^2 C_T^{M^{\sigma,{\rm per}}}(y,u)^2 dx dy du,
\end{align}
and
\begin{equation}
\delta^{\sigma, \partial}_T(\varphi_-^{0,\sigma},\varphi_+^{0,\sigma}) := -3\cdot 4 \int_{M^\sigma}\int_{M^\sigma} \sum_{p=1}^3 {4\choose p} C_T^{M^\sigma}(x,y)^{4-p} C_T^B(x,y \mid \varphi^{0,\sigma}_-, \varphi^{0,\sigma}_+)^{p} dx dy.	
\end{equation}

\begin{remark}
Let us stress that that $(\gamma^\sigma_T)_{T \geq 0}$, $(\delta^\sigma_T)_{T \geq 0}$, and $(\delta^{\sigma, \partial}_T(\varphi_-^{0,\sigma},\varphi_+^{0,\sigma}))_{T \geq 0}$ are constants. This is perhaps surprising for the mass renormalization. In principle, due to the loss of translation invariance, one might expect $\gamma^\sigma_T$ to be a non-constant renormalization function. It is important to our argument that this is not the case -- in fact, later on, we will show that $\gamma^\sigma_T$ agrees with the mass renormalization of the model on the torus (see \cite{BG20}) up to a finite correction. It can be shown that they diverge at rates
\begin{equation}
|\gamma_T^\sigma| \asymp
\log T, \, 
|\delta_T^\sigma| \asymp T,	
\end{equation}
where $\asymp$ means bounded from above and below as $T \rightarrow \infty$.
\end{remark}

\begin{remark} \label{remark: energy renormalization recombination}
The presence of $C^{M^{\sigma,{\rm per}}}$ in the second term of the energy renormalization $\delta^\sigma_T$ is to ensure that, given boundary conditions on the left $\varphi^0_-=\varphi^{0,-}_-$, the middle $\varphi^0=\varphi^{0,-}_+=\varphi^{0,+}_-$, and the right $\varphi^{0}_+=\varphi^{0,+}_+$,
\begin{equation}
\delta_T(\varphi^0_-,\varphi^0_+)=\sum_{\sigma=\pm }\delta^\sigma_T + \delta^{\sigma, \partial}_T(\varphi_-^{0,\sigma},\varphi_+^{0,\sigma}) - 3\cdot 4 \int_{M^\sigma} \int_{M^\sigma} C^B_T(x,y \mid \varphi^0, \varphi^0)^4 dx dy.
\end{equation}
An analogous statement holds for the torus. 
\end{remark}

Let $T \geq 0$. We define the regularized potential, which naturally decomposes into bulk and boundary-bulk contributions. The \emph{bulk regularized potential} is the map $V_T^\sigma$ given by
\begin{equation}
V_T^\sigma(\varphi):= \int_{M^\sigma} \llbracket \varphi_T^4\rrbracket (x) - \gamma_T^\sigma \llbracket \varphi_T^2\rrbracket (x) dx - \delta_T^\sigma, \qquad \forall \varphi \in S'(M^\sigma).	
\end{equation}
We emphasize that the Wick powers of $\varphi_T$ above are formally defined by the \emph{same} expressions as for the Gaussian case, i.e.\ with the same covariance $C^{M^\sigma}_T$ appearing in the constants. We will only use them when the argument is $\varphi \sim \mu^\sigma$. Given $\varphi^{0,\sigma}_-,\varphi^{0,\sigma}_+ \in S'(\mathbb T^2)$, the \emph{boundary-bulk regularized potential} is the map $V^{\sigma,\partial}_T(\cdot \mid \varphi^{0,\sigma}_-,\varphi^{0,\sigma}_+)$ defined for every $\varphi \in S'(M^\sigma)$ by
\begin{align}
V_T^\sigma(\varphi \mid \varphi^{0,\sigma}_-, \varphi^{0,\sigma}_+) &= \int_{M^\sigma} 4\llbracket \varphi_T^3 \rrbracket H^\sigma(\varphi^\sigma_+,\varphi^\sigma_-)_T+ 6\llbracket \varphi_T^2 \rrbracket  \llbracket H^\sigma(\varphi_+^\sigma,\varphi^\sigma_-)_T^2 \rrbracket 
\\
&\qquad + 4 \varphi_T \llbracket H^\sigma(\varphi^\sigma_+,\varphi^\sigma_-)^3_T\rrbracket -2\gamma_T^\sigma \varphi_T H^\sigma(\varphi^\sigma_+,\varphi^\sigma_-)
\\
&\qquad - \gamma^{\sigma}_T \llbracket H^\sigma(\varphi_+^\sigma,\varphi_-^\sigma)_T^2\rrbracket \, dx - \delta_T^{\sigma,\partial}, 
\end{align}
where we recall the definition of Wick powers of harmonic extensions of the boundary fields from the preceding subsection. 

\begin{remark}
It is natural to ask why the term $\int_{M^\sigma} \llbracket H^\sigma(\varphi_+^\sigma, \varphi^\sigma_-)_T^4 \rrbracket\, dx$ is not included in the boundary-bulk regularized potential above. On the one hand, since these potentials are defined almost surely in the boundary condition, this term is "deterministic" and hence will be cancelled by the normalization constant. However, the more severe issue arises when considering the gluing formula. For boundary fields sampled according to $\mu^0$, we will \emph{not} be able to control this term and it cannot be cancelled in a trivial fashion (e.g. by a further constant renormalization) since we will integrate over the boundary field for the gluing property and Markov property. As we shall see, handling it necessitates a changing the boundary measure in a singular way. 
\end{remark}

With the same notation as above, we define the \emph{regularized potential} $V_T^\sigma(\cdot \mid \varphi^{0,\sigma}_-,\varphi^{0,\sigma}_+)$ to be the sum of the bulk and boundary-bulk potentials, i.e.\
\begin{equation}
V_T^\sigma(\varphi \mid \varphi^{0,\sigma}_-,\varphi^{0,\sigma}_+) = V_T^\sigma(\varphi) + V^{\sigma,\partial}_T(\varphi \mid \varphi^{0,\sigma}_-,\varphi^{0,\sigma}_+), \qquad \forall \varphi \in S'(M^\sigma). 
\end{equation}

\begin{definition}
The approximate $\varphi^4_3$ measure on $M^\sigma$ with pair of boundary conditions $\varphi^{0,\sigma}_-,\varphi^{0,\sigma}_+ \in S'(\mathbb T^2)$ is the probability measure $\nu_T^\sigma(\cdot \mid \varphi_-^{0,\sigma},\varphi_+^{0,\sigma})$ with Radon-Nikodym derivative
\begin{equation}
\nu_T^\sigma(d\varphi \mid \varphi_-^{0,\sigma}, \varphi_+^{0,\sigma}) = \frac{1}{\mathcal Z_T^\sigma(0\mid \varphi_-^{0,\sigma},\varphi_+^{0,\sigma})} \exp(-V_T^\sigma(\varphi \mid  \varphi_-^{0,\sigma}, \varphi_+^{0,\sigma})) \mu^\sigma(d\varphi),
\end{equation}
where $\mathcal Z_T^\sigma(0\mid \varphi_-^{0,\sigma},\varphi_+^{0,\sigma})$ is a normalization constant. For every $f \in C^\infty(M^\sigma)$, the unnormalized Laplace transform of $f$ under this measure is the quantity
\begin{equation}
\mathcal Z^\sigma_T(f \mid \varphi^{0,\sigma}_-,\varphi^{0,\sigma}_-):=	\mathbb E_{\mu^\sigma} \left[ \exp\left(\langle f,\varphi\rangle_{M^\sigma} -V_T^\sigma(\varphi \mid \varphi^{0,\sigma}_-,\varphi^{0,\sigma}_+) \right) \right], \qquad \forall f \in C^\infty(M^\sigma),
\end{equation}
where $\langle \cdot,\cdot \rangle_{M^\sigma}$ denotes the $L^2(M^\sigma)$ inner product. Note that taking $f=0$ above indeed yields the normalization constant $\mathcal Z^\sigma_T(0 \mid \varphi^{0,\sigma}_-,\varphi^{0,\sigma}_+)$. 
\end{definition}

We now restate a more precise version of Theorem \ref{theorem: rough boundary conditions} on the construction of $\varphi^4_3$ with admissible boundary conditions. 

\begin{theorem} \label{thm: phi43 random bc}
Let $\sigma \in \{\emptyset,+,-\}$ and let $(\varphi^{0,\sigma_-},\varphi^{0,\sigma}_+)$ be independent admissible boundary conditions. Then almost surely in the boundary conditions, the following statements hold.
\begin{enumerate}
\item[(i)] For every $f \in C^\infty(M^\sigma)$, there exists a limiting unnormalized Laplace transform $\mathcal Z^\sigma(f \mid \varphi^{0,\sigma}_-,\varphi^{0,\sigma}_+) >0$ such that
\begin{equation}
	\lim_{T \rightarrow \infty} \mathcal Z_T^\sigma(f\mid \varphi^{\sigma,0}_-,\varphi^{\sigma,0}_+) = \mathcal Z^\sigma(f \mid \varphi^{\sigma,0}_-,\varphi^{\sigma,0}_+).
\end{equation}
\item[(ii)] There exists a probability measure $\nu^\sigma(\cdot \mid \varphi^{0,\sigma}_-,\varphi^{0,\sigma}_+)$ on $S'(M^\sigma)$ such that $\nu^\sigma_T(\cdot \mid \varphi^{0,\sigma}_-,\varphi^{0,\sigma}_+)$ converges weakly to $ \nu^\sigma(\cdot \mid \varphi^{0,\sigma}_-,\varphi^{0,\sigma}_+)$ as $T \rightarrow \infty$. 
\item[(iii)] For every $f \in C^\infty(M^\sigma)$, we have that $-\log \mathcal Z^\sigma(f \mid \varphi^{\sigma,0}_-,\varphi^{\sigma,0}_+)$ is given by an explicit variational problem. 
\end{enumerate}
\end{theorem}

We will call the probability measure $\nu^\sigma(\cdot \mid \varphi_-^\sigma, \varphi_+^\sigma)$ constructed in Theorem \ref{thm: phi43 random bc} the $\varphi^4_3$ measure on $M^\sigma$ with random boundary conditions $(\varphi_-^{0,\sigma}, \varphi_+^{0,\sigma})$. 

Let us remark that Theorem \ref{theorem: rough boundary conditions} follows from Theorem \ref{thm: phi43 random bc}. Parts (i) and (iii) of Theorem \ref{thm: phi43 random bc} are special cases of Theorem \ref{thm: bulk} below, and so we defer a discussion of the proof to after the latter theorem is stated. Let us comment, however, that (ii) is a standard consequence of (i). 

\begin{remark}
We stress that the inclusion of the energy renormalization $\delta_T^\sigma$ in the regularized potential is simply to make the unnormalized Laplace transforms converge. It does not play an intrinsic role in constructing the measure thanks to the normalization constant. However, on the other hand there is a degeneracy in the mass renormalization. One could add e.g.\ a smooth function to $\gamma^\sigma_T$ that is bounded as $T\rightarrow \infty$ and this will yield an equivalent measure to the one constructed above.
\end{remark}

\begin{remark}
The variational characterization of the unnormalized Laplace transform as stated in Theorems \ref{theorem: rough boundary conditions} and part (iii) of Theorem \ref{thm: phi43 random bc} is intentionally vague at this point since stating it will require some additional technical definitions. See Theorem \ref{thm: bulk} below. The relevance of the (limiting) variational problem is that we will use it to establish full convergence of the regularized measures, as opposed to just tightness. See also \cite{BG20} where this characterization was first establish for $\varphi^4_3$ models on the torus.
\end{remark}

\subsection{Identification of the boundary measure}

In order to identify the correct boundary law of $\varphi^4_3$, we will first attempt to understand where the naive approach to establishing a gluing of amplitudes fails. For simplicity, we will consider the case of gluing when the boundary conditions on $M$ are $0$ on both ends -- the resulting boundary measure that we will define will not depend on the choice of boundary conditions\footnote{We will see below that we put all boundary condition dependence into the amplitude.}. In this case, we will simply omit the boundary conditions from the notation, i.e.\ write $\mathcal Z_T(f)$ instead of $\mathcal Z_T(f\mid 0,0)$. 
 
Let $T \geq 0$ and $f \in C^\infty(M)$. By the domain Markov property of Proposition \ref{prop: domain markov} and the change of measure in Lemma \ref{lemma: tilde mu and mu abs cont}, we have that
\begin{equation}
\mathcal Z_T(f) = \mathbb E_{\mu^0} \left[ \mathcal Q \exp\left(-\int_M \llbracket H(\varphi^0)^4_T \rrbracket \, dx + \tilde \delta^0_T \right) \mathcal Z_T^-(f \mid 0,\varphi^0)\mathcal Z_T^+(f \mid \varphi^0,0)  \right],
\end{equation}
where
\begin{equation}
\mathcal Q:= C(\mathcal N) \exp\left( - \frac 12 \langle \varphi^0, (-\mathcal N_\infty) \varphi^0 \rangle_B - \langle f, H(\varphi^0) \rangle_{M^-} - \langle f, H(\varphi^0) \rangle_{M^+} \right),	
\end{equation}
and
\begin{equation}
\tilde \delta^0_T = -3\cdot 4 \int_M \int_M C^B_T(x,y)^4 dx dy. 	
\end{equation}
Note that in order to obtain the formula for $\delta^0_T$, we have used the definitions of the Wick powers above together with Lemma \ref{lemma: covariance DMP decomp}. The term $\mathcal E_T(\varphi^0)$ consists of terms which can be bounded using fairly straightforward analytic arguments (using smoothing properties of the operators involved). The quartic term, on the other hand, is much more severe. Under the $\mu^0$, one can show that the variance of the term
\begin{equation}
\llbracket H(\varphi^0)^4_T \rrbracket + \tilde \delta^0_T	
\end{equation}
contains a logarithmic divergence (see the remark below).

\begin{remark}
One can decompose this quartic term into homogeneous Wiener chaoses with respect to $\mu^0$ and compute the second moment using Wick's theorem (see e.g. the stochastic estimates in \cite{BG20}). The term $\delta^0_T$ renormalizes the divergent zeroth chaos (constant) term, but the term in the second chaos also contains a divergence. Let us remark that the zeroth chaos term is not exactly cancelled due to the difference in covariances between $\mu^0$ and $\tilde \mu^0$.
\end{remark}

The strategy we adopt is to instead exponentially tilt the boundary measure $\mu^0$ by this quartic density. This leads one to consider an \emph{interacting boundary measure} with formal density
\begin{equation}
\exp\left( - \int_M \llbracket H(\varphi^0)^4_T \rrbracket \, dx + \tilde \delta^0_T	\right)	\mu^0(d\varphi^0).
\end{equation}
There are two issues with this. The first issue, as hinted above, is that the variance of the random variable inside the exponential diverges as $T \rightarrow \infty$, so it is not clear how to interpret the potential, let alone the measure. We will use intuition coming from the construction of $\varphi^4_3$ to instead realize this measure as a \emph{singular perturbation} of $\mu^0$. This is in accordance with the predicted singularity in \cite{AL08}. We will discuss what this means on a more technical level shortly. Let us also emphasize that this is a key point of departure from any existing proofs of the Markov property or Segal's axioms for quantum fields of which we are aware. Furthermore, such a boundary measure would be \emph{non-Gaussian}, which further adds to the difficulties in working with it.

The second issue is that the candidate potential depends on the bulk geometry $M$. Furthermore, if one were to impose boundary conditions, a priori the potential would include dependence on these. Any reasonable candidate for the boundary measure for the gluing property must be independent of the ambient manifold and boundary conditions -- ultimately it should rely on the underlying boundary geometry, of which the bulk geometries are "cobordisms". In order to sidestep this issue, we will replace the harmonic extensions on $M$ (or $M^+$ or $M^-$) by their \emph{infinite cylinder} counterparts (recall that this is denoted by $\overline H$). We will put any residual terms into the amplitudes. 

We will now introduce the boundary measure and postpone further discussion of amplitudes to the next subsection. As usual, we will define this measure as a weak limit of approximations and thus we begin by defining the regularized boundary measure. Let $T \geq 0$. The boundary energy renormalization is the constant
\begin{equation}
\delta^0_T:= - 3\cdot 4 \int_{[-1,1]\times \mathbb T^2} \int_{[-1,1]\times\mathbb T^2} \overline C^B_T(x,y)^4 dx dy,
\end{equation}
where
\begin{equation}
	\overline C^B_T(x,y):= \mathbb E_{\mu^0}[\overline H\varphi_T(x) \overline H\varphi_T(y)], \qquad \forall x,y \in [0,1]\times\mathbb T^2.
\end{equation}
We define the \emph{regularized boundary potential} $V_T^0(\cdot)$ for every $\varphi^0 \in S'(\mathbb T^2)$ by
\begin{equation}
V_T^0(\varphi^0)= \int_{[-1,1]\times \mathbb{T}^2} \llbracket (\overline{H}\varphi_T^0)^4\rrbracket-\delta_T^0 \,  dx,
\end{equation}
 where here our convention for the Wick powers of the infinite cylinder harmonic extension is subtly\emph{different} -- we use the same definitions as above except $C^B_T(x,y)$ is replaced by $\overline C^B_T(x,y)$. 
 
 \begin{remark}
 Let us stress that the replacement of $C^B_T(x,y)$ by $\overline C^B_T(x,y)$ is crucial and allows the boundary theory to be \emph{indepedent of the geometry}. We will have to estimate the error terms from changing this covariance in our analytic estimates on the amplitudes below.
 \end{remark}

\begin{definition} \label{def: regularized boundary measure}
For every $T>0$, the \emph{regularized interacting boundary measure} is the probability measure $\nu^0_T$ on $S'(\mathbb T^2)$ with Radon-Nikodym derivative
\begin{equation}
\nu^0_T(d\varphi^0)=\frac{1}{\mathcal Z_T^0}\exp(-V_T^0(\varphi^0)) \mu^0(d\varphi^0),
\end{equation}
where $\mathcal Z_T^0 = \mathbb{E}_{\mu^0}[\exp(-V_T^0(\varphi^0))]$ is a normalization constant. For every $f \in C^\infty(\mathbb T^2)$, the unnormalized Laplace transform of $f$ under $\nu^0_T$ is the quantity 
\begin{equation}
\mathcal{Z}_T^0(f) := \mathbb E_{\mu^0} [\exp( \langle f, \varphi^0 \rangle_{\mathbb T^2}-V_T^0(\varphi^0)],
\end{equation}
where $\langle \cdot,\cdot \rangle_{\mathbb T^2}$ is the $L^2(\mathbb T^2)$ inner product.
\end{definition}

The following theorem constructs the boundary measure and establishes fundamental properties that we will need in the sequel. It will be proved in Section \ref{sec: boundary}. 

\begin{theorem} \label{thm: boundary}
For every $f \in C^\infty(\mathbb T^2
)$, there exists $\mathcal{Z}^0(f) > 0$ such that
\begin{equation}
\lim_{T \rightarrow \infty}\mathcal{Z}_T^0(f) = \mathcal{Z}^0(f). 	
\end{equation}
Consequently, there exists a unique probability measure$\nu^0$ on $S'(\mathbb T^2)$ such that the measures $\nu^0_T$ converge weakly to $\nu^0$ as $T \rightarrow \infty$.  

Furthermore, there exists random variables $W^0_\infty$ and $Z^0_\infty$ defined on the probability space $(\Omega, \mathcal F, \mathbb P)$ such that
\begin{enumerate}
	\item[(i)] $\nu^0 = {\rm Law}_{\mathbb P}(W^0_\infty+Z^0_\infty)$,
	\item[(ii)] ${\rm Law}_{\mathbb P}(W^0_\infty) = \mu^0$,
	\item[(iii)] $Z^0_\infty \in H^{1/2-\kappa}_z$ almost  surely,
	\item[(iv)] $\mathbb E[\|Z^0_\infty\|_{H^{1/2-\kappa}_z}^2]\leq C$. 
\end{enumerate}
\end{theorem}

We will call the probability measure $\nu^0$ constructed in Theorem \ref{thm: boundary} the interacting boundary measure/theory. As a consequence of properties (i)-(iv) above, $\nu^0$ describes an admissible law on boundary conditions.

\begin{remark}
The choice of integration domain being $[-1,1] \times \mathbb T^2$ inside the boundary potential is an arbitrary choice that makes the boundary measure independent of the length of $M$. Replacing $[-1,1]$ with another compact interval yields an equivalent measure to $\nu^0$. Perhaps a more natural choice would be the double infinite cylinder, $\mathbb R\times \mathbb T^2$, but we opted for a compact set to avoid convergence issues. 
\end{remark}

\subsection{Regularized amplitudes}

Let us return to the issue of defining regularized amplitudes. Fix admissible boundary conditions $\varphi^{0}_-,\varphi^{0}_+ \in S'(\mathbb T^2)$. Recall that by the domain Markov property Proposition \ref{prop: domain markov} and Lemma \ref{lemma: tilde mu and mu abs cont}, for every $f \in C^\infty(M)$,
\begin{equation}
\mathcal Z_T(f \mid \varphi^0_-,\varphi^0_+)  = \mathbb E_{\mu^0} \left[ \mathcal Q  \cdot  \mathcal O_T\cdot  \mathcal Z_T^-(f \mid 0,\varphi^0_-,\varphi^0)\mathcal Z_T^\sigma(f \mid \varphi^0,\varphi^0_+)  \right],
\end{equation}
where (as above) we have that
\begin{equation}
\mathcal Q:= C(\mathcal N) \exp\left( - \frac 12 \langle \varphi^0, (-\mathcal N_\infty) \varphi^0 \rangle_B - \langle f, H(\varphi^0_-,\varphi^0) \rangle_{M^-} - \langle f, H(\varphi^0,\varphi^0_+) \rangle_{M^+} \right),	
\end{equation}
and, moreover, 
\begin{equation}
\mathcal O_T  := \exp\left( - \int_{M^{-}} \llbracket H(\varphi^0_-,\varphi^0)_T^4 \rrbracket \, dx - \int_{M^+} \llbracket H(\varphi^0,\varphi^0_+)_T^4 \rrbracket \, dx + \tilde \delta^0_T \right),
\end{equation}
with
\begin{equation}
\tilde \delta^0_T = -3\cdot 4 \int_{M}\int_{M} C^B_T(x,y \mid \varphi^0_-,\varphi^0_+)^4 dx dy. 	
\end{equation}

We will decompose the term $\mathcal Q$ into \emph{free amplitudes}. Let ${\mathsf C}(\Delta_{\mathbb T^2}) >0$ be the constant identified in \cite[Corollary 1.3]{L03} such that 
\begin{equation}
\frac{\sqrt{\det(2\pi(-\Delta_{M^-}+m^2)^{-1})}\sqrt{\det(2\pi(-\Delta_{M^+}+m^2)^{-1})}}{\sqrt{\det(2\pi(-\Delta_{M}+m^2)^{-1})}} = \mathsf{C}(\Delta_{\mathbb T^2}) \sqrt{\det(2\pi(-\mathcal N)^{-1})}.
\end{equation}
The above determinants are interpreted as $\zeta$-regularized determinants, see \cite[Chapter 9]{S05}. Let us define
\begin{multline} \label{eqdef: free amplitude}
	\mathcal A^{{\rm free},-}_T(f \mid \varphi^0_-,\varphi^0)\\:= \frac{\sqrt{\det(2\pi(-\Delta_{M^-}+m^2)^{-1})}}{{\mathsf C}(\Delta_\mathbb T^2)\sqrt{\det(2\pi(-\mathcal N_0)^{-1/2})}}\exp\left( - \frac 14 \langle \varphi^0, (-\mathcal N_\infty^-)\varphi^0 \rangle_B - \langle f, H(\varphi_-^0,\varphi^0)\rangle_{M^-} \right).
\end{multline}
The other free amplitudes $\mathcal A^{{\rm free},+}(f \mid \varphi^0,\varphi^0_+)$ and $\mathcal A^{{\rm free}}(f \mid \varphi^0_-,\varphi^0_+)$ are defined analogously. With these definitions, we have that
\begin{align}
\mathcal Q &= \frac{C(\mathcal N)}{\det(2\pi(-\mathcal N_0)^{-1/2}){\mathsf C}(\Delta_{\mathbb T^2})^2} \sqrt{\det(2\pi(-\Delta_{M^-}+m^2)^{-1})}\sqrt{\det(2\pi(-\Delta_{M^+}+m^2)^{-1})}
\\
&\qquad \qquad \times \mathcal A^{{\rm free},-}(f \mid \varphi^0_-,\varphi^0)\mathcal A^{{\rm free},+}(f \mid \varphi^0,\varphi^0_+).
\end{align}

We will now rearrange the term $\mathcal O_T$ so that the interacting boundary measure appears. Observe that we may write
\begin{equation}
\int_{M^{-}} \llbracket H(\varphi^0_-,\varphi^0)_T^4 \rrbracket  \, dx = \int_{M^-} \llbracket H(\varphi^0_T)^4_T \rrbracket\, dx + E_T(a,-),
\end{equation}
and similarly for the analogous term on $M^+$ with corresponding commutator $E_T(a,+)$. Furthermore,
\begin{equation}
\int_M \llbracket H(\varphi^0)_T^4 \rrbracket \, dx = \int_{[-1,1]\times \mathbb T^2} \llbracket \overline H(\varphi^0)_T^4 \rrbracket \, dx + E_T(b),
\end{equation}
where $E_T(b)=E_T(b,-)+E_T(b,+)$ and
\begin{equation}
E_T(b,-)= \int_{M^-} \llbracket H(\varphi^0)^4_T \rrbracket \, dx - \int_{[-1,0]\times \mathbb T^2} \llbracket \overline H(\varphi^0)_T^4 \rrbracket \, dx,
\end{equation}
and similarly for $E_T(b,+)$. Finally, let us also define 
\begin{equation}
\Delta\delta^0_T:= \tilde\delta^0_T-\delta^0_T.
\end{equation}
Thus, by symmetry, we have that
\begin{equation}
\mathcal O_T = \exp(-V_T^0(\varphi^0))\exp(-E_T(a,-)-E_T(a,+)-E_T(b)+\Delta\delta_T^0). 
\end{equation}
In order to smuggle in the normalization constant for the interacting boundary measure, it is convenient to introduce the notation
\begin{equation}
	\mathcal Z^0_T(I^\sigma):= \mathbb E_{\mu^0} \left[ e^{- \int_{I^\sigma \times \mathbb T^2 } \llbracket  (\overline{H}\varphi_T^0)^4\rrbracket\, dx + \delta^0_T |I^\sigma| }  \right],
\end{equation}
where $I_\emptyset = [-1,1]$, $I_- = [-1,0]$, and $I_+ = [0,1]$. We now define
\begin{align} 
\begin{split} \label{eqdef: mathcal E term}
\mathcal E_T(\varphi^0_-,\varphi^0_+) &:= \mathcal Z_T^0(I^\emptyset)^{-1}\exp \left( -E_T(a,-)-E_T(a,+)-E_T(b) +  \Delta\delta^0_T \right),	
\\
\mathcal E_T^-(\varphi^0_-,\varphi^0) &:= \mathcal Z_T^0(I^-)^{-1}\exp \left( -E_T(a,-)-E_T(b,-) +  \frac 12 \Delta\delta^0_T \right),
\\
\mathcal E_T^+(\varphi^0,\varphi^0_+)^{-1} &:= \mathcal Z_T^0(I^+)^{-1}\exp \left( -E_T(a,+)-E_T(b,+) +  \frac 12 \Delta\delta^0_T \right).
\end{split}
\end{align}

The bulk amplitudes are defined as follows.
\begin{definition} \label{def: approx amplitude}
Let $\sigma \in \{\emptyset,-,+\}$ and $\varphi_-^{0,\sigma},\varphi_+^{0,\sigma} \in S'(\mathbb T^2)$. The regularized $\varphi^4_3$ amplitude on $M^\sigma$ is the map
\begin{equation}
	\mathcal A_T^\sigma:C^\infty(M^\sigma)\times S'(\mathbb T^2)\times S'(\mathbb T^2) \rightarrow \mathbb R
\end{equation}
defined by
\begin{equation}
	\mathcal A_T^\sigma(f \mid \varphi_-^\sigma,\varphi_+^\sigma):= \mathcal A^{\rm free, \sigma}(f \mid \varphi_-^{0,\sigma},\varphi_+^{0,\sigma}) \mathcal E_{T}^\sigma(\varphi_-^\sigma,\varphi_+^\sigma)  
	 \mathcal Z^\sigma_T(f \mid \varphi^\sigma_-,\varphi^\sigma_+). 
\end{equation}	
\end{definition}

We now state and prove the gluing property for amplitudes of the regularized $\varphi^4_3$ models.

\begin{proposition}\label{prop: gluing-cutoff}
Let $T>0$. Then for every $f \in C^\infty(M)$, $\varphi_-^0,\varphi_+^0 \in S'(\mathbb T^2)$, the amplitudes satisfy the following gluing property: 
\begin{equation} \label{eq: amplitudes gluing regularized}
\mathcal A_T(f\mid \varphi_-,\varphi_+) =  \mathbb E_{\nu^0_T} \left[ \mathcal A_T^-(f \mid \varphi_-,\varphi^0) \mathcal A_T^+(f \mid \varphi^0,\varphi_-) \right].
\end{equation}
\end{proposition}

\begin{proof}
From the computations and definitions above, we have that
\begin{align}
\mathcal A_T(f\mid \varphi_-,\varphi_+) &= \frac{C(\mathcal N)}{{\mathsf C}(\Delta_{\mathbb T^2})} \frac{\sqrt{\det(2\pi(-\Delta_{M^-}+m^2)^{-1})}\sqrt{\det(2\pi(-\Delta_{M^+}+m^2)^{-1})}}{\sqrt{\det(2\pi(-\mathcal N_0)^{-1/2})}\sqrt{\det(2\pi(-\Delta_{M}+m^2)^{-1})}} \\
&\qquad \times \frac{\mathcal Z^0_T(I^-)\mathcal Z^0_T(I^+)}{\mathcal Z^0_T(I)} \cdot \mathbb E_{\nu^0_T} \left[ \mathcal A_T^-(f \mid \varphi_-,\varphi^0) \mathcal A_T^+(f \mid \varphi^0,\varphi_-) \right].
\end{align}
First of all, by symmetry we have that
\begin{equation}
\mathcal Z^0_T(I)=\mathcal Z^0_T(I^-)\mathcal Z^0_T(I^+).
\end{equation}
Now, by the definition of ${\mathsf C}(\Delta_{\mathbb T^2})$, the other constant term is equal\footnote{Here, we are using that the ratio of zeta regularized determinants can be obtained by the spectral approximations considered in Lemma \ref{lemma: tilde mu and mu abs cont}. This follows from standard approximation arguments and we omit this.} to
\begin{equation}
C(\mathcal N) \times \frac{\sqrt{\det(2\pi(-\mathcal N)^{-1/2})}}{\sqrt{\det(2\pi(-\mathcal N_0)^{-1/2})}}	 = 1. 
\end{equation}
 
\end{proof}

\subsection{Amplitudes for the $\varphi^4_3$ model}

Our goal in the next two sections is to give a conditional proof of Theorem \ref{thm: main}. We will begin by isolating the key technical results required in order to establish the existence of the amplitudes for the $\varphi^4_3$ model without cutoffs and take the limit in the gluing property of Proposition \ref{prop: gluing-cutoff}. The proofs of these key results -- formalized in Theorems \ref{thm: bulk} and \ref{thm: enhancement converge} -- will then constitute most of the remainder of this article. 

Let us begin by discussing the convergence of the regularized $\varphi^4_3$ amplitudes. Although the regularized amplitudes were defined with respect to arbitrary distributional boundary conditions, here we will require the boundary conditions to be sampled from an admissible law. We will only consider boundary conditions sampled from $\nu^0$ and stress that the analysis carries over to the general case with almost no changes. At the base level, the difficulty of convergence comes from the fact that samples from $\nu^0$ are distribution-valued, in fact almost surely in $\mathcal C^{-\frac 12-\kappa}(\mathbb T^2)$. As such, there are boundary divergences that must be renormalized. Our renormalization procedure is based on the framework of \cite{BG20}. We will write the regularized amplitudes as a \emph{lifted amplitude}, defined on a larger space, and taking as argument an \emph{enhancement} of both boundary fields simultaneously. The lifted amplitude is deterministically defined and is continuous on this larger space -- proving this constitutes the "analytic part" of the argument and is formalized (along with some growth bounds) in Theorem \ref{thm: bulk}. The "probabilistic part" concerns the convergence in law of the random argument, the enhanced boundary field\footnote{Although the enhancement is of both boundary fields simultaneously, we will nevertheless refer to it as the enhanced boundary field.}, to a limiting enhanced boundary field -- this is formalized in Theorem \ref{thm: enhancement converge}.

\begin{figure}[htb]

\centering

\tikzset{every picture/.style={line width=0.75pt}} 

\begin{tikzpicture}[x=0.75pt,y=0.75pt,yscale=-1,xscale=1]

\draw    (113,169.33) -- (113,141.33) ;
\draw [shift={(113,139.33)}, rotate = 90] [color={rgb, 255:red, 0; green, 0; blue, 0 }  ][line width=0.75]    (10.93,-3.29) .. controls (6.95,-1.4) and (3.31,-0.3) .. (0,0) .. controls (3.31,0.3) and (6.95,1.4) .. (10.93,3.29)   ;
\draw    (354,137) -- (354,165) ;
\draw [shift={(354,167)}, rotate = 270] [color={rgb, 255:red, 0; green, 0; blue, 0 }  ][line width=0.75]    (10.93,-3.29) .. controls (6.95,-1.4) and (3.31,-0.3) .. (0,0) .. controls (3.31,0.3) and (6.95,1.4) .. (10.93,3.29)   ;
\draw    (200,122) -- (258,122) ;
\draw [shift={(260,122)}, rotate = 180] [color={rgb, 255:red, 0; green, 0; blue, 0 }  ][line width=0.75]    (10.93,-3.29) .. controls (6.95,-1.4) and (3.31,-0.3) .. (0,0) .. controls (3.31,0.3) and (6.95,1.4) .. (10.93,3.29)   ;

\draw (54.5,176.4) node [anchor=north west][inner sep=0.75pt]    {$\mathcal{A}_{T}( f\ \mid \varphi _{-} ,\varphi _{+})$};
\draw (41,109.4) node [anchor=north west][inner sep=0.75pt]    {$\boldsymbol{\mathcal{A}}_{T}\left( f\ ,\Xi _{T}^{\partial }( \varphi _{-} ,\varphi _{+})\right)$};
\draw (292.5,178.4) node [anchor=north west][inner sep=0.75pt]    {$\mathcal{A}_{\infty }( f\ \mid \varphi _{-} ,\varphi _{+})$};
\draw (279,109.4) node [anchor=north west][inner sep=0.75pt]    {$\boldsymbol{\mathcal{A}}_{\infty }\left( f\ ,\Xi _{\infty }^{\partial }( \varphi _{-} ,\varphi _{+})\right)$};

\end{tikzpicture}
\caption{The regularized amplitudes $\mathcal A_T(f \mid \varphi_-,\varphi_+)$ are lifted to $\boldsymbol{\mathcal{A}}_T(f, \Xi^\partial_T(\varphi_-,\varphi_+))$. Theorems \ref{thm: bulk} and \ref{thm: enhancement converge} yield convergence of these lifted amplitudes to $\boldsymbol{\mathcal{A}}_\infty(f, \Xi^\partial_T(\varphi_-,\varphi_+))$, that we then define to be the $\varphi^4_3$ amplitude. }
\end{figure}

The following theorem -- which is one of, if not the most important, key contributions of this article -- defines the lifted amplitudes, establishes their continuity, and also some exponential growth bounds. An outline of its proof is found in Section \ref{sec: renormalization of bulk amplitudes}, but we stress that the actual proof is spread over Sections \ref{sec: renormalization of bulk amplitudes}-\ref{sec: convergence of bulk amplitudes}.

\begin{theorem} \label{thm: bulk}
There exists a Banach space, $\boldsymbol{\mathfrak B}$ such that the following holds. Let $\sigma \in \{ \emptyset,-,+\}$.
\begin{enumerate}
\item[\rm (I)]\textbf{\textup{(Lifted amplitude).}} Let $T>0$. There exists measurable maps
\begin{itemize}
	\item $\Xi^\partial_T:\mathcal C^{-1/2-\kappa}(\mathbb T^2) \times \mathcal C^{-1/2-\kappa}(\mathbb T^2) \rightarrow \boldsymbol{\mathfrak B} \quad \hfill \textup{(Enhancement)}$
	\item $\boldsymbol{\mathcal A}^\sigma_T: C^\infty(M) \times \boldsymbol{\mathfrak B} \rightarrow \mathbb R \quad \hfill \textup{(Lift)}\hspace{2.2em}$
\end{itemize}
such that for every $f \in C^\infty(M_\sigma)$ and almost surely for every $(\varphi^\sigma_-,\varphi^\sigma_+) \sim \nu^0\otimes \nu^0$,
\begin{equation}\label{eq: lift of amplitudes}
\mathcal A_T^\sigma(f \mid \varphi_-^\sigma,\varphi_+^\sigma)=\boldsymbol{\mathcal A}^\sigma_T(f, \Xi^\partial_T(\varphi^\sigma_-,\varphi^\sigma_+)).
\end{equation}
\item[\rm (II)] {\rm \bf {(Limit amplitude and convergence).}} There exists a measurable map, called the limiting lifted amplitude,
\begin{equation}
	\boldsymbol{\mathcal A}^\sigma_\infty:C^\infty(M_\sigma)\times \boldsymbol{\mathfrak B} \rightarrow \mathbb R
\end{equation}
such the following convergence statement holds. Let $(\Xi^\partial_T) \subset \boldsymbol{\mathfrak B}$ and $\Xi^\partial_\infty \in \boldsymbol{\mathfrak B}$ be such that $\Xi^\partial_T \rightarrow \Xi^\partial_\infty$ in $\boldsymbol{\mathfrak B}$. Then for every $f \in C^\infty(M_\sigma)$,
\begin{equation} \label{eq: convergence of amplitudes}
	\lim_{T \rightarrow \infty} \boldsymbol{\mathcal A}^\sigma_T(f,\Xi^\partial_T) = \boldsymbol{\mathcal A}^\sigma_\infty(f, \Xi^\partial_\infty).
\end{equation}
\item[\textup{(III)}] \textup{\bf (Almost sure bounds)} For every $f \in C^\infty(M)$, there exists $c,C,p>0$ such that for every $T \in (0,\infty]$ and every $\Xi^\partial \subset \boldsymbol{\mathfrak B}$,
\begin{align}
\boldsymbol{\mathcal{A}}_{T}^\sigma(f,\Xi^\partial)
 \leq
  C \exp(c  \|\Xi^\partial \|^{p}_{\boldsymbol{\mathfrak{B}}}), 	
\end{align}
\end{enumerate}
\end{theorem}

\begin{remark}
Theorem \ref{thm: bulk} contains a few objects whose existence we postulate as part of the theorem, but for ease of reading at this stage we do not define them
\begin{enumerate}
\item[(i)] The Banach space $\boldsymbol{\mathfrak B}$ is called the \emph{enhanced boundary space}. We leave its definition implicit. It is obtained by taking the space $\boldsymbol{\mathfrak B}^0$ in Section \ref{sec: boundary} and products of intersections of Banach spaces appearing in the estimates on boundary terms in Section \ref{sec:stochastic}. 
\item[(ii)] The (cutoff) enhancement map $\Xi^\partial_T$ is a measurable mapping of the boundary data into the stochastic products defined in Section \ref{sec: boundary} and the (boundary-dependent) $\Upsilon_i$ terms defined in Section \ref{sec: renormalization of bulk amplitudes}. See Section \ref{subsec: renormalized var}. 
\item[(iii)] The definitions of the lifted amplitudes and limiting lifted amplitude are in terms of \emph{renormalized variational problems} that we develop across Sections \ref{sec: renormalization of bulk amplitudes}-\ref{sec: convergence of bulk amplitudes}. See in particular the outline of proof in Section \ref{sec: renormalization of bulk amplitudes}. 
\end{enumerate}
\end{remark}

Let us emphasize that Theorem \ref{thm: bulk}, parts I and II are \emph{deterministic statements}. They do not rely on the law of the boundary conditions. However, as clear from part I of the theorem, we need to understand the amplitudes evaluated with random arguments in $\boldsymbol{\mathfrak B}$. This is addressed in the theorem below, which yields the convergence in law of the enhanced boundary field to a limiting enhanced boundary field. Its proof can be found in Sections \ref{sec: boundary} and \ref{sec:stochastic} (corresponding to the purely boundary-dependent terms -- see Proposition \ref{prop: stochastic boundary v2} -- and the $\Upsilon_i$ terms mentioned above, respectively).

\begin{theorem}\label{thm: enhancement converge}
Almost surely for every $(\varphi_1,\varphi_2) \sim \nu^0\otimes \nu^0$, there exist a random variable
\begin{equation}
	\Xi^\partial_\infty(\varphi_1,\varphi_2)\in \boldsymbol{\mathfrak B} 
\end{equation}
called the \emph{limiting enhanced boundary field} such that the following holds.
\begin{enumerate}
\item[\textup{(I)}] {\rm \bf (Almost sure convergence)} Almost surely in $(\varphi_1,\varphi_2) \sim \nu^0\otimes \nu^0$, we have the following convergnce in the Banach space $\boldsymbol{\mathfrak B}$: 
\begin{equation}
\lim_{T \rightarrow \infty} \Xi^\partial_T(\rho_T \ast \varphi_1, \rho_T \ast \varphi_2) =\Xi^\partial_\infty(\varphi_1,\varphi_2)
\end{equation}
\item[\textup{(II)}] { \rm  \bf (Distributional convergence)} Almost surely in $\varphi \sim \nu^0$, we have the following convergence in law:
\begin{equation}
	\big( \Xi^\partial_T(\rho_T \ast \varphi, \cdot) \big) _* \nu^0_T\rightharpoonup \big( \Xi^\partial_\infty(\varphi, \cdot) \big)_* \nu^0.
\end{equation}
\end{enumerate}
\end{theorem}

\begin{remark}
Let us mention that Theorem \ref{thm: enhancement converge} could have been stated for arbitrary independent admissible boundary conditions with small modifications (especially to part II, where $\nu^0_T$ is replaced by a suitable approximation), but we will only apply it in the case of $\nu^0$.
\end{remark}

In light of Theorems \ref{thm: bulk} and \ref{thm: enhancement converge}, we may now rigorously define the $\varphi^4_3$ amplitudes (without cutoffs) and give meaning to the maps introduced in Theorem \ref{thm: main}. 

\begin{definition}
Let $\sigma \in \{\emptyset, -,+\}$. The $\varphi^4_3$ amplitude on $M^\sigma$ is the map
\begin{equation}
	\mathcal A_\infty^\sigma:C^\infty(M_\sigma)\times \mathcal C^{-1/2-\kappa}(\mathbb T^2) \times \mathcal C^{-1/2-\kappa}(\mathbb T^2) \rightarrow \mathbb R
\end{equation}
defined for every $f \in C^\infty(M_\sigma)$ and almost surely for every $\varphi_-,\varphi_+$ independent admissible boundary conditions by
\begin{equation}
	\mathcal A^\sigma_\infty(f \mid \varphi_-,\varphi_+) := \boldsymbol{\mathcal A}^\sigma_\infty (f, \Xi^\partial_\infty(\varphi_-,\varphi_+)),
\end{equation}
where the lifted amplitude $\boldsymbol{\mathcal A}^\sigma_\infty$ is constructed in Theorem \ref{thm: bulk} and the limiting enhanced boundary field is constructed in Theorem \ref{thm: enhancement converge}.
\end{definition}

\subsection{Proof of Segal's gluing property for $\varphi^4_3$} \label{subsec: proof of main}

We will now turn to the proof of Theorem \ref{thm: main}. Our starting point is the approximate gluing property established in Proposition \ref{prop: gluing-cutoff}. In passing from the almost sure and distributional convergence statements of Theorems \ref{thm: bulk} and \ref{thm: enhancement converge} to convergence in expectation, we will need the following upgrade of the bounds in Theorem \ref{thm: bulk} that can be seen as controlling the large field behaviour. Its proof is contained in Section \ref{sec: smallfield reduction}. 

\begin{lemma} \label{lem: large field}
There exists $\alpha > 0$ sufficiently small and a Banach space $\boldsymbol{\mathfrak B}'$ such that $\boldsymbol{\mathfrak B}' \subset \boldsymbol{\mathfrak B}$ is compactly 
embedded and the following estimate holds. There exists $c>0$ such that for every $f \in C^\infty(M)$ and $\nu^0\otimes\nu^0$ almost surely for  every $\varphi_-,\varphi_+ \in S'(\mathbb T^2)$, there exists $C>0$ such that for every $T>0$,
\begin{equation}
E_{\nu^0_T} \left[ \mathcal A_T^-(f \mid \rho_T \ast \varphi_-, \varphi^0)\mathcal A^+_T(f \mid \varphi^0, \rho_T \ast \varphi_+) e^{c \|\Xi_T^\partial (\rho_T \ast \varphi_-,\varphi^0)\|_{\boldsymbol{\mathfrak B}'}^\alpha +c\|\Xi^\partial_T(\varphi^0,\rho_T \ast \varphi_+)\|_{\boldsymbol{\mathfrak B}'}^\alpha }\right]	\leq C.
\end{equation}
\end{lemma}

\begin{remark}
The Banach space $\boldsymbol{\mathfrak B}'$ is obtained in a trivial fashion from $\boldsymbol{\mathfrak B}$. As shall be clear, $\boldsymbol{\mathfrak B}$ is comprised of function/distribution spaces (e.g. Sobolev, H\"older, and Besov spaces) of regularity which can be increased by a sufficiently small parameter without changing any of the statements above. We then use classical compact embedding theorems for these spaces by increasing the regularity of each space by such a sufficiently small parameter.
\end{remark}

We are now ready to give the proof of Theorem  \ref{thm: main}.
\begin{proof}[Proof of Theorem \ref{thm: main}]

Observe that for $T>0$, by Proposition \ref{prop: gluing-cutoff} we have 
\begin{equation} \label{eq: thm1 1}
\mathcal{A}_{T}(f \mid \rho_T \ast \varphi_{-},\rho_T \ast \varphi_{+})= \mathbb{E}_{\nu^{T}_{0}} \left[ \mathcal{A}_{T}^-( f\mid \rho_T \ast \phi_{-},\phi^0) \mathcal{A}_{T}^+(f \mid \varphi^0,\rho_T\ast \varphi_+)\right].
\end{equation}
 Furthermore, for the lefthand side of \eqref{eq: thm1 1}, by part I of Theorem \ref{thm: enhancement converge} and part II of Theorem \ref{thm: bulk} we have that
	\begin{equation}\lim_{T \to \infty} \mathcal{A}_{T}(f \mid \rho_T \ast \phi_{-},\rho_T \ast \phi_{+})= \lim_{T \to \infty} \boldsymbol{\mathcal{A}}_{T}(f,\Xi_{T}^\partial (\rho_T\ast \phi_{-},\rho_T \ast \phi_{+}))
	= \boldsymbol{\mathcal{A}}_{\infty}(f,\Xi_{\infty}^\partial (\phi_{-},\phi_{+})). 
	\end{equation}
Hence it suffices to estimate the expectation term in the righthand side of \eqref{eq: thm1 1}. We will use Skorokhod's embedding theorem to realize this distributional convergence as almost sure convergence and then use a truncation argument to deduce convergence of the expectations. Note that a priori this is up to a subsequence, but the limits will be independent of the choice of subsequence and hence one can deduce full convergence. 

By part II of Theorem \ref{thm: enhancement converge} and Skorokod's embedding theorem \cite{J98}, there exists a subsequence $T\rightarrow \infty$ (not relabelled) and $\boldsymbol{\mathfrak B}'$-valued random variables $(\Xi^{\partial,-}_T,\Xi^{\partial,+}_T)$ and $(\Xi^{\partial,-}_\infty, \Xi^{\partial,+}_\infty)$ defined on some common probability space -- which we abuse notation and denote as $(\Omega, \mathcal F, \mathbb P)$ with expectation written as $\mathbb E$ --  such that 
\begin{itemize}
	\item $\Xi^{\partial,-}_T $ is equal in law to the law of $\Xi^\partial_T(\rho_T \ast \varphi_-,\varphi^0)$ where $(\varphi_-,\varphi^0) \sim \nu^0\otimes \nu^0_T$, and similarly for $\Xi^{\delta,+}_T$.
	\item $\Xi^{\partial,-}_\infty$ is equal in law to the law of $\Xi^\partial_\infty(\varphi_-,\varphi^0)$ where $(\varphi_-,\varphi^0) \sim \nu^0\otimes\nu^0$, and similarly for $\Xi^{\partial,+}_\infty$.
	\item $\Xi^{\partial,-}_T$ (resp. $\Xi^{\partial,+}_T$) converges $\mathbb P$-almost surely in $\boldsymbol{\mathfrak B}$ to $\Xi^{\partial,-}_\infty$ (resp. $\Xi^{\partial,+}_\infty$).  
\end{itemize}

Let $K>0$ be arbitrary. Let us denote $\varphi^0_T \sim \nu^0_T$ for the rest of this proof.   By parts I and  II of Theorem \ref{thm: bulk} and the bounded convergence theorem,
\begin{align}
&\lim_{T\to \infty}\mathbb E\left[ \mathbbm{1}_{\max (||\Xi_{T}^{\partial,-}||_{\boldsymbol{\mathfrak{B}}'}, \|\Xi_T^{\partial,+}\|_{\boldsymbol{\mathfrak B}'})\leq K} \mathcal{A}^{-}_{T}(f \mid \rho_T \ast \varphi^T_{-},\varphi^{T}_{0}) \mathcal{A}^{+}_{T}(f\mid \varphi^T_{0},\rho_T \ast \varphi_{+})\right]
	\\
	&=
	\mathbb E\left[ \mathbbm{1}_{\max(||\Xi_{\infty}^{\partial,-}||_{\boldsymbol{\mathfrak{B}}'}, \|\Xi^{\partial,+}_\infty\|_{\boldsymbol{\mathfrak B}'})\leq K} \boldsymbol{\mathcal{A}}_{\infty}^{-}(f,\Xi_{\infty}(\varphi_{-},\varphi^0)) \boldsymbol{\mathcal{A}}^{+}_{\infty}(f,\Xi_{\infty}(\varphi^0,\varphi_{+}))\right]	
\end{align}
On the other hand, by a union bound and exponential Chebyshev's inequality, followed by Lemma \ref{lem: large field}, there exists $\alpha, c,C>0$ such that for every $T>0$,
\begin{align}
		&\mathbb E\left[\mathbbm{1}_{\max{(||\Xi_{T}^{\partial,-}||_{\boldsymbol{\mathfrak{B}}'},||\Xi_{T}^{\partial,+}||_{\boldsymbol{\mathfrak{B}}'})}\geq K} \mathcal{A}_{T}(f \mid \rho_T \ast \varphi_{-},\varphi^{T}_{0}) \mathcal{A}_{T}(f \mid \varphi^{T}_{0},\rho_T \ast \varphi_{+}) \right] 
		\\
		&\leq 2e^{-c K^{\alpha}} \mathbb E\left[ e^{c ||\Xi_{T}^{\partial,-}||^{\alpha}_{\boldsymbol{\mathfrak{B}}'}+ c||\Xi_{T}^{\partial,+}||^{\alpha}_{\boldsymbol{\mathfrak{B}}'}} \mathcal{A}_{T}(f \mid \rho_T\ast \varphi_{-},\varphi^{T}_{0}) \mathcal{A}_{T}(f \mid \varphi^T_{0},\rho_T \ast \varphi_{+})\right]
		\leq Ce^{-cK^\alpha}. 
	\end{align} 
Since $K$ is arbitrary, Theorem \ref{thm: main} now follows by the definition of the $\varphi^4_3$ amplitudes.

\end{proof}

\subsection{Proof of Markov property}

We now establish the Markov property, Theorem \ref{theorem: markov}.
\begin{proof} [Proof of Theorem \ref{theorem: markov}]
Recall the definition of $\nu^\ell$. Let $f, g \in C^\infty(M)$ such that $f$ is $M_\ell$-measurable and $g$ is $M\setminus M_\ell$-measurable. In order to prove the Theorem \ref{theorem: markov}, it is sufficient to show that 
	\begin{equation}
		\mathbb{E}_{\nu}[\exp(\langle f+g, \phi \rangle)]
		=\mathbb{E}_{\nu}\left[ \mathbb{E}_{\nu^{\ell}(\cdot|\phi_{-},\phi_{+})}[\exp(\langle f, \phi \rangle)] \exp(\langle g,\phi \rangle) \right].
	\end{equation}
	We will denote $\langle f,\phi\rangle^{N}=\langle f,\phi \rangle \wedge N$, and $\mathcal{A}_{T}(f^{N}|\cdot, \cdot)$ be the Amplitude where $\langle f,\phi \rangle $ is replaced by $\langle f,\phi \rangle \wedge N$.
	Furthermore denote by 
\begin{equation}
V^{[a,b]}(\phi_{T})
	=\int_{[a,b]\times \mathbb{T}^2} \llbracket \phi_{T}^4 \rrbracket -\gamma_{T} \int_{[a,b]\times{\mathbb{T}^2}}\llbracket \phi_{T}^2 \rrbracket -\delta_{T}|a-b|.
\end{equation}
By dominated convergence and Theorem \ref{thm: bulk}, we have that 
\begin{equation}
\mathbb{E}_{\nu}[\exp(\langle f+g, \phi \rangle)] =\lim_{N \to \infty} \lim_{T \to \infty} \mathbb{E}_{\nu^{T}}\left[\left(\exp(\langle f, \phi_{T} \rangle^N+\langle g, \phi \rangle)\right)\right].	
\end{equation}

Let us now fix $N$ and $T$. By using the definition of $\nu^T$, splitting the potential, and using the tower property of conditional expectation, we have that
\begin{align}
\begin{split} \label{eq: markov pf}
&\mathbb{E}_{\nu^{T}}\left[\left(\exp(\langle f, \phi_{T} \rangle^N+\langle g, \phi \rangle)\right)\right]	\\
&= \frac{1}{\mathcal Z_T}\mathbb{E}_{\mu}\left[\mathbb{E}_\mu\left[\exp(\langle f, \phi_{T} \rangle^N) \exp(- V^{[-\ell,\ell]}_{T}(\phi_{T})) \mid \mathcal{F}_{M_{l}}\right]  \exp(-\langle g,\phi_{T} \rangle) \exp(- V^{[-L,L]\setminus [-\ell,\ell]}_{T}(\phi_{T}) )\right].
\end{split}
\end{align}
Let us define $\varphi_-$ and $\varphi_+$ as the restriction of $\varphi \sim \mu$ to the tori $\{-\ell\}\times\mathbb T^2$ and $\{\ell\}\times\mathbb T^2$, respectively -- this almost surely exists thanks to the Markov property of the GFF. Then by the definition of the unnormalized Laplace transform, 
\begin{align}
\eqref{eq: markov pf}&= \frac{1}{\mathcal Z_{T}}\mathbb{E}_{\mu}[\mathcal{Z}^{(-\ell,\ell) }_T(f^N \mid (\phi_-)_{T},(\phi_+)_{T})\exp(\langle g, \phi\rangle)\exp(-V^{[-L,L]\setminus(-\ell,\ell)}_{T}(\phi))]\\
	&= \frac{1}{\mathcal Z_{T}}\mathbb{E}_{\mu}\Bigg[\frac{\mathcal Z_{T}^{(-\ell,\ell)}(0 \mid (\phi_{-})_T,(\phi_+)_T)}{\mathcal Z_{T}^{(-\ell,\ell)}(0 \mid (\phi_{-})_T,(\phi_+)_T)}\mathcal{Z}^{(-\ell,\ell)}_T(f^N \mid(\phi_-)_{T},(\phi_+)_{T})
	\\ & \qquad \qquad \qquad \qquad \qquad \qquad \qquad \qquad \times \exp(\langle g, \phi\rangle)\exp(-V^{[-L,L]\setminus(-\ell,\ell)}_{T}(\phi))\Bigg]\\
	&=\frac{1}{\mathcal Z_{T}}\mathbb{E}_{\mu}\Bigg[\frac{1}{\mathcal Z_{T}^{(-\ell,\ell)}(0 \mid (\phi_{-})_T,(\phi_+)_T)}\mathcal{Z}^{(-\ell,\ell)}_T(f^N \mid (\phi_-)_{T},(\phi^{0})_{T}) \\ & \qquad \qquad \qquad \qquad \qquad \qquad \qquad \qquad\times \exp(\langle g, \phi\rangle)\exp(-V^{[-L,L]}_{T}(\phi))\Bigg],
\end{align}
where in the last line we have unpacked the definition of $\mathcal Z_{T}^{(-\ell,\ell)}(0 \mid (\phi_{-})_T,(\phi_+)_T)$. The theorem now follows by using the definition of $\nu^\ell_T(\cdot \mid \cdot, \cdot)$ and taking limits thanks to Theorem \ref{thm: bulk}.
\end{proof}

\section{Construction of the boundary measure}
\label{sec: boundary}

In this section we will prove Theorem \ref{thm: boundary}. Following the strategy introduced in \cite{BG20}, our starting point is to represent $\mathcal Z^0_T(f)$ using a variational formula originally due to Bou\'e and Dupuis \cite{BD98}, and we do this in Section \ref{sec: bdry variational}.  In Section \ref{sec: bdry stochastic} we establish the regularity of the relevant Gaussian/multilinear functions of Gaussian random variables that occur in the variational problem. In Section \ref{sec: bdry unif} we use the regularity of these stochastic objects to obtain a \emph{renormalized} variational representation for $\mathcal Z^0_T(f)$. We then use harmonic analysis estimates to establish uniform bounds on this renormalized variational problem. In Section \ref{sec:boundary:convergence of LT}, assuming some technical results, we upgrade the uniform bounds into full convergence of $\mathcal Z^0_T(f)$ via a $\Gamma$-convergence argument for the variational problems, and thus prove Theorem \ref{thm: boundary}. Finally in Section \ref{sec: boundary technical}, we prove the remaining technical results to close the argument.

\subsection{Variational representation of $\mathcal Z^0_T(f)$}
\label{sec: bdry variational}

Let $\kappa_0>0$ be sufficiently small (fixed a posteriori). Denote by $(\Omega^0,\mathcal{F}^0,\mathbb{P}^0)$ the standard Wiener space, where $\Omega^0 = C(\mathbb{R}_+, \mathcal C^{-1-\kappa_0}(\mathbb T^2))$ and $\mathcal F^0$ is the Borel $\sigma$-algebra. For the remainder of \emph{this} section, we will simply write $\mathbb P=\mathbb P^0$ for notational convenience (with expectation $\mathbb E^0 = \mathbb E)$. The canonical coordinate process $X^0$ under $\mathbb P$ is an $L^2(\mathbb T^2)$-Brownian motion. We will abuse notation and assume that $\mathcal F$ is $\mathbb P$-complete and that the filtration generated by $X^0$ satisfies the usual conditions (on a filtered probability space). For every $t>0$, let $\sigma_t^0 = \rho_t \ast (-\Delta_{\mathbb T^2} + m^2)^{-1/2} \ast \rho_t$ and $J_t^0 = \sqrt{\partial_t \sigma_t^0}$, where the square root is defined in Fourier space after viewing $\sigma^0_t$ as a Fourier multiplier. We abuse notation and write $J_t^0(\cdot,\cdot):\mathbb T^2\times \mathbb T^2 \rightarrow \mathbb R$ for the kernel of $J_t^0$.

Define the stochastic process $(W^0_T)_{T\geq 0}$ by
\begin{equation}
W^0_T
:=
\int_0^T J^0_t dX^0_t, \qquad \forall T\geq 0. 	
\end{equation}
A covariance computation yields that for every $T>0$ we have that  ${\rm Law}_{\mathbb P}(W^0_T) = {\rm Law}_{\mu^0}(\varphi_T)$. Furthermore, again by a covariance computation, we have that the random variable $W^0_\infty := \lim_{T \rightarrow \infty} W_T^0$ exists (the convergence being in e.g.\ $L^2$) and is equal in law to $\mu^0$. 

Thanks to the equality in law, it is clear we can express $\mathcal Z^0_T(f)$ as an expectation of a functional on Wiener space with sufficiently good integrability properties (i.e.\ tamed in the sense of \cite{U14}). This allows us to apply the Bou\'e-Dupuis formula \cite{BD98} (we will use the version in \cite{U14}). In order to state this variational formula, we will need some notation. 

Let $\mathbb H^0$ denote the space of progressively measurable processes that are $\mathbb P$-almost surely in 
\begin{equation}
\mathcal H^0:= L^2_tL^2_z.	
\end{equation}
 For every $u \in \mathbb H^0$ and $T\geq 0$, let
\begin{equation}
Z_T^0(u)
=
\int_0^T J_t^0 u_t dt. 	
\end{equation}

The following proposition is a direct application of the Bou\'e-Dupuis formula to our setting.
\begin{proposition} \label{prop: bdry bd}
For every $f \in C^\infty(\mathbb T^2)$, 
\begin{equation} \label{eq: bdry bd}
-\log \mathcal{Z}_T^0(f)
=
\inf_{u \in \mathbb H^0} \mathbb E \left[ \langle f,W_T^0 + Z_T^0(u) \rangle_{L^2(\mathbb T^2)} + \mathcal{V}_T^0(W_T^0,Z_T^0(u)) + \frac 12 \|u\|_{\mathcal H^0}^2 \right],
\end{equation}
where
\begin{align}
\mathcal{V}_T^0(W_T^0, Z_T^0(u))
&=
\int_{[-1,1]\times \mathbb T^2}  4  \llbracket (\overline H W_T^0)^3 \rrbracket \overline H Z_T^0(u) + 6 \llbracket (\overline H W_T^0)^2 \rrbracket (\overline H Z_T^0(u))^2 
\\
&\quad \quad \quad \quad \quad 
+ 4 \overline H W_T^0 (\overline H Z_T^0(u))^3  + (\overline H Z_T^0(u))^4 \, dx - \delta_T^0. 
\end{align}
\end{proposition}

Above, we recall that $\llbracket (\overline H W_T^0)^2 \rrbracket$ and$\llbracket (\overline{H} W_T^0)^3\rrbracket$ denote the second and third Wick powers of $\overline H W^0_T$, respectively.

\begin{remark}
Note that in principle the term $\int_{[-1,1]\times\mathbb T^2}\llbracket (\overline H W_T^0)^4 \rrbracket dz$ should be appear in the integrand.  It is classical that this quantity can be written as an iterated stochastic integral and, hence, martingale. See \cite{N09}. However, by the martingale property, it has mean zero and hence can be omitted in light of the expectation in the variational formula \eqref{eq: bdry bd}. 
\end{remark}

\subsection{Regularity and convergence of the Wick powers} \label{sec: bdry stochastic}

 In order to understand better the divergences in the variational problem \eqref{eq: bdry bd}, we first\footnote{We will show later on that $Z_T^0(u) \in H^{1/2}(\mathbb T^2)$ for every $u \in L^2_tL^2_z$ with a uniform estimate on its norm in $T$. See Lemma \ref{lem: bdry a priori Z}. This allows us to deduce some determinstic regularity estimates on $\overline H Z^0_T(u)$ by properties of the harmonic extension $\overline H$. See Appendix \ref{appendix: harmonic extension estimates}.} need to understand the regularity of the Wick powers of $\overline H W_T^0$. The  is contained in the following proposition.

\begin{proposition} \label{prop: stochastic boundary v2}
For every $i\in\{1,2,3\}$, $\alpha \in (\max(0,\frac{i-2}{2}),i/2]$, every $\delta > 0$ sufficiently small, and for every $p\in[1,\infty)$ there exists $C_p>0$ such that  and 
\begin{equation} \label{eq:stoch-weighted-linfty}
\sup_T \mathbb E\Big[ \Big( \sup_{\tau \in [-1,1]}|\tau|^{\alpha}\| \llbracket (\overline H  W^0_T) ^i \rrbracket (\tau,\cdot)\|_{\mathcal C^{-i/2+\alpha-\delta}_z} \Big)^p \Big] \leq C_p. 
\end{equation}
Furthermore,  there exists a random variable $\llbracket (\overline H W^0_\infty)^i\rrbracket \in L^\infty(|\tau|^{\alpha},\mathcal C^{-i/2+\alpha-\delta}_z)$ -- where this weighted space is defined with the norm appearing inside the expectation \eqref{eq:stoch-weighted-linfty} --  such that for every sequence $(T_n) : T_n\rightarrow \infty$, we have
\begin{equation}
\llbracket (\overline H W^0_{T_n})^i\rrbracket\rightarrow \llbracket (\overline H W^0_\infty)^i\rrbracket \, \text{ in }L^\infty(|\tau|^{\alpha},\mathcal C^{-i/2+\alpha-\delta}_z).
\end{equation}
The estimates also hold with $\overline H$ replaced by $H$, and the single harmonic extension replaced by the double harmonic extension.
\end{proposition}

In order to prove this proposition, we will need to represent the Wick powers in terms of iterated stochastic integrals. Let us write $\overline P:\mathbb T^2 \times (\mathbb R\times \mathbb T^2) \rightarrow \mathbb R$, $(z,x) \mapsto \overline P(z,x)$ to denote the Poisson kernel corresponding to the kernel of the harmonic extension operator $\overline H$. Observe that
\begin{equation}
\overline HW^0_t = \sum_{n \in \mathbb N} \int_0^t J_s^0 dB^n_s \, \int\mathsf{e_n}(z) \overline P(z,\cdot ) dz.	
\end{equation}
For every $\tau \in \mathbb R$, let $\widehat{\overline P}_\tau(n)$ denote the Fourier transform of $\widehat{\overline P}(\tau,\cdot)$. Then 
\begin{equation}
\widehat {\overline H W_t^0(\tau,\cdot)}(n) = \int_0^t J_s^0 dB_s^n \, \widehat{\overline{P}}_\tau (n).	
\end{equation}
Hence by It\^o's formula,
\begin{equation}
\widehat{\llbracket (\overline HW_t^0)^k \rrbracket(\tau,\cdot)}(n) = \sum_{n_1+\dots+n_k = n} k!\int_{0 \leq t_1 \leq \dots \leq t_{k-1}\leq t } \prod_{j=1}^k J_{t_j}^0 dB_{t_j}^{n_j}  \, \prod_{j=1}^k \widehat{\overline{P}}_\tau(n_j).
\end{equation}
Computing the second moment (see \cite{N09} for computing second moments of iterated stochastic integrals), we find that
\begin{equation}
\mathbb E\Big[\Big|\widehat{\llbracket (\overline HW_t^0)^k \rrbracket(\tau,\cdot)}(n) \Big|^2 \Big] =  (k!)^2\int_{0\leq t_1 \leq \dots \leq t_{k-1}\leq t } \prod_{k=1}^i J_t^0(n_k)^2 dt_k \,  \prod_{j=1}^k \widehat{\overline{P}}_\tau(n_j)^2.
\end{equation}
We may now undo the simplex integration and compute these explicitly to obtain:
\begin{equation}
\mathbb E\Big[\Big|\widehat{\llbracket (\overline HW_t^0)^k \rrbracket(\tau,\cdot)}(n) \Big|^2 \Big] \leq   k!\sum_{n_1+\dots+n_k=n}\prod_{j=1}^k \frac{1}{\langle n_j \rangle }  \exp(-2|\tau| \sum_{j=1}^k \langle n_j\rangle  ),   
\end{equation}
where we have additionally used the explicit symbol of $\overline P_\tau$, which follows by straightforward Fourier space calculations thanks to the infinite cylinder. See Appendix \ref{appendix: harmonic extension estimates}.

Let us now prove the proposition. 

\begin{proof}[Proof of Proposition \ref{prop: stochastic boundary v2}]
Let us observe that the last statements of the proposition follow since the two harmonic extensions $\overline H$ and $H$ agree up to a smoothing term. It is sufficient by the martingale convergence theorem to show that these are uniformly bounded in $L^\infty(|\tau|^{\alpha},\mathcal C^{-i/2+\alpha-\delta}_z)$. Furthermore, by a Besov embedding argument, it is sufficient to take $p \in [1,\infty)$.   

Let $\tau > 0$ and let $\Delta_j$ be a Littlewood-Paley block (see Appendix \ref{appendix:besov}). Then by stationarity and hypercontractivity (see \cite{N09}),
\begin{equation}
2^{jsp}\mathbb E[\| \Delta_j  (\llbracket \overline HW_t^0)^k\rrbracket(\tau,\cdot)\|_{L^p_z}^p]	 \leq C_p \Big(2^{2js} \mathbb E[\| \Delta_j \llbracket(\overline HW_t^0)^k\rrbracket (\tau,\cdot)\|_{L^2_z}^2]\Big)^{p/2}. 
\end{equation}
Then by Parseval's theorem.
\begin{equation} \label{eq: prop bdry stoch 1}
2^{2js} \mathbb E[\| \Delta_j \llbracket(\overline HW_t^0)^k \rrbracket(\tau,\cdot)\|_{L^2_z}^2 = 2^{2js}\sum_{n} |\Delta_j(n)|^2 \mathbb E\Big[\Big|\widehat{\llbracket (\overline HW_t^0)^k \rrbracket(\tau,\cdot)}(n) \Big|^2 \Big].
\end{equation}
We may therefore bound the above by
\begin{align}
\eqref{eq: prop bdry stoch 1} &\leq k! 2^{2js}\sum_{n} |\Delta_j(n)|^2	 \sum_{n_1+\dots+n_k=n}\prod_{j=1}^k \frac{1}{\langle n_j \rangle }  \exp(-2|\tau| \sum_{j=1}^k \langle n_j\rangle  )\\&
=  C k! 2^{2js} \sum_{n} \Delta_j(n) \mathcal K_\tau \underbrace{\ast \dots \ast}_k \mathcal K_\tau (n),
\end{align}
where for any $\beta>0$ there exists $C_\beta>0$
\begin{equation}
\mathcal K_\tau(n) \leq C_\beta \frac{1}{|\tau|^\beta \langle n \rangle^{1+\beta  }},	
\end{equation}
for any $k\beta < 1$. Applying the bounds on convolutions from \cite[Appendix C.3]{TW18}, provided $\frac{1+\beta}2 > \frac{k-1}k$, we then obtain 
\begin{equation} 	
2^{2js} \sum_{n} \Delta_j(n) \mathcal K_\tau \underbrace{\ast \dots \ast}_k \mathcal K_\tau (n) \leq C 2^{2js} \sum_{n} \Delta_j(n) \frac{1}{|\tau|^{k\beta }} \frac{1}{\langle n \rangle^{k\beta+1}} \leq C'|\tau|^{-k\beta }\frac{2^{2j(1+s)}}{2^{j(k\beta + 1)}}.
\end{equation}
 Thus, for convergence, we need
 \begin{equation}
 s<-1/2 + k\beta /2,
 \end{equation}
 for every $\beta \in (\min(0,\frac{k-2}{k}),1)$. 

The convergence statement is a straightforward modification of the bounds. 	
\end{proof}

\subsection{Renormalized variational problem and uniform bounds} \label{sec: bdry unif}

Let us now return to the variational problem in \eqref{eq: bdry bd}. We begin with a deterministic regularity estimate on $Z_T^0(u)$ for $u \in \mathcal H^0=L^2_tL^2_x$ which will allow us to identify divergences. Its proof, which we omit, follows from a similar computation to \cite[Lemma 2]{BG20} and the regularizing properties of $(-\Delta_{\mathbb T^2}+m^2)^{-1/2}$. 
\begin{lemma}\label{lem: bdry a priori Z}
For every $u \in \mathcal H^0$, 
\begin{equation}
\sup_T \|Z^0_T(u)\|_{H^{1/2}_z} \leq \|u\|_{\mathcal H^0}.
\end{equation}
\end{lemma}

In light of Proposition \ref{prop: stochastic boundary v2} and Lemma \ref{lem: bdry a priori Z}, together with properties of the harmonic extension (see Appendix \ref{appendix: harmonic extension estimates}), it is clear that there is a potential divergence in the integrand involving $\llbracket (\overline H W_T^0)^3 \rrbracket$. In order to renormalize it, let us recall that
\begin{equation}
\delta^0_T:= - 3\cdot 4 \int_{[-1,1]\times \mathbb T^2} \int_{[-1,1]\times \mathbb T^2} \overline C_T^B(x,x')^4 dx dx',	
\end{equation}
where
\begin{equation}
\overline C_T^B(x,x') = \mathbb E_{\mu^0}[H \varphi_T(x) H\varphi_T(x')]. 	
\end{equation}

 For every $T>0$, let us define the map $\ell^0_T:\mathcal H^0 \rightarrow \mathcal H^0$ by  
  \begin{equation}\label{eqdef: bdry cov}
  \ell^0_T(u)_t := 
  4J_t^0 \overline H^* \llbracket (\overline HW_t^0)^3 \rrbracket \mathbbm 1_{t \leq T}+u_t, \qquad \forall t \geq 0.
  \end{equation}
  Above, $\overline H^*$ is the adjoint of $\overline H$ in $L^2([-1,1]\times \mathbb T^2)$ with kernel $\overline P^*:([-1,1]\times\mathbb T^2) \times \mathbb T^2 \rightarrow \mathbb R$. We may also extend the definition of $\ell^0_T$ to $T=\infty$, although it is no longer a map acting on $\mathcal H^0$ to itself, i.e.\ there is a loss of regularity. We will address this issue below in the context of the \emph{boundary remainder map}.    

\begin{lemma}
  \label{lem: bdry cubic} 
  For every $T>0$ and every $u \in \mathbb H^0$,
  \begin{equation}
  \mathbb E \left[ \int_{[-1,1]\times \mathbb T^2} 4\llbracket (\overline HW_T^0)^3\rrbracket \overline HZ_T^0(u)dx + \frac 12 \|u\|_{L^2_tL^2_z}^2 -\delta^0_T \right]=  \mathbb E\left[ \frac 12\|\ell^0_T(u)\|_{L^2_t L^2_z}^2 \right]. 	
  \end{equation}
\end{lemma}

\begin{proof}
By It\^o's formula, the stochastic integral being mean zero, and taking adjoints,
\begin{align}
\int_{[-1,1]\times\mathbb T^2}&4\llbracket (\overline HW_T^0)^3\rrbracket \overline HZ_T^0(u) dx 
\\
&= \int_{[-1,1]\times\mathbb T^2} \int_0^T 4\Big( \llbracket (\overline H W_t^0)^3 \rrbracket \overline H J_t^0(u_t) dt+ 4\overline HZ_t^0(u) d \llbracket (\overline HW_t^0)^3\rrbracket  \Big) dt dz
\\
&\overset{\mathbb E[\cdot]=0}{=} \int_{\mathbb T^2} \int_0^T 4 J_t^0 \overline H^* \llbracket (\overline HW_t^0)^3\rrbracket  \, u_t dt dz. 
\end{align}
Thus, by the change of variables \eqref{eqdef: bdry cov}, the lemma will follow if we can establish that 
\begin{equation} \label{eq: bdry cubic 1}
-\frac{4^2}2 \int_{\mathbb T^2} \int_0^T \mathbb E[(J_t^0 \overline H^* \llbracket (\overline HW_t^0)^3\rrbracket)^2] dt dz = \delta^0_T. 
\end{equation}
 
 Observe that
\begin{equation}
-\frac{2}{4^2}\times \text{LHS of } \eqref{eq: bdry cubic 1} = \int_0^T \int_{\mathbb T^2} \int_{\mathbb T^2} \int_{[-1,1]\times \mathbb T^2} \int_{[-1,1]\times\mathbb T^2} \mathbb (\dots) dxdx'dzdz' dt,	
\end{equation}
where
\begin{equation}
(\dots)= (J_t^0)^2(z,z') \overline P^*(z',x') \overline P^*(z,x) 	\mathbb E\Big[ \llbracket (\overline HW_t^0)^3\rrbracket (x) \llbracket (\overline HW_t^0)^3\rrbracket (x')] \Big].
\end{equation}
We will use that on functions mapping $ \mathbb R\times \mathbb T^2 \rightarrow \mathbb R$ that are compactly supported on $[-1,1]\times \mathbb T^2$, the (action of the) adjoint of $\overline H$ is equal to the adjoint of $\overline H$ in $L^2(\mathbb R\times \mathbb T^2)$ restricted to $L^2([-1,1]\times \mathbb T^2)$. Thus, for every $x,x' \in [-1,1]\times\mathbb T^2$,

\begin{equation}
\int_{\mathbb T^2} \int_{\mathbb T^2} (J_t^0)^2(z,z') \overline P^*(z',x') \overline P^*(z,x)  dz dz' = \dot{\overline C}_t^B(x,x').  	
\end{equation}
Hence, by standard properties of Wick powers,
\begin{equation}
-\frac{2}{4^2} \times \text{LHS of } \eqref{eq: bdry cubic 1} = \int_0^T \int_{[-1,1]\times\mathbb T^2}\int_{[-1,1]\times \mathbb T^2}\dot{\overline C}_t^B(x,x') \overline C_t^B(x,x')^3 dx dx'= -\frac{2}{4^2}\delta^0_T,	
\end{equation}
which upon rearrangement yields \eqref{eq: bdry cubic 1}. 
\end{proof}

\subsubsection{Definition of the renormalized variational problem}

Let us now define a renormalized variational problem. First, let us introduce some notation. For every $T \in (0,\infty]$, define the vector of stochastic processes with distribution-valued coordinates,
\begin{equation}
\Xi^0_T(X^0):= \left( (\overline HW_t^0)_{t \leq T}, (\llbracket (\overline HW_t^0)^2\rrbracket)_{t\leq T}, (\llbracket (\overline H W_t^0)^3 \rrbracket)_{t\leq T}  \right).
\end{equation} 
We stress that thanks to Proposition \ref{prop: stochastic boundary v2}, we include the case $T=\infty$.  Consider the Banach space
\begin{equation}
\boldsymbol{\mathfrak B}^0:=C_tL^4_\tau \mathcal C^{-1/2-\kappa_0'}_z \times C_t L^8_\tau \mathcal C^{-1/2-\kappa_0'}_z \times C_tL^1_\tau \mathcal C^{-1/2-\kappa_0'}_z,  
\end{equation} 
where $\kappa_0'>0$ is sufficiently small and fixed a posteriori.  Proposition \ref{prop: stochastic boundary v2} and properties of the harmonic extension implies that $\Xi^0_T(X^0) \in \boldsymbol{\mathfrak B}^0$ almost surely. In fact, we have the following convergence result. 

\begin{proposition}\label{prop: bdry stochastic}
For every $p \in [1,\infty)$, 
\begin{equation}
\lim_{T \rightarrow \infty}\mathbb E\left [ \| \Xi^0_T(X^0) - \Xi^0_\infty(X^{0}) \|_{\boldsymbol{\mathfrak B}^0}^p \right] = 0. 
\end{equation}
	
\end{proposition}

\begin{remark}
We will typically denote an element of $\boldsymbol{\mathfrak B}^0$ by $\Xi^0=(\Xi^{0,1},\Xi^{0,2},\Xi^{0,3})$. To distinguish this from the stochastic process, we will always keep the dependecy of $\Xi^0_T(X^0)$ on the underlying $L^2$-Brownian motion $X^0$ in the notation. Let us also remark that there is flexibility in the exact choice of $\boldsymbol{\mathfrak B}^0$ and our choice is motivated by the analytic estimates below. 	
\end{remark}

Let us now define the \emph{boundary remainder map}, which amounts to modifying the definition of $\ell^0_T$ to allow for the case  $T=\infty$. Let us introduce the Banach space 
\begin{equation}
\mathcal L^0:= L^2_t W^{1/2-\kappa_0',2+\kappa_0''}_z,	
\end{equation}
 where for every $\kappa_0'>0$ sufficiently small, $\kappa_0''>0$ is chosen so that $H^{1/2}_z \hookrightarrow W^{1/2-\kappa_0', 2+\kappa_0''}_z \hookrightarrow L^4_z$. For every $T\in (0,\infty]$, the boundary remainder map is the map $\ell^0_T:\boldsymbol{\mathfrak B}^0 \times \mathcal L^0 \rightarrow \mathcal L^0$ defined by
\begin{equation}
\ell^0_T(\Xi^0, u)_t:= 
  4J_t^0 \overline H^* \Xi^{0,3}_t \mathbbm 1_{t \leq T}+u_t, \qquad \forall t \geq 0.
\end{equation}
In the case $T<\infty$ and  $\Xi^0 = \Xi^0_T(X^0)$, we recover \eqref{eqdef: bdry cov} and simply write $\ell^0_T(\Xi^0_T(X^0),u)=\ell^0_T(u)$. 

For every $T\in(0,\infty]$, the \emph{residual boundary potential} is the map $\Phi^0_T:\boldsymbol{\mathfrak B}^0 \times \mathcal L^0 \rightarrow \mathbb R$  defined by 
\begin{equation}
\Phi_T^0(\Xi^0,u):= \int_{[-1,1]\times\mathbb T^2	} 6\Xi^{0,2} (\overline H Z_T^0(u))^2 + 4\Xi^{0,1} (\overline H Z_T^0(u))^3 dx, \qquad \forall (\Xi,u) \in \boldsymbol{\mathfrak B}^0\times \mathcal L^0.
\end{equation}
Additionally, let $F^0_T:\boldsymbol{\mathfrak B}^0 \times \mathcal{L}^0 \rightarrow \mathbb R$ denote the map defined for every $ (\Xi^0,u) \in \boldsymbol{\mathfrak B}^0\times \mathcal L^0$ by 
\begin{equation}
F^{0}_T(\Xi^0,u):= \Phi^0_T(\Xi^0,u) + \|\overline HZ_T^0(u)\|_{L^4([-1,1]\times\mathbb T^2)}^4 + \frac 12 \|\ell^0_T(\Xi^0,u)\|_{L^2_tL^2_z}^2.	
\end{equation}
The map $F^0_T$ consists of the residual boundary potential together with positive terms -- corresponding to an \emph{effective potential} and \emph{effective relative entropy} --  that will help us analyze the variational problem below.
Given $f \in C^\infty(\mathbb T^2)$, let $F^{0,f}_T:\boldsymbol{\mathfrak B}^0\times \mathcal{L}^{0} \rightarrow \mathbb R$ be the map defined by
\begin{equation}
	F^{0,f}_T(\Xi^0,u):= \langle f, Z^0_T(u)\rangle_{L^2_z} + F^0_T(\Xi^0,u), \qquad \forall (\Xi^0,u) \in \boldsymbol{\mathfrak B}^0\times \mathcal L^0.
\end{equation}

Write $\mathbb L^0$ to denote the space of progressively measurable processes that are almost surely in $\mathcal{L}^{0}$. 
\begin{definition}For every $T\in(0,\infty]$ and $f \in C^\infty(\mathbb T^2)$, the \emph{renormalized boundary cost function} is the map  $\mathbb F^{0,f}_T: \mathbb L^0 \rightarrow \mathbb R$ defined by
\begin{equation}\mathbb F^{0,f}_T(u) = \mathbb E\left [ F^{0,f}_T(\Xi^0_T(X^0), u) \right], \qquad \forall u \in \mathbb L^0.  
\end{equation}
In the case $f=0$,  $F^{0,0}_T=F^0_T$ and we write $\mathbb F^{0,0}_T=\mathbb F^0_T$. 
\end{definition}

As a direct consequence of Lemma \ref{lem: bdry cubic} and the definitions above, we obtain the following renormalized variational problem when $T<\infty$. 

\begin{proposition}\label{prop: bdry renormalized variationalproblem}

For every $T>0$ and $f \in C^\infty(\mathbb T^2)$,
\begin{equation} \label{eq: prop bdry unif 1}
-\log\mathcal Z^0_T(f)
=
\inf_{u \in \mathbb L^0} \mathbb F^{0,f}_T(u).	
\end{equation}
\end{proposition}

For the remainder of the section, we will analyze the convergence of the infima in the above renormalized variational problem. 

\subsubsection{Bounds on the residual boundary potential}

As a justification that we now have the "correct" renormalized problem, we will start by showing that the residual boundary potential contains no further (analytic) divergences. In particular, we show that the residual boundary potential admits deterministic bounds in terms of the effective potential and relative entropy terms in $F^0_T$ up to an error that can be controlled by stochastic estimates (c.f. Proposition \ref{prop: stochastic boundary v2}). We will actually show these bounds also hold for the residual boundary potential at $T=\infty$, even though we have not yet defined the variational problem at $T=\infty$. Let us write $\mathcal L^0_w$ and $\mathcal H^0_w$ for the spaces $\mathcal L^0$ and $\mathcal H^0$ endowed with their respective weak topologies.

\begin{proposition}\label{prop: bdry unif}
There exists $p_0\in[1,\infty)$ such that for every $\delta > 0$, there exists $C_\delta > 0$  such that for every $T\in(0,\infty]$ and $(\Xi^{0},u) \in \boldsymbol{\mathfrak B}^0\times \mathcal L^0$,\begin{equation} \label{eq: prop bdry unif 2}
|\Phi^0_T(\Xi^0,u)| \leq C_\delta \|\Xi^0\|_{\boldsymbol{\mathfrak B}^0}^{p_0} + \delta \left(  \|\overline H Z_T^0(u)\|_{L^4([-1,1]\times \mathbb T^2)}^4 +\|\ell^0_T(\Xi^0,u)\|_{\mathcal H^0}^2 \right). 	
\end{equation}
Furthermore, assume that $\Xi_T^0 \rightarrow \Xi_\infty^0$ in $\boldsymbol{\mathfrak B}^0$, $u^T \rightharpoonup u^\infty$ in $\mathcal L_w^0$, and $\ell^0_T(\Xi_T^0,u^T) \rightharpoonup \ell^0_\infty(\Xi_\infty^0, u^\infty)$ in $\mathcal H^0_w$. Then
\begin{equation} \label{eq: prop bdry unif 3}
\lim_{T\rightarrow \infty} \Phi^0_T(\Xi_T^0, u^T) = \Phi^0_\infty(\Xi_\infty^0, u^\infty). 
\end{equation}  
\end{proposition}

We begin with a preliminary analytical lemma.  

\begin{lemma}
  \label{lem: bdry quadratic}
For every $\delta > 0$ sufficiently small there exists $C>0$ such that, for every $T > 0$,
\begin{multline}
\Big|\int_{[-1,1]\times \mathbb T^2} \llbracket (\overline HW_T^0)^2 \rrbracket  (\overline HZ_T^0(u))^2 dx \Big| 
\\
\leq
C\||\tau|^{1/2 - \delta}	\llbracket (\overline H W_T^0)^2 (\tau,\cdot)\rrbracket\|_{L^4_\tau \mathcal C^{-1/2+\delta}_z} \| |\tau|^{-1/2+\delta} \overline H Z_T^0(u)(\tau,\cdot) \|_{L^2_\tau H^{1/2-\delta}_z} \|\overline HZ_T^0(u)\|_{L^4},
\end{multline}
and
\begin{equation}
\Big| \int_{[-1,1]\times\mathbb T^2} \overline HW_T^0 \cdot (\overline HZ_T^0(u))^3 dx \Big| \leq C\|\overline HW_T^0\|_{L^{8(1-\delta)}_\tau \mathcal C^{-1/2+\delta}_z} \| \overline H Z_T^0(u)\|^{
\frac{1/2-\delta}{1-\delta}}_{L^2_\tau H^{1-\delta}_z} \|\overline H Z_T^0(u)\|^{{2+\frac{1/2}{1-\delta}}}_{L^4}.
\end{equation}
\end{lemma}

\begin{proof}
Consider the first inequality. We will split the integral on $[-1,1]\times \mathbb T^2$ into an integral over $[-1,1]$ and an integral on the $\tau+\mathbb T^2$ slices. Fix such a $\tau \neq 0$. Then by duality and the fractional Leibniz rule, 
\begin{align}
\begin{split}
 \label{eq: bdry quadratic 1}
\Big|\int_{\mathbb T^2} &\llbracket (\overline HW_T^0)^2(\tau,\cdot) \rrbracket (\overline HZ_T^0(u)(\tau,\cdot))^2 dz\Big|	 \\& \qquad \leq \| \llbracket (\overline HW_T^0)^2(\tau,\cdot) \rrbracket\|_{\mathcal C^{-1/2+\delta}_z} \|\overline H Z_T^0(u)(\tau,\cdot)\|_{H^{1/2-\delta}_z} \|\overline H Z_T^0(u)(\tau,\cdot)\|_{L^4_z}.
\end{split}
\end{align}
Hence
\begin{equation}
\eqref{eq: bdry quadratic 1} \leq |\tau|^{1/2-\delta}\| \llbracket (\overline HW_T^0)^2(\tau,\cdot) \rrbracket\|_{\mathcal C^{-1/2+\delta}_z} \frac{1}{|\tau|^{1/2-\delta }}\|\overline H Z_T^0(u)(\tau,\cdot)\|_{H^{1/2-\delta}_z} \|\overline H Z_T^0(u)(\tau,\cdot)\|_{L^4_z}
\end{equation}
The first inequality now follows from integrating in the $\tau$-variable and using H\"older's inequality.

Let us now analyze the second inequality. Fix a $\tau$-slice as above. By duality, the fractional Leibniz rule applied twice, and interpolation, we have
\begin{align}
\Big| \int_{\mathbb T^2} &\overline HW_T^0(\tau,\cdot) \cdot (\overline H Z_T^0(u)(\tau,\cdot))^3 dz \Big| \\&\qquad \leq C \| \overline HW_T^0(\tau,\cdot)\|_{\mathcal C^{-1/2+\delta}_z} \|\overline H Z_T^0(u)(\tau,\cdot) \|_{H^{1-\delta}_z}^{\frac{1/2-\delta}{1-\delta}}\|\overline H Z_T^0(u)(\tau,\cdot)\|_{L^4_z}^{2+\frac{1/2}{1-\delta}}.
\end{align}
The second inequality follows by integrating in the $\tau$ variable and using H\"older's inequality.
\end{proof}

Let us now turn to the proof of Proposition \ref{prop: bdry unif}. 
\begin{proof}[Proof of Proposition \ref{prop: bdry unif}]
By Lemma \ref{lem: bdry cubic}, we have that \eqref{eq: prop bdry unif 1} holds with $\ell^0_T$ as above and 
\begin{equation}
\Phi^0_T= \int_{[-1,1]\times \mathbb T^2} 6 \llbracket (\overline H W_T^0)^2 \rrbracket (\overline H Z_T^0(u))^2 + 4 \overline HW_T^0 \cdot (\overline H Z_T^0(u))^3 dx.	
\end{equation}
In addition, by Lemma \ref{lem: bdry quadratic} we have that
\begin{align}
|\Phi^0_T| &\leq C\||\tau|^{1/2 - \delta}	\llbracket (\overline H W_T^0)^2 (\tau,\cdot)\rrbracket\|_{L^4_\tau \mathcal C^{-1/2+\delta}_z} 
\\ &\qquad \qquad \times \| |\tau|^{-1/2+\delta} \overline H Z_T^0(u)(\tau,\cdot) \|_{L^2_\tau H^{1/2-\delta}_z} \|\overline H Z_T^0(u)\|_{L^4([-1,1]\times\mathbb T^2)}
\\
&\quad + C\|\overline HW_T^0\|_{L^{8(1-\delta)}_\tau \mathcal C^{-1/2+\delta}_z} \| \overline H Z_T^0(u)\|_{L^2_\tau H^{1-\delta}_z} \|\overline H Z_T^0(u)\|_{L^4([-1,1]\times\mathbb T^2)}.	
\end{align}

Let us define
\begin{equation}
\mathbb W^{0,3}_T:=Z_T^0(J_\bullet ^0 \overline H^* \llbracket (HW_\bullet ^0)^3 \rrbracket )	
\end{equation}
and hence
\begin{equation}
\overline H Z_T (u)  =  - 4\overline  H\mathbb{W}_T^{0.3} + \overline H Z_T (\ell_{T}(u)).
\end{equation}
Using the
   the properties of $Z_T^0$, and that $\tau \mapsto \tau^{-1/2+\delta}$ is in $L^2([-1,1])$, we
  obtain
  \begin{equation}
   \||\tau|^{-1/2+\delta} Z_T^0(u)\|_{L^2_{\tau} H^{1/2-\delta}_z} \leq C( \|\mathbb W^{0,3}_T\|_{\mathcal C^{1/2-\delta}_z} + \|\ell_{T}^0(u)\|_{L^2_tL^2_z}). 	
  \end{equation}
  
  Furthermore, by \eqref{tool: H reg decay} and the estimate above, 
  \begin{equation}
  \|\overline H Z_T^0(u)\|_{L^2_\tau H^{1-\delta}_z}  \leq C\Big(\| \mathbb W^{0,3}_T\|_{\mathcal C^{1/2-2\delta}_z} +\|\ell_{T}^0(u)\|_{L^2_tL^2_z}\Big). 	
  \end{equation}

 Thus by applying Young's inequality to the estimates from Lemma \ref{lem: bdry quadratic}, we have shown that for any $\delta > 0$ sufficiently small, there exists $c_{\delta} > 0$
  and $p, q \in (1, \infty)$ such that
  \begin{align}
  |\Phi^0_T| &\leq C_\delta \Big( 	\|
   \overline  H W_T^0 \|^{p}_{L^{8 (1 - \delta)}_\tau \mathcal C_{z}^{- 1 / 2 + \delta}}
+ \||\tau|^{1/2-\delta} \llbracket (\overline H W_T^0)^2
    \rrbracket \|_{L^4_{\tau} \mathcal C^{- 1 / 2 + \delta}_{z}}^{q} +     \| \mathbb{W}_T^{0,3} \|_{\mathcal C^{1 / 2 - \delta}_{z}}^2  \Big)
    \\
    &\qquad
    + \delta \Big( \|\overline H Z_T^0(u)\|_{L^4([-1,1]\times\mathbb T^2)}^4 +  \frac 12 \|\ell_{T}^0(u)\|_{L^2_t L^2_z}^2 \Big). 
  \end{align}
 Now using the difinition of $\Xi^0_T$ we have thus established \eqref{eq: prop bdry unif 2}. 
  
The final property \eqref{eq: prop bdry unif 3} follows from Lemma \ref{lem: bdry quadratic} and the observation that $\Phi^0_T$ is a sum of multilinear terms in $\Xi^0_T, Z^0_T$. 
\end{proof}

\subsubsection{Limiting renormalized variational problem and ultraviolet stability}

Let us now define the natural limiting Laplace transform $\mathcal Z^0_\infty(f)$ in terms of a variational problem.
\begin{definition}
For every $f \in C^\infty(\mathbb T^2)$, define $\mathcal Z^0_\infty(f)$ by the variational formula 
\begin{equation}\label{eqdef: limiting ren var prob}
	\mathcal -\log \mathcal Z^0_\infty(f):= \inf_{u \in \mathbb L^0}\mathbb F^{0,f}_\infty(u). 
\end{equation}	
Here, we set the infimum to be $\infty$ if it does not exist. 
\end{definition}

Let us now show that the bounds we have already obtained allow us to bound $\mathcal Z^0_T(f)$ uniformly in $T$ -- so called \emph{ultraviolet stability} for the measures $(\nu^0_T)$. From this, one can deduce tightness of the measures $(\nu^0_T)$. Note that this is insufficient for our purposes, we require full convergence. 

\begin{proposition} \label{prop: bdry UV stable}
There exists $c,C>0$ such that for every $T \in (0,\infty]$ and every $f \in C^\infty(\mathbb T^2)$,
\begin{equation}
	\exp(-C-c\|f\|_{H^{-1/2+\kappa_0'}_z}^2 ) \leq \mathcal Z^0_T(f) \leq \exp(C+c\|f\|_{H^{-1/2+\kappa_0'}_z}^2). 
\end{equation}
\end{proposition}

\begin{proof}
Let us assume $f=0$ as the case of $f$ non-zero follows easily from this case by additionally using Lemma \ref{lem: bdry a priori Z} and duality. Let $\delta > 0$ be sufficiently small. Then by Propositions \ref{prop: bdry renormalized variationalproblem} and \ref{prop: bdry unif}, together with the definition of $\mathcal Z^0_\infty$ in \eqref{eqdef: limiting ren var prob}, we have that there exists $C_\delta>0$  and $p_0 \in [1,\infty)$ such that for every $u \in \mathbb L^0$,
\begin{equation} \label{eq: bdry infimum bound}
\mathbb F^{0,f}_T(u)  \geq -C_\delta \mathbb E[\|\Xi^0\|^{p_0}_{\boldsymbol{\mathfrak B}^0}]+ (1-\delta)\mathbb E\left[ \|\overline H Z_T^0(u)\|_{L^4([-1,1]\times \mathbb T^2)}^4+\frac 12 \|\ell^0_T(u)\|_{\mathcal H}^2 \right].
\end{equation}
In particular, by Proposition \ref{prop: bdry stochastic} and positivity of the terms on the righthand side, there exists $C>0$ such that
\begin{equation}
	-\log \mathcal Z^0_T(f) \geq -C.
\end{equation}

Let us now consider the reverse inequality. We may simply choose $u$ such that $\ell^0_T(u)=0$. It is then not hard to show, by similar arguments (except we use Proposition \ref{prop: bdry stochastic} again instead of positivity of terms) that, up to a redefinition of $C$,
\begin{equation}
	-\log \mathcal Z^0_T(f) \leq C. 
\end{equation}

\end{proof}

\subsection{Convergence of $\mathcal Z^0_T(f)$ and proof of Theorem \ref{thm: boundary}}
\label{sec:boundary:convergence of LT}

Let us not turn our attention to the convergence of $\mathcal Z^0_T(f)$. By Proposition \ref{prop: bdry renormalized variationalproblem}, we have reduced it to showing that the infima of the renormalized boundary cost functions converge. The renormalized cost functions have themselves a natural pointwise limit  (c.f. the second statement in Proposition \ref{prop: bdry unif}) and therefore we will show that the infima at $T<\infty$ converge to the infima of the latter (at $T=\infty$). We will do so by using the framework of $\Gamma$-convergence. In order to facilitate this, at the cost of certain technicalities, we will instead work with the space of joint laws of $(\Xi, u)$ since it has better compactness properties.

 Let us consider the following space of joint laws 
\begin{equation}
 \mathcal{X}^0:= \{ \mu \in \mathcal{M}_1( \boldsymbol{\mathfrak B}^0\times \mathcal{L}^{0}_w) : \exists u \in \mathbb{L}^0, \,  \mu = {\rm Law}_{\mathbb P}(\Xi^0_\infty(X), u) \text{ and } \mathbb E[\|u\|_{\mathcal{L}^0}^2]<\infty \}.	
\end{equation}
Let us emphasize that $\boldsymbol{\mathfrak B}^0$ is endowed with the strong topology, $\mathcal L^{0}_w$ is the space $\mathcal {L}^{0}$ endowed with the weak topology, and the laws are considered on the product space (equipped with the product topology). We will view $\mathcal X^0 \subset \mathcal Y^0$, where
 \begin{equation}
 \mathcal{Y}^0 := \{ \mu \in\mathcal{M}_1( \boldsymbol{\mathfrak B}^0\times \mathcal{L}^{0}_w) : \mathbb E_\mu \|u\|_{\mathcal{L}^2}^2 < \infty \}.  	
 \end{equation}
We endow $\mathcal X^0$ with the following notion of convergence. A sequence $(\mu_n) \subset \mathcal X^0$ converges to a probability measure $\mu \in \mathcal Y^0$, written as $\mu_n \rightarrow \mu$, if: 
 \begin{enumerate}
 \item[(i)] The measures converge weakly $\mu_n \rightharpoonup \mu$ in $\boldsymbol{\mathfrak B}^0\times \mathcal{L}^{0}_w$.
 \item[(ii)] The second moments are uniformly bounded: $\sup_n \mathbb E_{\mu_{n}} \|u\|_{\mathcal {L}^{0}}^2 < \infty$. 
 \end{enumerate}
 We let $\overline{\mathcal X}^0 \subset \mathcal Y^0$ denote the sequential closure of $\mathcal X^0$ in the space $\mathcal Y^0$. 

Let us now embed the renormalized variational problem in this space of laws. For every $T\in (0,\infty]$, define the functions ${\tilde {\mathbb F}}^{0,f}_T : \mathcal{Y}^0 \rightarrow \mathbb R$,
\begin{equation}
\tilde{\mathbb F}^{0,f}_T(\mu)
=
\mathbb E_\mu \Big[ F^{0,f}_T(\Xi^0_T(X^0),u)\Big], \qquad (\Xi^0(X^0),u)\sim \mu. 	
\end{equation}
As above, we set $\tilde{\mathbb F}^{0,0}_T =: \tilde{\mathbb F}^0_T$.
It is easy to see that for every $T \in (0,\infty]$,
\begin{equation} \label{eq: boundary v to mu}
\inf_{u \in \mathbb{L}^0} \mathbb F^{0,f}_T(u) = \inf_{\mu \in \mathcal{X}^0} {\tilde{\mathbb F}}^{0,f}_T(\mu). 	
\end{equation}

In order to prove Theorem \ref{thm: boundary}, we will show that $\mathcal Z^0_\infty(f)$ is nontrivial for every $f \in C^\infty(\mathbb T^2)$ and that the infima of $\tilde{\mathbb F}^{0,f}_T$ converge to the infima of $\tilde{\mathbb F}^{0,f}_\infty$. We aim to show that the infima are all located on the \emph{same} sequentially compact set as $T\rightarrow \infty$.

We will begin by showing that the infimum can be taken over the closure for every $T\in (0,\infty]$. 
This will be contained in two technical lemmas whose proof we will postpone to Section \ref{sec: boundary technical}. 
The first lemma deals with the case $T<\infty$ and mostly concerns regularity properties of the renormalized cost function under the notion of sequential convergence. 
The main difficulty with these types of regularity under approximation results is that our setup, which was formulated to treat the case $T\rightarrow \infty$, means that a priori any element in $\overline{\mathcal X}^0$ can be approximated with measures in $\mathcal X^0$, but only with uniform control in $\mathcal L^{0}$. This is problematic given the $\mathcal H$-norm term in the cost function. In the case $T<\infty$, however, the variational problem is not singular  -- we recall that it is supported on $\mathbb H^0$ (almost sure regularity in $\mathcal H$) and not just $\mathbb L^0$ (almost sure regularity in $\mathcal {L}^0$).
 
\begin{lemma}\label{lem:bdry-approx-finite}
Let $T>0$. For every $\mu \in \overline{\mathcal X}^0$ and every $K \in \mathbb N^*$, there exists $\mu_K \in \mathcal X^0$ such that  $\mu_K \rightarrow \mu$ and
\begin{equation}\label{eq: bdry mu-to-muK-Tfinite}
|\tilde{\mathbb F}^{0,f}_T(\mu) - \tilde{\mathbb F}^{0,f}_T(\mu_K)| < 1/K.
\end{equation}
\end{lemma}

We now turn to the second technical lemma, which addresses the case $T=\infty$. In this case, we cannot use any regularity boost to $u$ since the remainder term $\ell^0_\infty$ is no longer a mapping on $\mathcal H$. We will therefore have to do a more delicate analysis of the regularizing properties of $\ell^0_\infty$ --  see Section \ref{sec: boundary technical}, particularly Lemma \ref{lem: bdry rem properties}. 
\begin{lemma}\label{lem:bdry-approx-infty}
For every $\mu \in \overline{\mathcal X}^0$ and every $K \in \mathbb N^*$, there exists $\mu_K\in \mathcal X^0$ such that $\mu_K \rightarrow \mu$ and
\begin{equation}\label{eq: bdry mu-to-muK}
|\widetilde{\mathbb F}^{0,f}_\infty (\mu) - \widetilde{\mathbb F}^{0,f}_\infty(\mu_K)| < 1/K.
\end{equation}
and 
\begin{equation} \lim_{T\to \infty}{\mathbb F}^{0,f}_T (\mu_{K})={\mathbb F}^{0,f}_\infty (\mu_{K})\label{eq:T to infty}\end{equation}
\end{lemma}

As a direct corollary of  Lemmas \ref{lem:bdry-approx-finite} and \ref{lem:bdry-approx-infty}, we the minimization can be taken over the closure at both $T<\infty$ and $T=\infty$. 
\begin{corollary} \label{cor: boundary min closure}
For every $T \in (0,\infty]$, we have that
\begin{equation}
\inf_{\mu \in \mathcal{X}^0} {\tilde{\mathbb F}}^{0,f}_T(\mu) = \inf_{\mu \in \overline{\mathcal{X}}^0} {\tilde{\mathbb F}}^{0,f}_T(\mu).
\end{equation}
\end{corollary}

Let us now show that the infima of these functions concentrate on the same sequentially compact set, i.e.\ that they form an equicoercive family. For every $\mathrm C^0>0$, define the set $\mathcal{K}^0 = \mathcal K^0(\mathrm C^0) \subset \overline{\mathcal{X}}^0$ defined by
\begin{equation}
\mathcal{K}^0 := \{ \mu \in \overline{\mathcal{X}}^0 : \mathbb E_\mu[\|u\|_{\mathcal {L}^0}^2]\leq \mathrm C^0 \}.	
\end{equation}
We claim that this
is sequentially compact in $\overline{\mathcal X}^0$. In order to see this, let $(\mu_n) \subset \mathcal K^0$. Let $\boldsymbol{\mathfrak B}^0_{+}\subset \boldsymbol{\mathfrak B}^0$ denote the Banach space
\begin{equation}
\boldsymbol{\mathfrak B}^0_{+}	:= C_tL^4_\tau \mathcal C^{-1/2-\kappa_0'/2}_z \times C_t L^8_\tau \mathcal C^{-1/2-\kappa_0'/2}_z \times C_t L^1_\tau \mathcal C^{-\kappa_0'/2}_z
\end{equation} Note that the balls $B_{\boldsymbol{\mathfrak B}^0_+}(0,L) \times B_{\mathcal L^{0}}(0,L)$ are compact in $\boldsymbol{\mathfrak B}^0\times \mathcal{L}_w^0$ for every $L >0$ by standard compact embedding theorems for Besov spaces (see \cite{BCD11}) and the Banach-Alaouglu theorem. Hence by Markov's inequality, the family of measures $(\mu_n)$ are tight and therefore sequentially compact by Prokhorov's theorem \cite[Theorem 8.6.7]{B98}. Let $\mu$ be the subsequential limit in the weak topology. Clearly $\mu \in \mathcal K^0$ since $\mathbb E_\mu[\|u\|_{\mathcal L^{0}}^2] \leq \mathrm C^0$ by Fatou's lemma.

\begin{lemma}\label{lem: boundary equicoercive}
There exists $\mathrm C^0>0$ such that for every $T\in(0,\infty] $, 
\begin{equation}
\inf_{\mu \in \mathcal X^0} \tilde{\mathbb F}^{0,f}_T(\mu) = 	\inf_{\mu \in \mathcal K^0} \tilde{\mathbb F}^{0,f}_T(\mu),
\end{equation}
where $\mathcal K^0=\mathcal K^0(\mathrm C^0)$ is as above. 
\end{lemma}

\begin{proof}
Note that by Sobolev embedding and Proposition \ref{prop: bdry stochastic}, there exists $C>0$ such that for every $T\in(0,\infty]$, we have
\begin{equation}
\mathbb E_\mu [\|u\|_{\mathcal {L}^0}^2] \leq C + \mathbb E_\mu [ \| \ell^0_\infty(u)\|_{\mathcal {H}^0}^2].
\end{equation}
Moreover, by Proposition \ref{prop: bdry unif}, particularly \eqref{eq: bdry infimum bound}, there exists $C'>0$ such that
\begin{equation}
\mathbb E_{\mu}[\|\ell^0_\infty(u)\|_{\mathcal {H}}^2] \leq 2\tilde{\mathbb F}^{0,f}_T(\mu) + C'.
\end{equation}
Hence we may deduce the existence of $\mathrm C^0$ and therefore the compact set $\mathcal K^0$.
\end{proof}

We now establish that the family of functions $(\tilde{\mathbb F}^{0,f}_T)$ $\Gamma$-converge to $\tilde{\mathbb F}^{0,f}_\infty$. This consists of proving two conditions: i) a Fatou-type inequality \eqref{eq:fatou}, and ii) existence of a recovery sequence \eqref{eq:recovery}. 

\begin{lemma} \label{lem: boundary fatou}
Let $\mu \in \overline{\mathcal X}^0$. For every sequence  $ (\mu_n)  \subset \overline {\mathcal{X}}^0$ such that $\mu_n \rightarrow \mu \in \overline {\mathcal X}^0$ and every sequence $(T_n): T_n \rightarrow \infty$,
\begin{equation}
\tilde{\mathbb F}^{0,f}_{\infty}(\mu) \leq \liminf_{n \rightarrow \infty} \tilde{\mathbb F}^{0,f}_{T_n}(\mu_n). 	\label{eq:fatou}
\end{equation}
Furthermore, for every sequence $(T_n): T_n \rightarrow \infty$ there exists a sequence $(\mu_n)\subset \overline{\mathcal X}^0 : \mu_n \rightarrow \mu$ such that
\begin{equation}
\tilde{\mathbb F}^{0,f}_\infty(\mu) \geq \limsup_{n \rightarrow \infty} \tilde{\mathbb{F}}^{0,f}_{T_n}(\mu_n). 	\label{eq:recovery}
\end{equation}
Thus $\tilde{\mathbb F}^{0,f}_T$ $\Gamma$-converges to $\tilde{\mathbb F}^{0,f}_{\infty}$.
\end{lemma}

\begin{proof}

We will treat the case $f=0$ since the general case follows with minor modifications. Fix $\mu \in \overline{\mathcal X}^0$ Let us first show the existence of a recovery sequence \eqref{eq:recovery}. By Lemma \ref{lem:bdry-approx-infty} there exists a sequence  $(\mu_{K})\subset \mathcal X^0$ such that $\mu_K \rightarrow \mu$ and $\lim_{T\to \infty}|\tilde{\mathbb{F}}^0_{T}(\mu_{K})-\tilde{\mathbb{F}}^0_{\infty}(\mu_{K})|=0$ for every $K \in \mathbb N^*$ and $\lim_{K\to \infty}|\tilde{\mathbb F}^0_{\infty}(\mu_{K})-\tilde{\mathbb F}^0_{\infty}(\mu)|=0$. In particular, given any sequence $(T_n) : T_n \rightarrow \infty$, we can construct a subsequence of the $(\mu_K)$ satisfying \eqref{eq:recovery}.

We now turn to the Fatou-type inequality \eqref{eq:fatou}. Let $(\mu_n) \subset \overline{\mathcal X}^0$ be any sequence such that $\mu_n \rightarrow \mu$ and let us also fix $(T_n): T_n \rightarrow \infty$. Since $\mu_n \rightarrow \mu$, we have that ${\rm Law}_{\mu_n}(u) \rightarrow {\rm Law}_{\mu}(u)$ weakly on $\mathcal {L}^0_w$. Furthermore, the marginals on $\boldsymbol{\mathfrak B}^0$ are exactly the same for every $n$ -- in particular Proposition \ref{prop: bdry stochastic} applies. We need to address convergence of the remainder terms $\ell^0_T$ -- c.f. the conditions for pointwise convergence of the boundary residual potential in Proposition \ref{prop: bdry unif}. We may assume that $\liminf_{n \to\infty}\tilde{\mathbb{F}}_{T_n}^{0}(\mu_{n})<\infty$, otherwise the statement is trivial. 
Without loss of generality we may pass to a subsequence realizing the liminf (which we do not relabel). Thus we have that $\sup_n \tilde{\mathbb F}^0_{T_n}(\mu_n) < \infty$, which implies that
\begin{equation}
\sup_n \mathbb E_{\mu_n} \left[ \| \overline HZ^0_{T_n}(u) \|^4_{L^4([-1,1]\times \mathbb T^2)} + \|\ell^0_{T_n}(u)\|_{\mathcal H^0}^2\right]<\infty. 	
\end{equation}
Since we already have ${\rm Law}_{\mu_n}(\ell^0_T(u)) \rightarrow {\rm Law}_{\mu}(\ell^0_\infty(u))$ weakly in $\mathcal L^0_w$, by compactness we now have ${\rm Law}_{\mu_n}(\ell^0_T(u)) \rightarrow {\rm Law}_\mu(\ell^0_\infty(u))$ weakly in $\mathcal H^0_w$.

By Skorokhod's theorem \cite{J98}, up to a subsequence (not relabelled), there exists a probability space $(\tilde \Omega , \tilde{\mathcal F}, \tilde {\mathbb P})$ and random variables $\tilde \Xi^0_{T_n}, \tilde \Xi^0_\infty, \tilde u_n, \tilde u,   \tilde\ell^0_{T_n}, \tilde \ell^0_\infty$ such that we have
\begin{align}{\rm Law}_{\mu_n}( \Xi^0_{T_n}(X^0),u, \ell^0_{T_n}(u)) &= {\rm Law}_{\tilde{\mathbb P}}(\tilde \Xi^0_{T_n},\tilde u_n, \tilde\ell^0_{T_n}), \\ {\rm Law}_{\mu}(\Xi^0_\infty(X^0), u, \ell^0_\infty(u))&={\rm Law}_{\tilde{\mathbb P}}(\tilde \Xi^0_\infty, \tilde u, \tilde \ell^0_\infty),
\end{align}
and additionally such that $\tilde \Xi^0_{T_n} \rightarrow \tilde \Xi^0_\infty$ almost surely in $\boldsymbol{\mathfrak B}^0$, $\tilde u_n \rightarrow \tilde u$ almost surely in $\mathcal L^0_w$,  and $\tilde \ell^0_{T_n} \rightarrow \tilde\ell^0_\infty$ almost surely in $\mathcal H^0_w$. Furthermore, we have the following equalities almost surely. 
\begin{equation} \label{eq: bdry skorokhod drift}
	(\tilde u_n)_t:= -4J_t^0 \overline H^* (\tilde \Xi^{0,3}_{T_n})_t \mathbbm 1_{t \leq T_n} + (\tilde \ell^0_{T_n})_t, \quad  \tilde u_t:= -4J_t^0 \overline H^* (\tilde \Xi^{0,3}_\infty)_t + (\tilde \ell^0_\infty)_t.
\end{equation}
Observe that we have ${\rm Law}_{\mu_n}(u) = {\rm Law}_{\tilde {\mathbb P}}(\tilde u_n)$ and ${\rm Law}_{\mu}(u) = {\rm Law}_{\tilde{\mathbb P}}(\tilde u)$ and that $\tilde u_n \rightarrow \tilde u$ almost surely in $\mathcal L^0_w$.  

Writing $\tilde{\mathbb E}$ to denote expectation with respect to $\tilde{\mathbb P}$, we have that
\begin{equation}
\tilde{\mathbb F}^0_{T_n}(\mu_n)= \tilde{\mathbb E} \left[ \Phi^0_{T_n}(\tilde\Xi^0_{T_n}, \tilde u_n)+ \| \overline H Z^0_T (\tilde u_n) \|_{L^4([-1,1]\times \mathbb T^2)}^4+ \frac 12 \| \tilde \ell^0_{T_n}\|_{\mathcal H^0} ^2\right].
\end{equation}
By Proposition \ref{prop: bdry bd} and Fatou's lemma, we have that
\begin{equation}\label{eq: bd fatou 1}
\tilde{\mathbb E}[\Phi^0_\infty(\tilde \Xi_\infty^0, \tilde u)] \leq \liminf_{n\rightarrow \infty} \tilde{\mathbb E}[\Phi^0_{T_n}(\tilde\Xi^0_{T_n}, \tilde u_n)].
\end{equation}
Furthermore, by lower semicontinuity of the norm $\|\cdot\|_{\mathcal H^0}$ with respect to convergence in $\mathcal H^0_w$ and Fatou's lemma,
\begin{equation} \label{eq: bd fatou 2}
	\tilde{\mathbb E}[\|\tilde \ell^0_\infty\|_{\mathcal H^0}^2] \leq \liminf_{n \rightarrow \infty} \mathbb E[ \|\tilde\ell^0_{T_n}\|_{\mathcal H^0}^2].
\end{equation}
It remains to treat the $L^4$ term, the effective boundary potential. It is slightly awkward to consider it as a map of the $\tilde u_n$ directly, so let us instead use the equalities in law \eqref{eq: bdry skorokhod drift} and view it as a mapping $(\tilde \Xi^0_{T_n}, \tilde \ell^0_{T_n}) \mapsto \|\overline H Z^0_{T_n}(\tilde \Xi^0_{T_n}, \tilde \ell^0_{T_n})\|_{L^4([-1,1]\times \mathbb T^2)}^4$ instead. Observing that, by Sobolev embedding,
\begin{equation}
	\|\overline HZ_T^0(\tilde u_n) \|_{L^4([-1,1]\times \mathbb T^2)} \leq \left\| \int_0^T J_t^0 \overline H^* (\tilde \Xi^{0,3}_{T_n})_t dt \right\|_{L^4_z} + \|Z^0_T(\tilde \ell^0_{T_n}) \|_{H^{1/2}_z} 
\end{equation}
we conclude with help of Lemma \ref{lem: bdry a priori Z} that the aforementioned map is lower semicontinuous on $\boldsymbol{\mathfrak B}^0 \times \mathcal H^0_w$. Thus by Fatou's lemma, we have that
\begin{equation} \label{eq: bd fatou 3}
\tilde{\mathbb E}[\|\overline HZ_T^0(\tilde u)\|_{L^4([-1,1]\times\mathbb T^2)}^4] \leq \liminf_{n\rightarrow \infty} \tilde{\mathbb E}[ \|\overline HZ_T^0(\tilde u_n)\|_{L^4([-1,1]\times\mathbb T^2)}^4].
\end{equation}
Combining \eqref{eq: bd fatou 1}, \eqref{eq: bd fatou 2}, \eqref{eq: bd fatou 3} yields \eqref{eq:fatou}.  
\end{proof}

We will now prove our main theorem of the section.

\begin{proof}[Proof of Theorem  \ref{thm: boundary}]

By Lemma \ref{lem: boundary fatou}, we have that the functions $\tilde{\mathbb F}^{0,f}_T$ $\Gamma$-converge to $\tilde{\mathbb F}^{0,f}_\infty$. By Lemma \ref{lem: boundary equicoercive} and the\footnote{Strictly speaking, we have not specified a topology but only convergence of sequences. However, one can check that the proof of the fundamental theorem of $\Gamma$-convergence follows in our setting.} fundamental theorem of $\Gamma$-convergence \cite{D93}, we have that the minimizers converge. The theorem follows by the identity \eqref{eq: boundary v to mu} and Corollary \ref{cor: boundary min closure}.

Let us now turn our attention to items (i)-(iv). The fact that $\nu^{0}=\text{Law}_{\mu_\infty}(W_{\infty}^0+Z_\infty^0(u))$ and ${\rm Law}_{\mathbb P}(W^0_\infty) = \mu^0$, where $\mu_\infty$ is a minimizer of $\tilde{\mathbb{F}}^{0}_{\infty}$,
 follows by the same arguments as in \cite[Lemma 11 and Theorem 12]{BG23}. Item (iii) is implied by (iv), so let us establish that. 
Note that equicoercivity it follows that $\mathbb{E}_{\mu_\infty}[\|\ell^0_\infty(u)\|^{2}_{\mathcal H^0}]<\infty$ and that there exists $C>0$ such that  
\begin{equation}
   \| Z_\infty^0(u)\|_{ H^{1/2-\delta}_z} \leq C( \|\mathbb W^{0,3}_\infty\|_{\mathcal C^{1/2-\delta}_z} + \|\ell^{0}_\infty(u)\|_{\mathcal H^0}). 	
 \end{equation} 
 Item (iv) follows. 
 
\end{proof}

\subsection{Proof of approximation lemmas} \label{sec: boundary technical}

We now turn to the proofs of Lemmas \ref{lem:bdry-approx-finite} and \ref{lem:bdry-approx-infty}, thereby closing the proof of Theorem \ref{thm: boundary}. As a preliminary, we will need a technical result concerning approximation of measures in $\overline{\mathcal X}^0$ with measures of bounded support and sufficient moment properties. 

\begin{lemma} \label{lem: bdry lemma 14}
Let $T \in (0,\infty]$. Let $\mu \in \overline{\mathcal X}^0$ such that $\mathbb E_\mu[ \|Z_T^0(u)\|_{L^4_z}^4 + \|u\|_{\mathcal L^0}^2] < \infty$. Then there exist sequences $(\mu_L)_{L \in \mathbb N^*} \subset \overline{\mathcal X}^0$ and $(\mu_{L,n})_{n \in \mathbb N^*} \subset \mathcal X^0$ such that:
\begin{enumerate}
	\item[(i)] $\mu_L \rightharpoonup \mu$ on $\mathfrak S \times \mathcal L^0$ as $L \rightarrow \infty$ and $\mu_{L,n} \rightharpoonup \mu_L$ on $\mathfrak S \times \mathcal L_w^0$ as $n \rightarrow \infty$.
	\item[(ii)] For every $L\in \mathbb N^*$, we have that $\|u\|_{\mathcal L^0} \leq L$ $\mu_L$-almost surely. For every $n \in \mathbb N^*$, we have that $\|u\|_{\mathcal L^0} \leq L$ $\mu_{L,n}$-almost surely.
	\item[(iii)] We have the convergence properties: 
	\begin{align}
		\lim_{L \rightarrow \infty}\mathbb E_{\mu_L}[ \|\overline H Z_T^0(u)\|_{L^4([-1,1]\times\mathbb T^2) }^4] &= \mathbb E_\mu [ \|\overline HZ_T^0(u)\|_{L^4([-1,1]\times\mathbb T^2)}^4], \\ 
		\lim_{L \rightarrow \infty}\mathbb E_{\mu_L}[ \|u\|_{\mathcal L^0}^2] &= \mathbb E_\mu [ \|u\|_{\mathcal L^0}^2].
	\end{align}
\end{enumerate} 
\end{lemma}
\begin{proof}
The proof is a slight modification of the proof of \cite[Lemma 14]{BG20}, we omit the details. 
\end{proof}

We will also need the following regularizations. 
\begin{enumerate}
	\item[(i)] \textit{Space regularization.} Let $\eta^0 \in C^\infty_c(\mathbb T^2)$ and let $(\eta^0_\varepsilon)_{\varepsilon>0}$ denote the corresponding family of (spatial) mollifiers. Given $u \in \mathcal L^0$, Denote by ${\rm reg}_{\mathbb T^2,\varepsilon}:\mathcal L^0 \rightarrow \mathcal L^0$ the linear map defined for every $u \in \mathcal L^0$ by
	\begin{equation}\label{eqdef: space reg bdry}
		{\rm reg}_{\mathbb T^2,\varepsilon}(u):=u \ast_{\mathbb T^2} \eta^0_\varepsilon,
	\end{equation}
	where $*_{\mathbb T^2}$ denotes spatial convolution on $\mathbb T^2$. This map is continuous as a map from $\mathcal L^0 \rightarrow \mathcal H^0$ and from $\mathcal L^0_w \rightarrow \mathcal H^0_w$.
\item[(ii)]\textit{Space-time regularization.} Let $\zeta^0 \in C^\infty_c(\mathbb R_+\times \mathbb T^2)$ and let $(\zeta^0_\varepsilon)_{\varepsilon>0}$ be the corresponding family of (space-time) mollifiers. Denote by ${\rm reg}_{\mathbb R_+\times\mathbb T^2}:\mathcal L^0 \rightarrow \mathcal L^0$ the linear map defined for every $u \in \mathcal L^0$ by
\begin{equation}
{\rm reg}_{\mathbb R_+\times\mathbb T^2}(u)_t := e^{-\varepsilon t} u \ast_{\mathbb T^2} \zeta^0_\varepsilon(t,\cdot), \qquad \forall t \geq 0. 	
\end{equation}
This map is continuous as a map from $\mathcal L^0_w \rightarrow \mathcal H^0$. 
\end{enumerate}

We now turn to the proof of the approximation lemma when $T<\infty$. 

\begin{proof}[Proof of Lemma \ref{lem:bdry-approx-finite}]
Let $\mu \in \overline{\mathcal X}^0$ be such that $\tilde{\mathbb F}^{0,f}_T(\mu)<\infty$.  Recall that by \eqref{eq: bdry bd} for $T<\infty$, we have that 
\begin{equation} \label{eq: bdry cost fn T finite}
\tilde{\mathbb F}^{0,f}_T(\mu):= \mathbb E_\mu \left[ \langle f,W_T^0 + Z_T^0(u)\rangle_{L^2_z} + \mathcal V^0_T(W_T^0, Z_T^0(u)) + \frac 12 \|u\|_{\mathcal H^0}^2 \right].
\end{equation}
First, observe that by duality and Lemma \ref{lem: bdry a priori Z}
\begin{equation} \label{eq: bdry f term, t finite}
\mathbb E_{\mu} \left[ \langle f,W_T^0 + Z_T^0(u)\rangle_{L^2_z}\right] \geq - \|f\|_{H^{-1/2}_z}\|u\|_{\mathcal H^0}.
\end{equation}
Therefore, by Cauchy's inequality and since $\mathcal V_T^0(W_T,Z_T^0(u))$ is bounded from below, we deduce that $\tilde{\mathbb F}^{0,f}_T(\mu)<\infty$ implies $\mathbb E_\mu[ \|u\|_{\mathcal H^0}^2] < \infty$. Next, we may write $\mathcal V_T^0(W_T^0,Z_T^0(u)) = \Phi^{0}_T + \|\overline H Z_T^0(u)\|_{L^4([-1,1]\times\mathbb T^2)}^4$. For every $\varepsilon,\delta>0$ sufficiently small, by H\"older's inequality, there exists $C>0$ such that 
\begin{equation} \label{eq: bdry phi term, t finite}
|\Phi^{0}_T| \leq N_T  \| \overline H Z_T^0(u)\|_{L^4([-1,1]\times\mathbb T^2)}^3,
\end{equation}
where $N_T$ is an explicit positive random variable consisting of $L^p$ norms of the cutoff Wick powers (we are using that when $T<\infty$ this is smooth).

We now want to define a sufficient approximation of $\mu$ in $\mathcal X$ such that we have convergence of $\tilde{\mathbb F}^{0,f}$ along this sequence. We note that approximating sequences defined by the closure are insufficient to ensure this since it only gives us information about the $\mathcal L^0$-norm of $u$. Hence we opt to use Lemma~\ref{lem: bdry lemma 14}. Let $(\mu_L)_{L \in \mathbb N^*} \subset \overline{\mathcal X}^0$ and $(\mu_{L,n})_{n \in \mathbb N^*} \subset \mathcal X^0$ be the sequences constructed in Lemma \ref{lem: bdry lemma 14}. 

 We will regularize these two sequences and use that $\mathbb E_\mu[\|u\|_{\mathcal H^0}^2] < \infty$ to obtain the existence of two further sequences:
\begin{enumerate}
\item[(a)] $(\mu_L^{\varepsilon})\subset \overline{\mathcal X}^0$ such that 
\begin{equation} \label{eq: bdry Tfinite conv 1}
	\lim_{L \rightarrow \infty}\tilde{\mathbb F}^{0,f}_T(\mu_L^\varepsilon) = \tilde{\mathbb F}^{0,f}_T(\mu^\varepsilon), \quad \lim_{\varepsilon \rightarrow 0} \tilde{\mathbb F}^{0,f}_T(\mu^\varepsilon) = \tilde{\mathbb F}^{0,f}_T(\mu).
\end{equation}
\item[(b)] $(\mu_{L,n}^{\varepsilon,\delta}) \subset \mathcal X^0$ such that
\begin{equation} \label{eq: bdry Tfinite conv 2}
 \lim_{n \rightarrow \infty} \tilde{\mathbb F}^{0,f}_T(\mu_{L,n}^{\varepsilon,\delta}) = \tilde{\mathbb F}^{0,f}_T(\mu_L^{\varepsilon,\delta}), \quad \lim_{\delta \rightarrow 0}   \tilde{\mathbb F}^{0,f}_T(\mu_{L}^{\varepsilon,\delta}) = \tilde{\mathbb F}^{0,f}_T(\mu_L^\varepsilon).
\end{equation} 
\end{enumerate}
The desired sequence can then be constructed as a subsequence of $(\mu^{\varepsilon,\delta}_{L,n})$ by taking appropriate diagonal subsequences of the above limits.Let us now construct said sequences. We  begin by defining the regularizations:
\begin{align}
\mu^\varepsilon &:= ({\rm reg}_{\mathbb T^2, \varepsilon})_* \mu, &\mu^\varepsilon_L &:= ({\rm reg}_{\mathbb T^2, \varepsilon})_* \mu_L,
\\
\mu^{\varepsilon,\delta}_L&:= ({\rm reg}_{\mathbb R_+\times \mathbb T^2, \delta})_* \mu^\varepsilon_L, &\mu^{\varepsilon,\delta}_{L,n} &:= ({\rm reg}_{\mathbb R_+\times \mathbb T^2, \delta})_* ({\rm reg}_{\mathbb T^2, \varepsilon})_*\mu_{L,n}.
\end{align}

Let us first establish \eqref{eq: bdry Tfinite conv 1}. By the continuity properties of ${\rm reg}_{\mathbb T^2,\varepsilon}$, we have that $\mu_L^\varepsilon \rightharpoonup \mu^\varepsilon$ as $L \rightarrow \infty$ on $\boldsymbol{\mathfrak B}^0 \times \mathcal H^0$. By Skorokhod embedding we may write $\mu^\varepsilon_L = {\rm Law}_{\tilde{\mathbb P}}(\tilde{\Xi}^0,\tilde u_L^\varepsilon)$ and $\mu^\varepsilon = {\rm Law}_{\tilde{\mathbb P}}(\tilde{\Xi}^0,\tilde u^\varepsilon )$ on some probability space $(\tilde \Omega, \tilde{\mathbb P})$ such that $\tilde u_L^\varepsilon \rightarrow \tilde u^\varepsilon$ almost surely in $\mathcal H^0$. Writing
\begin{equation}
\tilde{\mathbb F}^{0,f}_T(\mu_L^\varepsilon) = \tilde{\mathbb E}[ F^{0,f}_T(\tilde\Xi^0, \tilde u_L^\varepsilon)], \quad \tilde{\mathbb F}^{0,f}_T(\mu^\varepsilon) = \tilde{\mathbb E}[ F^{0,f}_T(\tilde\Xi^0, \tilde u^\varepsilon)],
\end{equation}
the almost sure convergence of $(\tilde \Xi, \tilde u^\varepsilon_L) \rightarrow (\tilde\Xi, \tilde u^\varepsilon)$ ensures that $F_T^{0,f}(\tilde \Xi, \tilde u_L) \rightarrow F^{0,f}_T(\tilde \Xi, \tilde u)$ almost surely (thanks to the bounds established at the beginning of the proof).

We now wish to pass the limit inside the expectation. Recall that by construction, the  sequence $(\mu_L)$ satisfies
\begin{equation}
\lim_{L\rightarrow \infty} \mathbb E_{\mu_L}[\|Z_T^0(u)\|_{L^4}^4] = \mathbb E_\mu [ \|Z_T^0(u)\|_{L^4}^4], \quad 
\lim_{L\rightarrow \infty} \mathbb E_{\mu_L}[ \|u\|_{\mathcal {L}^0}^2] = \mathbb E_\mu[\|u\|_{\mathcal {L}^0}^2]. 
\end{equation}
By the continuity properties of ${\rm reg}_{\mathbb T^2,\varepsilon}$, we have that
\begin{equation}
\lim_{L\rightarrow \infty} \mathbb E_{\mu_L^\varepsilon}[\|Z_T^0(u)\|_{L^4}^4] = \mathbb E_{\mu^\varepsilon} [ \|Z_T^0(u)\|_{L^4}^4], \quad 
\lim_{L\rightarrow \infty} \mathbb E_{\mu_L^\varepsilon}[ \|u\|_{\mathcal H^{0}}^2] = \mathbb E_{\mu^\varepsilon}[\|u\|_{\mathcal H^{0}}^2]. 	
\end{equation}
Therefore, in order to conclude convergence of the cost functions, it is sufficient to show that the remainder terms in \eqref{eq: bdry cost fn T finite} are uniformly integrable. To do this, we show that they admit $1+\varepsilon$ moments, provided $\varepsilon$ is sufficiently small. Indeed, by Young's inequalities applied to \eqref{eq: bdry f term, t finite} and \eqref{eq: bdry phi term, t finite}, and using hypercontractivity, there exists $c_T>0$ such that
\begin{equation}
	\sup_L\tilde{\mathbb E}[|\langle f, W_T^0+Z_T^0(\tilde u^\varepsilon_L)\rangle_{L^2_z}|^{1+\varepsilon}] \leq c_T + \sup_L \tilde{\mathbb E}[\|\tilde u^\varepsilon_L\|_{\mathcal H^0}^2] < \infty,
\end{equation}
and
\begin{equation}
\sup_L \tilde{\mathbb E}[|\Phi^0_T|^{1+\varepsilon}]\leq c_T + \sup_L \tilde{\mathbb E}[\|\overline H Z_T^0(u)\|_{L^4([-1,1]\times\mathbb T^2)}^4] < \infty.	
\end{equation}
Thus, by Vitali's convergence theorem, we have that
\begin{equation} \label{eq: bdry Tfinite 1 1}
\lim_{L \rightarrow \infty} \tilde{\mathbb F}^{0,f}_T(\mu^\varepsilon_L) = \tilde{\mathbb F}_{\infty}^{0,f}(\mu^\varepsilon). 
\end{equation}

For the removal of $\varepsilon$, let us first observe that since $\|{\rm reg}_{\mathbb T^2,\varepsilon} u\|_{\mathcal H^0} \leq \|u \|_{\mathcal H^0}$ and that ${\rm reg}_{\mathbb T^2,\varepsilon} u \rightarrow u$ in $\mathcal H^0$, by dominated convergence we have that
\begin{equation}
\lim_{\varepsilon \rightarrow 0} \mathbb E_{\mu^\varepsilon}[\|u\|_{\mathcal H^0}^2] = \mathbb E_{\mu}[\|u\|_{\mathcal H^0}^2]. 
\end{equation}
For the convergence of the $L^4$ norm, note that $Z_T^0({\rm reg}_{\mathbb T^2,\varepsilon} u) = \eta^0_\varepsilon \ast_{\mathbb T^2} Z_T^0(u)$. Since $\| \eta^0_\varepsilon \ast_{\mathbb T^2} \overline H Z_T^0(u) - \overline H Z_T^0(u) \|_{L^4} \rightarrow 0$, and in particular $\|\eta^0_\varepsilon \ast_{\mathbb T^2} \overline H Z_T^0(u)\|_{L^4} \leq 2\| \overline H Z_T^0(u)\|_{L^4}$. Hence by dominated convergence we have that
\begin{equation}
\lim_{\varepsilon \rightarrow 0} \mathbb E_{\mu^\varepsilon}[\|\overline H Z_T^0(u)\|_{L^4([-1,1]\times\mathbb T^2 )}^4] = \mathbb E_\mu[ \|\overline H Z_T^0(u)\|_{L^4([-1,1]\times\mathbb T^2 )}^4].	
\end{equation}
Then by arguing verbatim as above (i.e.\ Skorokhod embedding and then a uniform integrability argument), we obtain
\begin{equation} \label{eq: bdry Tfinite 1 2}
\lim_{\varepsilon \rightarrow 0} \tilde{\mathbb F}^{0,f}_T(\mu^\varepsilon) = \tilde{\mathbb F}^{0,f}_T(\mu).
\end{equation}
Then \eqref{eq: bdry Tfinite 1 1} and \eqref{eq: bdry Tfinite 1 2} establish \eqref{eq: bdry Tfinite conv 1}. 

We now turn to \eqref{eq: bdry Tfinite conv 2}. Let us fix $L$ and $\varepsilon$. Note that
\begin{equation}
\sup_n \mathbb E_{\mu_{L,n}}[\|u\|_{\mathcal L^0}^2] < \infty \Rightarrow \sup_n \mathbb E_{\mu_{L,n}^\varepsilon}[\|u\|_{\mathcal H^0}^2] < \infty,
\end{equation}
and also that $\|u\|_{\mathcal H^0} \leq L$ $\mu_{L,n}^\varepsilon$-almost surely. Furthermore, by continuity properties of ${\rm reg}_{\mathbb T^2,\varepsilon}$, we have that $\mu_{L,n}^\varepsilon \rightharpoonup \mu_L^\varepsilon$ on $\boldsymbol{\mathfrak B}^0 \times \mathcal H^0_w$. We will need to upgrade this to convergence in the norm topology on $\mathcal{H}^0$, which is done by the space-time regularization map. Indeed, by properties of ${\rm reg}_{\mathbb R_+\times \mathbb T^2, \delta}$, we have that $\mu_{L,n}^{\varepsilon,\delta} \rightharpoonup \mu_{L,n}^\varepsilon$ on $\boldsymbol{\mathfrak B}^0 \times \mathcal H^0$. A similar Skorokhod embedding argument as above -- with a simplification due to the almost sure boundedness, thereby removing the need for uniform integrability -- then yields \eqref{eq: bdry Tfinite conv 2}. We omit the details. 
\end{proof}

We now turn to the case $T=\infty$. Here, the variational problem is singular and so we do not have the same regularizing properties (from $\mathcal L^0$ to $\mathcal H^0$). thus we will have to study the regularizing properties of the remainder map. The key properties we require are contained in the following lemma.

\begin{lemma} \label{lem: bdry rem properties}
For every $\varepsilon >0$, define the boundary remainder map ${\rm rem}^0_\varepsilon: \boldsymbol{\mathfrak B}^0 \times \mathcal L^0 \rightarrow \mathcal L^0$ for every $(\Xi^0,u) \in \boldsymbol{\mathfrak B}^0 \times \mathcal L^0$ by 
\begin{equation}
{\rm rem}^0_\varepsilon(\Xi^0, u)_t := - 4J_t^0 \overline{H}^* \Xi^{0,3}_t + {\rm reg}_{\mathbb T^2,\varepsilon}\Big( 4J_t^0 \overline{H}^* \Xi^{0,3}_t  + \ell^0_\infty(u)_t  \Big), \quad \forall t \geq 0.
\end{equation}
This a continuous linear map $\boldsymbol{\mathfrak B}^0 \times \mathcal L^0  \rightarrow \mathcal {L}^0$ and also as a map from $\boldsymbol{\mathfrak B}^0 \times \mathcal {L}^0_w \rightarrow \mathcal L^0_w$. It satisfies 
\begin{equation}\label{eq:rem-l}
\ell^{0}_\infty({\rm rem}_{\varepsilon}(\Xi^0, u))={\rm reg}_{\mathbb{T}^2,\varepsilon} \left( \ell^{0}_\infty(\Xi^0,u) \right).
\end{equation} 
Furthermore, the map satisfies the following properties: there exists $C>0$  and for every $\varepsilon > 0$ there exists $C_\varepsilon > 0$ such that:
\begin{enumerate}
\item[(i)] $\|{\rm rem}^0_\varepsilon(\Xi^0, u)\|_{\mathcal {L}^0} \leq C (\| \Xi^0\|_{\boldsymbol{\mathfrak B}^0} + \|u\|_{\mathcal L^0})$.
\item[(ii)] for $T \in (0,\infty]$, $\|\overline H Z^{0}_T({\rm rem}^0_\varepsilon(\Xi^0, u))\|_{L^4([-1,1]\times\mathbb T^2)} \leq C (\|\Xi^0\|_{\boldsymbol{\mathfrak B}^0} + \|\overline H Z^{0}_T(u)\|_{L^4([-1,1]\times\mathbb T^2)})$.\label{eq:bound-l-L4}
\item[(iii)] $\|\ell^0_\infty({\rm rem}_\varepsilon^0(\Xi^0, u))\|_{\mathcal H^0} \leq C_\varepsilon ( \|\Xi^0\|_{\boldsymbol{\mathfrak B}^0}+ \|u\|_{\mathcal {L}^0})$.\label{eq:bound-l-H}
\item[(iv)] $\|{\rm rem}^{0}_{\varepsilon}(\Xi^0,u)-u\|_{\mathcal{L}^0} \to 0$ if $u\in \mathcal{L}^0$ and $\|\ell^{0}_{\infty}({\rm rem}^{0}_{\varepsilon}(\Xi^0,u))-\ell_{\infty}^{0}(u)\|_{\mathcal{H}^0} \to 0$ if 
$\ell_{\infty}^{0}(u) \in \mathcal{H}^0$. 
\end{enumerate}
\end{lemma}

\begin{proof}
Properties (i) and (ii) are trivial by using linearity of ${\rm reg}_{\mathbb T^2, \varepsilon}$. Property (iii) follows from \eqref{eq:rem-l} and the definition of ${\rm reg_{\mathbb{T}^2,\varepsilon}}$. Property (iv) can be easily checked from the definitions. 
\end{proof}

In the case that $\Xi^0 = \Xi^0(X^0)$, we will omit it from notation in ${\rm rem}$. We now prove our main approximation theorem when $T=\infty$.

\begin{proof}[Proof of Lemma \ref{lem:bdry-approx-infty}]
  
For every $\kappa > 0$, let us define $\mu^\kappa:= ({\rm rem}^0_\kappa)_* \mu$. 
First we show the limit $\lim_{\kappa \to 0} \widetilde{\mathbb F}^{0,f}_\infty (\mu^{\kappa})=\widetilde{\mathbb F}^{0,f}_\infty (\mu)$. 
Note that we can represent the prelimit as 
\begin{align}
  \tilde{\mathbb F}^{0,f}_\infty(\mu^{\kappa}) &= \mathbb E_{\mu}\Big[ \langle f, Z^0_\infty({\rm rem}^0_\kappa u) \rangle_{L^2_z}+ \Phi^{0,f}_\infty( \Xi^0_\infty (X^0), {\rm rem}_{\kappa}^0u)
  \\ &\qquad \qquad +\|\overline H Z_\infty^0({\rm rem}_{\kappa}^0u)\|_{L^4([-1,1]\times\mathbb T^2)}^4 + \frac 12 \| \ell^0_\infty({\rm rem}_{\kappa}u)\|_{\mathcal H^0}^2 \Big].
\end{align}
Observe that $\Phi_{\infty}^{0,f}(\Xi^0_\infty,{\rm rem}_{\kappa}\, u) \to \Phi_{\infty}^{0,f}(\Xi^0_\infty,u)$ since it is a sum of multilinear bounded functional on $\boldsymbol{\mathfrak B}^0 \times \mathcal{H}^0$ (where we view $\Phi_{\infty}^0$ as a function of $\ell^{0}_{\infty}$ by a change of variables.)
From this we can see that the random variables under the expectation converge almost surely as $\kappa \to 0$ to 
\begin{equation}
 \langle f, Z^0_\infty(u) \rangle_{L^2_z}+ \Phi^{0,f}_\infty( \Xi^0_\infty, u)+\|\overline H Z_\infty^0(u)\|_{L^4([-1,1]\times\mathbb T^2)}^4 + \frac 12 \| \ell^0_\infty(u)\|_{\mathcal H^0}^2.
\end{equation}
By properties (ii),(iii), and (iv) of Lemma \ref{lem: bdry rem properties} we may thus apply dominated convergence and conclude that $\lim_{\kappa \to 0} \widetilde{\mathbb F}^{0,f}_\infty (\mu^{\kappa})=\widetilde{\mathbb F}^{0,f}_\infty (\mu)$.

Now let $(\mu_L)_{L \in \mathbb N^*} \subset \overline{\mathcal X}^0$ and $(\mu_{L,n})_{n \in \mathbb N^*} \subset \mathcal X^0$ be the sequence of measures constructed in Lemma \ref{lem: bdry lemma 14}. Let us define $\mu^\kappa_L:= ({\rm rem}^0_\kappa)_*\mu_L$ and $\mu^\kappa_{L,n}:= ({\rm rem}^0_\kappa)_* \mu_{L,n}$. 
 We will prove that  
$\lim_{L \to \infty} \widetilde{\mathbb F}^{0,f}_\infty (\mu_{L}^{\kappa})=\widetilde{\mathbb F}^{0,f}_\infty (\mu^{\kappa})$ -- by similar arguments one can also show $\lim_{n \to \infty} \widetilde{\mathbb F}^{0,f}_\infty (\mu_{L,n}^{\kappa})=\widetilde{\mathbb F}^{0,f}_\infty (\mu^{\kappa}_{L})$.
To the former claim, note that for every $\kappa > 0$ there exists $C_\kappa > 0$ such that   
\begin{multline}
  \|\overline H Z_\infty^{0}({\rm rem}^{0}_{\kappa}u)\|_{L^4([-1,1]\times\mathbb T^2)}^4+\|\ell_{\infty}^{0}({\rm rem}^{0}_{\kappa}u)\|_{\mathcal{H}^0}^2 \\\leq C_\kappa \left(  \|\overline H Z_\infty^{0}(u)\|_{L^4([-1,1]\times\mathbb T^2)}^4+\|u\|_{\mathcal{L}^0}^2 + \|\Xi^{0}_\infty \|^4_{\boldsymbol{\mathfrak B}^0}\right).\label{eq:gamma-bound-proof}\end{multline}
Recall also that 
\begin{equation}
\lim_{L \to \infty}\mathbb{E}_{\mu_{L}}\left[ \|\overline H Z_\infty^{0}(u)\|_{L^{4}([-1,1]\times\mathbb T^2)}^4+\|u\|_{\mathcal{L}^0}^2 \right] =\mathbb{E}_{\mu} \left[ \|\overline H Z_\infty^{0}(u)\|_{L^4([-1,1]\times\mathbb T^2)}^4+\|u\|_{\mathcal{L}^0}^2 \right].	
\end{equation}
 By Skorohod's theorem we can find a probability space $\tilde{\mathbb{P}}$, whose expectation we denote by $\tilde{\mathbb{E}}$, and random variables $\tilde{\Xi}^{0}_L, \tilde{u}_{L}$ such that ${\rm Law}(\tilde{\Xi}^{0}_L, \tilde{u}_{L})=\mu_{L}$ 
and $\lim_{L \to \infty}(\tilde{\Xi}^{0}_L, \tilde{u}_{L})=(\tilde{\Xi}^{0}, \tilde{u})$ almost surely in $\boldsymbol{\mathfrak B}^0\times \mathcal{L}^0$. Thus we have the representation 
\begin{align}
  \tilde{\mathbb F}^{0,f}_\infty(\mu^{\kappa}_{L}) &= \tilde{\mathbb E}\Bigg[ \langle f, Z^0_\infty({\rm rem}^0_\kappa \tilde u_L)\rangle_{L^2_z} +\Phi^{0}_\infty( \tilde{\Xi^{0}_L}, {\rm rem}_{\kappa}^0 \tilde{u}_{L})
  \\ &\qquad \qquad +\|\overline H Z_\infty^0({\rm rem}_{\kappa}^0 \tilde{u}_L)\|_{L^4([-1,1]\times\mathbb T^2)}^4 + \frac 12 \| \ell^0_\infty({\rm rem}^{0}_{\kappa}\tilde{u}_L)\|_{\mathcal H^0}^2 \Bigg],
\end{align}
Since $\|\overline H Z_\infty^0({\rm rem}_{\kappa}^0 \tilde{u}_L)\|_{L^4([-1,1]\times\mathbb T^2)}^4 + \frac 12 \| \ell^0_\infty({\rm rem}^{0}_{\kappa} \tilde{u}_{L})\|_{\mathcal H^0}^2$ is positive and converges and expectation it also converges in $L^1$, and so we may use it as a dominating function (recall that dominated convergence holds also if the dominating function is not fixed but converges in $L^1$).
Invoking Lemma \ref{lem: bdry quadratic} and \eqref{eq:gamma-bound-proof} we can thus apply dominated convergence to obtain $\lim_{L \to \infty} \widetilde{\mathbb F}^{0,f}_\infty (\mu_{L}^{\kappa})=\widetilde{\mathbb F}^{0,f}_\infty (\mu^{\kappa})$.

Finally it remains to prove 
$\lim_{T \to \infty} \widetilde{\mathbb F}^{0,f}_T (\mu_{L}^{\kappa})=\widetilde{\mathbb F}^{0,f}_\infty (\mu^{\kappa}_{L})$. Again we may represent 
\begin{align}
 \tilde{\mathbb F}^{0,f}_T(\mu^{\kappa}_{L}) &= \tilde{\mathbb E}\Bigg[ \langle f, Z^0_T({\rm rem}^0_\kappa \tilde u_L) \rangle_{L^2_z}+ \Phi^{0,f}_T( \tilde{\Xi_{L}^{0}}, {\rm rem}_{\kappa}^0 \tilde{u}_{L})
 \\ & \qquad \qquad \qquad +\|\overline H Z_T^0({\rm rem}_{\kappa}^0 \tilde{u}_L)\|_{L^4([-1,1]\times\mathbb T^2)}^4 + \frac 12 \| \ell^0_T({\rm rem}^{0}_{\kappa}\tilde{u}_L)\|_{\mathcal H^0}^2 \Bigg],
\end{align}
Now $Z_T^0({\rm rem}_{\kappa}^0 \tilde{u}_L), \ell^0_T({\rm rem}^{0}_{\kappa}\tilde{u}_L)$ converge to $Z_\infty^0({\rm rem}_{\kappa}^0 \tilde{u}_L), \ell^0_\infty({\rm rem}^{0}_{\kappa}\tilde{u}_L)$ in $L^4, \mathcal{H}$, respectively, almost surely in $\mu_{L}$ and the corresponding norms are bounded by $\|\tilde{u}_{L}\|_{\mathcal{L}^0}\leq L$, by the properties of ${\rm rem^0}$ and $Z_T^0$ (and harmonic extension). 
From the definition of $\Phi^{0,f}$ and the bounds from Lemma \ref{lem: bdry quadratic} it is not hard to see that $\lim_{T \to \infty}\Phi_T(\tilde{\Xi}^{0}_{L},\tilde u_{L})=\Phi_{\infty}(\tilde{\Xi}^{0}_L,\tilde{u}_{L})$ and $\Phi_{T}(\tilde{\Xi}_{L}^{0},\tilde u_{L})\leq C(\|\tilde{\Xi}^0_{L}\|_{\boldsymbol{\mathfrak B}^0}^{p_0}+\|\overline H Z^0_T(\tilde u_L) \|_{L^4}^4+1)$. We can also estimate the linear term similarly.  Thus we get the claim by dominated convergence and Vitali's convergence theorem. 
\end{proof}

\section{Algebraic renormalization of bulk amplitudes}
\label{sec: renormalization of bulk amplitudes}

The next four sections will be devoted to the proof of Theorem \ref{thm: bulk}. Without loss of generality, we will consider the case of $\sigma = \emptyset$, i.e.\ the amplitudes on $M$. In this section, our goal is to identify a \emph{renormalized variational representation} of the bulk amplitudes which is amenable to study in the limit $T \rightarrow \infty$. At the end of this section, we will give a brief discussion of how the next sections play a role in the proof.

\subsection{Bou\'e-Dupuis representation of bulk amplitudes}

We first recall the definition of cutoff amplitudes. 
Let $T>0$. Recall from Definition \ref{def: approx amplitude} that for every $f \in C^\infty(M)$ and $\varphi_-,\varphi_+ \in S'(\mathbb T^2)$, the cutoff amplitude is given by
\begin{equation}
	\mathcal A_T(f \mid \varphi_-,\varphi_+):= \mathcal A^{\rm free}(\varphi_-,\varphi_+) \mathcal E_{T}(\varphi_-,\varphi_+)  
	 \mathcal Z_T(f \mid \varphi_-,\varphi_+), 
\end{equation}	
where $\mathcal A^{\rm free}(\varphi_-,\varphi_+)$ is the free amplitude defined in \eqref{eqdef: free amplitude}, $\mathcal E_{T}(\varphi_-,\varphi_+)$ is defined in \eqref{eqdef: mathcal E term}, and $ \mathcal Z_T(f \mid \varphi_-,\varphi_+)$ is the unnormalized Laplace transform of $f$ under $\nu( \cdot \mid \varphi_-,\varphi_+)$. In studying the limit $T\rightarrow \infty$, the free amplitude is irrelevant. Thus we will consider the twisted amplitude 
\begin{equation}
	\mathsf A_T(f \mid \varphi_-,\varphi_+):= \mathcal E_{T}(\varphi_-,\varphi_+)  
	 \mathcal Z_T(f \mid \varphi_-,\varphi_+).
\end{equation}
To simplify further, we will consider the case $\mathsf A_T(f \mid \varphi_-,\varphi_+)$ with $f\equiv 0$ since the case $f \neq 0$ follows by straightforward modifications. We will henceforth just write $\mathsf A_T:= \mathsf A_T(0 \mid \varphi_-,\varphi_+)$ since the boundary fields $\varphi_-,\varphi_+$ are considered fixed unless stated otherwise.

We now turn to establishing a variational representation for $\mathsf A_T$. We will again use the Bou\'e-Dupuis formula as our starting point, just as in the preceding section and \cite{BG20}, so we will 

Let us fix $\kappa_0 > 0$ and consider the probability space $(\Omega,\mathcal{F},\mathbb{P})$, where the sample space is $\Omega = C(\mathbb{R}_+, \mathcal C^{-3/2-\kappa_0}(M))$ endowed with the Borel $\sigma$-algebra and under which the canonical coordinate process $X$ under $\mathbb P$ is an $L^2(M)$-Brownian motion. By abuse of notation, we will assume that $(\Omega, \mathcal F, \mathbb P)$ is $\mathbb P$-complete and the filtration $(\mathcal F)_{t \geq 0}$ generated by $X$ is augmented by $\mathbb P$-null sets.

\begin{remark}
Let us stress that we will assume that the process $X$ is independent of the process $X^0$ considered in Section \ref{sec: boundary}.	
\end{remark}

The process $X$ admits a Fourier decomposition. Recall the basis $\mathcal I_M = \{ \mathsf f_n : n \in \mathbb N^* \times \mathbb Z^2\}$ of Laplace eigenfunctions on $M$. Almost surely we have that
\begin{equation}
X_t = \sum_{n \in \mathbb N^* \times \mathbb Z^2} \mathsf f_n B_t^n,
\end{equation}
where $\{ B_t^n : n \in \mathbb N^* \times \mathbb Z^2 \}$ is a set of i.i.d. complex standard Brownian motions modulo the conjugacy constraint $B^{-n} = \overline {B^n}$. 

We will use the above representation to embed the GFF in this Wiener space. The regularization we will use imposes a cutoff in the $z$-direction and leaves the $\tau$-direction unregularized. Let $(\rho_t)_{t \geq 0}$ be as in the boundary theory and define the multiplier $J_t:\mathbb N^*\times \mathbb Z^2$ for every $(n_1,n_2,n_3) \in \mathbb N^* \times \mathbb Z^2$ by 
\begin{equation}
J_t(n_1,n_2,n_3) = \frac{1}{\sqrt{4\pi^2(n_1/2L)^2+4\pi^2(n_2^2+n_3^2)+m^2}} \sqrt{\partial_t \rho_t^2((n_2,n_3))}. 
\end{equation}
Here, by multiplier we are implicitly using that this acts on functions which can be periodized (such as those with value $0$ at the ends of $M$).

Let us now define the stochastic process $(W_t)_{t \geq 0}$ by
\begin{equation}
W_t(\tau,z) = \sum_{n \in \mathbb N^*\times  \mathbb Z^2} \mathsf f_n(\tau,z) \int_0^t J_s(n) dB_s^{n}. 	
\end{equation}
For every $t>0$, a covariance computation yields that the above series converges almost surely in $\mathcal C^{-1/2-\kappa_0}(M)$ and as $t\rightarrow \infty$ we have that $W_t \rightarrow W_\infty \in L^2(\mathcal C^{-1/2-\kappa_0}, \mathbb P)$. We have that ${\rm Law}_{\mathbb P}(W_\infty) = \mu$ and so $W_\infty$ is the GFF. Let us also note that for every $t\in(0,\infty]$, we have that the process in the $\tau$-variable has the following regularity:
\begin{equation}
\tau \mapsto W_t(\tau,\cdot) \in C^{ 1/2-\kappa_0}_\tau \mathcal C^{-1/2-\kappa_0}_z.	
\end{equation}
Finally, let us note that when $t<\infty$, $W_t \in C^{1/2-\alpha}(M)$ and in particular is a well-defined function.

The space of drifts in the variational problem consists of progressively measurable processes that are $\mathbb P$-a.s. in 
\begin{equation}
\mathcal H:=L^2_t L^2_x.	
\end{equation}
 We denote this space of processes by $\mathbb H$. For every $u \in \mathcal H$, consider the following map $u\mapsto Z_t(u)$,
\begin{equation}
	Z_t(u):= \int_0^t J_s(u) ds,
\end{equation}
where $J_s(u) = \sum_{n \in \mathbb N^* \times \mathbb Z^2} J_s(n) \hat u_s(n) \mathsf f_n$. Note that here we are abusing notation and identifying $u$ with its measure-zero modification on $M$ that makes it periodic.

The following is proposition is a direct consequence of the Bou\'e-Dupuis formula \cite{BD98, U14}. 
\begin{proposition} \label{prop: bulk bouedupuis}
For every $T>0$,
\begin{equation}
-\log \mathsf A_T = \inf_{u \in \mathbb H} \mathbb E \left[ \mathcal V_T(W_T,Z_T(u)) + \frac 12 \| u\|_{\mathcal H}^2 \right],	
\end{equation}
where
\begin{align}
\mathcal V_T(W_T,Z_T(u))&:= \int_{M} 4 \llbracket W_T^3 \rrbracket Z_T(u) + 6 \llbracket W_T^2 \rrbracket Z_T(u)^2 + 12 \llbracket W_T^2 \rrbracket H(\varphi_-,\varphi_+)_T Z_T(u) 
\\ &\qquad + 6 \llbracket W_T^2 \rrbracket \llbracket H(\varphi_-,\varphi_+)_T^2 \rrbracket 
 + 4 W_T Z_T(u)^3 + 12 W_T H(\varphi_-,\varphi_+)_T Z_T(u)^2 
 \\
 &\qquad + 12 W_T \llbracket H(\varphi_-,\varphi_+)_T^2 \rrbracket Z_T(u) + 4 W_T \llbracket H(\varphi_-,\varphi_+)_T^3 \rrbracket + Z_T(u)^4 
 \\
 &\qquad - 2\gamma_T W_T H(\varphi_-,\varphi_+)_T -2\gamma_T  W_T Z_T(u) -\gamma_T \llbracket H(\varphi_-,\varphi_+)_T^2 \rrbracket  
 \\
 &\qquad - 2 \gamma_T H(\varphi_-,\varphi_+)_T Z_T(u) -\gamma_T Z_T(u)^2 \,dx   - (\delta_T(\varphi_-,\varphi_+) - \Delta\delta_T^{0}) 
 \\
 &\quad + \int_M \llbracket H(\varphi_-,\varphi_+)_T^4 \rrbracket - \llbracket \overline H(\varphi_-)_T^4 \rrbracket - \llbracket \overline H(\varphi_+)_T^4 \rrbracket dx
 \\
 &
\quad + \int_{M \Delta \{ [-1,1]\times \mathbb T^2\}} \llbracket \overline H(\varphi_-)_T^4 \rrbracket + \llbracket \overline H(\varphi_+)_T^4 \rrbracket dx.
\end{align}
\end{proposition}

\begin{remark}
We recall that the Wick ordering of the terms $H(\varphi_-,\varphi)_T$ involve the covariance $C^B_T$, whereas the Wick ordering of the terms $\overline H(\varphi_-,\varphi_+)_T$ involve the covariance $\overline C^B_T$.	
\end{remark}

\subsection{Setup for the renormalization}

The variational problem in Proposition \ref{prop: bulk bouedupuis} contains divergences and this is seen by a power-counting heuristic. On the one hand, the regularity of Wick powers of the GFF is well-known. On the other hand, the drift terms satisfy the following regularity estimate whose proof is similar to \cite[Lemma 2]{BG20}.
\begin{lemma}
For every $u \in \mathcal H=L^2_tL^2_x$, we have that
\begin{equation}
\sup_T \|Z_T(u)\|_{H^1}^2 \leq \|u\|_{L^2_tL^2_x}^2.	
\end{equation}	
\end{lemma}

Below, we will analyze $\mathcal V_T$ term-by-term to isolate and renormalize divergences, leading us to a renormalized variational problem that we can handle. We will first identify \emph{bulk stochastic terms}, which produce divergences even in the case of smooth or zero boundary conditions $\varphi_-,\varphi_+$. In addition, we will use the regularity estimates developed in Section \ref{sec: boundary} to identify \emph{boundary-bulk stochastic terms}, which produce divergences because of the rough regularity of the boundary fields (recall the notion of admissible boundary condition). As in \cite{BG20}, we will use a combination of a paracontrolled ansatz on the drift, which induces a divergence in the drift entropy term, to twist these divergences into a form that we can renormalize by using the available mass and energy renormalization counter-terms. The goal is then to be left with remainder terms that be analytically estimated with the coercive terms, consisting of $\int_M Z_T(u)^4 \, dx$ and the residual drift entropy term left after the paracontrolled ansatz. See the end of the section, where we are more precise on the latter point. Let us emphasize that the required drift ansatz is more complicated than in the previous section on the boundary theory due to the need for mass renormalization in the $\varphi^4_3$ model. Furthermore, additional boundary-dependent terms are required in the drift ansatz, which produces extra difficulties and subtleties as compared with the periodic case \cite{BG20}. 

\begin{remark}
Let us emphasize that analyzing the boundary-bulk stochastic terms is a significant point of departure from \cite{BG20}. We will show that many of these terms admit a finite fractional moment -- measured in the appropriate norm -- uniformly in $T$. This will use the additional second moment requirement on the drift term $Z^0$ in the definition of admissible boundary conditions. As such, these norms will be bounded almost surely in the boundary data, and we will show that this is sufficient to obtain the convergence of the bulk amplitudes (again, almost surely in the boundary data). As a consequence, the quantitative bounds we obtain on the bulk amplitudes will depend in a non-trivial way on these norms -- in particular, we will obtain an exponential upper bound in terms of linear combinations of fractional moments of these norms.
\end{remark}

\begin{remark} \label{remark:boundarybulkstochastic}
On a technical level, in order to isolate divergences, we will use paraproduct and resonant product decompositions as done in \cite{BG20}. It will therefore  be convenient to do a periodization trick in \emph{some} integrands, where we replace the integral over $M$ by the integral over the torus $\mathcal M$, obtained from $M$ by identifying $\partial^-M$ and $\partial^+M$. Let us observe that $Z_T(v)$ and the Wick powers of $W_T$ are periodic thanks to the boundary values of $\{ \mathsf f_n \}$. For the harmonic extension terms, there is a discontinuity at the image of $\partial^- M$ and $\partial^+M$ under the identification -- thus we make a measure $0$ modification implicitly. As we shall see in later sections, this will only cause problems when we try to estimate these terms in strictly positive regularity function spaces.	
\end{remark}

Let us now turn to isolating the divergent integrals. By the power-counting heuristic mentioned above, the singular terms in $\mathcal V_T(W_T,Z_T(u))$ are given by the following terms, which are arranged in terms of joint homogenity in $Z_T(u)$ and $H(\varphi_-,\varphi_+)$. 
\begin{align}
	S_1 &:= \int_M 4 \llbracket W^3_T \rrbracket Z_T(u)\,  d x, \\
  S_2 &:= \int_{M} 12 \llbracket W^2_T \rrbracket H (\varphi_-, \varphi_+)_T Z_T(u)\, dx + \int_{\mathcal M} 6\llbracket W^2_T \rrbracket Z^2_T(u) \, dx, \\
  S_3 &:=  \int_{M} 12 W_T \llbracket H (\varphi_-, \varphi_+)_T^2 \rrbracket Z_T(u)
  dx +  4\llbracket H (\varphi_-, \varphi_+)^3_T \rrbracket Z_T(u) \, dx.
\end{align}
Let us stress here that we have used the periodization trick on the second term in $S_2$, since we will use paraproduct and resonant product decompositions.

We now turn to the counter-terms at our disposal. First, let us discuss mass renormalization. Recall that the renormalization constant $\gamma_T$ is defined in terms of the covariance of the periodic GFF. However, the covariance of the stochastic objects above are expressible in terms of the Dirichlet GFF and therefore we will consider the latter and estimate the error term coming from the replacement. Let us define, for every $x \in M$,
\begin{equation}
\gamma_T^M(x):= -3\cdot 4^2 \int_M C_T^M(x,y)^3	
\end{equation}
The mass renormalization counter-terms are contained in the term
\begin{align}
U_1 := &   \int_{\mathcal M} 
  - 2\gamma_T^M W_T H(\varphi_-,\varphi_+)_T -2\gamma_T^M  W_T Z_T(u) -\gamma_T^M \llbracket H(\varphi_-,\varphi_+)_T^2 \rrbracket  
 \\
 &\qquad - 2 \gamma_T^M H(\varphi_-,\varphi_+)_T Z_T(u) -\gamma_T^M Z_T(u)^2 dx.
\end{align}
For the energy renormalization, we will build up the required counter-term in an ad-hoc way as in \cite{BG20}, and show that the difference between this and the term $\delta_T(\varphi_-,\varphi_+)-\delta_0^T$ is finite. Thus, let us set 
\begin{equation}
 U_2 = -\delta_T^M,
\end{equation}
where $\delta_T^M$ is defined below and recalled in \eqref{eqdef: delta M}. Note that, just as for $\delta_T(\varphi_-,\varphi_+)$, the dependence of $\delta_T^M$ on the boundary data is only through whether $\varphi_-$ and $\varphi_+$ have nontrivial $W^0$-components, and thus we omit it from notation. 

The other terms in $\mathcal V_T$, minus the term involving $Z_T(u)^4$, are treated as remainder terms. This is because they do not contain divergences and can be estimated in terms of the coercive terms (which are discussed below) using analytic techniques\footnote{This is slightly inaccurate. The term $W_T Z_T^3$ requires use of the stochastic term $\mathbb W_T^{1\circ [3]}$. We stress that this term does not require additional renormalization to define. Thus the argument to estimate this term is mostly analytic after a double paraproduct decomposition. See the treatment of the analogous term in\cite{BG20}.}. We now list them.
\begin{align}
R_1 = & \int_M 12 W_T H (\varphi_-, \varphi_+)_T Z^2_T + 6
  \llbracket H (\varphi_-, \varphi_+)_T^2 \rrbracket Z^2_T \, dx, \\
  R_2 = &  \int_{ M} 4 W_T Z^3_T \mathd x + 4 H (\varphi_-, \varphi_+)_T Z^3_T \, dx, \\
  R_3 = & \int_M \llbracket H (\varphi_-, \varphi_+)_T^4 \rrbracket \mathd x -
  \int \llbracket \bar{H} (\varphi_-)_T^4 \rrbracket \mathd x - \int \llbracket
  \bar{H} (\varphi_+)_T^4 \rrbracket \mathd x \\
  &\quad + \int_{M \Delta \{ [-1,1]\times \mathbb T^2\}} \llbracket \overline H(\varphi_-)_T^4 \rrbracket + \llbracket \overline H(\varphi_+)_T^4 \rrbracket dx. 
  \\
  R_4 = &  - \int_M (\gamma_T - \gamma_T^M) \Big(2W_T Z_T + \llbracket H(\varphi_-,\varphi_+)_T^2\rrbracket + 2 H(\varphi_-,\varphi_+)_T Z_T + Z_T^2 \Big)
  \\
  & - (\delta_T(\varphi_-,\varphi_+) - \Delta\delta_T^0 - \delta_T^M).
\end{align}
Note that we will define additional remainder terms below that occur when handling the singular terms $S_1-S_3$.

The above definitions now allow us to rewrite the Bou\'e-Dupuis formula as in the following lemma. 

\begin{lemma}
For every $u \in \mathbb H$,
\begin{equation}
\mathbb E \left[ \mathcal V_T(W_T,Z_T(u)) + \frac 12 \| v\|_{\mathcal H}^2 \right]= \mathbb E\left[ \sum_{i=1}^3 S_i + \sum_{i=1}^2 U_i + \sum_{i=1}^4 R_i + \int_M Z_T(u)^4  + \frac 12 \| v\|_{\mathcal H}^2 \right]. 
\end{equation}
\end{lemma}

\subsubsection{Paracontrolled ansatz and coercive terms}

We now define our paracontrolled ansatz. We begin by introducing the stochastic objects that will appear in our analysis. First, we set
\begin{equation}
\mathbb W_T^2:= 12\llbracket W_t^2\rrbracket, \, \mathbb W_T^3 = 4\llbracket W_T^3\rrbracket , \, \mathbb W_T^{[3]}=\int_0^T J_t^2 \mathbb W_t^3 dt.	
\end{equation}
Let us define
\begin{align}
G_T = & \int^T_0 - J^2_t (\mathbb{W}^2_t \succ Z^{\flat}_t) - J^2_t \left(
  \mathbb{W}^2_t H \left( \varphi_- {, \varphi_+}  \right)_T \right) \mathd t,
\end{align}
where the paraproduct $\succ$ is on $\mathcal M$, i.e.\ we are implicitly using the periodicity of the objects above. Let us emphasize that $H(\varphi_-,\varphi_+)_T$ is \emph{not} dependent on $t$. 

The term $Z_t^\flat$ above is defined as follows.  Let $(\theta_t)$ be a family of smooth maps $\mathbb R^2 \rightarrow \mathbb R_+$ such that $t\mapsto \theta_t$ is continuously differentiable and the following is satisfied. There exists $c>0$ such that for all $t \geq 0$,
\begin{equation}
	\theta_t(\xi) \sigma_s(\xi) = 0, \quad \forall t \leq s,
\end{equation}	
and for $t>T_0$ for some $T_0$ sufficiently large, 
\begin{equation}
\theta_t(\xi) = 1, \qquad |\xi| \leq  ct.
\end{equation}
See \cite[Section 4]{BG20} for a construction. For $T>0$ fixed and larger than $T_0$,  let us define $u \mapsto Z_t^\flat(u)$ to be the process on $[0,T]$ defined by:
\begin{equation}
Z_{T;t}^\flat(u):= \mathcal F^{-1}_{\mathbb T^2}(\theta_t) \ast_{\mathbb T^2} Z_T(u), \qquad \forall t \leq T,
\end{equation}
where above we have used the Fourier (inverse) transform and spatial convolution on the periodic direction $\mathbb T^2$. 
By the first property of $\theta_t$, we have that $Z_{T;t}^\flat(u) = \mathcal F^{-1}_{\mathbb T^2}(\theta_t) \ast_{\mathbb T^2} Z_{T'}(u)$ for any $T'\geq t$. Hence we simply write $Z_t^\flat(u)$ in the sequel. By Mikhlin multiplier theorem, 
\begin{equation}
\| Z_t^\flat(u)\|_{L^4} \leq \|Z_T(u)\|_{L^4},
\end{equation}
provided $T>T_0$. We will assume throughout that $T>T_0$ implicitly.

\begin{remark}
The relevance of the $\flat$ is exactly to control terms involving $t<T$ and higher than quadratic homogeneity in the drift using analytic estimates. This is because, in our coercive terms, the quartic term is $\|Z_T(u)\|_{L^4}$ and it is not clear how this term controls homogeneity for $t<T$. We refer to \cite{BG20} for more discussion. 	
\end{remark}

We define the full drift ansatz.
\begin{definition} \label{definition: full ansatz}
For every $T>0$ and every $u \in \mathbb H$, let us write $\ell_T^{\varphi_-,\varphi_+}(u)=\ell_T(u) \in \mathbb H$ to denote the drift defined by
\begin{align} 
   u_t - \ell_T (u)_t 
  &=  - J_t \mathbb{W}^3_t - J_t (\mathbb{W}^2_t \succ Z^{\flat}_t) - J_t
  (\mathbb{W}^2_t H (\varphi_-, \varphi_+)_T) \\
  &\qquad \qquad  - 12 J_t (W_t \llbracket H
  (\varphi_-, \varphi_+)_T^2 \rrbracket) - 4 J_t \llbracket H (\varphi_-,
  \varphi_+)_T^3 \rrbracket ,
\end{align}
when $t \leq T$, and $\ell_T(u) = 0$ for $t > T$.
\end{definition}

We now turn to the coercive terms. 
\begin{definition}
For every $T>0$ and every $u \in \mathbb H$, define
\begin{equation} \label{eqdef: CT-coercive}
\mathcal C_T(u):= \|Z_T(u)\|_{L^4_x}^4 + \frac 12 \|\ell_T(u) \|_{\mathcal H}^2.	
\end{equation}
\end{definition}

\subsection{Step-by-step renormalization}

We now proceed to do a step-by-step renormalization of the singular terms $S_1-S_3$. Below, we will often write $\equiv$ to mean equal up to mean zero terms (with respect to $\mathbb P$). Also, we will drop the dependency of the integral on the spatial domain when clear from context. We remind the reader that only the second term in $S_2$ is periodized, and otherwise the integration domain is $M$.

\subsubsection{Renormalization of $S_1$}

The singular term $S_1$ is controlled by an ansatz on the drift and a component of the energy renormalization $U_2$. 

\begin{lemma} \label{lemma: delta 1}
  Let
  \[ u_t = - J_t \mathbb{W}_t^3 + w_t \label{eq:firstansatz} \]
  and define
  \begin{equation}
  \delta^{1,M}_T:= -\frac 12 \mathbb E\left[ \int^T_0 \int (J_t \mathbb{W}^3_t)^2 \mathd x \mathd t \right].
  \end{equation}
  Then
  \begin{equation}
    S_1 + \frac{1}{2} \int^T_0 \| u \|^2_{L^2} \mathd t - \delta^{1,M}_T  \equiv  \frac{1}{2}
    \int^T_0 \| w \|^2_{L^2} \mathd t.
  \end{equation}
\end{lemma}

\begin{proof}
This is a direct consequence of It\^o's formula. Note that the stochastic integral $\int_0^T u_t \cdot d \mathbb W^3_t $ is mean zero when $T<\infty$ since $\mathbb E[\|u\|_{\mathcal H}^2]<\infty$. 

\end{proof}

\subsubsection{Renormalization of $S_2$}

We now turn our attention to the term $S_2$. The ansatz \eqref{eq:firstansatz} propagates in and thus we will have to analyze its effect on $S_2$. It is convenient to introduce the \emph{intermediate ansatz}
\begin{equation} \label{eq: intermediate ansatz}
Z_T(u) = - \mathbb W^{[3]}_T + K_T,	
\end{equation}
where $K_T = Z_T (w)$. Then we have that
\begin{align}
  S_2 = & -12 \int \llbracket W^2_T \rrbracket H (\varphi_-, \varphi_+)_T
  \mathbb{W}_T^{[3]} \mathd x + 12 \int \llbracket W^2_T \rrbracket H
  (\varphi_-, \varphi_+)_T K_T \mathd x 
  \\
  \quad &+ 6 \int \llbracket W^2_T \rrbracket
  (\mathbb{W}^{[3]}_T)^2 \mathd x - 12 \int \llbracket W^2_T \rrbracket
  \mathbb{W}^{[3]}_T K_T \mathd x + 6 \int \llbracket W^2_T \rrbracket K^2_T
  \mathd x.
\end{align}
Let us define
\begin{equation}
\delta^{2,M}_T := - 6 \mathbb E \left [ \int \llbracket W^2_T \rrbracket
  (\mathbb{W}^{[3]}_T)^2 \mathd x  \right].
\end{equation}
Using Wick's theorem to cancel the first term and $\delta^{2,M}_T$ to cancel the third, we obtain
\begin{equation}
S_2 - \delta^{2,M}_T \equiv \tilde{S}_2,
\end{equation}
where
\begin{equation}
 \tilde S_2:= 12 \int \llbracket W^2_T \rrbracket H (\varphi_-, \varphi_+)_T
   K_T \mathd x - 12 \int \llbracket W^2_T \rrbracket \mathbb{W}^{[3]}_T K_T
   \mathd x + 6 \int \llbracket W^2_T \rrbracket K^2_T \mathd x.
\end{equation}
These divergences will be treated using the counterterm $U_1$ containing the mass renormalization.

Let us begin by further isolating the divergences using a paraproduct decomposition, which we can use thanks to the periodization trick.
\begin{lemma}
  \begin{align} \label{eq: S2tilde}
  \begin{split}
   \tilde{S}_2 &= -12 \int \llbracket W^2_T \rrbracket \circ
     \mathbb{W}^{[3]}_T K_T \mathd x \\ &\qquad \qquad + 12 \left( \int (\llbracket W_T^2
     \rrbracket \succ Z^{\flat}_T) K_T + \int (\llbracket W_T^2 \rrbracket H
     (\varphi_-, \varphi_+)_T) K_T \mathd x \right) + R_5, 
   \end{split}
   \end{align}
  where
  \begin{align}
    R_5 
    := & - 12 \int (\llbracket W^2_T \rrbracket \prec \mathbb{W}^{[3]}_T) K_T
    \mathd x + 6 \int (\llbracket W^2_T \rrbracket \prec K_T) K_T \mathd x -
    6 {\rm Com}_1 (\llbracket W^2_T \rrbracket, K_T) \\
    & + 12 \int (\llbracket W^2_T \rrbracket \succ (Z_T - Z^{\flat}_T)) K_T
    \mathd x, 
  \end{align}
  and 
  \begin{equation}
  	 {\rm Com}_1 (\llbracket W^2_T \rrbracket, K_T) = \int (\llbracket W^2_T
     \rrbracket \succ K_T) K_T \mathd x - \int (\llbracket W^2_T \rrbracket
     \circ K_T) K_T \mathd x. 
   \end{equation}
\end{lemma}

\begin{proof}
The proof is by paraproduct decomposition and reorganization of terms. We first decompose
  \begin{align}
    & 12 \int \llbracket W^2_T \rrbracket \mathbb{W}^{[3]}_T K_T \mathd x
    \nonumber\\
    = & 12 \int (\llbracket W^2_T \rrbracket \succ \mathbb{W}^{[3]}_T) K_T
    \mathd x + 12 \int (\llbracket W^2_T \rrbracket \circ \mathbb{W}^{[3]}_T)
    K_T \mathd x + 12 \int (\llbracket W^2_T \rrbracket \prec
    \mathbb{W}^{[3]}_T) K_T \mathd x. \nonumber
  \end{align}
  The last term is contained in $R_5$. Furthermore
  \[ 6 \int \llbracket W^2_T \rrbracket K^2_T \mathd x = 12 \int (\llbracket
     W^2_T \rrbracket \succ K_T) K_T \mathd x + 6 \int (\llbracket W^2_T
     \rrbracket \prec K_T) K_T \mathd x - 6 {\rm Com}_1 (\llbracket W^2_T \rrbracket,
     K_T) . \]
  The last two terms are contained in $R_5$. Finally,
  \[ -12 \int (\llbracket W^2_T \rrbracket \succ \mathbb{W}^{[3]}_T) K_T \mathd
     x + 12 \int (\llbracket W^2_T \rrbracket \succ K_T) K_T \mathd x = 12
     \int (\llbracket W^2_T \rrbracket \succ Z_T) K_T \mathd x \]
  and
  \[ 12 \int (\llbracket W^2_T \rrbracket \succ Z_T) K_T \mathd x = 12 \int
     (\llbracket W^2_T \rrbracket \succ Z^{\flat}_T) K_T \mathd x + 12 \int
     (\llbracket W^2_T \rrbracket \succ (Z_T - Z^{\flat}_T)) K_T \mathd x. \]
  The last term is contained in $R_5$ as well. 
\end{proof}

We postpone treatment of the resonant product term in \eqref{eq: S2tilde} (see $\Upsilon_1$ below), and instead analyze the terms in the parentheses by using a second drift ansatz.

\begin{lemma}
  Let
  \begin{equation} w_t = - J_t (\mathbb{W}^2_t \succ Z^{\flat}_t) - J_t (\mathbb{W}^2_t H
     (\varphi_-, \varphi_+)_T) + v_s \label{eq:secondansatz} \end{equation}
  where $h_s$ is defined by this relation. Then 
  \begin{align}
    & 12 \left( \int (\llbracket W_T^2 \rrbracket \succ Z^{\flat}_T) K_T +
    \int \llbracket W_T^2 \rrbracket H (\varphi_-, \varphi_+)_T K_T \mathd x
    \right) + \frac{1}{2} \int^T_0 \| w_s \|^2_{L^2} \mathd x \\
    \equiv & - \frac 12\int^T_0 \int (J_t \mathbb{W}^2_t \succ Z^b_t + J_t \mathbb{W}^2_t H
    (\varphi_-, \varphi_+)_T)^2 \mathd x \mathd t + \frac{1}{2} \int^T_0 \| v_s
    \|^2_{L^2} \mathd x + R_6
  \end{align}
  where
  \[ R_6 := 12 \int^T_0 \int (\llbracket W_t^2 \rrbracket \succ
     \dot{Z}^{\flat}_t) K_t \mathd x \mathd t. \]
\end{lemma}

\begin{proof}
  Applying Ito's formula on the products $(\llbracket W_T^2 \rrbracket \succ Z^{\flat}_T) K_T$ and $\llbracket W_T^2 \rrbracket K_T$, we get
  
  \begin{align}
    & 12 \left( \int (\llbracket W_T^2 \rrbracket \succ Z^{\flat}_T) K_T +
    \int (\llbracket W_T^2 \rrbracket H (\varphi_-, \varphi_+)_T) K_T \mathd x
    \right) \nonumber\\
    = & 12 \left( \int_0^T \int (\llbracket W_t^2 \rrbracket \succ
    Z^{\flat}_t) J_t w_t \mathd t + \int^T_0 \int (\llbracket W_t^2 \rrbracket H
    (\varphi_-, \varphi_+)_T) J_t w_t \mathd x \right) + {\rm mart} +  R_6,
  \end{align}
  where
  \begin{equation}
  {\rm mart} := \sum_{i\leq j-2}\int_0^T \int K_t \Delta_i Z_t^\flat d( \Delta_j \mathbb W_t^2) + \int_0^T \int  K_t H(\varphi_-,\varphi_+)_T d\mathbb W_t^2.
  \end{equation}
  The first stochastic integral is a martingale thanks to finite $L^4$ norm to the fourth power after expectation. The second stochastic integral is a martingale and thus mean zero because $\mathbb E[\| u \|_{\mathcal H}^2]~<~\infty$. Hence, ${\rm mart} \equiv 0$.  
  
  Now from the definition of $v$ we get the statement by completing the
  square. 
\end{proof}

We further expand
\begin{eqnarray*}
  &  & \int^T_0 \int (J_t (\mathbb{W}_t^2 \succ Z^{\flat}_t) + J_t (\mathbb{W}_t^2 H
  (\varphi_-, \varphi_+)_T))^2 \mathd x \mathd t\\
  & = & \int^T_0 \int (J_t (\mathbb{W}_t^2 \succ Z^{\flat}_t))^2 \mathd x \mathd
  t + \int^T_0 \int (J_t (\mathbb{W}_t^2 H (\varphi_-, \varphi_+)_T))^2 \mathd x
  \mathd t\\
  &  & + 2 \int^T_0 \int J_t (\mathbb{W}_t^2 \succ Z^{\flat}_t)\,  J_t(
  \mathbb{W}_t^2 H (\varphi_-, \varphi_+)_T) \mathd x \mathd t.
\end{eqnarray*}
All in all we have proven that 
\begin{align}
    \mathbb{E} &[S_1 + S_2] - \delta^{1,M}_T - \delta^{2,M}_T + \frac 12 \int_0^T \|u_t\|_{L^2_x}^2 dt
    \\& = -\frac 12 \mathbb{E} \left[ \int^T_0 \int (J_t
    (\mathbb{W}^2_t \succ Z^{\flat}_t))^2 \mathd x \mathd t \right] -\frac 12\mathbb{E}
    \left[ \int^T_0 \int (J_t (\mathbb{W}^2_t H (\varphi_-, \varphi_+)_T))^2
    \mathd x \mathd t \right] \\
    & -\frac 12 \mathbb{E} \left[ \int^T_0 \int J_t (\mathbb{W}^2_t \succ Z^{\flat}_t) \, 
    J_t (\mathbb{W}_t H (\varphi_-, \varphi_+)_T) \mathd x \mathd t \right]
    \\
    &-12 \mathbb E \left[ \int \llbracket W^2_T \rrbracket \circ
     \mathbb{W}^{[3]}_T K_T \mathd x \right] 
    \\
    & + \mathbb E R_5 + \mathbb E R_6 + \frac 12 \int_0^T \|v_t\|_{L^2_x}^2 dt.
  \end{align}

The first term on the r.h.s needs renormalization, and we put it together with a term from $U_{1}$ to see the cancellation. But first we rewrite this as follows:     \begin{equation}
     \frac 12 \int^T_0 \int (J_t
    \mathbb{W}^2_t \succ Z^{\flat}_t)^2 \mathd x \mathd t = R_7 +\frac 12 \int^T_0 \int (J_t \mathbb{W}^2_t
     \circ J_t \mathbb{W}^2_t)  (Z^{\flat}_t)^2 \mathd x \mathd t,
     \end{equation}
     where
\begin{equation} 
R_7 := \frac 12 \int^T_0 \int (J_t (\mathbb{W}^2_t \succ
     Z^{\flat}_t))^2 \mathd x \mathd t -\frac 12 \int^T_0 \int (J_t \mathbb{W}^2_t
     \circ J_t \mathbb{W}^2_t)  (Z^{\flat}_t)^2 \mathd x \mathd t.
\end{equation}
We now treat the remaining term. Note that
\begin{equation}
\gamma_T^M(x):= - \mathbb E \left[ \int_0^T J_t \mathbb W_t^2 \cdot  J_t \mathbb W_t^2 (x) dt \right].
\end{equation}For every $T>0$, define  
\begin{equation}
\Upsilon_{1, T} : = - \int^T_0 \int (\frac 12 J_t \mathbb{W}^2_t \circ J_t
     \mathbb{W}^2_t + \dot \gamma_t^M)  (Z^{\flat}_t)^2 \mathd x \mathd t  - \int (\mathbb W_T^2 \circ \mathbb W_T^{[3]} + 2\gamma_T^M W_T )K_T. 
\end{equation}
This is the bulk stochastic term

We now analyze the boundary-dependent  remaining terms. Using the independence of the boundary conditions and $\mathbb P$, we have that 
\begin{multline}
    \mathbb{E} \left[ \int^T_0 \int (J_t (\mathbb{W}^2_t H (\varphi_-,
    \varphi_+)_T))^2 \mathd x \mathd t \right] \\
    =  3\cdot 4^2 \int \int C_T^M (x, y)^3 H (\varphi_-, \varphi_+)_T (x) H
    (\varphi_-, \varphi_+)_T (y) \mathd x \mathd y.
  \end{multline}
We will want to understand how the righthand side behaves as a random variable when $\varphi_-,\varphi_+ \sim \nu^0$. To this end, we introduce the following boundary-bulk stochastic term. For every $\varphi_-,\varphi_+ \in \mathcal C^{-1/2-\kappa_0}_z$, define  
\begin{align}
\Upsilon_{2, T}(\varphi_-,\varphi_+) :=& \frac{12^2}{3} \int \int C_T^M (x, y)^3 H
     (\varphi_-, \varphi_+)_T (x) H (\varphi_-, \varphi_+)_T (y) \mathd x \mathd y
     \\&- \int \gamma^M_T  \llbracket H (\varphi_-, \varphi_+)^2_T
     \rrbracket \mathd x - \delta_T^{3,M},
\end{align}
where
\begin{equation}
\delta_T^{3,M} :=  3\cdot 4^2 \int \int C_T^M (x, y)^3 \mathbb E_{\tilde\mu^0\otimes \tilde\mu^0} \left[ H
     (\varphi_-, \varphi_+)_T (x) H (\varphi_-, \varphi_+)_T (y) \right] \mathd x \mathd y,
\end{equation}
where we recall that $\tilde\mu^0$ is the boundary GFF coming from the domain Markov property.

We are now left with one boundary-bulk term, which we will treat as a remainder term.
\begin{equation}
R_8
    :=  \int^T_0 \int (J_t (\mathbb{W}^2_t \succ Z^{\flat}_t)) (J_t(
    \mathbb{W}^2_t \, H (\varphi_-, \varphi_+)_T)) \mathd x \mathd t - \int 2 \gamma^M_T H (\varphi_-, \varphi_+)_T Z_T \mathd x.
\end{equation}

\subsection{Renormalization of $S_3$}

Using the intermediate ansatz, we have
\begin{align}
  S_3 = & 12 \int W_T \llbracket H (\varphi_-, \varphi_+)^2_T \rrbracket K_T
  \mathd x + 4 \int \llbracket H (\varphi_-, \varphi_+)^3_T \rrbracket K_T
  \mathd x \\
  & - 12 \int W_T \llbracket H (\varphi_-, \varphi_+)^2_T \rrbracket
  \mathbb{W}^{[3]}_T \mathd x - 4 \int \llbracket H (\varphi_-, \varphi_+)^3_T
  \rrbracket \mathbb{W}^{[3]}_T \mathd x \nonumber
  \\
  &\equiv 12 \int W_T \llbracket H (\varphi_-, \varphi_+)^2_T \rrbracket K_T
  \mathd x + 4 \int \llbracket H (\varphi_-, \varphi_+)^3_T \rrbracket K_T
  \mathd x.
\end{align}
We will now renormalize these remaining singular terms by using the full drift ansatz to twist these divergences into a form that can be tamed using the the (remaining) counter-terms in the mass renormalization $U_1$. 

Let $L_T=Z_T(v)$. Then by \eqref{eq:secondansatz}, $K_T = L_T + G_T$, from which we obtain
\begin{equation} S_3 = 12 \int W_T \llbracket H (\varphi_-, \varphi_+)^2_T \rrbracket L_T
   \mathd x + 4 \int \llbracket H (\varphi_-, \varphi_+)^3_T \rrbracket L_T
   \mathd x + R_9,
 \end{equation}
  where
\begin{equation} 
R_9 := 12 \int W_T \llbracket H (\varphi_-, \varphi_+)^2_T \rrbracket G_T
   \mathd x + 4 \int \llbracket H (\varphi_-, \varphi_+)^3_T \rrbracket G_T
   \mathd x. 
   \end{equation}
 
 We will now do the further drift ansatz. Recall the definition of $\ell_T(u)$. Notice that 
\begin{equation} 
v_t = 12 J_t W_t \llbracket H (\varphi_-, \varphi_+)^2_T \rrbracket + 4
     J_t \llbracket H (\varphi_-, \varphi_+)^3_T \rrbracket + \ell_T(u)_t.
\end{equation}
 By It\^o's formula and completing the square, we therefore have the following lemma.
\begin{lemma}
For every $T>0$ and every $u \in \mathbb H$, 
  \begin{align}
    & 12\mathbb{E} \int W_T \llbracket H (\varphi_-, \varphi_+)^2_T \rrbracket
    L_T \mathd x + 4\mathbb{E} \int \llbracket H (\varphi_-, \varphi_+)^3_T
    \rrbracket L_T \mathd x + \frac{1}{2} \mathbb{E} \int^T_0 \| v_t \|^2
    \mathd t \nonumber\\
    = & -\frac 12 \mathbb{E} \int^T_0 \int  (12 J_t (W_t \llbracket H (\varphi_-,
    \varphi_+)^2_T \rrbracket) + 4 J_t \llbracket H (\varphi_-, \varphi_+)^3_T
    \rrbracket)^2 \mathd t \mathd x + \frac{1}{2} \int^T_0 \| \ell_T(u)_t \|^2 \mathd
    t \nonumber. 
  \end{align}
\end{lemma}

The first two terms on the righthand side are singular, but they do not depend on the drift. We will therefore define a suitably renormalized version of them (using the remaining terms in $U_2$) as new boundary-bulk stochastic objects. We first expand and use that $\mathbb E[W_T]=0$ to obtain
\begin{align}
  & \frac 12 \mathbb{E} \int^T_0 \int  (12 J_t (W_t \llbracket H (\varphi_-,
  \varphi_+)^2_T \rrbracket) + 4 J_t\llbracket H (\varphi_-, \varphi_+)^3_T
  \rrbracket)^2 \mathd t \mathd x \\
  = & \frac{12^2}2\mathbb{E} \int^T_0 \int (J_t (W_t \llbracket H (\varphi_-,
  \varphi_+)^2_T \rrbracket))^2 \mathd t \mathd x + \frac{4^2}2 \mathbb{E} \int^T_0 \int
   (J_t\llbracket H (\varphi_-, \varphi_+)^3_T \rrbracket)^2 \mathd t \mathd x.
\end{align}
Let us define
\begin{align}
\delta^{4,M}_T&:= \frac{12^2}2 \mathbb E\,  \mathbb E_{\tilde\mu^0\otimes \tilde\mu^0}\left[ \int_0^T \int (J_t(\mathbb W_t \,\llbracket  H(\varphi_-,\varphi_+))^2_T \rrbracket )^2 dx dt \right],
\\
\delta^{5,M}_T&:= \frac{4^2}2   \mathbb E_{\tilde\mu^0\otimes \tilde\mu^0}\left[ \int_0^T \int J_t(\llbracket  H(\varphi_-,\varphi_+)^3_T\rrbracket )^2 dx dt \right], 
\end{align}
where we stress we that the expectations are once again taken with respect to the boundary GFF coming from the domain Markov property.
We now define the two boundary-bulk stochastic terms,
\begin{align}
\Upsilon_{3,T}&:= -\frac{12^2}2 \mathbb{E} \int^T_0 \int  (J_t(W_t \llbracket H
     (\varphi_-, \varphi_+)^2_T \rrbracket))^2 \mathd t \mathd x + \delta^{4,M}_T,
\\
	\Upsilon_{4,T} &:= -\frac{4^2}2  \mathbb{E} \int^T_0 \int 
     ( J_t\llbracket H (\varphi_-, \varphi_+)^3_T \rrbracket)^2 + \delta^{5,M}_T.
\end{align}

\subsection{Renormalized variational problem and strategy of proof} \label{subsec: renormalized var}

The following proposition summarizes the outcome of the computations above.

\begin{proposition} \label{prop: renormalized cost fcn}
For every $u \in \mathbb H$, 
\begin{equation}
\mathbb E \left [ \mathcal V_T(W,u) + \frac 12 \|u\|_{L^2_tL^2_x}^2 \right] = \mathbb E\left[ \sum_{i=1}^9 R_i + \sum_{i=1}^4\Upsilon_{i,T} + \mathcal C_T(u)\right]	
\end{equation}
\end{proposition}

As we shall see in the subsequent sections, the function on the righthand side does not contain further divergences and thus we have arrived at the renormalized cost function and renormalized variational problem, as formalized below. As preparation, we can now define the boundary-bulk enhancement. Let $\Xi^\partial_T(\varphi_-,\varphi_+)$ consist of the Wick powers of the harmonic extenstion of $\varphi_-,\varphi_+$, together with $\Upsilon_2,\Upsilon_3$, and $\Upsilon_4$.

\begin{definition} \label{def: renormalized cost fcn}
Let $T>0$. The \emph{renormalized cost function} is the map given, for every $u \in \mathbb H$, by
\begin{equation}
\mathbb F_T^{\Xi^\partial_T(\varphi_-,\varphi_+)}(u):= \mathbb E \left[ \Phi^{\varphi_-,\varphi_+}_T(u) + \mathcal C_T(u) \right],	
\end{equation}
where
\begin{equation}
\Phi_T^{\varphi_-,\varphi_+}(u) = 	\sum_{i=1}^9 R_i + \sum_{i=1}^4\Upsilon_{i,T}.
\end{equation}
The \emph{renormalized variational problem} is given by
\begin{equation}
\inf_{u \in \mathbb H} 	\mathbb F_T^{\Xi^\partial_T(\varphi_-,\varphi_+)}(u)
\end{equation}

\end{definition}

\begin{remark} \label{rem: renormalized cost fcn f}
The renormalized cost function can be trivially adapted to include the term $f$. In this case, we will denote the renormalized cost function $\mathbb F_T^{f; \Xi^\partial_T(\varphi_-,\varphi_+)}$ and the corresponding potential term is denoted $\Phi^{f; \varphi_-,\varphi_+}_T$.	
\end{remark}

Let us now describe briefly the organization of Sections \ref{sec:stochastic}-\ref{sec: convergence of bulk amplitudes} and thus the proof of Theorem \ref{thm: bulk}. The actual convergence argument for Theorem \ref{thm: bulk} is completed in Section \ref{sec: convergence of bulk amplitudes} via the theory of $\Gamma$-convergence. This splits into two parts: establishing $\Gamma$-convergence and establishing equicoercivity of the renormalized cost functions. The $\Gamma$-convergence proof is similar to the one in the boundary case, although the treatment of the remainder term at $T=\infty$ is more subtle. For the equicoercivity, we need to establish uniform bounds on the terms $R_i$ and $\Upsilon_i$ in terms of the coercive terms $\mathcal C_T(u)$ and constants. For the remainder terms, this follows by analytic arguments in Section \ref{sec:equicoercive}. The $\Upsilon_i$ terms are analyzed in Section \ref{sec:stochastic}, where we also establish Theorem \ref{thm: enhancement converge}. In the case of the boundary-bulk terms, as mentioned in Remark \ref{remark:boundarybulkstochastic}, we obtain that these terms are bounded in expectation in the bulk almost surely in the realization of the boundary data.

\section{Stochastic estimates}
\label{sec:stochastic}

In this section, we will estimate the stochastic terms $\Upsilon_{i,T}$ for $i=1,\dots,4$. In the case of the bulk terms $i=1$, we will define the \emph{bulk enhancement} of $W$ and the desired estimates will follow from estimates on this process. In the case of the boundary-bulk terms $i=2,3,4$, we will estimate them when $\varphi_-,\varphi_+$ are i.i.d. distributed according to an admissible boundary condition, although we will mainly apply this when their law is given by $\nu^0_T$. We will in fact prove convergence of these terms to a limiting enhanced boundary field. At the end, we will upgrade these bounds to convergence and use this, together with stochastic estimates developed in Section \ref{sec: boundary}, to prove Theorem \ref{thm: enhancement converge}.

\subsection{Bulk enhancement and $\Upsilon_{1,T}$} \label{subsec: ups1}

The \emph{bulk enhancement} of $W$ at scale $T$ is the vector of stochastic processes $\Xi_T(W)$ defined by
\begin{equation}
	\Xi_T(W):= \left( (W_t)_{t \leq T}, (\llbracket W_t^2 \rrbracket)_{t \leq T}, (\mathbb W_t^{[3]})_{t\leq T}, (\mathbb W_t^{1 \diamond [3]})_{t \leq T}, (\mathbb W_t^{2\diamond 2})_{t \leq T}, (\mathbb W_t^{2\diamond[3]})_{t \leq T} \right).
\end{equation}
Above, the first three stochastic processes on the righthand side have already been defined. The remaining ones are defined by 
\begin{equation}
\mathbb W_T^{1 \diamond [3]}:= W_T \circ \mathbb W_T^{[3]}, \, \mathbb W_T^{2 \diamond 2}:= J_t \mathbb W_t^2 \circ J_t \mathbb W_t^2 + 2\dot\gamma_t^M, \, \mathbb W_T^{2 \diamond [3]}:= \mathbb W_T^2 \circ \mathbb W_T^{[3]}+2\gamma_T^MW_T.  
\end{equation}
This vector naturally lives on a Banach space 
\begin{equation}
	\mathfrak S:= \mathcal C^{-1/2-\kappa}_x\times \mathcal C^{-1-\kappa}_x \times \mathcal C^{1/2-\kappa}_x \times \mathcal C^{-\kappa}_x \times L^1_t\mathcal C^{-\kappa}_x \times \mathcal C^{-1/2-\kappa}_x,
\end{equation}
where we endow $\mathfrak S$ with the sum of the norms of the individual spaces in the product.

The following proposition gives uniform in $T$ moment bounds on the components of $\Xi_T(W)$ and resulting convergence to a \emph{limiting bulk enhancement} of $W$ in probability. Its proof follows by straightforward albeit tedious modifications of \cite[Section 8]{BG20} into real space and uses the kernel estimates on Dirichlet Green functions found in Appendix \ref{sec: Dirichlet covariance estimates}. We omit the details.  
\begin{proposition} \label{prop: bulk enhancement}
For every $p \in [1,\infty)$, there exists $C>0$ such that for every $T>0$,
\begin{equation}
	\mathbb E[\|W_T\|_{\mathcal C^{-1/2-\kappa}_x}^p + \|\llbracket W_T^2\rrbracket \|_{\mathcal C^{-1-\kappa}_x}^p + \| \mathbb W_T^{[3]}\|_{\mathcal C^{1/2-\kappa}_x}^p] \leq C
\end{equation}
and
\begin{equation}
\mathbb E[\|\mathbb W_T^{1 \diamond [3]}\|_{\mathcal C^{-\kappa}_x}^p]+\mathbb E[\| \mathbb W_T^{2\diamond 2}  \|_{L^1_t \mathcal C^{-\kappa}_x}^p] + \mathbb E[\| \mathbb W_T^{2\diamond[3]}\|_{\mathcal C^{-1/2-\kappa}_x}^p] \leq C.	
\end{equation}
Furthermore, there exist random variables 
\begin{equation}
W_\infty, \, \llbracket W_\infty^2\rrbracket, \, \mathbb W_\infty^{[3]}, \, \mathbb W_\infty^{1 \diamond [3]}, \, \mathbb W_\infty^{2\diamond 2}, \, \mathbb W_\infty^{2\diamond[3]}	
\end{equation}
and limiting bulk enhancement
\begin{equation}
	\Xi_\infty(W):= \left( (W_t)_{t \leq \infty}, (\llbracket W_t^2 \rrbracket)_{t \leq \infty}, (\mathbb W_t^{[3]})_{t\leq \infty}, (\mathbb W_t^{1 \diamond [3]})_{t \leq \infty}, (\mathbb W_t^{2\diamond 2})_{t \leq \infty}, (\mathbb W_t^{2\diamond[3]})_{t \leq \infty} \right)
\end{equation}
 such that, for every $p \in [1,\infty)$,
 \begin{equation}
 	\lim_{T \rightarrow \infty} \mathbb E[\|\Xi_T(W)-\Xi_\infty(W)\|_{\mathfrak S}^p] = 0. 
 \end{equation}
\end{proposition}

As a direct consequence of the preceding lemma together with duality and standard estimates, we obtain the following bound on $\Upsilon_{1,T}$.
\begin{lemma}
For every $\delta > 0$ sufficiently small, there exists $C=C(\delta) > 0$ such that for every $T>0$,
\begin{equation}
	\mathbb E[|\Upsilon_{1,T}|] \leq C_\delta + \delta \mathbb E[\mathcal C_T(u)].  
\end{equation}
\end{lemma}

\subsection{Stochastic estimates on $\Upsilon_{2,T}$}

We prove an estimate on the (almost) first moment of $\Upsilon_{2,T}$ when $\varphi_-,\varphi_+$ are i.i.d. according to an admissible law on boundary conditions. 
\begin{lemma} \label{lem: upsilon2}
Let $\mathbb P^0$ be a law on admissible boundary conditions $\phi^0:=(W^0,Z^0)$. Then for every $\alpha \in (0,1)$, there exists $C>0$ such that, for every $T>0$,
\begin{equation}
\mathbb E_{\mathbb P^0\otimes \mathbb P^0} [|\Upsilon_{2,T}(W^0_-+Z^0_-,W^0_++Z^0_+)|^\alpha]\leq C + \mathbb E_{\mathbb P^0\otimes \mathbb P^0}[\|Z^0_+\|_{H^{1/2-\kappa}_z}^2+\|Z^0_-\|_{H^{1/2-\kappa}_z}^2]. 
\end{equation}
Furthermore, the pushforward under the regularized boundary measure $(\Upsilon_{2, T})^* (\nu^0_T\otimes \nu^0_T)$ converges in law to a random variable 
  $(\Upsilon_{2,\infty})^* (\nu^0\otimes \nu^0)$.
  \end{lemma}
  
 Lemma \ref{lem: upsilon2} will follow from Lemmas \ref{lem: upsilon2 a} and \ref{lem: upsilon2 b}. We split $\Upsilon_{2,T}$ into two terms, one which depends purely on $W^0_-$ and $W^0_+$, and one which has dependency on the regular parts $Z^0_-,Z^0_+$, and the latter lemmas handle these parts separately. To this end, let us write
  \begin{equation}
  \Upsilon_{2,T}(W_-^0+Z^0_-,W^0_++Z^0_+)=:\Upsilon_{2,T}^a(W^0_-,W^0_+) + \Upsilon^b_{2,T}(W^0_-,Z^0_-,W^0_+,Z^0_+),	
  \end{equation}
  where
  \begin{align}
  \Upsilon_{2,T}^a(W^0_-,W^0_+)&:= \frac{12^2}3 \int \int C^M_T(x,y)^3 H(W^0_-,W^0_+)_T(x) H(W^0_-,W^0_+)_T(y) dx dy 
  \\
  &\quad - \int \gamma_T^M \llbracket H(\varphi_-,\varphi_+)^2_T\rrbracket dx - \delta^{3,M}_T.
  \end{align}

The first lemma handles $\Upsilon^a_{2,T}$, for which we can use Gaussian computations. We estimate its variance. Recall that $\mu^0$ is the marginal of $W^0$ under $\mathbb P^0$.
 \begin{lemma}\label{lem: upsilon2 a}
 There exists $C>0$ such that, for every $T>0$,
 \begin{equation} \label{eq: upsilon2 a 2ndmoment}
 	\mathbb E_{\mu^0\otimes \mu^0}[|\Upsilon_{2,T}^a|^2] \leq C. 
 \end{equation}
    
 \end{lemma}
 
 \begin{proof}
Since the covariance of $\mu^0$ is equal to the covariance of $\tilde \mu^0$ up to an infinitely smoothing term, it is sufficient to do this estimate where the expectation is taken with respect to $\tilde \mu^0$. Recall the definitions of the covariances $C^M_T$ and $C^B_T$. Let us write $\Upsilon^{a}_{2,T}=I_1+I_2$, where 
\begin{align}
  I_1&:= \frac{12^2}{3} \int \int C_T^M (x, y)^3 H
     (\varphi_-, \varphi_+)_T (x) H (\varphi_-, \varphi_+)_T (y) dx dy  - \int   \gamma^M_T  H (\varphi_-, \varphi_+)_T^2
      \mathd x
      \\
      &= \frac{12^2}3 \int \int C^M_T(x,y)^3 H(\varphi_-,\varphi_+)_T(x) \left( H(\varphi_-,\varphi_+)_T(y) - H(\varphi_-,\varphi_+)_T(x) \right),
  \end{align}
  and
  \begin{align}
  I_2&:= \int \gamma_T^M \mathbb E_{\tilde\mu^0\otimes \tilde\mu^0}[ H(\varphi_-,\varphi_+)_T^2]  dx - \delta^{3,M}_T
  \\
  &= \frac{12^2}3 \int \int C^M_T(x,y)^3 \left( C^B_T(x,x)-C^B_T(x,y) \right)
  = - \mathbb E_{\tilde\mu^0\otimes \tilde\mu^0}[I_1].
  \end{align}
  Hence,
  \begin{equation}
  \mathbb E_{\tilde\mu^0\otimes \tilde\mu^0}[|\Upsilon_{2,T}^a|^2] = \mathrm {Var}_{\tilde\mu^0\otimes\tilde\mu^0}(I_1).
  \end{equation}

We now compute the second moment of $I_1$. By expanding and using Wick's theorem, 
  \begin{align}
  &\frac{3^2}{12^4}\mathbb E_{\tilde\mu^0\otimes\tilde\mu^0}[I_1^2] \\ &= \int \int \int \int  C^M_T(x_1,y_1)^3 C^M_T(x_2,y_2)^3 \mathbb E_{\tilde\mu^0 \otimes \tilde\mu^0} \Bigg[ H(\varphi_-,\varphi_+)(x_1)H(\varphi_-,\varphi_+)(x_2) 
  \\
  &\qquad \qquad 
  \times \Big( H(\varphi_-,\varphi_+)(y_1) - H(\varphi_-,\varphi_+)(x_1) \Big)\Big(H(\varphi_-,\varphi_+)(y_2) - H(\varphi_-,\varphi_+)(x_2) \Big) \Bigg] 
  \\
  &=\int \int \int \int  C^M_T(x_1,y_1)^3 C^M_T(x_2,y_2)^3 \Bigg\{ \left( C^B_T(x_1,y_1) - C^B_T(x_1,x_1) \right) \left( C^B_T(x_2,y_2) - C^B_T(x_2,x_2) \right)
  \\ & \qquad \qquad \qquad +C^B_T(x_1,x_2) \left( C^B_T(y_1,y_2) - C^B_T(x_1,y_2) - C^B_T(x_2,y_1) + C^B_T(x_1,x_2) \right)
  \\
  & \qquad \qquad \qquad 
   +  \left( C^B_T(x_1,y_2) - C^B_T(x_1,x_2) \right) \left( C^B_T(x_2,y_1) - C^B_T(x_1,x_2) \right) \Bigg\}.
  \end{align}
  However, note that
  \begin{align}
  \frac{3^2}{12^4}\mathbb E_{\tilde\mu^0\otimes\tilde\mu^0}[I_1]^2 =  \int \int \int \int & C^M_T(x_1,y_1)^3 C^M_T(x_2,y_2)^3
  \\ & \left( C^B_T(x_1,x_1) - C^B_T(x_1,y_1) \right) \left( C^B_T(x_2,x_2) - C^B_T(x_2,y_2) \right). 
  \end{align}
  Hence, by symmetry considerations
  \begin{align}
  &\frac{3^2}{12^4} \mathbb E_{\tilde\mu^0 \otimes \tilde\mu^0}[|\Upsilon_{2,T}^a|^2] \\
  &=\int \int \int \int  C^M_T(x_1,y_1)^3 C^M_T(x_2,y_2)^3 \Bigg\{ \left( C^B_T(x_1,y_2) - C^B_T(x_1,x_2) \right) \left( C^B_T(x_2,y_1) - C^B_T(x_1,x_2) \right) 
  \\ & \qquad \qquad +C^B_T(x_1,x_2) \left( C^B_T(y_1,y_2) - C^B_T(x_1,y_2) - C^B_T(x_2,y_1) + C^B_T(x_1,x_2) \right) \Bigg\} 
   \\
   &=\int \int \int \int  C^M_T(x_1,y_1)^3 C^M_T(x_2,y_2)^3 \Bigg\{\left( C^B_T(x_1,y_2) - C^B_T(x_1,x_2) \right) \left( C^B_T(x_2,y_1) - C^B_T(x_1,x_2) \right)
   \\ & \quad \quad \qquad + \left( C^B_T(x_1,x_2) - C^B_T(x_1,y_2) \right) \left( C^B_T(y_1,y_2) - C^B_T(x_1,y_2)\right) \Bigg\}
\\ 
 =2 \int &\int \int \int  C^M_T(x_1,y_1)^3 C^M_T(x_2,y_2)^3 \Bigg\{ \left( C^B_T(x_1,x_2) - C^B_T(x_1,y_2) \right) \left( C^B_T(y_1,y_2) - C^B_T(x_1,y_2)\right)\Bigg\}.
  \end{align}
  
  We denote the righthand side of the above as $J$. In order to estimate this, we will need the estimate on the gradient of the Poisson kernel $C^B_T$, see Lemma \ref{lem: gradient Cb}. Let us write $\ell(x_i,y_j)$ to denote the straight line connecting $x_i$ and $y_j$. Note that for all $\xi \in \ell(x_i,y_j)$, $d(\xi,B) \geq \min (d(x_i,B), d(y_j,B))$. Therefore, by the mean-value theorem and the gradient estimate, for every $\epsilon,\delta > 0$ sufficiently small, 
  \begin{align}
  	|C^B_T(x_1,x_2) - C^B_T(x_1,y_2)| &\leq C \left(\sup_{\xi \in \ell(x_2,y_2)}|\nabla_{e_2} C^B_T(x_1,\xi )||x_2-y_2| \right)^\varepsilon
  	\\
  	&\quad \quad \times \left( \frac{1}{|x_1-x_2|^{1-\epsilon}} + \frac{1}{|x_1-y_2|^{1-\epsilon}}\right)
  	\\
  	&\leq C \frac{|x_2-y_2|^\varepsilon}{{\rm min}(d(x_2,\partial M)^{(2-\delta)\epsilon},d(y_2,\partial M)^{(2-\delta)\epsilon}) \cdot   d(x_1,\partial M)^{\delta\epsilon}}  
  	\\
  	&\quad \quad \quad \times \left( \frac{1}{|x_1-x_2|^{1-\epsilon}} + \frac{1}{|x_1-y_2|^{1-\epsilon}}\right).
  	 \end{align}
  	 In particular, by redefining $\delta$, we obtain the estimate
  	 \begin{align}
  	 &|C^B_T(x_1,x_2) - C^B_T(x_1,y_2)| \\   \leq& C \frac{|x_2-y_2|^\delta}{{\rm min}(d(x_2,\partial M)^{\delta},d(y_2,\partial M)^{\delta}) \cdot   d(x_1,\partial M)^{\delta}}
  	\left( \frac{1}{|x_1-x_2|^{1-\epsilon}} + \frac{1}{|x_1-y_2|^{1-\epsilon}}\right).
  	 \end{align}
  Hence, we may bound

  \begin{align}
  |J|\\ \leq& C \int \int \int \int \frac{\max(d(x_2,\partial M)^{-\delta},d(y_2,\partial M)^{-\delta})\max (d(x_1,\partial M)^{-\delta},d(y_1,\partial M)^{-\delta})}{|x_1-y_1|^{3-\delta}|x_2-y_2|^{3-\delta}} 
  \\
  &\times d(x_1,\partial M)^{-\delta}d(y_2,\partial M)^{-\delta}
  \\
  &\times \left( \frac{1}{|x_1-x_2|^{1-\epsilon}} + \frac{1}{|x_1-y_2|^{1-\epsilon}}\right)\left( \frac{1}{|y_1-y_2|^{1-\epsilon}} + \frac{1}{|y_2-x_1|^{1-\epsilon}}\right)
  \\
  &=: J_1 + J_2 + J_3 + J_4,
  \end{align}
  where the $J_i$ are defined according to expanding the brackets in the last line above (see below). Note that, by symmetry, $J_2=J_3$ (these correspond to the cross terms in the expansion above). 
  
  We will first estimate $J_1$. Note that, also by symmetry,
  \begin{align}
  J_1&:= C \int\int\int\int \frac{\max(d(x_2,\partial M)^{-\delta},d(y_2,\partial M)^{-\delta})\max (d(x_1,\partial M)^{-\delta},d(y_1,\partial M)^{-\delta})}{|x_1-y_1|^{3-\delta}|x_2-y_2|^{3-\delta}|x_1-x_2|^{1-\delta}|y_1-y_2|^{1-\delta}}
  \\&\qquad\qquad\ \times  d(x_1,\partial M)^{-\delta}d(y_2,\partial M)^{-\delta}
  \\
  &\leq C' \int \int \int \int \left( \frac{1}{|x_1-y_1|^{3-\delta}|x_2-y_2|^{3-\delta}|x_1-x_2|^{1-\delta}|y_1-y_2|^{1-\delta}} \right)
  \\
  &\quad\quad \times  \Big( d(x_2,\partial M)^{-2\delta}d(x_1,\partial M)^{-2\delta} + d(y_2,\partial M)^{-2\delta}d(x_1,\partial M)^{-2\delta} 
  \\
  &\qquad \qquad \qquad + d(x_2,\partial M)^{-2\delta}d(y_1,\partial M)^{-2\delta} + d(y_2,\partial M)^{-2\delta}d(y_1,\partial M)^{-2\delta}\Big) 
  \\
  &= 2C' \int \int \int \int \left( \frac{d(x_1,\partial M)^{-2\delta}d(x_2,\partial M)^{-2\delta} + d(x_1,\partial M)^{-2\delta}d(y_2,\partial M)^{-2\delta}}{|x_1-y_1|^{3-\delta}|x_2-y_2|^{3-\delta}|x_1-x_2|^{1-\delta}|y_1-y_2|^{1-\delta}} \right)
  \\
  &=: J_1(a) +J_1(b).
  \end{align}
  We estimate $J_1(a)$ first. Integrating out $y_1$ and $y_2$ respectively, we obtain
  \begin{align}
  J_1(a) &\leq C'' \int \int \int \frac{d(x_1,\partial M)^{-2\delta}d(x_2,\partial M)^{-2\delta}}{|x_1-y_2|^{1-2\delta}|x_2-y_2|^{3-\delta}|x_1-x_2|^{1-\delta}}
  \\
  &\leq C'''\int \int \frac{d(x_1,\partial M)^{-2\delta}d(x_2,\partial M)^{-2\delta}}{|x_1-x_2|^{2-4\delta}} \leq C_4
  \end{align}
where the last inequality is by Young's convolution inequality.
We now estimate $J_1(b)$. Integrating out $y_1$ and then $x_2$, and then using Young's convolution inequality, we obtain
\begin{align}
J_1(b) &\leq C'' \int \int \int \frac{d(x_1,\partial M)^{-2\delta}d(y_2,\partial M)^{-2\delta}}{|x_1-y_2|^{1-2\delta}|x_2-y_2|^{3-\delta}|x_1-x_2|^{1-\delta}}
\\
&\leq C''' \int \int \frac{d(x_1,\partial M)^{-2\delta}d(y_2,\partial M)^{-2\delta}}{|x_1-y_2|^{2-4\delta}} \leq C_4.
\end{align}   

Let us now turn to $J_2$ (which in turn bounds $J_3$ by symmetry). As by the above manipulations (with an extra straightforward Cauchy-Schwarz step, c.f. the estimate of $J_4$ below), we obtain 
\begin{align}
J_2 &\leq C' \int \int \int \int \frac{d(x_1,\partial M)^{-4\delta} + d(x_2,\partial M)^{-4\delta} + d(y_1,\partial M)^{-4\delta} + d(y_2,\partial M)^{-4\delta}}{|x_1-y_1|^{3-\delta}|x_2-y_2|^{3-\delta}|x_1-x_2|^{1-\delta}|y_1-y_2|^{1-\delta}}	
\\
&=: J_2(a) + J_2(b) + J_2(c) + J_2(d).
\end{align}
We will only estimate $J_2(a)$, the other terms follow by similar considerations. Integrating out $y_2$ and then $y_1$, and then using Young's convolution inequality, 
\begin{align}
J_2(a) &\leq C'' \int \int \int \frac{d(x_1,\partial M)^{-4\delta}}{|x_1-y_1|^{3-\delta}|x_2-y_1|^{1-2\delta}|x_1-x_2|^{1-\delta}}
\\
&\leq C''' \int \int \frac{d(x_1,\partial M)^{-4\delta}}{|x_1-x_2|^{2-4\delta}} \leq C_4. 
\end{align}

Finally, let us estimate $J_4$. 
\begin{align}
J_4 &:= C \int\int\int\int \frac{\max(d(x_2,\partial M)^{-\delta},d(y_2,\partial M)^{-\delta})\max (d(x_1,\partial M)^{-\delta},d(y_1,\partial M)^{-\delta})}{|x_1-y_1|^{3-\delta}|x_2-y_2|^{3-\delta}|x_1-y_2|^{2-2\delta}}
\\
&\qquad \qquad \times d(x_1,\partial M)^{-\delta}d(y_2,\partial M)^{-\delta}
\\
  &\leq C' \int \int \int \int \frac{d(x_1,\partial M)^{-4\delta} + d(x_2,\partial M)^{-4\delta} + d(y_1,\partial M)^{-4\delta} + d(y_2,\partial M)^{-4\delta}}{|x_1-y_1|^{3-\delta}|x_2-y_2|^{3-\delta}|x_1-y_2|^{2-2\delta}}	
 \\
 &=: J_4(a) + J_4(b) + J_4(c) + J_4(d). 
\end{align}
We will analyze $J_4(a)$, the other terms follow from similar considerations. Integrating out $y_1$, then $x_2$ and then $y_2$, we obtain
\begin{align}
J_4(a) &\leq C'' \int \int \int \frac{d(x_1,\partial M)^{-\delta}}{|x_2-y_2|^{3-2\delta}|x_1-y_2|^{2-2\delta}}
\\
&\leq C''' \int \int  \frac{d(x_1,\partial M)^{-\delta}}{|x_1-y_2|^{2-2\delta}} \leq C_4 \int d(x_1,\partial M)^{-\delta} \leq C_5. 
\end{align}

 \end{proof}

\begin{lemma}\label{lem: upsilon2 b}
For every $\alpha \in (0,1)$, there exists $C>0$ such that, for every admissible boundary law $\mathbb P^0$ and every $T>0$,
\begin{equation} \label{eq: upsilon2 b estimate}
\mathbb E_{\mathbb P^0\otimes \mathbb P^0} [|\Upsilon_{2,T}^b(W^0_-,Z^0_-,W^0_+,Z^0_-)|^\alpha]\leq C + \mathbb E_{\mathbb P^0\otimes\mathbb P^0}[\|Z^0_+\|_{H^{1/2-\kappa}_z}^2+\|Z^0_-\|_{H^{1/2-\kappa}_z}^2]. 
\end{equation}
\end{lemma}

\begin{proof}
By symmetry we have that $\Upsilon^b_{2,T}=N_1+N_2$, where 
\begin{align}
N_1&:= \frac{2\cdot 12^2}3 \int \int C_T^M(x,y)^3 H(W^0_-, W^0_+)(x) H(Z^0_-, Z^0_+)(y) dx dy
\\
&\quad \quad  - 2\int \gamma_T^M H(W^0_-,W^0_+)(x) H(Z^0_-,Z^0_+)(x) dx,
\\
N_2&:= \frac{12^2}3 \int \int C_T^M(x,y)^3 H(Z^0_-, Z^0_+)(x) H(Z^0_-, Z^0_+)(y) dx dy - \int \gamma_T^M H(Z^0_-,Z^0_+)(x)^2 dx.
\end{align}
We will only estimate $N_1$ as the estimate for $N_2$ follows by similar arguments. Recall the interpolation estimate of Proposition \ref{prop: M covariance interp bdry},
\begin{equation}
|C^M_T(x,y)| \leq C \frac{d(x,\partial M)^\delta }{|x-y|^{1+2\delta}}.
\end{equation}
Using this, the definition of $\gamma_T^M$, and the Cauchy-Schwarz inequality,
\begin{align}
|N_1| &= \frac{2\cdot 12^2}3 \left| \int \int C^M_T(x,y)^3 H(W^0_-,W^0_+)(x) \left( H(Z^0_-,Z^0_+)(y)-  H(Z^0_-,Z^0_+)(x)\right) \right|
\\
&\leq C \int \int d(x,\partial M)^\delta|H(W_-^0,W_+^0)(x)| \frac{|H(Z^0_-,Z^0_+)(y)-  H(Z^0_-,Z^0_+)(x)|}{|x-y|^{3+2\delta}} 
\\
&\leq C' \left( \int \int \frac{d(x,\partial M)^{2\delta} |H(W_-^0,W_+^0)(x)|^{2}}{|x-y|^{3-2\delta}} dx dy\right)^{1/2}
\\
&\qquad \qquad \left( \int \int \frac{| H(Z^0_-,Z^0_+)(y)-  H(Z^0_-,Z^0_+)(x)|}{|x-y|^{3+6\delta}} \right)^{1/2}
\\
&\leq C'' \| d(x,\partial M)^\delta H(W^0_-,W^0_+)\|_{L^2} \|H(Z^0_-,Z^0_+)\|_{H^{6\delta}}
\\ \label{eq: Upsilon2b deterministic}
&\leq C''' (\|W_-^0\|_{\mathcal C^{-1/2-\kappa}_z}^2+\|W^0_+\|_{\mathcal C^{-1/2-\kappa}_z}^2)+ \|Z^0_-\|_{H^{6\delta}_z}^2+\|Z^0_+\|_{H^{6\delta}_z}^2.
\end{align} 

The lemma now follows by combining the estimate above together with Proposition \ref{prop: bdry stochastic} to bound the moments of norms of $W^0_-,W^0_+$ that appear.
\end{proof}

\subsection{Stochastic estimates on $\Upsilon_{3,T}$}

Recall that 
\begin{equation}
\Upsilon_{3,T}(\varphi_-,\varphi_+):= -\frac{12^2}2 \mathbb{E} \int^T_0 \int  (J_t(W_t \llbracket H
     (\varphi_-, \varphi_+)^2_T \rrbracket))^2 \mathd t \mathd x + \delta^{4,M}_T.
\end{equation}
We will prove a \emph{fractional moment estimate} on $\Upsilon_{3,T}$ when $\varphi_-,\varphi_+$ are i.i.d. distributed according to an admissible law on boundary conditions.

\begin{lemma} \label{lem: upsilon3}
For every $\alpha > 0$ sufficiently small, there exists $C>0$  such that, for every admissible boundary law $\mathbb P^0$ and every $T>0$ 
  \[  \mathbb E_{\mathbb P^0\otimes \mathbb P^0}[ |\Upsilon_{3, T}(W^0_-+Z^0_-,W^0_++Z^0_+)|^\alpha ]\leq C+ \mathbb E_{\mathbb P^0\otimes \mathbb P^0}[\| Z^0_-\|_{H^{1/2-\kappa}_z}^2+\|Z^0_+\|_{H^{1/2-\kappa}_z}^2]. \]
  Furthermore, the pushforward under the regularized boundary measure $(\Upsilon_{3, T})^* (\nu^0_T\otimes \nu^0_T)$ converges in law to a random variable 
  $(\Upsilon_{3,\infty})^* (\nu^0\otimes \nu^0)$.
\end{lemma}

Lemma \ref{lem: upsilon3} follows from Lemmas \ref{lem: upsilon3 a} and \ref{lem: upsilon3 b} below. As before, we split $\Upsilon_{3,T}=:\Upsilon_{3,T}^a + \Upsilon_{3,T}^b$, where $\Upsilon_{3,T}^a$ depends only on $W^0_-$ and $W^0_+$. First note that
\begin{align}
\mathbb{E} \int^T_0 \int J_t (W_t \llbracket H
     (\varphi_-, \varphi_+)_T^2 \rrbracket)^2 \mathd t \mathd x	= \frac 12 \int \int  C^M_T(x,y)^2 \llbracket H
     (\varphi_-, \varphi_+)_T^2 \rrbracket^2(x) \llbracket H
     (\varphi_-, \varphi_+)_T^2 \rrbracket^2(y).
\end{align}
Thus, we define
\begin{equation}
\Upsilon^a_{3,T}:=-\frac{12^2}{4} \int \int C^M_T(x,y)^2 \llbracket H(W^0_-,W^0_+)^2_T\rrbracket ^2(x) \llbracket H(W^0_-,W^0_+)^2_T\rrbracket (y) + \delta^{4,M}_T. 
\end{equation}

\begin{lemma}\label{lem: upsilon3 a}
There exists $C>0$ such that, for every $T>0$,
\begin{equation} \label{eq: upsilon3 a estimate}
\mathbb E_{\mu^0\otimes \mu^0}[|\Upsilon^a_{3,T}|^2] \leq C.	
\end{equation}
\end{lemma}

\begin{proof}
As before, we may take expectations with respect to $\tilde \mu^0$ and use the infinitely smoothing correction to the covariance to extend the estimates to $\mu^0$. Expanding and using the definition of $\delta^{4,M}_T$ together with Wick's theorem, we have that
\begin{align}
&\mathbb E_{\tilde \mu^0\otimes \tilde \mu^0}[ |\Upsilon_{4,T}^a|^2]\\
&= \frac{3\cdot 12^3}2 \int \int \int \int C^M_T(x_1,y_1)^2C^M_T(x_2,y_2)^2 E_{\tilde \mu^0\otimes\tilde\mu^0} \Big[ \llbracket H
     (\varphi_-, \varphi_+)^2_T \rrbracket^2(x_1) \llbracket H
     (\varphi_-, \varphi_+)^2_T \rrbracket^2(y_1)  
    \\
     &\qquad \qquad \qquad \qquad \qquad \times \llbracket H
     (\varphi_-, \varphi_+)^2_T \rrbracket^2(x_2) \llbracket H
     (\varphi_-, \varphi_+)^2_T \rrbracket^2(y_2)\Big]- (\delta_T^{4,M})^2
\\
&= \frac{3\cdot 12^3}2 \int \int \int \int C^M_T(x_1,y_1)^2C^M_T(x_2,y_2)^2 \Bigg\{(2!)^2 C^B_T(x_1,x_2)^2 C^B_T(y_1,y_2)^2 
\\
&\qquad \qquad \qquad \qquad + 2 \cdot 2^2 C^B_T(x_1,y_1) C^B_T(x_2,y_2) C^B_T(x_1,x_2) C^B_T(y_1,y_2) \Bigg\}=: \Iota + \Iota\Iota.
\end{align}

By the covariance estimates on $C^M_T$ and $C^B_T$, see Appendices \ref{sec: Dirichlet covariance estimates} and \ref{appendix: harmonic extension}, together with standard bounds on convolutions, we have that
\begin{align}
\Iota &\leq C \int \int \int \int \frac{1}{|x_1-y_1|^2|x_2-y_2|^2|x_1-x_2|^2|y_1-y_2|^2}
\\
&\leq C' \int \int \int \frac{1}{|x_2-y_1||x_2-y_2|^2 |y_1-y_2|^2}\leq C'' \int \int  \frac{1}{|x_2-y_1|^2} \leq C'''. 
\end{align}
For the second term, we will additionally use the boundary covariance interpolation estimate of Proposition \ref{prop: M covariance interp bdry}. For every $\delta > 0$ sufficiently small,
\begin{align}
\Iota\Iota &\leq C \int \int \int \int \frac{d(y_1,\partial M)^{-\delta}d(y_2,\partial M)^{-\delta}}{|x_1-y_1|^{3-\delta}|x_2-y_2|^{3-\delta}|x_1-x_2||y_1-y_2|}
\\
&\leq C'\int \int \int \frac{d(y_1,\partial M)^{-\delta}d(y_2,\partial M)^{-\delta}}{|x_2-y_1|^{1-\delta}|x_2-y_2|^{3-\delta}|y_1-y_2|}\leq C'' \int \int \frac{d(y_1,\partial M)^{-\delta}d(y_2,\partial M)^{-\delta}}{|y_1-y_2|^{2-2\delta}} \leq C'''. 
\end{align}	
\end{proof}

We now turn to estimating $\Upsilon^b_{3,T}$. 

\begin{lemma}\label{lem: upsilon3 b}
For every $\alpha > 0$ sufficiently small, there exists $C>0$  such that, for every admissible boundary law $\mathbb P^0$ and every $T>0$ 
  \[  \mathbb E_{\mathbb P^0\otimes \mathbb P^0}[ |\Upsilon_{3, T}^b(W^0_-,Z^0_-,W^0_+,Z^0_+)|^\alpha ]\leq C+ \mathbb E_{\mathbb P^0\otimes \mathbb P^0}[\| Z^0_-\|_{H^{1/2-\kappa}_z}^2+\|Z^0_+\|_{H^{1/2-\kappa}_z}^2]. \]
Furthermore, the pushforward under the regularized boundary measure $(\Upsilon_{3, T}^b)^* (\nu^0_T\otimes \nu^0_T)$ converges in law to a random variable 
  $(\Upsilon_{3,\infty}^b)^* (\nu^0\otimes \nu^0)$.
\end{lemma}

\begin{proof}
We will only prove the estimate since the convergence statement holds by similar arguments as with $\Upsilon^b_{2,T}$. We may write
\begin{align}
	\Upsilon^b_{3,T}&= -\frac{12^2}4\sum_{\substack{1 \leq i \leq 2 \\ 0 \leq j \leq 2}} \int \int C_T^M(x,y)^2  {2 \choose i} {2 \choose j}
	\\
	&\times \llbracket H(W^0_-,W^0_+)^{2-i}_T \rrbracket (x) H(Z^0_-,Z^0_+)^{i}_T(x) \llbracket H(W^0_-,W^0_+)^{2-j}_T \rrbracket (y) H(Z^0_-,Z^0_+)^{j}_T (y).
\end{align}
We will treat $j=0$ and $j\geq 1$ separately and call these $N_0$ and $N_{\geq 1}$ respectively. Additionally, we will split $N_{\geq 1} = \sum_{1 \leq i \leq 2, j\geq 1} N_{\geq 1}(i,j)$. Furthermore, let us drop the subscript $T$ from the harmonic extension terms for ease of notation. 

First, let us consider $j \geq 1$ and fix $i=1,2$. Then by the covariance interpolation inequality of Proposition \ref{prop: M covariance interp bdry} and Young's inequality,
\begin{align}
|N_{\geq 1}(i,j)| &\leq C \int \int \frac{1}{|x-y|^{3-2\delta}} d(x,\partial M)^{1/2-\delta }|\llbracket H(W^0_-,W^0_+)^{2-i} \rrbracket (x) H(Z^0_-,Z^0_+)^{i}(x)|
\\ &\qquad \qquad
\times d(y,\partial M)^{1/2-\delta}|\llbracket H(W^0_-,W^0_+)^{2-j} \rrbracket (y) H(Z^0_-,Z^0_+)^{j} (y)| dx dy
\\
&\leq C' \|d(\cdot,\partial M)^{1/2-\delta }\llbracket H(W^0_-,W^0_+)^{2-i} \rrbracket H(Z^0_-,Z^0_+)^{i}\|_{L^2}
\\
&\qquad \times \|d(\cdot,\partial M)^{1/2-\delta }\llbracket H(W^0_-,W^0_+)^{2-j} \rrbracket H(Z^0_-,Z^0_+)^{j}\|_{L^2}.
\end{align}
Now since $i,j \geq 1$, we are considering at most the first Wick power and can treat the terms symmetrically. Thus, by splitting into $\tau$ and $z$ variables, and applying a uniform bound on the integrals,
\begin{align}
&\|d(\cdot,\partial M)^{1/2-\delta }\llbracket H(W^0_-,W^0_+)^{2-i} \rrbracket H(Z^0_-,Z^0_+)^{i}\|_{L^2}
\\
\leq& C\Big\| d(\tau,\partial M)^{1/2-\delta} \|H(W^0_-,W^0_+)^{2-i}(\tau,\cdot)\|_{L^\infty_z} \|H(Z^0_-,Z^0_+)^i(\tau,\cdot) \|_{L^2_z} \Big\|_{L^2_\tau }.
\end{align}
We now use the interpolation estimates \eqref{tool: H reg decay}. For every $\tau$,
\begin{align}
d(\tau,\partial M)^{1/2-\delta} \|H(W^0_-,W^0_+)^{2-i}(\tau,\cdot)\|_{L^\infty_z} \leq C d(\tau,\partial M)^{-2\delta}(\|W^0_-\|_{\mathcal C^{-1/2-\delta}_z}+ \|W^0_+\|_{\mathcal C^{-1/2-\delta}_z}).
\end{align}
Hence, 
\begin{align}
\|d(\cdot,\partial M)^{1/2-\delta }&\llbracket H(W^0_-,W^0_+)^{2-i} \rrbracket H(Z^0_-,Z^0_+)^{i}\|_{L^2}
\\
&\leq C' (\|W^0_-\|_{\mathcal C^{-1/2-\delta}_z}+ \|W^0_+\|_{\mathcal C^{-1/2-\delta}_z}) \left\| d(\tau,\partial M)^{-2\delta} \|H(Z_-^0,Z_+^0)^i(\tau,\cdot) \|_{L^2_z}\right\|_{L^2_\tau}.
\end{align}
Observe that $$\left\| d(\tau,\partial M)^{-2\delta} \|H(Z_-^0,Z_+^0)^i(\tau,\cdot) \|_{L^2_z}\right\|_{L^2_\tau}=\left\| d(\tau,\partial M)^{-2\delta} \|H(Z_-^0,Z_+^0)(\tau,\cdot) \|_{L^{2i}_z}^i\right\|_{L^2_\tau}.$$ Thus, by Sobolev embedding and the fact that $2i \leq 4$, we have that
\begin{equation}\label{eq: upsilon3 b eq1}
\Big\| d(\tau,\partial M)^{-2\delta} \|H(Z_-^0,Z_+^0)^i(\tau,\cdot) \|_{L^2_z}\Big\|_{L^2_\tau}	 \leq C \left\| d(\tau,\partial M)^{-2\delta} \|H(Z_-^0,Z_+^0)(\tau,\cdot) \|_{H^{1/2}_z}^i \right\|_{L^2_\tau }. 
\end{equation}
Thus, by trading regularity for boundary blow-up (see \eqref{tool: H reg decay}, we have that
\begin{align}
\eqref{eq: upsilon3 b eq1} &\leq C'\left\| d(\tau,\partial M)^{-(2+i)\delta} \|H(Z_-^0,Z_+^0)(\tau,\cdot) \|_{H^{1/2-\delta}_z}^i \right\|_{L^2_\tau }
\\
&\leq C'' (\| Z_-^0\|_{H^{1/2-\delta}_z}^i+\|Z_+^0\|_{H^{1/2-\delta}_z}^i).	
\end{align}

It remains to treat the case $j=0$. By the kernel estimates in Appendix \ref{sec: Dirichlet covariance estimates} and Young's convolution inequality, we have that
\begin{align}
|N_0(i)| &\leq C \int \int \frac{1}{|x-y|^{3-\delta}} | d(x,\partial M)^{1/2-2\delta}  H(W_-^0,W_+^0)^{2-i}(x) H(Z^0_-,Z^0_+)^i(x) |
\\
&\qquad\qquad \qquad \qquad \times  |d(y,\partial M)^{1/2+\delta}\llbracket H(W^0_-,W^0_+)^2 \rrbracket(y)| 
\\
&\leq C' \| d(\cdot,\partial M)^{1/2-2\delta}  H(W_-^0,W_+^0)^{2-i} H(Z^0_-,Z^0_+)^i \|_{L^2} \| d(\cdot,\partial M)^{1/2+\delta} \llbracket H(W^0_-,W^0_+)^2 \|_{L^2}.
\end{align}
The first product on the righthand side we have done above, so we will only estimate the second term. Note that by splitting into $(\tau,z)$ integration, multiplying and dividing by $d(\tau,\partial M)$, we obtain
\begin{align}
 &\| d(\cdot,\partial M)^{1/2+\delta} \llbracket H(W^0_-,W^0_+)^2 \rrbracket  \|_{L^2} 
 \\
 &\leq \|d(\tau,\partial M)^{1/2+\delta} \| 	\llbracket H(W^0_-,W^0_+)^2 \rrbracket (\tau,\cdot) \|_{L^\infty_z} \|_{L^2_\tau }
 \\
 &\leq C' \| |d(\tau,\partial M)|^{1+\delta/2} \| \llbracket H(W^0_-,W^0_+)^2 \rrbracket (\tau,\cdot) \|_{L^\infty_z} \|_{L^\infty_\tau}.
\end{align}
The lemma now follows by combining the estimates above and using Proposition \ref{prop: bdry stochastic} to bound the moments of norms of $W^0_-,W^0_+$ that appear.
\end{proof}

\subsection{Stochastic estimates on $\Upsilon_{4,T}$}

Recall that
\begin{equation}
\Upsilon_{4,T} := -\frac{4^2}2  \int^T_0 \int 
     ( J_t\llbracket H (\varphi_-, \varphi_+)^3_T \rrbracket)^2 + \delta^{5,M}_T.	
\end{equation}
We will prove a fractional moment estimate on $\Upsilon_{4,T}$ when $\varphi_-,\varphi_+$ are i.i.d. distributed according to an admissible law on boundary conditions. 

\begin{lemma} \label{lem: upsilon4}
For every $\alpha > 0$ sufficiently small, there exists $C>0$  such that, for every admissible boundary law $\mathbb P^0$ and every $T>0$ 
\begin{equation}  \mathbb E_{\mathbb P^0\otimes \mathbb P^0}[ |\Upsilon_{4, T}(W^0_-+Z^0_-,W^0_++Z^0_+)|^\alpha ]\leq C+ \mathbb E_{\mathbb P^0\otimes \mathbb P^0}[\| Z^0_-\|_{H^{1/2-\kappa}_z}^2+\|Z^0_+\|_{H^{1/2-\kappa}_z}^2]. \end{equation}
\end{lemma}

Lemma \ref{lem: upsilon4} follows from Lemmas \ref{lem: upsilon4 a} and \ref{lem: upsilon4 b} below. As before, we split $\Upsilon_{4,T}=:\Upsilon_{4,T}^a + \Upsilon_{4,T}^b$, where $\Upsilon_{4,T}^a$ depends only on $W^0_-$ and $W^0_+$. As a preliminary, note that by Fubini's theorem and the self-adjointedness of $J_t$, we have that
\begin{equation}
\int_0^T \int  ( J_t \llbracket H (\varphi_-, \varphi_+)_T^3 \rrbracket)^2 dt = \int \int C_T^M(x,y)   \llbracket H (\varphi_-, \varphi_+)_T^3 \rrbracket(x)  \llbracket H (\varphi_-, \varphi_+)_T^3 \rrbracket(y).
\end{equation}
Hence let us define
\begin{equation}
\Upsilon^{a}_{4,T}:= -\frac{4^2}{2} \int \int C_T(x,y) \llbracket H(W^0_-,W^0_+)_T^3\rrbracket (x) \llbracket H(W^0_-,W^0_+)^3_T\rrbracket (y) + \delta^{5,M}_T.  	
\end{equation}

\begin{lemma}\label{lem: upsilon4 a}
There exists $C>0$ such that, for every $T>0$,
\begin{equation} \label{eq: upsilon4 a estimate}
\mathbb E_{\mu^0\otimes \mu^0}[|\Upsilon^a_{4,T}|^2] \leq C.	
\end{equation}
\end{lemma}

\begin{proof}
Again, it is sufficient to  prove \eqref{eq: upsilon4 a estimate} with the expectation taken with respect to $\tilde \mu^0\otimes \tilde \mu^0$. By Wick's theorem and the definition of $\delta^{5,M}_T$,
\begin{align}
\mathbb E_{\tilde \mu^0\otimes \tilde \mu^0} [ \left(\Upsilon_{4,T}^{a}\right)] &=  4^3 \int \int \int \int C^M_T (x_1, y_1) C^M_T (x_2, y_2) 
\\
&\qquad \qquad \mathbb E_{\tilde\mu^0\otimes \tilde\mu^0} \Bigg[  \prod_{i=1}^2  \llbracket H (W^0_-, W^0_+)_T^3 \rrbracket(x_i) \prod_{i=1}^{2}  \llbracket H (W^0_-, W^0_+)_T^3 \rrbracket(y_i)\Bigg]- (\delta^{5,M}_T)^2 
\\
&=  4^3 \int \int \int \int C^M_T (x_1, y_1) C^M_T (x_2, y_2) \Bigg\{   2 \cdot (3!)^2 C^B_T (x_1, x_2)^3 C^B_T (y_1, y_2)^3 
    \\
    &\qquad  \qquad + 2 \cdot 3^2\cdot 2^2 C^B_T(x_1,y_1) C^B_T(x_2,y_2) C^B_T(x_1,x_2)^2 C^B_T(y_1,y_2)^2 
    \\
    &\qquad \qquad  + 2\cdot (3!)^2 C^B_T(x_1,y_1)^2 C^B_T(x_2,y_2)^2  C^B_T(x_1,x_2)C^B_T(y_1,y_2) 
     \Bigg\}
    \\
    &=: \Iota + \Iota\Iota + \Iota\Iota\Iota.
\end{align}
Note that above, subtracting $(\delta^{5,M}_T)^2$ corresponds we have subtracting the most divergent term
  \[ 4^3 \cdot (3!)^2 \int \int \int \int C^M_T (x_1, y_1) C^M_T (x_2, y_2) C^B_T (y_1, x_1)^3 C^B_T
     (x_2, y_2)^3.\]
     
We begin with the estimate of $\Iota$. Let $\delta > 0$ be sufficiently small. Using the kernel estimates on the Poisson kernel $C^B_T$ and Dirichlet covariance $C^M_T$ contained in Appendices \ref{sec: Dirichlet covariance estimates} and \ref{appendix: harmonic extension}, together with standard bounds on convolutions, we have that
\begin{align}
\Iota &\leq C \int \int \int \int \frac{d(x_2,\partial M)^{-\delta}d(y_2,\partial M)^{-\delta}}{|x_1-y_1||x_2-y_2||x_1-x_2|^{3-\delta}|y_1-y_2|^{3-\delta}}	
\\
&\leq C' \int \int \int \frac{d(x_2,\partial M)^{-\delta}d(y_2,\partial M)^{-\delta}}{|y_1-x_2|^{1-\delta}|x_2-y_2||y_1-y_2|^{3-\delta}}
\leq C_2 \int \int  \frac{d(x_2,\partial M)^{-\delta}d(y_2,\partial M)^{-\delta}}{|x_2-y_2|^{2-2\delta}} \leq C_3. 
\end{align}

  Now consider $\Iota\Iota$. Again, by the covariance estimates ,
  \begin{align}
  \Iota\Iota &\leq C \int \int \int \int \frac{1}{|x_1-y_1|^2|x_2-y_2|^2 |x_1-x_2|^{2} |y_1-y_2|^{2}}	
  \\
  &\leq C' \int \int \int \frac{1}{|y_1-x_2|^{1}|x_2-y_2|^2 |y_1-y_2|^{2}}
 \leq C'' \int \int \frac{1}{|x_2-y_2|^{2}} \leq C_3. 
  \end{align}
  
  We now turn to $\Iota\Iota\Iota $. By the covariance estimates, for $\delta > 0$ sufficiently small we have
\begin{align}
    \Iota\Iota\Iota &\leq C
     \int \int \int \int \frac{d (y_1, \partial M)^{- \delta}d (y_2, \partial M)^{- \delta}}{| x_1 -
    y_1 |^{3 - \delta}| x_2 - y_2 |^{3 -
    \delta}| x_1 - x_2 || y_1 - y_2 |} 
    \\
    &\leq C' \int \int \int \frac{d (y_1, \partial M)^{- \delta}d (y_2, \partial M)^{- \delta}}{| x_2 -
    y_1 |^{1 - \delta}| x_2 - y_2 |^{3 -
    \delta}| y_1 - y_2 |}
     \leq C_2 \int \int  \frac{d (y_1, \partial M)^{- \delta}d (y_2, \partial M)^{- \delta}}{| y_1 - y_2 |^{2-2\delta}} \leq C_3.
  \end{align}

\end{proof}

We now turn to estimating $\Upsilon^b_{4,T}$.

\begin{lemma}\label{lem: upsilon4 b}
Let $\mathbb P^0$ be an admissible law on boundary conditions $(W^0,Z^0)$. For every $\alpha > 0$ sufficiently small, there exists $C>0$  such that, for every $T>0$ 
  \begin{equation}  \mathbb E_{\mathbb P^0\otimes \mathbb P^0}[ |\Upsilon_{4, T}^b(W^0_-,Z^0_-,W^0_+,Z^0_+)|^\alpha ]\leq C+ \mathbb E_{\mathbb P^0\otimes \mathbb P^0}[\| Z^0_-\|_{H^{1/2-\kappa}_z}^2+\|Z^0_+\|_{H^{1/2-\kappa}_z}^2]. \end{equation}
\end{lemma}

\begin{proof}
Let us write
\begin{align}
\Upsilon^b_{4,T}=- \sum_{\substack{1 \leq i \leq 3 \\ 0 \leq j \leq 3}} {3 \choose i}{3\choose j} \frac{4^2}{2} N(i,j),
\end{align}
where 
\begin{align}
N(i,j) &:= \int \int C_T^M(x,y) \llbracket H(W^0_-,W^0_+)_T^{3-i} \rrbracket(x) H(Z^0_-, Z^0_+)_T^i(x) 
\\
&\qquad \qquad \qquad \llbracket H(W^0_-,W^0_+)_T^{3-j} \rrbracket (y) H(Z^0_-,Z^0_+)_T^j(y) dx dy.
\end{align}
We will drop the subscript $T$ in the harmonic extension terms below for ease of notation.

We begin by estimating when $j\geq 1$. By the covariance estimates of Appendix \ref{sec: Dirichlet covariance estimates} and Young's convolution inequality,
\begin{align}
|N(i,j)| &\leq C \int \int \frac{d(x,\partial M)^{1-\delta}d(y,\partial M)^{1-\delta}}{|x-y|^{3-2\delta}} |\llbracket H(W^0_-,W^0_+)^{3-i} \rrbracket(x) H(Z^0_-, Z^0_+)^i(x)|
\\
&\qquad \qquad | \llbracket H(W^0_-,W^0_+)^{3-j} \rrbracket (y) H(Z^0_-,Z^0_+)^j(y)| dx dy
\\
&\leq C' \|d(\cdot,\partial M)^{1-\delta}\llbracket H(W^0_-,W^0_+)^{3-i} \rrbracket H(Z^0_-, Z^0_+)^i \|_{L^2} 
\\
&\qquad \qquad \times \|d(\cdot,\partial M)^{1-\delta}\llbracket H(W^0_-,W^0_+)^{3-j} \rrbracket H(Z^0_-, Z^0_+)^j \|_{L^2}.
\end{align}
Due to the symmetric roles of $i$ and $j$ when $j \geq 1$, the following estimates are sufficient to conclude. Let us first assume $i=1$. Then
\begin{align}
 \|d(\cdot,\partial M)^{1-\delta}&\llbracket H(W^0_-,W^0_+)^{2} \rrbracket H(Z^0_-, Z^0_+) \|_{L^2}	
 \\
 &\leq C \| d(\cdot,\partial M)^{1+\delta}\llbracket H(W^0_-,W^0_+)^{2} \rrbracket \|_{L^\infty} \| d(\cdot,\partial M)^{-2\delta} H(Z^0_-, Z^0_+)  \|_{L^2}.
\end{align}
The first factor on the righthand side can be treated by using Proposition \ref{prop: stochastic boundary v2} and properties of $H$. For the second factor, arguing as in preceding estimates, we have that
\begin{equation}
\| d(\cdot,\partial M)^{-2\delta} H(Z^0_-, Z^0_+)  \|_{L^2} 	\leq C ( \|Z^0_-\|_{H^{1/2-\delta}_z} + \|Z^0_+\|_{H^{1/2-\delta}_z}).
\end{equation}

Let us now assume $i=2$. In this case, we estimate
\begin{align}
\|d(\cdot,\partial M)^{1-\delta}& H(W^0_-,W^0_+) H(Z^0_-, Z^0_+)^2 \|_{L^2}
\\
&\leq \|d(\cdot,\partial M)^{1/2+\delta} H(W^0_-,W^0_+) \|_{L^\infty} \|d(\cdot,\partial M)^{1/2-2\delta} H(Z^0_-,Z^0_+)^2 \|_{L^2}.
\end{align}
As before, the first term on the righthand side can be estimated by properties of $H$. For the second term, notice that by splitting into $(\tau,z)$ coordinates, using the 2d Sobolev embedding and using the positive exponent on the distance to boundary, we have that
\begin{equation}
\|d(\cdot,\partial M)^{1/2-2\delta} H(Z^0_-,Z^0_+)^2 \|_{L^2} \leq C\|d(\cdot,\partial M)^{1/2-2\delta} \|H(Z^0_-,Z^0_+)(\tau,\cdot) \|_{H^{1/2}}^2\|_{L^{2}_{
  \tau}},
\end{equation}
and this can be estimated as above. 

Finally, we assume $i=3$. By splitting into $(\tau,z)$ integration and using Sobolev embedding, we have that
\begin{align}
\begin{split} \label{eq: upsilon4 b eq1}
\|&d(\cdot,\partial M)^{1-\delta} H(Z^0_-, Z^0_+)^3\|_{L^2}
	\\
	&= \|d(\tau,\partial M)^{1-\delta} \|H(Z^0_-,Z^0_+)(\tau,\cdot)\|_{L^6_z}^3 \|_{L^2_\tau} 
	\leq C \|d(\tau,\partial M)^{1-\delta} \|H(Z^0_-,Z^0_+)(\tau,\cdot)\|_{H^{2/3}_z}^3 \|_{L^2_\tau}.
\end{split}	
\end{align}
Hence by exchanging regularity for boundary blow-up as above, we have that
\begin{equation}
\eqref{eq: upsilon4 b eq1}
	\leq C'' (\|Z^0_-\|_{H^{1/2-\delta}_z}^3 + \|Z^0_+\|_{H^{1/2-\delta}_z}^3).	
\end{equation}

Finally, we estimate when $j=0$, allowing for arbitrary $i \geq 1$. Then we have, by interpolating with the boundary blow-up as above,
\begin{align}
|N(i,0)| 
&\leq C'   \| |d(\cdot,\partial M)|^{1+\delta}\llbracket H(W^0_-,W^0_+)^{3-i} \rrbracket \|_{L^\infty}\, \times  \||d(\cdot ,\partial M)|^{3/2+\delta}\llbracket H(W^0_-,W^0_+)^3 \rrbracket \|_{L^\infty}
\\
&\qquad \qquad \times \int \int \frac{d(x,\partial M)^{-1+\delta}d(y,\partial M)^{-1+\delta}}{|x-y|^{3/2+4\delta}} |H(Z^0_-,Z^0_+)^i(x)| 
\end{align} 
It remains to estimate the integral since the terms involving $W^0_-,W^0_+$ are sufficient for our purposes. Let us write $x=(\tau_x,z_x)$ and $y = (\tau_y,z_y)$. By splitting into $(\tau,z)$ integration and using that $|x-y| \geq |z_x - z_y|$, then integrating $z_y$ and $\tau_y$ first we have that
\begin{align}
\int \int &\frac{d(x,\partial M)^{-1+\delta}d(y,\partial M)^{-1+\delta}}{|x-y|^{3/2+4\delta}} |H(Z^0_-,Z^0_+)^i(x)| dx dy
\\
&\leq C \int d(\tau_x,\partial M)^{-1+\delta} \int  |H(Z^0_-,Z^0_+)^i(\tau_x,z_x)|
\\
&\qquad \qquad \times  \int d(\tau_y,\partial M)^{-1+\delta}\int \frac{1}{|z_x - z_y|^{3/2+4\delta}} dz_y d\tau_y dz_x d\tau_x
\\
&\leq C' \int d(\tau,\partial M)^{-1+\delta} \|H(Z^0_-,Z^0_+)(\tau_x,\cdot)\|_{L^i_z}^i d\tau 
\leq C'' ( \|Z^0_-\|_{H^{1/2-\kappa}_z}^3 + \|Z^0_+\|_{H^{1/2-\kappa_z}}^3),
\end{align}
where we have used that $i \leq 3$ so that by (the 2d) Sobolev embedding, $H^{1/2-\kappa}_z \hookrightarrow L^{i}_z$. 

The lemma now follows by combining the estimates above together with Proposition \ref{prop: bdry stochastic} to estimate the moments of the norms of the $W^0_-,W^0_+$ that appear. 
\end{proof}

\subsection{Proof of Theorem \ref{thm: enhancement converge}}

We begin with convergence in law of the random variables $\Upsilon_{i,T}$. We will only consider the cases where the boundary fields are i.i.d. according to $\nu^0_T$ or $\nu^0$, but one can generalize it to sequences of admissible boundary conditions that converge appropriately.

\begin{lemma} \label{lem: upsiloni convergence}
For every $i=2,3,4$, we have that $\Upsilon_{i,T}(\varphi_-^T,\varphi_+^T)$, where $\varphi_-^T,\varphi_+^T$ are i.i.d. according to $\nu^0_T$, converges in law to a random variable $\Upsilon_i^\infty$. 

The same convergence statement to the same random variables hold if $\varphi^T_-=(\varphi_-)_T$ and $(\varphi_+)_T$, where $\varphi_-,\varphi_+$ are i.i.d. according to $\nu^0$. 
	
\end{lemma}

\begin{proof}	
Without loss of generality, let us consider $i=2$, the other statements follow by similar arguments. We will also only prove the first statement in the lemma, as the other follows similarly. We will work in the coupling measures where $\nu^0_T$ is represented as the sum of $(W^0_T, Z^0_T)$ and $\nu^0$ is represented as the sum of $(W^0_\infty, Z^0_\infty)$. By Skorokhod embedding, we may take a subsequence (not relabelled) and assume that these random variables are defined on the same probability space and converge almost surely. 

First of all, recall the splitting 
\begin{equation}
\Upsilon_{2,T}(W^0_{-,T}+Z^0_{-,T}, )=\Upsilon_{2,T}^a(W^0_{-,T},W^0_{+,T})+\Upsilon_{2,T}^b(W^0_{-,T},Z^0_{-,T},W^0_{+,T},W^0_{-,T}).	
\end{equation}
For the first term, the convergence statement follows from \eqref{eq: upsilon2 a 2ndmoment} and martingale convergence upon splitting $\Upsilon^a_{2,T}$ into terms in homogeneous Wiener chaoses. 

We now turn our attention to the second term. To do this, we view
$\Upsilon_{2,T}^b$ as a \emph{fixed} multilinear functional of the $W^0$'s, $Z^0$'s, \emph{and} the covariance $C_T^M$. Note that this can be done for every $\Upsilon_{i,T}$ -- the $\delta$ renormalizations are absorbed in the $\Upsilon_{i,T}^a$ term, and the mass renormalization term in $\Upsilon_{3,T}^b$ can be rewritten in terms of $C^M_T$. By using the analytic arguments used to establish \eqref{eq: upsilon2 b estimate}, see \eqref{eq: Upsilon2b deterministic}, we therefore have that these multilinear maps are bounded on the relevant Banach spaces. Thus, it follows from pointwise convergence of the $W^0$'s and $Z^0$'s, thanks to the Skorokhod embedding. The limiting random variable obtained is unique due to convergence to $W^0_\infty, Z^0_\infty$, and thus we obtain the full convergence in law as required.

\end{proof}

We now prove Theorem \ref{thm: enhancement converge}.

\begin{proof}[Proof of Theorem \ref{thm: enhancement converge}] This follows from Proposition \ref{prop: stochastic boundary v2} and Lemma \ref{lem: upsiloni convergence}. 
	
\end{proof}

\section{Estimates on remainder terms}
\label{sec:equicoercive}
In this section, we will prove estimates on the remainder terms $R_i$. Recall that the bulk enhancement at scale $T$ is the vector of stochastic processes $\Xi_T(W)$ defined in Section \ref{subsec: ups1}. We stress that below, $\mathbb E$ is  expectation with respect to bulk fields and not boundary fields.

\begin{proposition}\label{prop: remainder terms}
 For every $\delta>0$ and for every $(\varphi_-,\varphi_+)$ admissible boundary conditions, there exists $p=p(\delta)>0$ and $C=C(\delta)>0$ such that, for every $i=1,\dots,8$ and every $T>0$,
\begin{equation} \label{eq: ri estimates}
\mathbb E[|R_i|] \leq C(1+\|\Xi^\partial_T(\varphi_-,\varphi_+)\|_{\boldsymbol{\mathcal B}}^p) + \delta \mathbb E[\mathcal C_T(u)], \qquad \forall u \in \mathcal H.  	
\end{equation}
\end{proposition}
Proposition \ref{prop: remainder terms} will follow from Lemmas \ref{lemma: regularity estimates bdry field}-\ref{lemma: R9} below. We will only prove the required estimates on the terms $R_1, R_3, R_4, R_8, R_9$. The terms $R_5-R_7$ contain purely bulk terms and can be estimated as in \cite{BG20} with straightforward modifications. The term $R_2$ can also be estimated as in \cite{BG20} with straightforward modifications.

\begin{remark}
Let us mention that the constant term on the righthand side of \eqref{eq: ri estimates} comes from contributions that can be bounded by $\mathbb E[\|\Xi_T(W)\|_{\mathfrak S}^p]$. This yields a constant uniform in $T$ in light of Proposition \ref{prop: bulk enhancement}. Strictly speaking our estimates are not fully deterministic, unlike for the periodic case \cite{BG20}. 
\end{remark}

As preparation for these lemmas, we first establish regularity estimates on the boundary fields and on the drift term $K_T$ induced by the intermediate ansatz.

\subsection{Regularity estimates on harmonic extensions of boundary fields}

The main result of this subsection, Lemma \ref{lemma: regularity estimates bdry field} below, is a collection of fine regularity estimates on the boundary fields. This is a quantitative extension of Proposition \ref{prop: stochastic boundary v2} to account for the $Z^0$ term in the coupling representation of the boundary field. We will make heavy use of the regularizing properties of the harmonic extension described in Appendix \ref{appendix: harmonic extension}.

\begin{lemma} \label{lemma: regularity estimates bdry field}
Let $i\in{1,2,3}$. For every
	 $\alpha\geq 0$ such that $\alpha \in ({\rm max}(0,\frac{i-2}2), i/2]$, and $\kappa>0$ sufficiently small, there exists $C,p,\kappa'>0$ where $\kappa'$ is sufficiently small such that the following estimate holds. For every admissible boundary condition $(W^0,Z^0)$ and for every $T\in (0,\infty]$, 
	\begin{align}
		&\sup_{\tau \in [-L,L]} |d(\tau,\partial M)|^\alpha \|\llbracket H(W^0+Z^0)_T^{i} \rrbracket \|_{\mathcal{C}_{z}^{-i/2+\alpha-\kappa}} \\
		& \qquad \qquad \leq C \sum_{j=1}^3 \sup_{\tau \in [-L,L]} |d(\tau,\partial M)|^\alpha \|\llbracket H(W^0)_T^{j} \rrbracket \|^p_{\mathcal{C}^{-j/2+\alpha-\kappa'}_z} + \|(Z^0)_T\|^{p}_{H^{1/2-\kappa'}_z}.
	\end{align}
\end{lemma}

\begin{proof}
We will drop $T$ from the notation when clear from context. The case $i=1$ follows from Proposition \ref{prop: stochastic boundary v2} and the regularity properties of $Z^0$ and harmonic extension. 

We now treat the case $i=2$, in which case $\alpha \in (0,1]$. We split according to 
\begin{equation}
\llbracket H(W^0)^2\rrbracket + 2 H( W^0)H(Z^0) + H(Z^0)^2. 
\end{equation}
The desired estimate on the first term follows from the stochastic regularity estimates of Proposition \ref{prop: stochastic boundary v2}. 

For the mixed term, we will do a paraproduct decomposition in the $z$-direction, denoted $\succ_z$ and also $\circ_z$ for the resonant product. Provided $\alpha>\kappa$, we have
\begin{align*}
&|d(\tau,\partial M)|^\alpha \|H(W^0)(\tau,\cdot) \succ_z H(Z^0)(\tau,\cdot)\|_{\mathcal C^{-1+\alpha-\kappa}_z} \\
&\leq C \Big(|d(\tau,\partial M)|^\alpha \| H(W^0)(\tau,\cdot) \|_{\mathcal C^{-1/2+\alpha-\kappa/2}_z} \Big) \|H(Z^0)(\tau,\cdot)\|_{\mathcal C^{-1/2-\kappa/2}_z},\\
&\leq C' \Big(|d(\tau,\partial M)|^\alpha \| H(W^0)(\tau,\cdot) \|_{\mathcal C^{-1/2+\alpha-\kappa/2}_z} \Big)^2+ \|H(Z^0)(\tau,\cdot)\|^2_{ H^{1/2-\kappa/4}_z},
\end{align*} 
where in the last step we have used the 2d Besov embedding $H^{s+\kappa'}_z \hookrightarrow \mathcal C^{s-1}_z$ for $\kappa'$ arbitrarily small. The desired estimate now follows by properties of the harmonic extension.
 For the resonant product term, we may then use the Besov embedding and resonant product estimate to obtain
\begin{align*}
	&|d(\tau,\partial M)|^\alpha  \|H(W^0)(\tau,\cdot) \circ_z H(Z^0)(\tau,\cdot)\|_{\mathcal C^{-1+\alpha-\kappa}_z} \\
&\leq C|d(\tau,\partial M)|^{\alpha} \|H(W^0)(\tau,\cdot) \circ_z H(Z^0)(\tau,\cdot)\|_{H^{\alpha-2\kappa}_z} \\
&\leq C'|d(\tau,\partial M)|^{\alpha} \|H(W^0)(\tau,\cdot)\|_{\mathcal C_z^{-1/2+\alpha-\kappa}} \|H(Z^0)(\tau,\cdot)\|_{H^{1/2-\kappa}_z}.
\end{align*}
The desired estimate then follows. The  remaining paraproduct term can be estimated by similar arguments.

 We now turn to the term quadratic in $Z^0$ (still in the case $i=2$). For every $\kappa'>0$ sufficiently small, by Besov embedding there exists $C>0$ such that
\begin{equation}
|d(x,\partial M)|^\alpha \|H(Z^0)(\tau,\cdot)^2\|_{\mathcal C^{-1+\alpha-\kappa}_{z}} \leq C |d(x,\partial M)|^\alpha \|H(Z^0)(\tau,\cdot)^2\|_{B^{1/2+\alpha}_{4/3-\kappa',1,z}}. 
\end{equation}
Furthermore, by the fractional Leibniz rule and Sobolev embedding we have that for every $\kappa'''> 0$ sufficiently small and $\kappa''>0$ sufficiently small depending on $\kappa'''$, there exists $C',C''>0$ such that 
\begin{align}
|d(x,\partial M)|^\alpha &\|H(Z^0)(\tau,\cdot)^2\|_{B^{1/2+\alpha}_{4/3-\kappa',1,z}} \\ &\leq C' \|Z^0\|_{H^{1/2-\kappa''}_{z}} \|H(Z^0)(\tau,\cdot)\|_{L^{4-\kappa''}_{z}}\leq C''\|Z^0\|^2_{H^{1/2-\kappa'''}_{z}}.
\end{align}
The desired estimate follows.

Finally, let us do the estimates when $i=3$, in which case $\alpha \in (1/2,3/2]$. We expand:
\begin{equation} \label{eq: hw3 term}
\llbracket H(W^0)^3 \rrbracket + 3\llbracket H(W^0)^2 \rrbracket H(Z^0) + 3H(W^0) H(Z^0)^2 + H(Z^0)^3.
\end{equation}
The desired estimate on the first term in \eqref{eq: hw3 term} is a consequence of Proposition \ref{prop: stochastic boundary v2}. For the second term in \eqref{eq: hw3 term}, we will again do a ($z$-direction) paraproduct decomposition. First observe that, by Besov embedding, for every $\kappa'>0$ sufficiently small, there exists $C>0$ such that 
\begin{align}
&|d(\tau,\partial M)|^\alpha  \|\llbracket H(W^0)(\tau,\cdot)^2 \rrbracket \succ_{z} H(Z^0)(\tau,\cdot)\|_{\mathcal C^{-3/2+\alpha-\kappa}_z}  
\\
&\leq C |d(\tau,\partial M)|^\alpha  \|\llbracket H(W^0)(\tau,\cdot)^2 \rrbracket \succ_{z} H(Z^0)(\tau,\cdot)\|_{B^{-1+\alpha-\kappa'}_{4-\kappa',4-\kappa', z}}. 	
\end{align}
Furthermore, by paraproduct estimates for every $\kappa''>0$ sufficiently small, there exists $C>0$ such that
\begin{align}
&|d(\tau,\partial M)|^{\alpha} \|\llbracket H(W^0)(\tau,\cdot)^2 \rrbracket \succ_{z} H(Z^0)(\tau,\cdot)\|_{B^{-1+\alpha-\kappa'}_{4-\kappa',4-\kappa', z}} 
\\
&\leq |d(\tau,\partial M)|^{\alpha} \| \llbracket H(W^0)(\tau,\cdot)^2 \rrbracket \|_{\mathcal C^{-1+\alpha-\kappa'}_z} \|H(Z^0)(\tau,\cdot)\|_{L^{4-\kappa''}_{z}}.
\end{align}
The desired estimate follows since, by Sobolev embedding, for every $\kappa'''>0$ sufficiently small, there exists $C>0$ such that $\|H(Z^0)(\tau,\cdot)\|_{L^{4-\kappa''}_z} \leq C \|H(Z^0)(\tau,\cdot)\|_{H^{1/2-\delta}_z}$. 

For the resonant product term, by Besov embedding, for every $\kappa'>\kappa$ sufficiently close to $\kappa$, there exists $C>0$ such that
\begin{align}
&|d(\tau,\partial M)|^\alpha  \|\llbracket H(W^0)(\tau,\cdot)^2 \rrbracket \circ_z H(Z^0)(\tau,\cdot) \|_{\mathcal C_{z}^{-3/2+\alpha-\kappa}}
\\&\leq C|d(\tau,\partial M)|^\alpha  \|\llbracket H(W^0)(\tau,\cdot)^2 \rrbracket \circ_z H(Z^0)(\tau,\cdot) \|_{H_{z}^{-1/2+\alpha-\kappa'}}.
\end{align}
Furthermore by resonant product estimates, there exists $C>0$ such that
\begin{align}
&|d(\tau,\partial M)|^\alpha  \|\llbracket H(W^0)(\tau,\cdot)^2 \rrbracket \circ_z H(Z^0)(\tau,\cdot) \|_{H^{-1/2+\alpha-\kappa'}_{z}} 
\\&\leq C|d(\tau,\partial M)|^\alpha \|\llbracket H(W^0)(\tau,\cdot)^2 \rrbracket \|_{\mathcal C^{-1+\alpha-\kappa'/2}_{z}} \|H(Z^0)\|_{H^{1/2-\kappa'/2}_{z}}.
\end{align}
The desired estimate now follows by arguments similar to above.  The remaining paraproduct term follows by similar estimates as above.

We now turn to the third term in \eqref{eq: hw3 term} and do another paraproduct decomposition. By paraproduct estimates, note that given $\kappa'>0$ sufficiently small, for every $\kappa''>0$ sufficiently small, there exists $C>0$ such that 
\begin{align}
&|d(\tau,\partial M)|^\alpha \| H(W^0)(\tau,\cdot) \succ_z H(Z^0)(\tau,\cdot)^2 \|_{B^{-1/2-\kappa'+\alpha}_{2-\kappa',z}} 
\\ &\leq C |d(\tau,\partial M)|^\alpha  \| H(W^0)(\tau,\cdot)\|_{\mathcal C^{-1/2-\kappa'/2+\alpha}_z} \| H(Z^0)(\tau,\cdot)^2\|_{L_{z}^{2-\kappa''}}.
\end{align}
The lefthand side can be estimated (as above) using Besov embedding. The desired estimate holds by straightforward arguments. For the resonant product term, by Besov embedding and resonant product estimates, we have that for every $\kappa'>0$ sufficiently small there exists $C,C'>0$ such that
\begin{align}
	&|d(\tau,\partial M)|^\alpha \| H(W^0)(\tau,\cdot) \circ_z H(Z^0)(\tau,\cdot)^2 \|_{\mathcal C^{-3/2-\kappa+\alpha}_z} \\
	&\leq C|d(\tau,\partial M)|^\alpha \| H(W^0)(\tau,\cdot) \circ_z H(Z^0)(\tau,\cdot)^2 \|_{B^{-\kappa/2+\alpha}_{4/3-\kappa',z}}\\
&\leq C'|d(\tau,\partial M)|^\alpha \|H(W^0)(\tau.\cdot)\|_{\mathcal C^{-1/2-\kappa/4+\alpha}_z}\|H(Z^0)(\tau,\cdot)^2\|_{B^{1/2-\kappa/4}_{4/3-\kappa',z}}.
\end{align}
From here we may proceed as above. Finally, let us estimate the last term. First note that by Besov embedding, for every $\kappa'>0$ sufficiently small, there exists $C>0$ such that
\begin{equation}
|d(\tau,\partial M)|^\alpha \|H(Z^0)(\tau,\cdot)^3\|_{\mathcal C^{-3/2+\alpha-\kappa}_z}
\leq C |d(\tau,\partial M)|^\alpha \|H(Z^0)(\tau,\cdot)^3\|_{B^{\alpha}_{4/3-\kappa',z}}.	
\end{equation}
By the fractional Leibniz estimate, interpolation estimate \eqref{tool: H reg decay}, and Sobolev embedding, for every $\kappa''>0$ and $\kappa '''>0$ sufficiently small, there exists $C,C'>0$ such that 
\begin{align}
|d(\tau,\partial M)|^\alpha \|H(Z^0)(\tau,\cdot)^3\|_{B^{\alpha}_{4/3-\kappa',z}} &\leq C |d(\tau,\partial M)|^\alpha \|H(Z^0)(\tau,\cdot)\|_{B^\alpha_{4-\kappa''}} \|Z^0\|_{L^4_z}^2
\\ &\leq C' \|Z^0\|_{H^{1/2-\kappa'''}_z}^3.	
\end{align}
The desired estimate follows.

\end{proof}

\subsection{Estimate on $K_T$}

Recall in \eqref{eq: intermediate ansatz}  we introduced an intermedate ansatz $Z_T(u) = -\mathbb W_T^{[3]} + K_T(u)$, where
\begin{align}
\begin{split} \label{eq: kt definition}
K_T(u)&=\int^T_0 J^2_t (\mathbb{W}_t^2 \succ Z^{\flat}_t) + J^2_t
    ( \mathbb{W}_t^2 H (\varphi_-, \varphi_+)_T) + 12 J^2_t W_t \llbracket H
     (\varphi_-, \varphi_+)_T^2 \rrbracket 
     \\ &\qquad + 4J^2_t \llbracket H (\varphi_-,
     \varphi_+)_T^3 \rrbracket \mathd t + Z_T (\ell_T(u)).
     \end{split}	
\end{align}
The intermediate ansatz results in a loss of $H^1$ regularity for $Z_T(u)$ as $T\rightarrow \infty$  since the term $\mathbb W_T^{[3]}$ has almost sure regularity $\mathcal C^{1/2-\kappa}$. It turns out that the remainder term $K_T(u)$ \emph{almost} has the optimal regularity, as we shall now demonstrate.
\begin{lemma} \label{lem: KT estimate}
There exists $C>0$ and $\kappa'>0$ sufficiently small such that, for every $\varphi_-,\varphi_+ \in \mathfrak X$ and for every $T>0$,
\begin{align}
&\mathbb E[\|K_T\|_{H^{1-\kappa}}^2] 
\\&\qquad \leq C \Big(  1+  \| H(\varphi_-,\varphi_+)_T\|_{L^{2-\kappa'}}^2 + \| d(\cdot,\partial M)^{1/2} \llbracket H(\varphi_-,\varphi_+)_T^2 \rrbracket \|_{L^{2-\kappa'}}^2 
\\ &\qquad \qquad  +\| d(\cdot,\partial M) \llbracket H(\varphi_-,\varphi_+)^3 \rrbracket \|_{L^{2-\kappa'}}^2 \Big) + \mathbb E\left[ \|Z_T\|_{L^4}^4 + \frac 12 \int_0^T \| \ell_T(u)\|_{L^2_x}^2 dt \right].
\end{align}
\end{lemma}

\begin{proof}
For the first term in \eqref{eq: kt definition}, we may estimate as in \cite[Section 7]{BG20} to obtain
\begin{equation}
\left\| \int^T_0 J^2_t (\mathbb{W}_t^2 \succ Z^{\flat}_t) dt \right\|_{H^{1-\kappa}} \leq \|\mathbb W_t^2 \|_{L^\infty_t \mathcal C^{-1-\kappa/2}} \|Z_T\|_{L^2},
\end{equation}
which can then be bounded by Young's inequality and stochastic estimates. Here, we have implicitly used the regularizing properties of $J_t$ established in Appendix \ref{appendix: regularizing jt}. 
	
For the second term in \eqref{eq: kt definition}, first observe that by the properties of $J_t$ (again, see Appendix \ref{appendix: regularizing jt}) and the Cauchy-Schwarz inequality,
\begin{equation}
\left\| \int_0^T J^2_t( \mathbb{W}_t^2 H (\varphi_-, \varphi_+)_T) dt \right\|_{H^{1-\kappa}}^2 \leq C \int_0^T \| J_t ( \mathbb{W}_t^2 H (\varphi_-, \varphi_+)_T) \|^2_{H^{-\kappa/2}} dt.
\end{equation}
Hence, by Fubini's theorem, regularizing properties of $J_t$, the integral kernel representation of Sobolev norms\footnote{Any function with Dirichlet boundary conditions is periodic and hence can be viewed as an element of the torus. Integral representations of (inhomogeneous) Sobolev norms using the kernel of the periodic Laplacian is standard. We have implicitly used the decay of the kernel in this estimate}, and kernel bounds on $C^M_T$ (see Appendix \ref{sec: Dirichlet covariance estimates}), we have that
\begin{align}
& \int_0^T \mathbb E \left[ \|J_t(\mathbb W_t^2 H(\varphi_-,\varphi_+)_T)\|_{H^{-\kappa/2}}^2 \right]
\\
&\leq C \int_0^T \int \frac{|H(\varphi_-,\varphi_+)_T(x)H(\varphi_-,\varphi_+)_T(y)| \mathbb E[\mathbb W_t^2(x) \mathbb W_t^2(y)]}{|x-y|^{1-\kappa/2}}  \frac{dt}{\langle t \rangle^{1+\kappa/4}}
\\
&\leq C \int_0^T \int \frac{|H(\varphi_-,\varphi_+)_T(x)H(\varphi_-,\varphi_+)_T(y)|}{|x-y|^{3-\kappa/2}}  \frac{dt}{\langle t \rangle^{1+\kappa/4}}\leq C \| H(\varphi_-,\varphi_+)_T\|_{L^{2-\epsilon(\kappa)}}^2,
\end{align}
where the last inequality is by Young's convolution inequality and $\epsilon$ is a positive function of $\kappa$ that is $o(1)$ as $\kappa \downarrow 0$.

For the third term in \eqref{eq: kt definition}, by properties of $J_t$ and the Cauchy-Schwarz inequality,
\begin{equation}
\mathbb E \Bigg[ \left\| \int_0^T J^2_t (W_t \llbracket H (\varphi_-,  \varphi_+)^2 _T \rrbracket ) dt \right\|_{H^{1-\kappa}}^2 \Bigg] 
    \leq C \int_0^T \langle t \rangle^{-1-\kappa/2} \mathbb E [ \| W_t \llbracket H (\varphi_-,  \varphi_+)^2_T \rrbracket  \|_{H^{-1-\kappa/2}}^2] dt.
\end{equation}
It suffices to bound the expectation in the integrand uniformly in $t$. Note that, by Fubini's theorem and the integral representation of Sobolev norms,
\begin{align}
&\mathbb E[	\| W_t \llbracket H (\varphi_-,  \varphi_+)^2_T \rrbracket  \|_{H^{-1-\kappa/2}}^2] 
\\&\quad = \int \int (-\Delta^{\rm per}+m^2)^{-1-\kappa/2}(x,x') C_t^M(x,x')  \llbracket H (\varphi_-,  \varphi_+)^2_T \rrbracket (x)  \llbracket H (\varphi_-,  \varphi_+)^2_T \rrbracket (x') dx dx'.
\end{align}
Thus by standard covariance estimates on the periodic Laplacian and the covariance estimates on $C^M_T$ in Appendix \ref{sec: Dirichlet covariance estimates}, for every $\alpha\in (0,1)$, there exists $C>0$ such that
\begin{equation}
|(-\Delta^{\rm per}+m^2)^{-1-\kappa}(x,x') C_t^M(x,x') | \leq C \frac{d(x,\partial M)^\alpha d(x',\partial M)^\alpha}{|x-x'|^{1+2\alpha}|x-x'|^{1-\kappa}}. 
\end{equation}
Hence, 
\begin{multline}
\mathbb E [ \| W_t \llbracket H (\varphi_-,  \varphi_+)^2_T \rrbracket  \|_{H^{-1-\kappa/2}}^2]
\\
\leq C \int \int \frac{d(x,\partial M)^\alpha d(x',\partial M)^\alpha}{|x-x'|^{2+2\alpha-\kappa}}  \llbracket H (\varphi_-,  \varphi_+)^2_T \rrbracket(x) \llbracket H (\varphi_-,  \varphi_+)^2_T \rrbracket(x') dx dx'.
\end{multline}
Taking $\alpha = 1/2$, then there exists $\epsilon = \epsilon(\kappa) = o(1)$ as $\kappa \downarrow 0$ such that, by Young's convolution inequality,
\begin{align}
\int_0^T \mathbb E & [ \| J_t (W_t \llbracket H^2 (\varphi_-,  \varphi_+) ) \|_{H^{-\kappa/2}}^2] dt	
\leq C \| d(\cdot,\partial M)^{1/2} \llbracket H^2(\varphi_-,\varphi_+) \rrbracket \|_{L^{2-\epsilon}}^2,	
\end{align}
as required.

For the fourth term in \eqref{eq: kt definition}, let us recall that $C^M_T= \partial_t\rho_t^2\ast C^M_\infty$ is a positive operator with square root $\sqrt{\partial_t \rho_t^2} \ast (C^M_\infty)^{1/2}$. We will use the regularizing properties of $(C^M_\infty)^{1/2}$ stated in Appendix \ref{subsec: regularizing cthalf}. By computing the $t$-integral and using regularizing properties of $(C_\infty^M)^{1/2}$, there exists $C>0$ such that
\begin{align}
\Bigg\| \int_0^T J_t^2 \llbracket H(\varphi_-,\varphi_+)^3_T \rrbracket dt \Bigg\|_{H^{1-\kappa}}^2 &=  \left\|  C_T^M\left[  \llbracket H(\varphi_-,\varphi_+)^3_T \rrbracket(y) \right]\right\|_{H^{1-\kappa}}^2 
\\&\leq C \left\| (C_T^M)^{1/2} \left[\llbracket H(\varphi_-,\varphi_+)^3_T \rrbracket \right] \right \|_{H^{-\kappa}}^2.
\end{align}
Let us define the kernel
\begin{equation}
\mathcal C(x,x'):= (C_\infty^M)^{1/2} \ast (-\Delta^{\rm per}+m^2)^{-\kappa}\ast (C_\infty^M)^{1/2}(x,x'). 
\end{equation}
By the kernel estimates on $(C_\infty^M)^{1/2}$ together with standard kernel estimates on $(-\Delta^{\rm per}+m^2)^{-\kappa}$, we have that
\begin{align}
|\mathcal C(x,x')| &\leq C \int \int \frac{d(x,\partial M)^\alpha d(x',\partial M)^\beta dydy'}{|x-y|^{2+\alpha}|y-y'|^{3-\kappa}|y'-x'|^{2+\beta}} 
\leq C'\int  \frac{d(x,\partial M)^\alpha d(x',\partial M)^\beta dy}{|x-y|^{2+\alpha}|y-x'|^{2+\beta-\kappa}}
\\ 
&\leq C''\frac{d(x,\partial M)^\alpha d(x',\partial M)^\beta }{|x-x'|^{1+\alpha+\beta-\kappa}}.
\end{align}
We choose $\alpha=\beta = 1-\varepsilon$ for $\varepsilon$ small enough depending on $\kappa$. Hence, by using the integral representation of Sobolev norms the definition of $\mathcal C(x,x')$, we have
\begin{align}
&\left\| (C_T^M)^{1/2} \left[ \llbracket H(\varphi_-,\varphi_+)^3_T \rrbracket \right] \right \|_{H^{-\kappa}}^2
\\
&\qquad \leq C' \int\int   |\mathcal C(x,x')|\left|\llbracket H(\varphi_-,\varphi_+)^3 \rrbracket(x) \llbracket H(\varphi_-,\varphi_+)^3 \rrbracket(x') \right| dxdx'
\\
&\qquad \leq C'' \int \int \frac{d(x,\partial M)^{1-\varepsilon} d(x',\partial M)^{1-\varepsilon}}{|x-x'|^{3-\kappa}}\left| \llbracket H(\varphi_-,\varphi_+)^3 \rrbracket(x) \llbracket H(\varphi_-,\varphi_+)^3 \rrbracket(x') \right| dx dx'
\\
&\qquad = C''' \| d(\cdot,\partial M)^{1-\varepsilon} \llbracket H(\varphi_-,\varphi_+)^3 \rrbracket \|_{L^{2-\epsilon'(\kappa)}}^2,
\end{align}
where the last line follows by Young's convolution inequality with $\epsilon'(\kappa) = o(1)$ as $\kappa \downarrow 0$. 

The estimate on $Z_T(\ell_T(u))$ follows by straightforward bounds on $Z_T$. 
\end{proof}

\subsection{Estimate of $R_1$}

Recall that 
\begin{equation}
R_1 =  \int_{M} 12 W_T H (\varphi_-, \varphi_+)_T Z^2_T + 6 
  \llbracket H (\varphi_-, \varphi_+)_T^2 \rrbracket Z^2_T \, dx =: R_1(a) + R_1(b).
\end{equation}

We will begin with a useful interpolation estimate between Sobolev regularity and boundary blow-up. 
\begin{lemma}\label{lemma: weight Sobolev}
For every $\delta>0$, for every $\alpha, \beta >0$ such that $\alpha+\beta<1$, there exists a constant $C$ such that, for every $f\in H^{\alpha+\beta}(M)$,
\begin{equation}
\|d(x,\partial M)^{-\alpha}f\|_{L^2_{\tau} H^{\beta}_{z}} \leq C \|f\|_{H^{\alpha+\beta+\delta}_x}	
\end{equation}
\end{lemma}
\begin{proof}
Note that $\langle n_1 \rangle^{\alpha} \langle (n_2,n_3) \rangle^{\beta} \leq  \langle (n_1,n_2,n_3) \rangle^{\alpha+\beta}$. Hence 
\begin{equation} \label{eq: HtauHzHx}
\|f\|_{H^{\alpha}_{\tau}H^{\beta}_{z}} \leq \|f\|_{H^{\alpha+\beta}}.	
\end{equation}
 Therefore, by H\"older's inequality and Sobolev embedding,
\begin{align}
\|d(x,\partial M)^{-\alpha}f\|_{L^2_{\tau} H^{\beta}_{z}} &\leq \left(\int d(\tau,\partial M)^{-\frac{\alpha}{\alpha+\delta}}d\tau\right)^{\frac{1}{2\alpha+2\delta}}  \|\|f\|^2_{H^{\beta}_{z}}\|_{L^{\frac{2}{1-2\alpha-2\delta}}_\tau}\\
	&\leq C  \|\|f\|^2_{H^{\beta}_{z}}\|_{H^{\alpha+\delta}_{\tau}} \leq C'  \|f\|_{H^{\alpha+\beta+\delta}}.
\end{align}
\end{proof}

We will estimate $R_1(b)$ first since it is the more straightforward term to analyze.

\begin{lemma}
There exists $\kappa >0$ sufficiently small such that, for every $\delta> 0 $ sufficiently small, there exists $C > 0$  such that, for every $T>0$,
  \begin{align} 
  \mathbb E [| R_1(b) |]&\leq C  \Big( 1+ \| d^{1-\kappa}(\cdot,\partial M) \llbracket H(\varphi_-,\varphi_+)^2_T \rrbracket \|_{L^4_\tau L^\infty_z}^4  \\&\qquad +  \| d(\cdot,\partial M)^{1/2-\kappa}  \llbracket H(\varphi_-,\varphi_+) ^2_T \rrbracket \|_{L^{4/3}_\tau L^\infty_z}^4 \Big) + \delta (\| K_T
     \|^2_{H^{1 - \kappa}} + \| K_T \|^4_{L^4} ).
\end{align}
\end{lemma}

\begin{proof}
Recalling the intermediate ansatz \eqref{eq: intermediate ansatz}, we further decompose $R_1(b)$ into 
\begin{equation}
R_1(b) = \int \llbracket H(\varphi_-,\varphi_+)_T^2 \rrbracket (K_T^2 + 2 \mathbb W_T^{[3]} K_T + (\mathbb W_T^{[3]})^2 ) =: \sum_{i=1}^3 R_1(b;i). 
\end{equation}

Let us first estimate $R_1(b,1)$.  
Let $\kappa > 0$ be sufficiently small. We split into $(\tau,z)$ coordinates. By H\"older's inequality, symmetry, followed by H\"older's inequality again, we have
\begin{align}
|R_1&(b,1)| 
\\ &\leq \int \| d(\tau,\partial M)^{1-\kappa} \llbracket H(\varphi_-,\varphi_+) ^2_T \rrbracket (\tau,\cdot)\|_{L^\infty_z} \| d(\tau,\partial M)^{-1+\kappa} K_T(\tau,\cdot)\|_{L^2_z} \|K_T(\tau,\cdot)\|_{L^4_z} d\tau  \\
&\leq \| d^{1-\kappa}(\cdot ,\partial M) \llbracket H(\varphi_-,\varphi_+)_T^2 \rrbracket \|_{L^4_\tau L^\infty_z}  \| d(\cdot, \partial M)^{-1+\kappa} K_T\|_{L^2_x} \|K_T\|_{L^4_x}
\\
&\leq
C \| d^{1-\kappa}(\cdot,\partial M) \llbracket H(\varphi_-,\varphi_+)^2_T \rrbracket \|_{L^4_\tau L^\infty_z}^4 + \delta \left( \| d(\cdot, \partial M)^{-1+\kappa} K_T\|_{L^2_x}^2 +  \|K_T\|_{L^4_x}^4 \right).
\end{align}
The desired estimate holds by using Lemma \ref{lemma: weight Sobolev} on the penultimate term of the righthand side.

For the second term, we proceed similarly. By H\"older's inequality,   
\begin{align}
|R_1(b,2)| \leq \| d(\cdot,\partial M)^{1/2-\kappa}  \llbracket H(\varphi_-,\varphi_+) ^2_T \rrbracket \|_{L^{4/3}_\tau L^\infty_z} \| d(\cdot,\partial M)^{-1/2+\kappa} \mathbb W^{[3]}_T\|_{L^\infty_x} \| K_T \|_{L^4_x}.
\end{align}
Similarly, for the final term we have
\begin{equation}
|R_1(b,3)| \leq \|d^{1-2\kappa}(\cdot,\partial M)  \llbracket H(\varphi_-,\varphi_+)_T ^2 \rrbracket  \|_{L^1_\tau L^\infty_z} \| d^{-1/2+\kappa}(\cdot,\partial M) \mathbb W_T^{[3]} \|_{L^\infty_x}.
\end{equation}
The desired estimates follow by straightforward arguments. 
 \end{proof}

We now turn to the estimate of $R_1(a)$. 
 
 \begin{lemma}
 There exists $p>0$ such that, for every $\delta > 0$ sufficiently small, there exists $C=C(\delta)>0$ such that, for every $T > 0$, 
  \begin{align} 
  \mathbb E [| R_1(a) |] &\leq C  \Big(1+\| H(\varphi_-,\varphi_+)_T\|_{L^{8/5}}^2 + \| d (\cdot,
    \partial M)^{1 / 2 + 2 \kappa} H (\varphi_-, \varphi_+)_T 
    \|_{L^\infty}^p
    \\
    &\qquad + \| d (\cdot, \partial M)^{1 + 2 \kappa} H (\varphi_-, \varphi_+)_T) \|_{L_{\tau}^{\infty} \mathcal{C}_z^{1 / 2 + \kappa}}^p\Big) + \delta (\| K_T
     \|^2_{H^{1 - \delta}} +  \| K_T \|^4_{L^4} ).
\end{align}

 \end{lemma}
 
 \begin{proof}
By Wick's theorem, $\mathbb E[W_T (\mathbb W_T^{[3]})^2]=0$. Hence by the intermediate ansatz \eqref{eq: intermediate ansatz}, we obtain
  \begin{align} \label{eq: R_1a}
    R_1(a) 
    \equiv  & 
    \int 12 W_T  H (\varphi_-, \varphi_+) \mathbb{W}_T^{[3]}K_T \mathd x + \int
    12 W_T H (\varphi_-, \varphi_+) K^2_T \mathd x.
  \end{align}
  For the first term in \eqref{eq: R_1a}, by duality and the Cauchy-Schwarz inequality, we have that for every $\delta > 0$, there exists $C=C(\delta) > 0$ such that
  \begin{equation}
  \left| \int W_T H (\varphi_-, \varphi_+) \mathbb{W}_T^{[3]} K_T dx \right|  \leq  \mathbb E\left [ C \| W_T H (\varphi_-, \varphi_+) \mathbb{W}_T^{[3]}
    \|^2_{H^{- 1 + \kappa}} + \delta \| K_T \|_{H^{1 - \kappa}}^2 \right]. 
  \end{equation}
 We analyze the first term on the righthand side above. Recall that $\Delta^{\rm per}$ is the Laplacian on $\mathcal M$. By the integral representation of Sobolev norms, we may write 
  \begin{multline}
  \mathbb E \left[ \| W_T H (\varphi_-, \varphi_+) \mathbb{W}_T^{[3]}
    \|^2_{H^{- 1 + \kappa}} \right] \\= \int \int (-\Delta^{\rm per}+m^2)^{-1+\kappa} H(\varphi_-,\varphi_+)(x) H(\varphi_-,\varphi_+)(y) \mathbb E[\mathcal E(x,y)],
  \end{multline}
  where $\mathcal E(x,y):= W_T(x) \mathbb W_T^{[3]}(x) W_T(y) \mathbb W_T^{[3]}(y)$. Notice that 
  \begin{equation}
  \mathbb E[\mathcal E(x,y)] = \int_0^T \int_0^T \int \int  \dot{C}_t^M(x,z)  \dot{C}_s^M(y,z') \mathbb  E\left[ W_T(x) W_T(y) \mathbb W_t^3(z) \mathbb W_s^3(z')  \right] dz dz' dt ds.
  \end{equation}
  By Wick's theorem, we have that
  \begin{align}
  \mathbb  E\left[ W_T(x) W_T(y) \mathbb W_t^3(z) \mathbb W_t^3(z')  \right] = \mathcal S_1(x,y,z,z') + \mathcal S_2(x,y,z,z') + \mathcal S_3(x,y,z,z'),
  \end{align}
  where
  \begin{align}
  \mathcal S_1(x,y,z,z') &= 6C_T^M(x,y) C_{t\wedge s}^M(z,z')^3, \\
  \mathcal S_2(x,y,z,z') &= 18 C_t^M(x,z)C_{t\wedge s}^M(z,z')^2 C_s^M(y,z'),
  \\
  \mathcal S_3(x,y,z,z') &= 18 C_{s}^M(x,z') C_{t\wedge s}^M(z,z')^2 C_t^M(y,z). 
  \end{align}
  For each $i$, define 
  \begin{equation}
  \mathcal S_i'(x,y):= \int_0^T \int_0^T \int  \int  \dot{C}_t^M(x,z) \dot{C}_s^M(y,z') \mathcal S_i dz dz' dt ds. 
  \end{equation}
By the interpolation estimates on the covariances, see Propositions \ref{prop: M covariance interp T} and \ref{prop: M derivative covariance interp}, and standard convolution bounds, we have
\begin{align}
&\mathcal S_1'(x,y)\\
 &\leq C \int_0^T \int_0^T \int \int \frac{1}{\langle t\rangle ^{1+2\delta} \langle s \rangle^{1+2\delta}  |x-z|^{1+2\delta }|y-z'|^{1+2\delta} }
 \\ 
 &\hspace{60mm} \times \frac{1}{|x-y|} \frac{(t\wedge s)^\delta \log\left( 1 + \frac{|z-z'|^2}{(z_1-z_1)^2} \right) }{|z-z'|^{3-\delta}} dzdz'dsdt
\\
&\leq C' \frac{1}{|x-y|} \int \frac{1}{|x-z|^{1+2\delta}|y-z|^{1+\delta}} dz \leq \frac{C''}{|x-y|}.
\end{align}
Similarly, integrating out the $t$ and $s$ variables and $z'$, we obtain
\begin{align}
\mathcal S_2'(x,y) &\leq C \int_0^T \int_0^T \int \int \frac{1}{\langle t\rangle ^{1+2\delta} \langle s \rangle^{1+2\delta}  |x-z|^{2+2\delta }|y-z'|^{2+2\delta} } \frac{1}{|z-z'|^2} dz dz' ds dt
\\
&\leq C' \int \frac{1}{|x-z|^{2+2\delta}|y-z|^{1+2\delta}}dz \leq \frac{C''}{|x-y|^{4\delta}}.  
\end{align}
By symmetry, we also have
\begin{equation}
\mathcal S_3'(x,y) \leq \frac{C''}{|x-y|^{4\delta}}. 
\end{equation}
Putting these together and using standard bounds on the Bessel potential kernel, we obtain
\begin{align}
&\mathbb E \left[ \| W_T H (\varphi_-, \varphi_+)_T \mathbb{W}_T^{[3]}
    \|^2_{H^{- 1 + \kappa}} \right] 
    \\&\leq C \int \int \frac{|(-\Delta+m^2)^{-1+\kappa}(x,y)|}{|x-y|} |H(\varphi_-,\varphi_+)_T(x)| |H(\varphi_-,\varphi_+)_T(y)| dx dy
    \\
    &\leq C' \int \int \frac{1}{|x-y|^{2+2\kappa}}	|H(\varphi_-,\varphi_+)_T(x)| |H(\varphi_-,\varphi_+)_T(y)|\leq C'' \| H(\varphi_-,\varphi_+)_T\|_{L^{8/5}}^2,
\end{align}
where the last inequality comes from using Young's convolution inequality (duality form) with exponents $(p,q)$ such that $2/p+1/q=2$ (in this case, we choose $q=4/3$ and $p=8/5$).

  We now estimate the second term in the term $R_1(a)$ given in \eqref{eq: R_1a}. By a paraproduct decomposition in the $z$-variable, 
  \begin{align} \label{eq: R_1a para decomp}
    W_T H (\varphi_-, \varphi_+)_T & = W_T \succ_z H (\varphi_-, \varphi_+)_T +
    W_T \preccurlyeq_z H (\varphi_-, \varphi_+)_T
  \end{align}
  We will need the following preliminary estimates on these terms. For the first term in \eqref{eq: R_1a para decomp}, by paraproduct estimates, for any $\kappa > 0$, there exists $C=C(\kappa)>0$ such that
  \begin{multline}
  \begin{split} \label{eq: R1_a prelim 1}
    \sup_\tau \| d (\tau, \partial M)^{1 / 2 + 2 \kappa} W_T \succ_z H (\varphi_-, \varphi_+)_T
    (\tau, \cdot) \|_{\mathcal{C}_z^{- 1 / 2 - \kappa}}  
    \\ \leq  C\sup_\tau \| W_T (\tau, \cdot) \|_{\mathcal{C}_z^{- 1 / 2 - \delta}} \| d (\tau,
    \partial M)^{1 / 2 + 2 \kappa} H (\varphi_-, \varphi_+)_T (\tau, \cdot)
    \|_{L_z^{\infty}} .
    \end{split}
  \end{multline}
   For the second term in \eqref{eq: R_1a para decomp}, by para- and resonance product estimates and similar symmetry reasons, 
  \begin{align}
  \begin{split} \label{eq: R1_a prelim 2}
     \sup_\tau &\| d (\tau, \partial M)^{1 + 2 \delta} (W_T \preccurlyeq_z H (\varphi_-, \varphi_+)_T
    (\tau, \cdot)) \|_{L_z^{\infty}} \\
    &\leq  C \| W_T  \|_{L_{\tau}^{\infty} \mathcal{C}_z^{- 1 / 2 -
    \kappa}} {\| d (\cdot, \partial M)^{1 + 2 \kappa} H (\varphi_-, \varphi_+)_T) \|_{L_{\tau}^{\infty} \mathcal{C}_z^{1 / 2 + \kappa}}}. 
    \end{split}  
  \end{align}
  
By using a paraproduct decomposition and inserting weight factors, we may write
  \begin{equation}
  \int W_T H(\varphi_-,\varphi_+) K_T^2 = R_1(a,2,1) + R_1(a,2,2),
  \end{equation}
  where
  \begin{align}
  R_1(a,2,1) &:=  \int d (x, \partial M)^{1/2 + 2\kappa} (W_T \succ_z H (\varphi_-, \varphi_+)_T) (d
    (x, \partial M)^{-1 / 4 - \kappa} K_T)^2 \mathd x, 
    \\
    R_1(a,2,2) &:= \int d (x, \partial M)^{1 + 2 \kappa} (W_T \preccurlyeq_z H (\varphi_-, \varphi_+)_T)
    (d (x, \partial M)^{-1 / 2 - \kappa} K_T)^2 \mathd x.	
  \end{align}
  By duality, the fractional Leibniz rule, and the preliminary estimate \eqref{eq: R1_a prelim 1}
  \begin{align}
  |R_1(a,2,1)| &\leq C  \sup_\tau  \| d (\tau , \partial M)^{1 / 2 + 2 \delta} W_T \succ_z H (\varphi_-, \varphi_+)_T
    (\tau, \cdot) \|_{\mathcal{C}_z^{- 1 / 2 - \kappa}} \\ &\qquad \times  \| (d (\cdot, \partial M)^{-1 / 4 -
    \kappa} K_T) \|^2_{L^2_\tau H^{1 / 2 + \kappa }_z}.
  \end{align}
  By similar reasoning and \eqref{eq: R1_a prelim 2},
  \begin{align}
  |R_1(a,2,2)|&\leq C \| d (\cdot, \partial M)^{1 + 2 \kappa} W_T \preccurlyeq_z H (\varphi_-,
    \varphi_+)_T  \|_{L^{\infty}} \| (d (\cdot, \partial M)^{-1 / 2 - \delta}
    K_T) \|^2_{L^2}. 
  \end{align}
  Observe that by Lemma \ref{lemma: weight Sobolev},
  \begin{equation}
  \max(\| (d (\cdot, \partial M)^{-1 / 4 -
    \kappa} K_T) \|^2_{L^2_\tau H^{1 / 2 + \kappa }_z} ,\| (d (\cdot, \partial M)^{-1 / 2 - \delta}
    K_T) \|^2_{L^2}) \leq \|K_T\|_{H^{3/4+\kappa}}^2. 
  \end{equation}
The rest of the lemma now follows by a straightforward interpolation argument and we omit the details.
 
 \end{proof}

\subsection{Estimate of $R_3$}

Recall that
\begin{align}
R_3 = & \int_M \llbracket H (\varphi_-, \varphi_+)_T^4 \rrbracket -
  \llbracket \bar{H} (\varphi_-)_T^4 \rrbracket - \llbracket
  \bar{H} (\varphi_+)_T^4 \rrbracket \mathd x \\
  &\qquad \qquad \qquad \qquad + \int_{M \Delta \{ [-1,1]\times \mathbb T^2\}} \llbracket \overline H(\varphi_-)_T^4 \rrbracket + \llbracket \overline H(\varphi_+)_T^4 \rrbracket dx.
\end{align}

\begin{lemma}
For every $\kappa > 0$ sufficiently small, there exists $C>0$ such that for every $T>0$ and $\varphi_-,\varphi_+ \in \mathcal C^{-1/2-\kappa}_z$, 
\begin{equation}
|R_3| \leq \sum_{\sigma \in \{\pm \}}C \Big(1+ \|\varphi_\sigma \|_{\mathcal C^{-1/2-\kappa}_z}^4 + \sum_{i=1}^3  \sup_{\tau \in M} | d(\tau, \partial M)^{1-\kappa}\| \llbracket (\bar{H} (\varphi_\sigma))^i \rrbracket (\tau, \cdot)
    \|_{\mathcal{C}^{- 1 + \kappa/2}_z} |^4 \Big)
\end{equation}

\end{lemma}

\begin{proof}
Let us first observe that the bound on the integral supported on $M\Delta \{[-1,1]\times \mathbb T^2\}$ is straightforward to estimate, therefore we only estimate the first integral. We will drop the subscript $T$ for convenience.

  Note that we can write
\begin{equation}
H(\varphi_-,\varphi_+) = \overline H(\varphi_-) + \overline H(\varphi_+) + S(\varphi_-,\varphi_+),	
\end{equation}
where $S(\varphi_-,\varphi_+)$ is $m$-harmonic with boundary data $\bar H(\varphi_-)|_{\partial^+ M}$ and $\bar H(\varphi_+)|_{\partial^- M}$, and therefore $S(\varphi_-,\varphi_+)$ is smooth up to the boundary. Using this decomposition, we may write
  \begin{align} 
  \begin{split}\label{eq: r3 eq1}
 \llbracket H &(\varphi_-, \varphi_+)^4 \rrbracket \\& = \sum_{0 \leq i \leq 3}
     \binom{4}{i} \llbracket (\bar{H} (\varphi_-) + \bar{H} (\varphi_+))^i
     \rrbracket \cdot S (\varphi_-, \varphi_+)^{4-i} + \llbracket (\bar{H} (\varphi_-) +
     \bar{H} (\varphi_+))^4 \rrbracket.
   \end{split}
  \end{align}
  
  We begin by estimating 
  \begin{equation}
  \int \llbracket (\bar{H} (\varphi_-) +
     \bar{H} (\varphi_+))^4 \rrbracket -  \llbracket \bar{H} (\varphi_-)^4 \rrbracket  - 
       \llbracket \bar{H} (\varphi_+)^4 \rrbracket.
  \end{equation}
  By the Binomial theorem for Wick powers, 
  \begin{multline}
    \llbracket (\bar{H} (\varphi_-) + \bar{H} (\varphi_+))^4 \rrbracket -
    \llbracket (\bar{H} (\varphi_-))^4 \rrbracket - \llbracket (\bar{H}
    (\varphi_+))^4 \rrbracket \\ = \sum_{1 \leq j \leq 3} \binom{4}{j}
    \llbracket (\bar{H} (\varphi_-))^j \rrbracket \llbracket (\bar{H}
    (\varphi_+))^{4 - j} \rrbracket. 
      \end{multline}
  We now estimate the integrals above for each $j \in \{1,2,3\}$ in a unified way. Let $c>0$ be sufficiently small and write $M=M^{\rm bulk}(c) \sqcup M^{\rm bdry}(c)$, where ${\rm dist}(M^{\rm bulk}(c), \partial M)\geq c$. Then we split the integral to bulk and boundary contributions as follows:
  \begin{multline}
    \int \llbracket (\bar{H} (\varphi_-))^j \rrbracket \llbracket (\bar{H}
    (\varphi_+))^{4 - j} \rrbracket \mathd x 
    \\=  \int_{M^{\rm bulk}(c)} \llbracket (\bar{H} (\varphi_-))^i
    \rrbracket \llbracket (\bar{H} (\varphi_+))^{4 - i} \rrbracket \mathd x +
    \int_{M^{\rm bdry}(c)} \llbracket (\bar{H} (\varphi_-))^i \rrbracket
    \llbracket (\bar{H} (\varphi_+))^{4 - i} \rrbracket \mathd x.   
  \end{multline}
  The first integral, corresponding to the bulk contribution, is clearly bounded by $\| \varphi_- \|^4_{\mathcal{C}^{- 1 /
  2 - \delta}_z} + \| \varphi_+ \|^4_{\mathcal{C}^{- 1 / 2 - \delta}_z}$ due to
  the smoothing of the harmonic extension and H\"older's inequality. For the second term -- the boundary contribution -- it is convenient to split it further to $M^{\rm bdry}_\pm(c)$ according to which boundary component is nearest. Without loss of generality, we estimate $M^{\rm bdry}_-(c)$. By duality and smoothing properties of the $m$-harmonic extension, there exists $C=C(c)>0$ such that
  \begin{align}
    &\left| \int_{M^{\rm bdry}_-(c)} \llbracket (\bar{H} (\varphi_-))^j \rrbracket
    \llbracket (\bar{H} (\varphi_+))^{4 - j} \rrbracket\right| 
    \\&
    \leq \sup_{\tau \in M^{\rm bdry}_-(c)} \| \llbracket (\bar{H} (\varphi_+))^{4
    - j}(\tau,\cdot) \rrbracket \|_{\mathcal{C}^{1 - \kappa/2}_z} \\ &\qquad \qquad \qquad \times  \int  \frac{d(\tau, \partial M_-)^{1-\kappa}}{d(\tau, \partial M_-)^{1-\kappa}}\| \llbracket (\bar{H} (\varphi_-))^j \rrbracket (\tau, \cdot)
    \|_{\mathcal{C}^{- 1 + \kappa/2}_z} d\tau
    \\
    &\leq C\|\varphi_-\|^{4-j}_{\mathcal C_z^{-1/2-\kappa}} \sup_{\tau \in M^{\rm bdry}_-(c)} | d(\tau, \partial M_-)^{1-\kappa}\| \llbracket (\bar{H} (\varphi_-))^j \rrbracket (\tau, \cdot)
    \|_{\mathcal{C}^{- 1 + \kappa/2}_z} |.
  \end{align} 
  
 We now must estimate the integrals corresponding to the terms $i \in \{0,1,2,3\}$ in \eqref{eq: r3 eq1}, that is
  \begin{equation}
  \int \sum_{0 \leq i \leq 3} \binom{4}{i} \llbracket (\bar{H} (\varphi_-) + \bar{H}
  (\varphi_+))^i \rrbracket S (\varphi_-, \varphi_+)^{4-i}
  \end{equation}
  In the case $i=0$, we use the fact
  that by elliptic regularity theory and smoothing properties of $\bar H$,
  \begin{equation}
  \|S(\varphi_-,\varphi_+)\|_{C^2} \leq C (\|\varphi_-\|_{\mathcal C^{-1/2-\delta}_z}+\|\varphi_+\|_{\mathcal C^{-1/2-\delta}_z}).
  \end{equation}
Hence\footnote{Note that in this case, the desired estimate simply follows from the maximum principle. We will use the stronger regularity estimate below. }, 
  \begin{equation}
 \left| \int S(\varphi_-,\varphi_+)^{4} \right| \leq C'{\rm max}\left( \|\varphi_-\|_{\mathcal C^{-1/2-\delta}_z}, \|\varphi_+\|_{\mathcal C^{-1/2-\delta}_z} \right)^4. 
  \end{equation}
  For the other terms, we expand further
   \begin{align}
    \sum_{1 \leq i \leq 3} \binom{4}{i} &\llbracket (\bar{H} (\varphi_-) + \bar{H}
    (\varphi_+))^i \rrbracket S (\varphi_-, \varphi_+)^{4-i}
    \\ &
    =\sum_{1 \leq i \leq 3} \sum_{k = 0}^i \binom{4}{i} \binom{i}k \llbracket (\bar{H}
    (\varphi_-))^k \rrbracket \llbracket \bar{H} (\varphi_+)^{i-k}
    \rrbracket S (\varphi_-, \varphi_+)^{4-i}  
    \end{align}
    We will estimate this integrals in a unified way. Let $c>0$ be sufficiently small and $M^{\rm bulk}(c)$, $M^{\rm bdry}_-(c)$ and $M^{\rm bdry}_+(c)$ be as above. We split the integrand into boundary and bulk contributions in the obvious way:
    \begin{align}
    \int \llbracket (\bar{H} (\varphi_-))^k \rrbracket \llbracket \bar{H}
    (\varphi_+)^{i-k} \rrbracket S (\varphi_-, \varphi_+)^{4-i} = I^{\rm bulk}+I^{\rm bdry}_- + I^{\rm bdry}_+. 	
    \end{align}
    The desired estimate on $I^{\rm bulk}$ is trivial because all terms are smooth. By symmetry, we only consider $I^{\rm bdry}_-$. By duality 
    \begin{align}
    |I^{\rm bdry}_-| &\leq \sup_{\tau \in M^{\rm bdry}_-(c)} \| 	\llbracket \bar{H} (\varphi_+)^{i-k}
    \rrbracket S (\varphi_-, \varphi_+)^{4-i}(\tau,\cdot)\|_{H^{1-\kappa}_z} 
    \\
    &\qquad \times \int_{M_-^{\rm bdry}(c)} \frac{d(\tau,\partial M_-)^{1-\kappa}}{d(\tau,\partial M_-)^{1-\kappa}} \| \llbracket \overline H(\varphi_-)^k\rrbracket \|_{H^{-1+\kappa}_z} d\tau 
    \\
    &\leq C {\rm max}( \|\varphi_-\|^{4-k}_{\mathcal C^{-1/2-2\kappa}_z}, \|\varphi_+\|^{4-k}_{\mathcal C^{-1/2-2\kappa}_z}) 
    \\&\qquad \times \sup_{\tau \in M^{\rm bdry}_-(c)}|d(\tau,\partial M_-)^{1-\kappa} \| \llbracket \overline H(\varphi_-)^k\rrbracket \|_{H^{-1+\kappa}_z}|. 
    \end{align}
    The desired estimate follows.
 \end{proof}

\subsection{Estimate of $R_4$}

We write $R_4 = R_4(a) + R_4(b)$, where 
\begin{align}
 R_4(a) &=  - \int_M (\gamma_T - \gamma_T^M) \Big(2W_T Z_T + \llbracket H(\varphi_-,\varphi_+)^2_T\rrbracket + 2 H(\varphi_-,\varphi_+)_T Z_T + Z_T^2 \Big) \, dx,\\
  R_4(b) &= - \delta_T(\varphi_-,\varphi_+) + \Delta\delta^0_T + \delta^M_T. 
\end{align}

In order to estimate $R_4(a)$, we begin with a preliminary estimate on the function $\gamma_T - \gamma_T^M$. Recall that $C_T^{M^{\rm per}}$ is the covariance of the massive GFF on $M$ with periodic boundary conditions and that
\begin{equation}
\gamma_T^M(x) = -\frac{2\cdot 12^2}3 \int  C_T^M(x,y)^3 dy,\qquad  
\gamma_T(x) = -\frac{2\cdot 12^2}{3} \int C_T^{M^{\rm per}}(x,y)^3 dy.
\end{equation}
\begin{lemma} \label{lemma: gammat diff}
For every $p \in [1,\infty)$, there exists $C=C(p) > 0$ such that, for every $T > 0$,
  \begin{align}
   \| \gamma_T  - \gamma_T^M \|_{L^p} \leq C.
   \end{align}
\end{lemma}

\begin{proof}
It is sufficient to show that for every $\kappa > 0$, there exists $C=C(\kappa)>0$ such that for every $x \in M$,
\begin{equation}
|\gamma_T(x) - \gamma_T^M(x)| \leq C d(x, \partial M)^{-\kappa}.
\end{equation}
  Recall that from the Markov property (see Lemma \ref{lemma: covariance DMP decomp}), $C_T^{M^{\rm per}}(x,y) = C_T^{B}(x,y) + C_T^M(x,y)$. Therefore, we have
  \begin{equation}
    \gamma_T 
    = -\frac{2\cdot 12^2}3\sum_{1 \leq i \leq 3} \binom{3}{i} \int C_T^B (x, y)^i C_T^M (x,
    y)^{3 - i} \mathd y + \gamma_T^M (x).
  \end{equation}
  We now use the kernel estimates on $C_T^M$ and $C_T^B$ contained in Appendices \ref{sec: Dirichlet covariance estimates} and \ref{appendix: harmonic extension} to obtain the following. For every $\kappa > 0$, there exists $C=C(\kappa)>0$ such that for every $i \in \{1,2,3\}$,
  \begin{equation}
  	\left| \int C_T^B (x, y)^i C_T^M (x, y)^{3 - i} dy \right|\leq C\left| \int \frac{1}{| x - y |^{3 - i}} \frac{d (x,
    \partial M)^{- \kappa}}{| x - y |^{i - \kappa}} d y \right| \leq C' d(x,\partial M)^{-\kappa}.
  \end{equation}
  
  The desired estimate follows. 
\end{proof}

We now estimate $R_4(a)$.

\begin{lemma}
 For every $\delta >0$, there exists $C>0$ such that 
  \begin{align}
    \mathbb{E} | R_4(a) | &\leq  C \left( 1 +  \| H (\varphi_-, \varphi_+) \|^2_{L_{\tau}^{\infty}
    \mathcal{C}_z^{- 1 / 2 - \kappa}} + \| \llbracket (H (\varphi_-,
    \varphi_+))^2 \rrbracket \|^2_{L^2_\tau \mathcal{C}_z^{- 1 -
    \kappa}} \right) 
    \\ &\qquad \qquad  + \delta \left( \| K_T \|^2_{H^{1 - \kappa}} + \| Z_T \|^4_{L^4} \right).
  \end{align}
\end{lemma}

\begin{proof}
By the intermediate ansatz \eqref{eq: intermediate ansatz},
\begin{equation} \label{eq: R4a}
 R_4(a) \equiv   - \int (\gamma_T - \gamma_T^M) \Big( 2W_T K_T + \llbracket H(\varphi_-,\varphi_+)^2\rrbracket + 2 H(\varphi_-,\varphi_+) K_T +Z_T^2  \Big).
\end{equation}
Note that the last term in \eqref{eq: R4a} can be bounded by the Cauchy-Schwarz and Young inequalities: for every $\delta > 0$, there exists $C>0$ such that
\begin{equation}
\left| \int (\gamma_T - \gamma_T^M) Z_T^2 \right| \leq \| \gamma_T - \gamma_T^M\|_{L^2} \| Z_T \|_{L^4}^2 \leq C\|\gamma_T-\gamma_T^M\|_{L^2}^2 + \delta \|Z_T\|_{L^2}^4.
\end{equation}
The required estimate follows from Lemma \ref{lemma: gammat diff}.

For the first and third terms in \eqref{eq: R4a}, we will split $\tau$-slices. Note that by translation invariance in the axial (z) direction, we have that  \[(\gamma_T - \gamma^M_T) (x) = (\gamma_T - \gamma^M_T) (\tau).\]
Therefore, for every $\delta > 0$, there exists $C'>0$ such that
  \begin{align}
    &\Big| \int (\gamma_T - \gamma^M_T) (W_T + H (\varphi_-, \varphi_+))
    K_T d x \Big| 
    \\
    &\leq 
    C  (\| W_T 
    \|_{L_{\tau}^{\infty} \mathcal{C}_z^{- 1 / 2 - \kappa}} + \| H (\varphi_-,
    \varphi_+) \|_{L_{\tau}^{\infty} \mathcal{C}_z^{- 1 / 2 -
    \kappa}})  \int |(\gamma_T - \gamma^M_T) (\tau)| \| K_T(\tau,\cdot) \|_{H_z^{1 - \kappa}} d\tau
    \\
    &\leq 
    C' \left(  \| (\gamma_T - \gamma^M_T) \|^4_{L_{\tau}^2} + \| W_T\|^4_{L_{\tau}^{\infty} \mathcal{C}_z^{- 1 / 2 - \kappa}} + \| H
    (\varphi_-, \varphi_+) \|^4_{L_{\tau}^{\infty} \mathcal{C}_z^{-
    1 / 2 - \kappa}} \right) + \delta \| K_T \|^2_{L_{\tau}^2 H_z^{1 - \kappa}}. 
  \end{align}
  We then use \eqref{eq: HtauHzHx} to conclude.
      
 For the second term in \eqref{eq: R4a}, by splitting the integral, using duality and uniformly bounding the integral, we obtain
   \begin{equation}
    \left| \int (\gamma_T - \gamma^M_T) \llbracket (H (\varphi_-,
    \varphi_+))^2 \rrbracket \mathd x \right|
   \leq C \| \gamma_T - \gamma^M_T \|^2_{L_{\tau}^2} + \delta \| \llbracket (H
    (\varphi_-, \varphi_+))^2  \rrbracket \|^2_{L_{\tau}^2
    \mathcal{C}_z^{- 1 - \kappa}}.
  \end{equation}
  \end{proof}
  
We now turn to the energy renormalization terms $R_4(b)$. The term $\Delta\delta^0_T$ is bounded by the following lemma.
\begin{lemma}
We have that
\begin{equation}
\delta^0_T = -2 \cdot 4! \int \int C_T^B(x,y)^4 dx dy + O(1).
\end{equation}	
\end{lemma}

\begin{proof}
Note that we may replace the integral over $[-1,1]\times \mathbb T^2$ by an integral over $M$ up to an $O(1)$ term by smoothing properties of the harmonic extension $\bar H$. Let us compute
\begin{align}
E &= \left( \int \llbracket \bar H\varphi_T(x)^4 \rrbracket dx \right)^2 - \left( \int \llbracket  H\varphi_T(x)^4 \rrbracket dx \right)^2
\\
&= \left( \int \llbracket \overline H \varphi_T(x)^4 \rrbracket - \llbracket H \varphi_T(x)^4 \rrbracket  \right) \left( \int \llbracket \overline H \varphi_T(x)^4 \rrbracket + \llbracket H \varphi_T(x)^4 \rrbracket  \right) =: E_1 \times E_2. 
\end{align}
Writing $H \varphi_T = \bar H \varphi_T + S \varphi_T $, by the Binomial theorem for Wick powers we have
\begin{equation}
E_1 = \sum_{0 \leq i \leq 3}\int \binom 4i \llbracket \overline H \varphi_T(x)^{i} \rrbracket (S\varphi_T)^{4-i}. 
\end{equation}
By Wick's theorem,
\begin{align}
\sum_{0 \leq i \leq 3} &\binom 4i \int \int \mathbb E \left[ \llbracket \overline H \varphi_T(x)^{i} \rrbracket (S\varphi_T)^{4-i}  \llbracket \overline H \varphi_T(x')^4 \rrbracket \right] + \mathbb E \left[ \llbracket \overline H \varphi_T(x)^{i} \rrbracket (S\varphi_T)^{4-i}   \llbracket H \varphi_T(x')^4 \rrbracket\right]	
\\
&=
\sum_{0 \leq i \leq 3} 4! \binom 4i \int \int \overline C_T^B(x,x')^i \bar K^S_T(x,x')^{4-i} + K_T^{H,\bar H}(x,x')^i K_T^S(x,x')^{4-i},
\end{align}
where
\begin{align}
\bar K^S_T(x,x') &= \mathbb E[ S\varphi_T(x) \bar H\varphi_T(x')],
\\
K^S_T(x,x') &= \mathbb E[ S\varphi_T(x) H\varphi_T(x')],
\\
K_T^{H,\overline H}(x,x') &= \mathbb E[H \varphi_T(x) \overline H \varphi_T(x')].
\end{align}
We claim that there exists $C>0$ such that for every $T > 0$,
\begin{equation} \label{eq: kernel KS smooth claim}
\sup_{x,x'} {\rm max}\left( K_T^S(x,x'), \bar K_T^S(x,x') \right) \leq C. 
\end{equation}
The lemma then follows since the $(C_T^B)^i$ and $(K_T^{H, \bar H})^i$ are locally integrable for $i =1,2,3$.  See the covariance estimates in Appendix \ref{appendix: harmonic extension}.

We now prove the claim. Without loss of generality, let us consider $K_T^S$. Writing in terms of the convolution kernels, we have that
\begin{equation}
K_T^S(x,x') = \int_{\mathbb T^2} \int_{\mathbb T^2} S(x,y) P(x',y') C_T^0(y,y') dy dy'.
\end{equation}
The term $S(x,y)$ is the kernel  of $S$ and is smooth. The Poisson kernel is of order $P(x',y') \asymp \frac{1}{|x'-y'|}$. The kernel of $(-\Delta_{\mathbb T^2}+m^2)^{-1}$ is of order $C_T^0(y,y') \asymp \frac{1}{|y-y'|}$. By these considerations together with translation invariance,
\begin{equation}
|K_T^S(x,x')| \leq C \|S\|_\infty \int_{\mathbb T^2} \frac{1}{|x'-y'|} \int_{\mathbb T^2}  \frac{1}{|y-y'|} dy dy' \leq C'. 
\end{equation}
\end{proof}

It remains to estimate $\delta_T(\varphi_-,\varphi_+)-\delta_T^M$. Recall that 
\begin{equation}\label{eqdef: delta M}
\delta_T^M := \sum_{i=1}^5 \delta_T^{i,M},	
\end{equation}
where
\begin{align}\label{eqdef: delta M 1}
\delta_T^{1,M}&:=- 12 \int \int C_T^{M} (x, y)^4 \mathd x \mathd y,
\\ \label{eqdef: delta M 2}
\delta_T^{2,M}&:= 6 \cdot 4^2 \cdot 3^2 \cdot 2 \int \int \int
    C_T^{M} (x, y)^2 C_T^{M} (y, y')^2 C_T^{M}
    (y', x)^2 \mathd x \mathd y \mathd y',
\\ \label{eqdef: delta M 3}
\delta_T^{3,M}&:=\frac{12^2}{3} \int \int C_T^M (x, y)^3 \mathbb E_{\mu^0\otimes \mu^0} \left[ H
     (\varphi_-, \varphi_+) (x) H (\varphi_-, \varphi_+) (y) \right] \mathd x \mathd y,
\\ \label{eqdef: delta M 4}
\delta^{4,M}_T&:= \frac{12^2}2 \mathbb E\otimes \mathbb E_{\mu^0\otimes \mu^0}\left[ \int_0^T \int J_t(\mathbb W_t \,\llbracket  H(\varphi_-,\varphi_+)^2 \rrbracket )^2 dx dt \right],
\\ \label{eqdef: delta M 5}
\delta^{5,M}_T&:= \frac{4^2}2   \mathbb E_{\mu^0\otimes \mu^0}\left[ \int_0^T \int J_t(\llbracket  H(\varphi_-,\varphi_+)^3\rrbracket )^2 dx dt \right]. 
\end{align}

\begin{lemma}
There exists $C>0$ such that, for every $T >0$,
\begin{equation}
|\delta_T(\varphi_-,\varphi_+)-\delta_T^M| \leq C. 
\end{equation}	
\end{lemma}

\begin{proof}
We may rewrite $\delta^{3,M}_T,\delta^{4,M}_T$, and $\delta^{5,M}_T$ in terms of kernels as follows:
\begin{align} \label{eqdef: delta M 6}
\delta^{3,M}_T &=  \frac{12^2}{3} \int \int C_T^M (x, y)^3  C_T^B(x,y)  \mathd x \mathd y,
\\ \label{eqdef: delta M 7}
\delta^{4,M}_T &= \frac{12^2}2 \int \int C_T^M(x,y)^2 C_T^B(x,y)^2 dx dy,
\\ \label{eqdef: delta M end}
\delta^{5,M}_T &= 3! \cdot \frac{4^2}2 \int \int C_T^M(x,y) C_T^B(x,y)^3 dx dy. 
\end{align}
By Remark \ref{remark: energy renormalization recombination} (or direct inspection), it thus follows that
\begin{align}
|\delta_T&(\varphi_-,\varphi_+)-\delta_T^M|\\&\leq C \int \int \int \left| C_T^{M^{\rm per}} (x, y)^2 C_T^{M^{\rm per}} (y, y')^2
     C_T^{M^{\rm per}} (y', x)^2-  C_T^M (x,
     y)^2 C_T^M (y, y')^2 C_T^M (y', x)^2 \right|.
\end{align}
It suffices to show that
\begin{equation}
\int \int \int \left| C_T^{M^{\rm per}} (x, y)^2 C_T^{M^{\rm per}} (y, y')^2
     C_T^{M^{\rm per}} (y', x)^2-  C_T^M (x,
     y)^2 C_T^M (y, y')^2 C_T^M (y', x)^2 \right| \leq C.	
\end{equation}
Since $C^{M^{\rm per}}_{T}(x,y)=C_{T}^{M}(x,y)+C_{T}^{B}(x,y)$, it is sufficient to show that there exists $C>0$ such that for every $1 \leq i \leq 2$ and $0 \leq j,k \leq 2$, and every $T>0$,
  \begin{equation} 
    \int \int \int \left| {C^M_T}  (x, y)^{2 - i} {C^B_T}  (x,
    y)^i C_T^M \left( y, y' \right)^{2 - j} C_T^B \left( y, y' \right)^j 
    C_T^M (x, y')^{2 - k} C_T^M (x, y')^k \right|
    \leq C.
  \end{equation}
In order to establish this, recall the pointwise covariance estimates of Appendices \ref{sec: Dirichlet covariance estimates} and \ref{appendix: harmonic extension}:
\begin{equation}
 |C_T^M(x,y)|, |C_T^B(x,y)| \leq \frac{C}{|x-y|}.
\end{equation}
Furthermore, by the covariance interpolation estimates (see Appendix \ref{appendix: harmonic extension}), for every $0 < \alpha < 1$ we have 
\begin{equation}
|C^B_T(x, y)| \leq 
  C\frac{1}{| x - y |^{1 - \alpha}} \frac{1}{d (x, \partial M)^{\alpha}}.	
\end{equation}
  Thus we get for any $1 \leq i \leq 2$ and $ 0 \leq j, k \leq 2$,
  \begin{align}
   \int \int \int &\left| {C^M_T}  (x, y)^{2 - i} {C^B_T}  (x,
    y)^i C_T^M \left( y, y' \right)^{2 - j} C_T^B \left( y, y' \right)^j 
    C_T^M (x, y')^{2 - k} C_T^M (x, y')^k \right| \mathrm{d}x\mathrm{d}y\mathrm{d}y'
    \\
    &\leq C  \int \int \int \frac{1}{| x - y |^{2 - \alpha}}
    \frac{1}{d (x, \partial M)^{\alpha}} \frac{1}{| y - y' |^2} \frac{1}{| x - y' |^2}\mathrm{d}x\mathrm{d}y\mathrm{d}y'
    \\
   &\leq C \int \int \frac{1}{| x - y |^{2 - \alpha}} \frac{1}{d
    (x, \partial M)^{\alpha}} \frac{1}{| x - y |} \mathrm{d}x\mathrm{d}y \leq \int \frac{1}{d (x, \partial M)^{\alpha}}\mathrm{d}x \leq C.
  \end{align} 
\end{proof}

\subsection{Estimate of $R_8$}

Recall that
\begin{equation}
R_8= \int^T_0 \int (J_t (\mathbb{W}^2_t \succ Z^{\flat}_t)) (J_t(
    \mathbb{W}^2_t \, H (\varphi_-, \varphi_+)_T)) \mathd x \mathd t - \int 2 \gamma^M_T H (\varphi_-, \varphi_+)_T Z_T \mathd x.
\end{equation}
We will prove the following estimate on $R_8$. 
\begin{lemma} \label{lem: r8 bound}
For every $\delta > 0$ sufficiently small, there exists $C_\delta > 0$ such that
  \begin{align}
    \mathbb E[|R_8|]\leq 
    C_\delta(1+\|H(\varphi_-,\varphi_+)\|_{B^{s}_{p,p}}^{r}) +  \delta \mathbb E [\mathcal{C}_T (u)],
  \end{align}
  where $p \in [1,2)$ is sufficiently close to $2$, $s$ is sufficiently small, and  is as below and $r \in [1,\infty)$ is sufficiently large.\end{lemma}

As a first step, we prove a regularity estimate on $H(\varphi_-,\varphi_+)$ measured in a Besov norm of strictly positive regularity $s>0$ taken sufficiently small.  The difficulty comes from the discontinuity at the image of $\partial M_-$ and $\partial M_+$ under the identification.

\begin{lemma}\label{lem: hext pos reg}
Let $\varphi_-,\varphi_+ \in \mathcal C^{-1/2-\kappa}_z$. Then for every $p \in [1,2)$, there exists $s>0$ sufficiently small and $C>0$ such that
\begin{equation}
	\|H(\varphi_-,\varphi_+)\|_{B^s_{p,p}} \leq C( 1 + \|\varphi_-\|_{\mathcal C^{-1/2-\kappa}_z} +\|\varphi_+\|_{\mathcal C^{-1/2-\kappa}_z}).
\end{equation}	
\end{lemma}

\begin{proof}
Note that $B^s_{p,p,\tau}B^s_{p,p,z} \hookrightarrow B^s_{p,p}$. Thus, by this and standard embeddings, it is sufficient to establish the above bound for $W^s_{p,\tau}B^s_{p,p,z}$. Recall that
\begin{equation}
\| H(\varphi_-,\varphi_+)\|_{W^s_{p,\tau}B^s_{p,p,z}}^p =\int_{L\mathbb T} \int_{L\mathbb T} \frac{\|H(\varphi_-,\varphi_+)(\tau,\cdot)-H(\varphi_-,\varphi_+)(\tau',\cdot)\|^p_{B^s_{p,p,z}}}{|\tau-\tau'|^{d+sp}} d\tau d\tau' =: I.
\end{equation}
We may assume that we shift by $+L$ so that $ H(\varphi_-,\varphi_+)$ is continuous, in fact smooth, as a function of $\tau$ except possibly at $\tau=0$. It is therefore sufficient to estimate this integral on the set $\Sigma:=(-\varepsilon,\varepsilon)^2$. 

We first estimate the integral $I$ on the set $\Sigma_1:= \{ \tau < \tau' < 0 : \tau,\tau' \in \Sigma \}$. Let $\delta>0$ be sufficiently small. Then
\begin{align}
I_1&:= \int_{\Sigma_1}  \frac{\|H(\varphi_-,\varphi_+)(\tau,\cdot)-H(\varphi_-,\varphi_+)(\tau',\cdot)\|^p_{B^s_{p,p,z}}}{|\tau-\tau'|^{1+sp}} d\tau d\tau'
\\
&\leq \int_{\Sigma_1}  \frac{\|H(\varphi_-,\varphi_+)(\tau,\cdot)-H(\varphi_-,\varphi_+)(\tau',\cdot)\|^{p(1-\delta)}_{B^s_{p,p,z}}}{|\tau-\tau'|^{1+sp}} 
\\
&\qquad \qquad \qquad \qquad \times \left(\int_\tau^{\tau'} \| \partial_{\tau''}H(\varphi_-,\varphi_+)(\tau'',\cdot)\|_{B^s_{p,p,z}}d\tau''\right)^{p\delta} d\tau d\tau'.
\end{align}
By properties of the harmonic extension (see Appendix \ref{appendix: harmonic extension}), we have that
\begin{equation}
	\|H(\varphi_-,\varphi_+)(\tau,\cdot)\|_{B^s_{p,p,z}} \leq c_1 \tau^{-1/2-\kappa-s} \|H(\varphi_-,\varphi_+)\|_{L^\infty_\tau B^{-1/2-\kappa}_{p,p,z}}.
\end{equation}
Thus, we may bound
\begin{align}
	I_1 &\leq c_2 \|H(\varphi_-,\varphi_+)\|^{p(1-\delta)}_{L^\infty_\tau B^{-1/2-\kappa}_{p,p,z}} 
	\\ &\qquad \qquad \times  \int_{\Sigma_1} \frac{\tau^{-(1/2+\kappa+s)p(1-\delta)}}{|\tau-\tau'|^{1+sp}} \left(\int_\tau^{\tau'} \| \partial_{\tau''}H(\varphi_-,\varphi_+)(\tau'',\cdot)\|_{B^s_{p,p,z}}d\tau''\right)^{p\delta} d\tau d\tau'.
\end{align}
The condition $p<2$ is necessary here to ensure that the singularity at $\tau=0$ is integrable. 

Let us now estimate the derivative term. Recall that $H(\varphi_-,\varphi_+) = \overline H(\varphi_-) + \overline H(\varphi_+) + S(\varphi_-,\varphi_+)$. Upon identication of $-L$ with $L$, and the shift of $+L$, in the regime $\tau < 0$ we have that,
\begin{equation}
	\|\partial_\tau H(\varphi_-,\varphi_+)(\tau,\cdot)\|_{B^s_{p,p,z}} \leq \|\partial_\tau \overline H(\varphi_-)(\tau,\cdot)\|_{B^s_{p,p,z}} + \|\partial_\tau \overline H(\varphi_+)(\tau,\cdot)\|_{B^s_{p,p,z}}. 
\end{equation}
By the explicit expression for $\overline H$, we have that
\begin{equation}
	\partial_\tau \overline H(\varphi_+) = \mathcal N_0 \overline H(\varphi_+),
\end{equation}
where we recall $\mathcal N_0 = -\sqrt{-\Delta_{\mathbb T^2}+m^2}$. On the set $\tau'' \in (\tau,\tau')$, we may estimate
\begin{equation}
	\| \mathcal N_0 \overline H(\varphi_+)(\tau'',\cdot) \|_{B^s_{p,p,z}} \leq c_3 \|\mathcal N_0 \overline H(\varphi_+)\|_{L^\infty_\tau B^{-3/2-\kappa}_{p,p,z}} \tau^{-3/2-\kappa-s}.
\end{equation}
Thus,
\begin{align}
&\left(\int_\tau^{\tau'} \| \partial_{\tau''}H(\varphi_-,\varphi_+)(\tau'',\cdot)\|_{B^s_{p,p,z}}d\tau''\right)^{p\delta} 
\\
&\leq c_3 (\|\mathcal N_0 \overline H(\varphi_+)\|_{L^\infty_\tau B^{-3/2-\kappa}_{p,p,z}}^{p\delta}+\|\mathcal N_0 \overline H(\varphi_+)\|_{L^\infty_\tau B^{-3/2-\kappa}_{p,p,z}}^{p\delta}) \tau^{-(3/2+\kappa+s)p\delta}|\tau-\tau'|^{p\delta}.
\end{align}
Thus the bound  on $I_1$ reduces to estimating 
\begin{equation}
\int_{\Sigma_1} \frac{\tau^{-(1/2+\kappa+s)p(1-\delta)} \tau^{-(3/2+\kappa+s)p\delta}}{|\tau-\tau'|^{1+(s-\delta)p}} d\tau d\tau.
\end{equation}
By first fixing $p<2$, then $\delta$ and $s$ sufficiently small such that $s<\delta$, we find that the above integrals are integrable near $\tau=0$ and thus we obtain the desired bound.

By symmetry considerations, to finish the lemma it remains to estimate the integral $I$ on the set $\Sigma_2:= \{\tau < 0 < \tau' : \tau,\tau' \in \Sigma \}$ --- we call this $I_2$. By similar estimates as above, we have that
\begin{align}
I_2 &:= \int_{\Sigma_2}  \frac{\|H(\varphi_-,\varphi_+)(\tau,\cdot)-H(\varphi_-,\varphi_+)(\tau',\cdot)\|^p_{B^s_{p,p,z}}}{|\tau-\tau'|^{1+sp}} d\tau d\tau'
\\
&\,\leq c_5 \|H(\varphi_-,\varphi_+)\|^{p}_{L^\infty_\tau B^{-1/2-\kappa}_{p,p,z}}  \int_{\Sigma_2} \frac{\tau^{-(1/2+\kappa+s)p}}{(\tau'-\tau)^{1+sp}} d\tau d\tau'.
\end{align}
Integrating first in $\tau'$ and then in $\tau$, we have that the integral on the righthand side above is finite provided $p<2$ and $s,\kappa>0$ is sufficiently small.

\end{proof}

We now turn to estimate on $R_8$.

\begin{proof}[Proof of \ref{lem: r8 bound}]
The first integral of $R_8$ is given by 
\begin{equation}
I_1:= \int^T_0 \int (J_t (\mathbb{W}^2_t \succ Z^{\flat}_t)) (J_t(
    \mathbb{W}^2_t \, H (\varphi_-, \varphi_+)_T)) \mathd x \mathd t 	
\end{equation}
We will begin by putting this in a more tractable form, by moving the $J_t$ on to $\mathbb W_t^2$. Let us define $I_1'$ by the decomposition
\begin{equation} \label{eq: r8 i1}
I_1= \int_0^T \int  (J_t \mathbb{W}^2_t \succ
     Z^{\flat}_t ) (J_t
    (\mathbb{W}^2_t \, H (\varphi_-, \varphi_+)_T)) dx dt +  I_1',
\end{equation}
We first estimate $I_1'$. By the commmutator lemma applied to $J_t$ and regularizing properties of $J_t$, both found in Appendix \ref{sec: Dirichlet covariance estimates}, we have that
 \begin{equation} \| J_t (\mathbb{W}^2_t \succ Z^{\flat}_t) - J_t \mathbb{W}^2_t \succ
     Z^{\flat}_t \|_{H^{1 / 4 - 3 \kappa}} \leq  \frac{C}{\langle t\rangle ^{1 / 2 +
     \kappa}} \| \mathbb{W}^2_t \|_{\mathcal{C}^{- 1 - \kappa}} \| Z^{\flat}_t
     \|_{H^{1 / 4 - \kappa}} . 
 \end{equation}
 Hence
\begin{equation}
|I_1'| \leq C \int_0^T \frac{1}{\langle t\rangle ^{1 / 2 +
     \kappa}} \| \mathbb{W}^2_t \|_{\mathcal{C}^{- 1 - \kappa}} \| Z^{\flat}_t
     \|_{H^{1 / 4 - \kappa}} \|J_t(
    \mathbb{W}^2_t \, H (\varphi_-, \varphi_+)_T)\|_{H^{-1/4+3\kappa}}  dt.
\end{equation}
In order to estimate the integral on the righthand side, observe that
\begin{multline}
\| J_t (\mathbb{W}^2_t \, H (\varphi_-, \varphi_+)_T)\|_{H^{-1/4+3\kappa}} \\ \leq \frac{C}{\langle t\rangle ^{7/4-3\kappa}} \|\mathbb{W}^2_t \, H (\varphi_-, \varphi_+)_T \|_{L^2} \leq \frac{C}{\langle t\rangle ^{7/4-3\kappa}}\|\mathbb W_t^2 \|_{L^\infty} \|H(\varphi_-,\varphi_+)_T\|_{B^s_{p,p}}.
\end{multline}
Above, the first inequality follows by the regularizing properties of $J_t$, and the second inequality follows by H\"older's inequality and Sobolev embedding (where $s>0$ is sufficiently small so that $p<2$ is sufficiently close to 2). 
Adding this back into the estimate of $I_1'$ and using properties of the $\flat$, we obtain that
\begin{equation}
|I_1'| \leq C  \| \langle t \rangle^{-9/4+2\kappa} \|\mathbb W_t^2 \|_{L^\infty_x} \|\mathbb W_t^2\|_{\mathcal C^{-1-\kappa}_x} \|_{L^1_t} \cdot \|Z_T\|_{H^{1/4-\kappa}} \|H(\varphi_-,\varphi_+)_T \|_{B^s_{p,p}},
\end{equation}
which can then be estimated by Young's inequality.

We now estimate the first term on the rightahnd side of \eqref{eq: r8 i1}. By a paraproduct decomposition,
\begin{multline}
\int_0^T \int  (J_t \mathbb{W}^2_t \succ
     Z^{\flat}_t ) (J_t
    \mathbb{W}^2_t \, H (\varphi_-, \varphi_+)_T) dx dt\\= \int_0^T \int (J_t \mathbb{W}^2_t \succ
     Z^{\flat}_t ) ( J_t \mathbb{W}^2_t \succ H (\varphi_-, \varphi_+)_T  )dx dt + I_1''.	
\end{multline}
By using commutator estimates as above, taking $p<2$ (for Lemma 
\ref{lem: hext pos reg} to apply), we have 
  \begin{multline}
   \| J_t (\mathbb{W}^2_t \succ H (\varphi_-, \varphi_+)_T)  -
    J_t \mathbb{W}^2_t \succ H (\varphi_-, \varphi_+)_T  \|_{B_{p,
    p}^{\kappa}} 
    \\ \leq  \frac{C}{\langle t\rangle ^{1 / 2 + \kappa}} \| \mathbb{W}^2_t \|_{\mathcal{C}^{- 1 -
    \kappa}} \| H (\varphi_-, \varphi_+)_T \|_{B_{p, p}^{3\kappa}}. 
  \end{multline}
  On the other hand, by para and resonant product estimates,
  \begin{align}
    \| J_t (\mathbb{W}^2_t \preccurlyeq H (\varphi_-, \varphi_+)_T) \|_{B_{p,
    p}^{\kappa}} \leq C \frac{1}{\langle t\rangle ^{3 / 2 + \kappa}} \| \mathbb{W}^2_t
    \|_{L^\infty} {\| H
    (\varphi_-, \varphi_+)_T \|_{B_{p, p}^{2\kappa}}}.
  \end{align}
Writing $q$ as the H\"older conjugate of $p$,
\begin{align}
|I_1''| &:= \left|\int_0^T \int  (J_t \mathbb{W}^2_t \succ
     Z^{\flat}_t ) (J_t(
    \mathbb{W}^2_t \, H (\varphi_-, \varphi_+)_T) - J_t (\mathbb{W}^2_t \succ H (\varphi_-, \varphi_+)_T) ) dx dt \right|	
    \\
    &\leq C \int_0^T   \frac{1}{\langle t \rangle ^{1 / 2 + \kappa}} \|J_t \mathbb{W}^2_t \succ
     Z^{\flat}_t \|_{B^{-2\kappa}_{q,q}}(\| \mathbb{W}^2_t
    \|_{\mathcal{C}^{- 1 - \kappa}} + \langle t\rangle^{-1}\|\mathbb W_t^2\|_{L^\infty}) {\| H
    (\varphi_-, \varphi_+)_T \|_{B_{p, p}^{3\kappa}}}.
\end{align}
Provided $p<2$ is sufficiently close to 2, then $q$ can be taken sufficiently close to $2$ as well (we need $q<4$). By properties of $\flat$ and a paraproduct estimate, we obtain that 
\begin{align}
|I_1''|&\leq C' \|Z_T\|_{L^4}\|H(\varphi_-,\varphi_+)_T\|_{B^{3\kappa}_{p,p}} 
\\
&\qquad \qquad \times (\| \langle t\rangle^{-3/2-\kappa} \| \mathbb W_t^2\|_{\mathcal C^{-1-\kappa}_x}^2 \|_{L^1_t} + \| \langle t \rangle^{-5/2-\kappa} \|\mathbb W_t^2 \|_{L^\infty_x} \|\mathbb W_t^2\|_{\mathcal C^{-1-\kappa}_x} \|_{L^1_t}).
\end{align}
This can then be estimated using Young's inequality.

To summarize, we have written
\begin{equation}
I_1 = \int_0^T \int (J_t \mathbb{W}^2_t \succ
     Z^{\flat}_t ) ( J_t \mathbb{W}^2_t \succ H (\varphi_-, \varphi_+)_T )dx dt +I_1'+I_1'',	
\end{equation}
and have obtained estimates on $I_1'$ and $I_1''$. We now turn to estimating the final form of the integrand together with the remaining term in $R_8$, which is the renormalization counter term. To this end, we define
 \begin{equation}
 I_2:= \int_0^T \int (J_t \mathbb{W}^2_t \succ
     Z^{\flat}_t ) ( J_t \mathbb{W}^2_t \succ H (\varphi_-, \varphi_+)_T  )dx-\int 2\gamma_T^M H(\varphi_-,\varphi_+)_T Z_T dx.
 \end{equation}
 By It\^o's formula on the product $\gamma_T^M Z_T$ and reorganizing terms, we may write
 \begin{equation}
 	I_2 = I_3+I_2',
 \end{equation}
 where
 \begin{equation}
 I_3:= 	\int^T_0 \int (J_t \mathbb{W}^2_t \succ Z^{\flat}_t)
    (J_t \mathbb{W}^2_t \succ H (\varphi_-, \varphi_+)_T)
    -(J_t \mathbb{W}^2_t \circ J_t \mathbb{W}^2_t)
    Z^{\flat}_t H (\varphi_-, \varphi_+)_T dx dt,
 \end{equation}
 and where
 \begin{align}
	I_2' &:= I_2'(a) + I_2'(b),
	\\
	I_2'(a) &:=  \int^T_0 \int (J_t \mathbb{W}^2_t \circ J_t \mathbb{W}^2_t
    + 2\dot{\gamma}^M_t) Z^{\flat}_t H (\varphi_-, \varphi_+)_T dx dt,
    \\
    I_2'(b)&:= \int^T_0 \int 2 \gamma_t^M \dot{Z}^{\flat}_t H (\varphi_-,
    \varphi_+)_T dxdt + \int 2\gamma_T^M (Z_T - Z^{\flat}_T)
    H (\varphi_-, \varphi_+)_T dx. 
\end{align}

We start by estimating $I_2'(a)$. Recall that
  \begin{equation}
  	\mathbb W_t^{2 \diamond 2}:= J_t \mathbb W_t^2 \circ J_t \mathbb W_t^2 + 2\dot\gamma_t^M.
  \end{equation} 
  Thus, by duality and fractional Leibniz, followed by Young's inequality,
  \begin{equation}
   |I_2'(a)|\leq  C\|\mathbb W_t^{2\diamond 2} \|_{L_t^1 \mathcal{C}^{- \kappa}} \| Z^{\flat}_t
    \|_{L_t^{\infty} B_{5 / 2, 5 / 2}^{\kappa}} \| H (\varphi_-, \varphi_+)_T
    \|_{B_{5 / 3, 5 / 3}^{\kappa}},
\end{equation}
from which the desired estimate follows.

We now estimate $I_2'(b)$. Note that $\gamma_T^M(\tau,\cdot)=\gamma_T^M(\tau)$ is constant for every $\tau$. By splitting into $\tau$-slices, we have that
\begin{align}
    \Bigg| \int &\gamma_T^M (Z_T - Z^{\flat}_T) H (\varphi_-, \varphi_+)_T dx \Bigg| \\
    &\leq C \int \frac{|\gamma_T^M(\tau)|}{T^{1/4-\kappa}} \|Z_T(\tau,\cdot)-Z_T^\flat(\tau,\cdot)\|_{H^{1/2-\kappa}_z} \|H(\varphi_-,\varphi_+)(\tau,\cdot)_T\|_{H^{-1/4}_z} d\tau 
   \\    &\leq  T^{- 1 / 4 + \kappa} \| \gamma_T^M \|_{L^\infty} (\| Z_T \|^2_{L^2_\tau H^{1 / 2 -
    \kappa}_z} + \| H (\varphi_-, \varphi_+)_T \|^2_{L_{\tau}^2 H^{- 1 / 4}_z}),  
    \end{align}
    where the first line is by duality and Bernstein's inequality (using spectral support considerations of $Z_T-Z_T^\flat$ in the $z$-variable), and the second line is by the Cauchy-Schwarz and Young inequalities. Using that $\|\gamma_T^M\|_{L^\infty_x} = \|\gamma_T^M\|_{L^\infty_\tau}= O(\log T)$, the desired estimate on this term follows.
    
  For the other integral in $I_2'(b)$, we will estimate it up to a mean-zero term. By Wick's theorem it is sufficient to estimate the integral with $\dot Z_t^\flat$ replaced by $\dot K_t^\flat$. Recall the $\flat$ interpolation estimate from \cite{BG20}: for every sufficiently regular function, 
  \begin{equation}
  \| \dot f_t^\flat \|_{B^s_{p,p}} \leq \frac{C}{\langle t \rangle^{1+s-s'}} \|f_T\|_{B^{s'}_{p,p}}.
  \end{equation}
  Then, by splitting into $(\tau,z)$ variables and using duality,
  \begin{multline}
  	\int_0^T \|\gamma_t^M\|_{L^\infty} \|\dot K_t^\flat \|_{H^\kappa} \|H(\varphi_-,\varphi_+)_T\|_{L^2_\tau H^{-\kappa}_z} dt \\\leq C' \int_0^T \|\gamma_t^M\|_{L^\infty} \frac{1}{\langle t\rangle ^{2}}\|K_T\|_{H^{1-\kappa}} \|H(\varphi_-,\varphi_+)_T\|_{L^2_\tau H^{-\kappa}_z} dt,
  \end{multline}
  which can then be estimated by using that $\|\gamma_t^M\|_{L^\infty} = O(\log t)$, Young's inequality, and the estimates on $K_T$ from Lemma \ref{lem: KT estimate}.

Finally, we estimate $I_3$. We will slightly adapt the quadrilinear estimate in \cite[Proposition 11]{BG20}. First, observe that for every $f,g,h$,
  \begin{align}
  \int_0^T \int (f_t \succ g_t)(f_t \succ h_t) &= \int_0^T \int (f_t \circ (f_t \succ h_t)) g) dx dt + {\rm com}_1 \\ &= \int_0^T \int (f_t\circ f_t)g_th_t + {\rm com}_2 + {\rm com}_1,
  \end{align}
  where 
  \begin{align}
  	{\rm com}_1 &:= \int_0^T \int (f_t \succ g_t) (f_t \succ h_t) - (f_t \circ (f_t \succ h_t))g_t dx dt,
  	\\
  	{\rm com}_2 &:= \int_0^T \int \Big( (f_t \circ (f_t \succ h_t)) - (f_t\circ f_t)h_t \Big) g_t dx dt.
  \end{align}
  We will set $f_t=J_t \mathbb W_t^2$, $g_t=Z_t^\flat$ and $h_t=H(\varphi_-,\varphi_+)_T$.

  Let us first estimate ${\rm com}_1$. Let $p^{-1}+q^{-1}=1$. Then by the trilinear estimate \cite[Proposition 10]{BG20} and paraproduct estimate,
  \begin{align}
  |{\rm com}_1| &\leq C \int_0^T \|J_t \mathbb W_t^2 \|_{\mathcal C^{-\kappa}}\| Z_t^\flat \|_{B^{2\kappa}_{q,q}} \|J_t\mathbb W_t^2 \succ H(\varphi_-,\varphi_+)_T \|_{B^{-\kappa}_{p,p}} dt
  \\
  &\leq C' \int_0^T \|J_t \mathbb W_t^2 \|_{\mathcal C^{-\kappa}}^2 \|Z_t^\flat\|_{B^{2\kappa}_{q,q}} \|H(\varphi_-,\varphi_+)\|_{L^p} dt.
  \end{align}
  Thus, taking $q>2$ sufficiently close to $2$ such that there exists $\kappa'>0$ such that $H^{1/2-\kappa'} \hookrightarrow B^{2\kappa}_{q,q}$, we have that (using properties of $\flat$) 
  \begin{equation}
  	|{\rm com}_1| \leq C'' \|Z_T\|_{H^{1/2-\kappa'}} \|H(\varphi_-,\varphi_+)_T\|_{L^p} \int_0^T \|J_t \mathbb W_t^2 \|_{\mathcal C^{-\kappa}}^2 dt, 
  \end{equation}
  which can then be put in the desired form by standard arguments using the regularizing properties of $J_t$ in Appendix \ref{sec: Dirichlet covariance estimates}. 
  
 We now turn to ${\rm com}_2$. For the second term, by the trilinear estimate \cite[Proposition 8]{BG20} and paraproduct estimates, 
 \begin{equation}
 |{\rm com}_2|\leq C \int_0^T \| J_t \mathbb W_t^2 \|_{\mathcal C^{-\kappa}}^2 \|H(\varphi_-,\varphi_+)_T\|_{B^{3\kappa}_{p,p}} \|Z_t^\flat\|_{L^q} dt. 	
 \end{equation}
 The desired inequality follows by similar arguments as above provided $p<2$ is chosen sufficiently close to $2$ such that $q\leq 4$.

\end{proof}

\subsection{Estimate of $R_9$}

Recall that
\begin{equation}
R_9 = 12 \int W_T \llbracket H (\varphi_-, \varphi_+)_T^2 \rrbracket G_T
   \mathd x + 4 \int \llbracket H (\varphi_-, \varphi_+)^3_T \rrbracket G_T
   \mathd x =: R_9(a) + R_9(b),
\end{equation}
where
\begin{equation}
G_T = \int^T_0 - J^2_t (\mathbb{W}^2_t \succ Z^{\flat}_t) - J^2_t \left(
  \mathbb{W}^2_t H \left( \varphi_- {, \varphi_+}  \right)_T \right) \mathd t.
\end{equation}

\begin{lemma} \label{lemma: R9}
Let $\epsilon > 0$ be arbitrarily small and $p = 4/3(1+\epsilon)$ and let $q = 4(1+\epsilon)/(3\epsilon)$. 
For every $\delta > 0$, there exists $C>0$ such that
  \begin{align}
    \mathbb E[| R_9 |] \leq &C \Big( \| d (x, \partial M)^{1 / 2 + 2\kappa} \llbracket H^3
    (\varphi_-, \varphi_+) \rrbracket \|^p_{L_{\tau}^{\infty}
    H_z^{- 1 + \kappa}} 
    \\
    &\quad + \| d (x,\partial M)^{1 / 2 - \kappa} \llbracket H^2 (\varphi_-,
    \varphi_+) \rrbracket \|^2_{L^{4/3}_\tau L^{\infty}_{z}} \Big)      \\
    &\quad +\mathbb E \left[C  \| \mathbb{W}^2_t \|^q_{L^\infty_t \mathcal C^{-1-\kappa}} + \delta \| Z_T \|^4_{L^4} \right]. 
  \end{align}
  \end{lemma}
\begin{proof}
Observe that, by Wick's theorem,
\begin{align}
R_9(a) &\equiv \tilde R_9(a) := - 12 \int W_T \llbracket H(\varphi_-, \varphi_+)_T^2 \rrbracket  \int_0^T J^2_t (\mathbb{W}^2_t \succ Z^\flat_t) \mathrm{d}t \mathrm{d}x,
\\
R_9(b) &\equiv \tilde R_9(b) := -4 \int \llbracket H (\varphi_-, \varphi_+)^3_T \rrbracket \int_0^T J^2_t (\mathbb{W}^2_t \succ Z^\flat_t) \mathrm{d}t \mathrm{d}x.
\end{align}

We estimate $\tilde R_9(a)$ first. We will split into $(\tau,z)$ variables. By duality and H\"older's inequality,
\begin{align}
|\tilde R_9(a)|	&\leq \int \|  W_T (\tau, \cdot) \llbracket H (\varphi_-, \varphi_+)_T^2
    \rrbracket (\tau, \cdot) \|_{H_z^{- 1 + \kappa}} \left\| \int_0^T J^2_t (\mathbb{W}^2_t \succ Z^\flat_t)(\tau,\cdot) dt\right\|_{H^{1-\kappa}_z}   d\tau
    \\
&\leq \| W_T  \llbracket H (\varphi_-, \varphi_+)^2_T
    \rrbracket \|_{L^{4/3}_\tau H^{-1+\kappa}_z}   \left\| \int_0^T J^2_t (\mathbb{W}^2_t \succ Z^\flat_t) dt\right\|_{L^4_\tau H^{1-\kappa}_z}. 
\end{align}
In order to estimate the second term on the RHS, note that
\begin{equation}
  \left\| \int_0^T J^2_t (\mathbb{W}^2_t \succ Z^b_t) dt\right\|_{L^4_\tau H^{1-\kappa}_z} 
  \leq   \left\| \int_0^T J^2_t (\mathbb{W}^2_t \succ Z^b_t)dt\right\|_{B^{1-\kappa}_{4,1}} \leq C \| \mathbb W_t^2 \|_{L^\infty_t \mathcal C^{-1-\kappa/2}} \|Z_T\|_{L^4}.
\end{equation}

We now turn our attention to the remaining term. Let $q<4$ and $p>4/3$ be H\"older conjugate. By Young's inequality and the above estimate, we thus must estimate
\begin{equation}
I_0(\varphi_-,\varphi_+)=I_0:= \mathbb E\left [  \| W_T  \llbracket H (\varphi_-, \varphi_+)^2_T
    \rrbracket \|_{L^{4/3}_\tau H^{-1+\kappa}_z} ^p \right].
\end{equation}
By hypercontractivity and Minkowski's integral inequality, we may bound this by
\begin{align}
I_0 &\leq C \mathbb E  \left [  \| W_T  \llbracket H (\varphi_-, \varphi_+)^2_T
    \rrbracket \|_{L^{4/3}_\tau H^{-1+\kappa}_z} ^2 \right]^{p/2}	\\&\leq \left\| \mathbb E \left [  \| W_T  \llbracket H (\varphi_-, \varphi_+)^2_T
    \rrbracket \|_{H^{-1+\kappa}_z} ^2 \right]^{1/2} \right\|^p_{L^{4/3}_\tau}.\label{eq:boundWH}
\end{align}
Now to bound the expectation, observe that $(m^2-\Delta_{per})^{-(1+2\kappa)}$ on $\mathbb{T}^2$ has a kernel bounded by $|z-z'|^{-4\kappa}$. 
Thus for fixed $\tau$ we can bound 
\begin{align}
  &\mathbb E \left [  \| W_T  \llbracket H (\varphi_-, \varphi_+)^2_T
    \rrbracket \|_{H^{-1+\kappa}_z} ^2 \right]\\
\leq &\int  \int  \frac{1}{|z-z'|^{4\kappa}} C^{M}(x,x') |\llbracket H (\varphi_-, \varphi_+)^2_T(\tau,z) \rrbracket \llbracket H (\varphi_-, \varphi_+)^2_T(\tau,z')  \rrbracket | \mathrm{d} z \mathrm{d} z'\\
\leq& \int  \int \frac{1}{|z-z'|^{4\kappa}} \frac{d(\tau,\partial M)^{1-6\kappa}}{|z-z'|^{2-6\kappa}} \mathrm{d}z \mathrm{d}z' \|\llbracket H (\varphi_-, \varphi_+)^2_T(\tau,\cdot)  \rrbracket\|_{L^\infty_{z}}^2 \\
 \leq & C \| d(\tau,\partial M) ^{1/2-3\kappa} \llbracket H (\varphi_-, \varphi_+)^2_T(\tau,\cdot)  \rrbracket\|^2_{L^\infty_{z}}
\end{align}
Now taking this to the power $2/3$ and integrating in $\tau$ we obtain the desired bound.

We now turn to $\tilde R_9(b)$. As above, it is sufficient to estimate, with $p$ as above,
\begin{equation}
I:= \int \| \llbracket  H(\varphi_-,\varphi_+)^3 \rrbracket (\tau,\cdot) \|_{H^{-1+\kappa}_z}^p d\tau.
\end{equation}
Hence we may interpolate with the distance to the boundary to obtain 
\begin{equation}
|I| \leq C \| d(\tau,\partial M)^{\frac{1+4\kappa}{2}}\llbracket  H(\varphi_-,\varphi_+)^3 \rrbracket (\tau,\cdot) \|_{L^\infty_\tau H^{-1+\kappa}_z}^p \int d(\tau,\partial M)^{-\frac{p(1+4\kappa)}2} d\tau.
\end{equation}
The integral in $\tau$ on the righthand side is finite provided $p$ is sufficiently close to $4/3$.

 \end{proof}

\section{Convergence of bulk amplitudes}

\label{sec: convergence of bulk amplitudes}

In this section, we show that the variational representation of the bulk amplitudes converge to a limiting variational representation using the theory of $\Gamma$-convergence. As a consequence, we prove Theorem \ref{thm: bulk}. Our proof of convergence follows the $\Gamma$-convergence argument of Section \ref{sec: boundary} (and as a result, of \cite{BG20}) with the crucial difference coming in the handling of the remainder map.

\subsection{Relaxation of the variational problem}

We begin by relaxing the variational problem on an appropriate space of joint laws that has sufficient compacity properties to establish $\Gamma$-convergence. Recall that $\mathcal H = L^2_t L^2_x$. Define
\begin{equation}
	\mathcal L:=L^2_t W^{-1/2-\delta,3}_x . 
\end{equation}
This choice is motivated by the fact that the mapping $u\to Z(u)$ is a compact map $\mathcal{L} \mapsto L^{\infty}_t L^{4}_x$.

Recall the limiting bulk enhancement $\Xi_\infty(W)$ defined in Proposition \ref{prop: bulk enhancement}. We consider the following spaces of joint laws:
\begin{align}
\mathcal Y&:= \{ \mu \in P(\boldsymbol{\mathfrak S}\times \mathcal L) : \mathbb E_{\mu}[\|u\|_{\mathcal L}^2]<\infty \},	
\\
\mathcal X &:= \{ \mu \in \mathcal Y: {\rm Law}_\mu(\Xi) = {\rm Law}_{\mathbb P}(\Xi_\infty(W)) \}.
\end{align}
We endow $\mathcal Y$ and $\mathcal X$ with the following notion of sequential convergence. We say that a sequence $(\mu_n)$ converges to $\mu$ in $\mathcal Y$ if
\begin{enumerate}
	\item[(i)] $\mu_n \rightharpoonup \mu$ on $\boldsymbol{\mathfrak S}\times \mathcal L_w$, 
	\item[(ii)] $\sup_n \mathbb E_{\mu_n}[ \|u\|_{\mathcal L}^2] <\infty$.
\end{enumerate}
Let $\overline{\mathcal X}$ denote the sequential closure of $\mathcal X$ in $\mathcal Y$ under this notion of convergence. Let us stress that, since the marginal of the first components are fixed for any sequence in $\mathcal X$, the marginal of the first component is also fixed in $\overline{\mathcal X}$.

We now define our cost functions. For $T \in (0,\infty]$, define the functional on $\overline{\mathcal X}$
\begin{equation} \label{eqdef: tilde FF}
	\tilde{\mathbb F}^{f; \Xi_T^\partial (\varphi_-,\varphi_+)}_T(\mu)
 =\mathbb E_{\mu}\left[\Phi^{f;\varphi_-,\varphi_+}_T(u)+\mathcal C_T(u)\right], \qquad  \mu \in \overline{\mathcal X},
\end{equation}
where $\mathbb{E}_{\mu}$ denotes the expectation with respect to the measure $\mu$, $\Phi_T^{f;\varphi_-,\varphi_+}$ is as in Definition \ref{def: renormalized cost fcn} (see also Proposition \ref{prop: renormalized cost fcn} and Remark \ref{rem: renormalized cost fcn f}), and $\mathcal C_T(u)$ is as in \eqref{eqdef: CT-coercive}.

\begin{remark} \label{remark: cost fcn tinfty}
We stress that the function $u \mapsto \Phi_T^{f; \varphi_-,\varphi_+}(u) + \mathcal C_T(u)$ is deterministically well-defined for $u \in \mathcal L$. It does not require progressive measurability of $u$, which is not necessarily valid anymore. Furthermore, let us emphasize that we have implicitly extended this function to include the case $T=\infty$. Since $u \in \mathcal L$, the definition of $\mathcal C_\infty(u)$ is straightforward. Furthermore, by the stochastic estimates in Section \ref{sec:stochastic} the terms $\Upsilon^i_T$ in $\Phi_T^{f; \varphi_-,\varphi_+}$ can be extended in a straightforward manner to include $T=\infty$. It remains to justify why we may define the $R_i$ terms at $T=\infty$. This follows from viewing the $R_i$ as (sums of) multilinear maps and deducing continuity of these maps from the estimates used to prove Proposition \ref{prop: remainder terms}. We omit further details.
\end{remark}

The following proposition is an immediate consequence of the definition \eqref{eqdef: tilde FF}, Definition \ref{def: renormalized cost fcn}, and Remarks \ref{rem: renormalized cost fcn f} and \ref{remark: cost fcn tinfty}. 

\begin{proposition} \label{prop: relaxed renormalized cost fcn}
Let $\mu = {\rm Law}_{\mathbb P}(\Xi, u)$ for $v$ progressively measurable with respect to $W$ and almost surely in $\mathcal L$ . Then for every $T \in (0,\infty]$,
\begin{equation}
\tilde{\mathbb F}^{f; \Xi^\partial_T(\varphi_-,\varphi_+)}_T(\mu) = \mathbb F^{f; \Xi^\partial_T(\varphi_-,\varphi_+)}_T(v).	
\end{equation}
	
\end{proposition}

\subsection{Approximation properties of the functionals}

We begin with some technical approximation properties of the renormalized cost functions involved in the variational problem, analogous to those in the case of the boundary theory. 

The first technical property concerns approximation of the cost functions when $T<\infty$. Its proof is similar to the proof of Lemma \ref{lem:bdry-approx-finite} with minor adaptations required, and thus we omit it.

\begin{lemma}\label{lem:bulk-law-technical-1}
	Let $T > 0$. For every $\mu \in \overline{\mathcal X}$ and every $K \in \mathbb N^*$, there exists $\mu_K \in \mathcal X$ such that $\mu_K \rightarrow \mu$ and
\begin{equation}\label{eq: bulk mu-to-muK-Tfinite}
|\tilde{\mathbb F}^{f;\varphi_0,\varphi_+}_T(\mu) - \tilde{\mathbb F}^{f;\varphi_0,\varphi_+}_T(\mu_K)| < 1/K.
\end{equation}
\end{lemma}

We now turn to approximation properties of the cost function when $T=\infty$.

\begin{lemma} \label{lem: bulk law-technical-2}
The limiting renormalized cost function has the following approximation properties. 	
\begin{enumerate}
\item[(i)] For every $\mu \in \overline{\mathcal X}$ and every $K \in \mathbb N^*$, there exists $\mu_K\in \mathcal X$ such that
\begin{equation}\label{eq: bulk mu-to-muK}
|\widetilde{\mathbb F}^{f;\varphi_0,\varphi_+}_\infty (\mu) - \widetilde{\mathbb F}^{f;\varphi_0,\varphi_+}_\infty(\mu_K)| < 1/K.
\end{equation}
\item[(ii)] Let $\mu \in \mathcal X$. If $\Xi^{\infty}(\phi_{0}^T,\phi_{+}^T)\to \Xi^{\infty}(\phi_{0}^\infty,\phi_{+}^\infty)$ in $\boldsymbol{\mathcal B}$, then
\begin{equation} \label{eq: bulk mu_K T}
\lim_{T \rightarrow \infty} \widetilde {\mathbb F}^{f;\varphi^T_0,\varphi^T_+}_T(\mu) = \widetilde{\mathbb F}^{f;\varphi_0,\varphi_+}_\infty(\mu). 
\end{equation}
\end{enumerate}
\end{lemma}

In order to prove Lemma \ref{lem: bulk law-technical-2} it is necessary to have a bound on $\mathbb{E}[\|\ell_{\infty}^{\varphi_-,\varphi_+}(u)\|^{2}_{\mathcal{H}}]$. Since the ansatz map is more complicated now, this will require some significant modifications as compared with the proof of Lemma \ref{lem:bdry-approx-infty}, which is the analogous result for the boundary theory. For convenience, let us write
\begin{equation}
\mathbb U_t:= J_t \mathbb W_t^3 + J_t (\mathbb{W}^2_t H
     (\varphi_-, \varphi_+))+ 12 J_t W_t \llbracket H (\varphi_-, \varphi_+)^2 \rrbracket + 4
     J_t \llbracket H (\varphi_-, \varphi_+)^3 \rrbracket
\end{equation}
so that, we have
\begin{equation}
\ell(u)_t:= \ell^{\varphi_-,\varphi_+}_\infty(u)_t= u_t+ \mathbb U_t + J_t(\mathbb W_t^2 \succ Z_t^\flat(u)).
\end{equation}

We will now regularize $\ell(u)$. In order for our regularization to exist, we will need to modify the paracontrolled part of the ansatz. For $\eta > 0$, define
\begin{equation}
\mathbb W_t^{2,\leq }:=  P_{|k|\leq \|\mathbb W_t^2\|_{\mathcal C^{-1-\kappa/2}}^\eta}, \, \mathbb W_t^{2,>}:= \mathbb W_t^2 - \mathbb W_T^{2,\leq},
\end{equation}
where $P_{|k|\leq C}$ denotes the projection on the fourier modes $\leq C$ on $\mathbb{T}^2\times L\mathbb{T}$. Note that here we are implicitly viewing these objects as periodic distributions.
We will choose $\eta$ so that, by Bernstein's inequality, there exists $c>0$ sufficiently small such that
\begin{equation}
\| \mathbb W_t^{2,>}\|_{\mathcal C^{-1-\kappa}} \leq c.
\end{equation}
Define the regularized ansatz map 
\begin{equation}
\tilde \ell(u) := - J_t ( \mathbb W_t^{2,\leq} \succ Z_t^\flat (u) )+\ell^\infty(u),
\end{equation}
corresponding to taking the full remainder of $u$ and subtracting the low frequency part of the paracontrolled term.  Furthermore, by using the definition and the triangle inequality, we have that there exists $C>0$ such that
\begin{equation}
\|\tilde \ell (u)\|_{\mathcal L} \leq C\|u\|_{\mathcal L} + C.
\end{equation}

We now define the regularized remainder map ${\rm rem}_{\epsilon}u$ as the implicit solution of the equation
\begin{equation}
{\rm rem}_\varepsilon \, u
=
- \mathbb U_t - J_t (\mathbb W_t^{2,>}\succ Z_t^\flat({\rm rem}_\epsilon u ))+ {\rm reg}_{\varepsilon;x} \Big( \tilde \ell (u) \Big),
\end{equation}
where, for every $\tilde u \in \mathcal L$, ${\rm reg}_{\varepsilon; x} \tilde u$ is the convolution of $\tilde u$ (viewed as a periodic distribution) in space at scale $\varepsilon$ -- c.f. \eqref{eqdef: space reg bdry}.
The fact that this equation has a solution follows immediately from standard properties of convolution, paraproduct estimates, and the Banach contraction mapping theorem in the space $\mathcal L$. We omit the details.

\begin{remark} 
The use of a contraction mapping argument in the regularization justifies the need for the high frequency decomposition of $\mathbb W_t^2$ depending on the norm, and for $c$ has to be chosen sufficiently small for this to work).	
\end{remark}

Lemma \ref{lem: bulk law-technical-2} follows from the following properties of the regularized remainder map by adapting the arguments used in the proof of Lemma \ref{lem:bdry-approx-infty} using Lemma \ref{lem: bdry rem properties}.

\begin{lemma}\label{lem:bulk-rem}
There exists $C>0$ such that for every $\varepsilon > 0$, the following statements hold.
\begin{enumerate}
\item[(i)] For every $u \in \mathcal L$, 
\begin{equation}
\|{\rm rem}_\varepsilon(u)\|_{\mathcal L} \leq C \left(\| \Xi_\infty(W)\|_{\mathfrak S} + \|\Xi^\partial_\infty(\varphi_-,\varphi_+)\|_{\boldsymbol{
\mathcal B}}+ \|u\|_{\mathcal L}\right).
\end{equation}
\item[(ii)] For every $T \in (0,\infty]$ and $u \in \mathcal L$, 
\begin{equation}
\|Z_T({\rm rem}_\varepsilon(u))\|_{L^4} \leq \left(\|\Xi_\infty(W)\|_{\mathfrak S} + \|\Xi^\partial_\infty(\varphi_-,\varphi_+)\|_{\boldsymbol{
\mathcal B}} + \|Z_T(u)\|_{L^4}\right).
\end{equation}
\item[(iii)] For every $\varepsilon > 0$, there exists $c_\varepsilon > 0$ such that, for every $u \in \mathcal L$,
\begin{equation}
	\|\ell^{\varphi_-,\varphi_+}_\infty({\rm rem}_\varepsilon(u))\|_{\mathcal H}^2 \leq c_\varepsilon \left( (1+\|\Xi_\infty (W)\|_{\mathfrak S} + \|\Xi^\partial_\infty(\varphi_-,\varphi_+)\|_{\boldsymbol{
\mathcal B}})^4 + \|Z_\infty(u)\|_{L^4}^4 + \|u\|_{\mathcal L}^2 \right).
\end{equation}
Furthemore, the map $u \to \ell_{\infty}({\rm rem}_{\epsilon})$ is a continuous from $\mathcal{L}_{w} \to \mathcal{H}$
\end{enumerate}
\end{lemma}

\begin{proof}
Property (i) and (ii) are straightforward consequences of the definition of ${\rm rem}$ and ${\rm reg}$. Let us turn to property  (iii). Note that
\begin{align}
\ell^\infty({\rm rem}_\epsilon u) &= \ell^\infty \left( - \mathbb U_t - J_t (\mathbb W_t^{2,>}\succ Z_t^\flat({\rm rem}_\epsilon u ))+ {\rm reg}_{\varepsilon;x} \Big( \tilde \ell (u) \Big) \right)
\\
&=  - \mathbb U_t - J_t (\mathbb W_t^{2,>}\succ Z_t^\flat({\rm rem}_\epsilon u ))+ {\rm reg}_{\varepsilon;x} \Big( \tilde \ell (u) \Big)
\\
&\qquad + \mathbb U_t + J_t \left( \mathbb W_t^2 \succ Z_t^\flat \left( - \mathbb U_t - J_t (\mathbb W_t^{2,>}\succ Z_t^\flat({\rm rem}_\epsilon u ))+ {\rm reg}_{\varepsilon;x} \Big( \tilde \ell (u) \Big)\right)  \right) 
\\
&= J_t ( \mathbb W_t^{2,\leq } \succ Z_t^\flat({\rm rem}_\epsilon u)) + {\rm reg}_{\epsilon;x} (\tilde\ell(u)).
\end{align}
Then the desired result follows by standard estimates.
\end{proof}

\subsection{Convergence and the proof of Theorem \ref{thm: bulk}}

We now turn to the proof of Theorem \ref{thm: bulk}. The key inputs are the uniform equicoercivity and $\Gamma$-convergence of the bulk cost functions (which represent the amplitudes). Using the uniform bounds on the remainder terms together with the approximation lemmas above, the following two lemmas follow by identical arguments as in the proofs of Lemmas \ref{lem: boundary equicoercive} and \ref{lem: boundary fatou} in the boundary case -- we omit the details.

\begin{lemma}[Equicoercivity] \label{lem: bulk equicoercivity}
Let $(\Xi^\partial_T)_{T>0}$ converge to $\Xi^\partial_\infty$ in $\boldsymbol{\mathcal B}$. There exists a sequentially compact set $\mathcal K = \mathcal K(\sup_{0 < T \leq \infty}\|\Xi^\partial_T\|_{\boldsymbol{\mathcal B}})\subset \overline{\mathcal X}$ such that, for every $T \in (0,\infty]$,
\begin{equation}
\inf_{\mu \in \overline{\mathcal X}} \tilde{\mathbb F}^{f; \Xi_T^\partial }_T(\mu) = \inf_{\mu \in \mathcal K} \tilde{\mathbb F}^{f; \Xi^\partial_T}_T(\mu). 
\end{equation}
\end{lemma}

\begin{lemma}[$\Gamma$-convergence] \label{lem: bulk gamma convergence}
Let $f\in C^\infty(M)$ and $\Xi^\partial_\infty \in \boldsymbol{\mathcal B}$.
\begin{enumerate}
\item[(i)] (Fatou) For every $\mu \in \overline{\mathcal X}$, for every sequence $\mu_n \rightarrow \mu$ in $\overline{\mathcal {X}}$, every sequence $(T_n) \rightarrow \infty$, and every sequence $\Xi^\partial_{T_n} \rightarrow \Xi^\partial_\infty $ in $\boldsymbol{\mathcal B}$,
\begin{equation}
\tilde{\mathbb F}^{f;\Xi^\partial_\infty}_\infty(\mu) \leq \liminf_{n\rightarrow \infty} \tilde {\mathbb F}^{f;\Xi^\partial_{T_n}}_{T_n}(\mu_n). 
\end{equation}
\item[(ii)] (Recovery sequence) For every $\mu \in \overline{ \mathcal X}$, for every sequence $(T_n) \rightarrow \infty$, and every sequence $\Xi^\partial_{T_n} \rightarrow \Xi^\partial_\infty$ in $\boldsymbol{\mathcal B}$, there exists a sequence $\mu_n \rightarrow \mu$ in $\overline {\mathcal X}$ such that
\begin{equation}
	\tilde{\mathbb F}^{f;\Xi^\partial_\infty}_\infty( \mu) \geq \limsup_{n\rightarrow \infty} \tilde {\mathbb F}^{f;\Xi^\partial_{T_n}}_{T_n}(\mu_n).
\end{equation} 
\end{enumerate}	
\end{lemma}

\begin{proof}[Proof of Theorem \ref{thm: bulk}]
Part I follows from Propositions \ref{prop: renormalized cost fcn} and \ref{prop: relaxed renormalized cost fcn}. Part III follows from the bounds established on the $\Upsilon^i$ in Section \ref{sec:stochastic} and the $R_i$ in Section \ref{sec:equicoercive} (see Proposition \ref{prop: remainder terms}),  and the extension of these estimates to $T=\infty$ (see Remark \ref{remark: cost fcn tinfty}).

It remains to establish Part II, the convergence of the amplitudes. As a consequence of the $\Gamma$-convergence of Lemma \ref{lem: bulk gamma convergence}, the equicoercivity of Lemma \ref{lem: bulk equicoercivity}, and the fundamental theorem of $\Gamma$-convergence\cite{D93}, we have that for every $f$ and every sequence $\Xi^\partial_T \rightarrow \Xi^\partial_\infty$ in $\boldsymbol{\mathcal B}$,
\begin{equation}
\lim_{T \rightarrow \infty} \min_{\mu \in \overline{\mathcal X}}{\tilde{\mathbb F}}^{f;\Xi^\partial_{T}} _{T}(\mu) =  \min_{\mu \in \overline{\mathcal X}}{\tilde{\mathbb F}}^{f; \Xi^\partial_{T}}_{\infty}(\mu). 
\end{equation}
Consequently, if $\Xi^\partial_{T} \to \Xi^\partial_{\infty}$ in $\boldsymbol{\mathcal{B}}$, then  
\begin{align}
\lim_{T \rightarrow \infty} \boldsymbol{\mathcal{A}}_T^\sigma(f,\Xi_{T}^\partial )
=
\boldsymbol{\mathcal{A}}_\infty^\sigma(f,\Xi_{\infty}^\partial), \qquad \forall f \in C^\infty(M),
\end{align}
as required.  
\end{proof}

\section{Smallfield reduction}
\label{sec: smallfield reduction}

In this section, we prove Lemma \ref{lem: large field} which in turn fills the remaining unproven claim used in the proof of Theorem \ref{thm: main}.

\begin{proof}[Proof of Lemma \ref{lem: large field}]
For every $\varphi_-,\varphi_+$ i.i.d. distributed according to $\nu^0$, let us define
\begin{align}
E_T(\varphi_-,\varphi_+)&:= E_{\nu^0_T} \Big[ \mathcal A_T^-(f \mid \rho_T \ast \varphi_-, \varphi^0)\mathcal A^+_T(f \mid \varphi^0, \rho_T \ast \varphi_+)
\\
&\qquad\qquad \qquad  \times  e^{c \|\Xi_T^\partial (\rho_T \ast \varphi_-,\varphi^0)\|_{\boldsymbol{\mathfrak B}'}^\alpha +c\|\Xi^\partial_T(\varphi^0,\rho_T \ast \varphi_+)\|_{\boldsymbol{\mathfrak B}'}^\alpha }\Big].
\end{align}
Recall that we wish to establish that $\nu^0 \otimes \nu^0$-almost surely, there exists $C>0$ such that\begin{equation}
\sup_T E_T(\varphi_-,\varphi_+)	\leq C
\end{equation} 

Let $T>0$ be fixed. By Definition \ref{def: approx amplitude} and the Markov property for $\mu$, Proposition \ref{prop: domain markov} --- see the calculations in the proof of Proposition \ref{prop: gluing-cutoff} --- we have that $E_T(\varphi_-,\varphi_+)$ is equal to
\begin{equation}
\mathcal A^{\rm free}(\rho_T \ast \varphi_-, \rho_T \ast \varphi_+) \mathcal E_T(\varphi_- \ast \rho_T, \varphi_+ \ast \rho_T) \mathbb E_{\mu}\left[\exp\left( \langle f, \varphi \rangle - U_T^\alpha(\varphi_T \mid \varphi_-,\varphi_+) \right) \right].
\end{equation}
Above, the potential $U_T^\alpha(\varphi, \varphi_-,\varphi_+)$ is equal by definition to the quantity
\begin{equation}
V_T(\varphi_T \mid \rho_T \ast \varphi_-,\rho_T \ast \varphi_+ )+ c \|\Xi_{T}^\partial(\rho_T\ast \varphi_-,(\varphi_T)_0)\|^{\alpha}_{\boldsymbol{\mathcal{B}}'}+c \|\Xi_{T}^\partial((\varphi_T)_0,\rho_T \ast \varphi_+)\|^{\alpha}_{\boldsymbol{\mathcal{B}}'},
\end{equation}
where $(\varphi_T)_0=\varphi_T(0,\cdot)$ is the restriction of $\varphi_T$ to $\tau=0$. 

Since the constant may depend on $\varphi_-,\varphi_+$ and also $f$, we can ignore the contribution of the free amplitude and the term $\mathcal E_T(\varphi_- \ast \rho_T, \varphi_+ \ast \rho_T)$. It is sufficient to analyze the expectation term, which we denote by $E_{2,T}$.

The term $E_{2,T}$ can be analyzed similarly as the bulk amplitudes in Sections \ref{sec: renormalization of bulk amplitudes}-\ref{sec: convergence of bulk amplitudes}. Applying the Bou\'e-Dupuis formula and using the same notation as before, the desired estimate on $E_{2,T}$ amounts to establishing that there exists $C>0$ such that for every $T>0$,
\begin{equation} \label{eq: smallfield claim}
	\inf_{v \in \mathbb H} \mathbb F_{2,T}(v)\geq -C,
\end{equation}
where the cost function $\mathbb F_{2,T}(v)$ is given by
\begin{align}
\mathbb F_{2,T}(v) &= \mathbb E \Bigg[ \mathcal V_T(W_T,Z_T(v))  + c\| \Xi_T^\partial(\rho_T\ast\varphi_-, (W_T+Z_T(v)+H(\varphi_-,\varphi_+)_T)_0)\|^\alpha  
\\
&\qquad + c\| \Xi_T^\partial(\rho_T \ast \varphi_+, (W_T+Z_T(v)+H(\varphi_-,\varphi_+)_T)_0 \|^\alpha + \frac 12 \| v\|_{L^2_t L^2_x}^2 \Bigg].
\end{align}
We note that, since we are only interested in bounds, a $\Gamma$-convergence argument to deduce full convergence is not necessary.

The potential term $\mathcal V_T(W_T,Z_T(v))$ is treated identically as in Section \ref{sec: renormalization of bulk amplitudes} and it remains to modify the arguments to estimate the second and third terms. By symmetry, we will estimate the second term, i.e.\ the one with $\varphi_-$. First of all, note that $(W_T)_0$ exists by the Markov property and its law is given by boundary Gaussian with covariance given by the Dirichlet-to-Neumann kernel mollified at scale $T$, i.e.\ the pushforward of $\tilde \mu^0$ under the regularization map $\varphi^0 \mapsto (\varphi^0)_T$. For the term $Z_T(v)$, applying the ansatz established in Section \ref{sec: renormalization of bulk amplitudes}, by the intermediate ansatz \eqref{eq: intermediate ansatz} we have that $Z_{T}(v)=-\mathbb W_t^{[3]}+K_{T}$. The first component can be restricted to the $\tau=0$ slice with regularity in $\mathcal C^{1/2-}_z$, and by trace theorem we have that the restriction of the $H^{1-}$ component belongs to $H^{1/2-}$. Finally, the harmonic extension term is smooth away from $\partial M$ and hence smooth on $B$. Thus,
\begin{equation}
(Z_T(v)+H(\varphi_-,\varphi_+)_T)_0 \in H^{1/2-}(B).
\end{equation}Combining the above, we have that $(W_T + Z_T(v) + H(\varphi_-,\varphi_+)_T)_0$ is an admissible law on boundary conditions under the measure $\mathbb P$ arising from the Bou\'e-Dupuis formula. By abuse of notation, we will therefore write this term as $W_T^0 + Z_T^0$. 

\begin{remark}
Although in the definition of admissible law, $W_T^0$ is distributed according to the boundary Gaussian field $\mu^0$, we may replace $\mu^0$ by the boundary Gaussian with Dirichlet-to-Neumann covariance $\tilde\mu^0$ with straightforward modifications. One way to see this is to use the Bou\'e-Dupuis representation of $(W_T)_0$ in terms of a stochastic series. By direct calculation, one can then show that $(W_T)_0$ is equal almost surely to a random variable distributed according to $\mu^0$ plus a smoothing term of Sobolev regularity greater than $\frac 12$. 
\end{remark}

Recall that $\Xi^\partial_T(\rho_T \ast \varphi_-, W_T^0 + Z_T^0)$ consists of the Wick powers of harmonic extensions the individual boundary fields as well as the random variables $\{ \Upsilon_{2,T}, \Upsilon_{3,T}, \Upsilon_{4,T}, \Upsilon_{5,T}\}$ defined in Section \ref{sec: renormalization of bulk amplitudes}. We emphasize from the outset that the term $\rho_T \ast \varphi_-$ is treated as deterministic in this lemma. For the Wick powers, we may directly apply Proposition \ref{prop: stochastic boundary v2}. For the random vairables $\Upsilon_{i,T}$, we have shown in Section \ref{sec:stochastic} that there exists $\alpha > 0$ such that, for every $\delta > 0$, there exist constants $c_1,c_2>0$ --- depending on $\varphi_-$ --- such that
\begin{equation}
\mathbb E[|\Upsilon_{i,T}|^\alpha] \leq c_1+\delta \mathbb E[\|Z^0_T\|_{H^{1/2-}_z}^2] \leq c_2 + \delta \left( \mathbb E \left[ \|Z_T(v)\|_{L^4_x}^4 + \frac 12 \|\ell_T^{\varphi_-,\varphi_+}(v)\|_{L^2_tL^2_x}^2\right] \right),
\end{equation}
where $\ell^{\varphi_-,\varphi_+}_T(v)$ is the residual drift entropy induced by the full drift ansatz, see Definition \ref{definition: full ansatz}.
Thus we have established that, with $\alpha$ chosen as above, for every $\delta > 0$ there exists $C>0$ depending on $\varphi_-$, such that
\begin{equation}
	\mathbb{E}[\| \Xi_T^\partial(\rho_T \ast \varphi_-, (W_T+Z_T(v)+H(\varphi_-,\varphi_+))_0)\|_{\boldsymbol{\mathcal B}'}^\alpha] 
\leq C+\delta \left(\mathbb{E} \|Z_{T}(v)\|^4_{L^4}+\frac 12 \|\ell_{T}^{\varphi_-,\varphi_+}(v)\|^2_{L^2_tL^2_x}\right).
\end{equation}
Adding these estimates to the estimates developed in Sections \ref{sec:stochastic} and \ref{sec:equicoercive} establishes \eqref{eq: smallfield claim}. 

\end{proof}

\section{The $\varphi^4_3$ Hamiltonian} \label{sec:spectral}

In this section we prove Theorem \ref{thm: hamiltonian}. Namely, we will construct the $\varphi^4_3$ Hamiltonian and prove some fundamental spectral properties such as discreteness of spectrum.  

\subsection{Tightness of the ground state of regularized $\varphi^4_3$ Hamiltonians} \label{subsec: tightness ground state}

We begin by proving a technical result that will be crucial to our strategy to construct the $\varphi^4_3$ Hamiltonian. For every $T\geq 0$, we will introduce a cutoff Hamiltonian with discrete spectrum and which has a simple lowest eigenvalue whose associated normalized eigenvector is positive almost surely. These are called ground states for the cutoff Hamiltonian. We will then prove that these ground states form a tight sequence.  

Let $T \in [0,\infty)$. Recall the cutoff boundary measure $\nu^0_T$. Define the Hilbert space $\mathcal H_T := L^2(\nu^0_T)$. 
  Let $(\mathbf P_\tau^T)_{\tau \geq 0}$ be the semigroup acting on bounded functions $f:\mathcal H_T \rightarrow \mathbb R$ by
\begin{equation}
\mathbf P_\tau^T f(\varphi) := \int f(\varphi') \mathcal A_\tau^T(\varphi,\varphi') \nu^0_T(d\varphi'), \qquad \forall \varphi \in \mathcal H_T,	\qquad \forall \tau > 0,
\end{equation}
 where $\mathcal A_\tau^T$ is the regularized $\varphi^4_3$ amplitude on $[0,\tau]\times\mathbb T^2$ defined in Definition \ref{def: approx amplitude}. Furthermore, let $\mathbf P^T_0 = {\rm Id}_{\mathcal H_T}$. 
 
 The following proposition follows from the classical theory of Schr\"odinger operators with semibounded potentials. We refer to \cite[Chapter 11-12]{RS78} for further detail (see also the proof of discrete spectrum in Section \ref{subsec: spectral properties of H}).
 
 \begin{proposition} \label{proposition: cutoff hamiltonian}
 Let $T\in[0,\infty)$. There exists a self-adjoint linear map $\mathbf H_T:D(\mathbf H_T)\rightarrow \mathcal H_T$, where $D(\mathbf H_T) \subset \mathcal H_T$ is a dense linear subspace, such that $\mathbf H_T$ is the infinitesimal generator of $(\mathbf P^T_\tau)_{\tau \geq 0}$. The linear map $\mathbf H_T$ has discrete spectrum $\sigma(\mathbf H_T)\subset [E_0^T, \infty)$ where $E_0^T \in \mathbb R$ is the lowest eigenvalue which is simple and its associated normalized eigenvector $e_0^T$ can be chosen such that $e_0^T>0$ $\nu^0_T$-almost surely.
 \end{proposition}

\begin{remark}
Our regularized $\varphi^4_3$ Hamiltonian $\mathbf H_T$ differs from the Hamiltonian considered in \cite{G68}, where a spectral cutoff is used directly in (the quadratic form of) the Hamiltonian.
\end{remark}

The bounds on the amplitude derived in Theorem \ref{thm: bulk} imply that the lowest eigenvalues are uniformly bounded in $T\geq 0$, as formalized in the proposition below. The proof follows from a simplification\footnote{Here, the simplification comes from the fact we already know the spectrum is discrete, so no spectral measure decomposition is required.} of the arguments used below to prove the corresponding result when $T=\infty$, see Proposition \ref{proposition: lower bound spectrum H}.

\begin{proposition} \label{proposition: cutoff spectral lower bound}
There exists $C >0$ such that
\begin{equation}
-C \leq \inf_T E_0^T \leq \sup_T E_0^T \leq C.	
\end{equation}
\end{proposition}

We now turn to "tightness" for the \emph{ground states} $ e_0^T$. One difficulty is that our bounds on the amplitudes in Theorem \ref{thm: bulk} are not strong enough to use $L^p$ estimates on eigenvectors for $p>2$. We therefore use an interpolation space called the Orlicz space. We recall some standard theory, see \cite{HH19}. Given a convex, lower semicontinuous function $\Psi: \mathbb R_{\geq 0} \rightarrow \mathbb R_{\geq 0}$, the Orlicz space associated to $\Psi$  is the Banach space with norm 
\begin{equation}
	\inf \left\{ k \geq 0 : \int \Psi(|f|/k) d\nu^0 \leq 1 \right \},
\end{equation}
and the space is the set of functions such that $\int \Psi(|f|) d\nu^0 < \infty$. 
We will exclusively work with $\Psi(t) = \Psi_{p,\alpha}(t) := t^p (\log_+ t)^\alpha$, when $p, \alpha \geq 1$. In this case, we write the space as $L^p (\log_+ L)^\alpha$. We begin with an elementary observation: there exists $C>0$ such that, for every $f \in L^p (\log_+ L)^\alpha$,
\begin{equation} \label{eq: orlicz norm bound}
	\|f\|_{L^p (\log_+ L)^ \alpha} \leq  \max \left \{ \left(\int |f|^p (\log_+ |f|)^\alpha d\nu^0 \right)^{1/p}, 1 \right\}. 
\end{equation}
To see this, let us choose 
\begin{equation}
k = \max \left \{ \left(\int |f|^p (\log_+ |f|)^\alpha d\nu^0 \right)^{1/p}, 1 \right\} \geq 1
\end{equation}
Then, since $k \geq 1$, 
\begin{equation}
	\int |f/k|^p (\log_+ |f| -\log_+ k)^\alpha d\nu^0 \leq \frac 1{k^p} \int |f|^p (\log_+|f|)  \leq 1,
\end{equation} 
which establishes \eqref{eq: orlicz norm bound}.

We begin by showing that the ground state eigenvector $L^2 (\log_+ L)^\alpha$ Orlicz norm is uniformly bounded as $T \rightarrow \infty$ for every $\alpha \geq 1$.  
\begin{lemma}\label{lemma: eigenvector bounds}

For every $\alpha \geq 1$, there exists $c_0>0$ such that
    \begin{equation}
        \sup_{T} \int (e^T_0)^{2} (\log_+ e_0^T)^\alpha d\nu^0_T \leq c_0.
    \end{equation}
\end{lemma}

\begin{proof}
Recall that by definition of $e_0^T$, for every $\tau > 0$, almost surely in $\varphi$ we have \begin{equation}
        e_0^T(\phi)=e^{E^T_0 \tau}\int \mathcal{A}^T_\tau(\phi,\phi') e_{0}^T(\phi')\nu_{T}^0(d\varphi').
    \end{equation} 
   By Proposition \ref{proposition: cutoff spectral lower bound}, the term $e^{E_0^T \tau}$ is uniformly bounded in $T$ (provided $\tau$ is fixed). 
   
   Let us now consider the integral term. By the Cauchy-Schwarz inequality and gluing property (with cutoff) of Proposition \ref{prop: gluing-cutoff}, there exists $c_1>0$ such that for every $T>0$,
    \begin{equation}
        \left(\int \mathcal{A}^T_\tau(\phi,\phi') e^T_{0}(\phi')\nu_{T}^0(d\varphi')\right)^2 \leq c_1\mathcal{A}^T_{2\tau}(\phi,\phi)
    \end{equation}
    We deduce that
    \begin{equation} \label{eq: evec decay squared}
    	e_0^T(\varphi)^2 \leq c_2 \mathcal A_{2\tau}^T(\varphi,\varphi). 
    \end{equation}
    Furthermore, from the bounds on the amplitudes in Theorem \ref{thm: bulk}, we deduce the existence of $c_3,p>0$ such that 
    \begin{equation} \label{eq: evec decay ptwise}
    e_0^T(\varphi) \leq c_2 \exp(\| \Xi_T^\partial(\varphi,\varphi)\|_{\boldsymbol{\mathfrak B}([0,2\tau]\times \mathbb T^2)}^p).
    \end{equation}

 By \eqref{eq: evec decay squared} and \eqref{eq: evec decay ptwise}, and the bound on the amplitude, for every $\alpha > 0$ there exists $c_3>0$ such that, uniformly in $T$,
    \begin{equation} \label{eq: evec decay lq}
    \int e_0^T(\varphi)^{2} (\log_+ e_0^T(\varphi))^\alpha	\nu^0_T(d\varphi) \leq c_3 \int \mathcal{A}^T_{2\tau}(\phi,\phi) \|\Xi^\partial_T(\varphi,\varphi)\|_{\boldsymbol{\mathfrak B}([0,2\tau]\times \mathbb T^2)}^{p\alpha } \nu^0_T(d\varphi).
    \end{equation}
    The right-hand side of \eqref{eq: evec decay lq} can then be estimated by a straightforward adaptation of the reasoning in Section \ref{sec: smallfield reduction} to prove Lemma \ref{lem: large field}. 
   
\end{proof}

\begin{remark}
In the proof of Lemma \ref{lemma: eigenvector bounds} we use that $e_0^T$ is an eigenvector heavily. However, we only use that it is associated to the lowest eigenvlaue $E_0^T$ when we assert that $e^{E_0^T \tau}$ is bounded uniformly in $T$. For higher order eigenvalues, the bound may in principle depend on $T$.	We will instead apply this same argument to the eigenvectors of $\mathbf H$ \emph{after} we have proved that its spectrum is discrete.
\end{remark}

\begin{remark}
The estimate on $(\log_+ e_0^T)^\alpha$ could be replaced by $\exp ( (\log_+ e_0^T)^{\alpha'})$ for $\alpha '$ small enough. However, the function $t \mapsto \exp( (\log_+ t)^\alpha)$ does not necessarily satisfy convexity for arbitrarily small $\alpha$.
\end{remark}

Let us also point out that the bounds above, particularly \eqref{eq: evec decay lq}, together with the gluing property can be used to estimate the amplitude in a higher order (Orlicz) norm. 
\begin{corollary} \label{corollary: lp bound amplitude}
For every $\alpha \geq 1$, there $c_\alpha>0$ such that
\begin{align}
\sup_{T \in [0,\infty)}  \|\mathcal A_\tau^T (\varphi,\varphi')	\|_{L^2(\log_+ L)^\alpha (\nu^0_T; L^2( \nu^0_T))} &\leq c_\alpha, \\
\sup_{T \in [0,\infty)} \|A^T_\tau (\varphi,\varphi)\|_{L^2(\log_+L)^\alpha } &\leq c_\alpha.
\end{align}

\end{corollary}

In order to upgrade Lemma \ref{lemma: eigenvector bounds} to a true tightness result on the eigenvectors, we will express $e_0^T$ in terms of the simultaneous coupling of the boundary fields $(\nu^0_T)_{T \geq 0}$ and $\nu^0=\nu^0_\infty$.   

Recall that $\nu^0_T ={\rm Law}_{\mathbb P^0}( W^0_T+Z^0_T)$ and $\nu^0 ={\rm Law}_{\mathbb P^0}( W^0_\infty+Z^0_\infty)$, where $(\Omega^0, \mathcal F^0, \mathbb P^0)$ is the probability space and $(W^0+Z^0)$ is the process constructed in Theorem \ref{thm: boundary} and Section \ref{sec: boundary}.  We implicitly redefine our Hilbert spaces and operators as acting on this coupling space. Thus have ${\rm Law}_{\nu^0_T}(e_0^T) = {\rm Law}_{\mathbb P^0}(e_0^T(W^0_T+Z^0_T))$.

We now turn to tightness of the ground states.

\begin{proposition} \label{prop: tightness eigenvectors}
There exists a random variable $e_0$ defined on the coupling space that is a measurable function of $W^0_\infty+Z^0_\infty$, i.e.\ $e_0=e_0(W^0_\infty+Z^0_\infty)$, such that $\|e_0\|_{L^2(\mathbb P^0)}=1$ and, up to a subsequence, 
\begin{equation}e_0^T(W_T^0+Z^0_T) \rightarrow e_0(W^0_\infty+Z^0_\infty),	
\end{equation}
 where the convergence can be taken as either $\mathbb P^0$-almost sure or in $L^2(\mathbb P^0)$. Furthermore, $e_0>0$ almost surely with respect to $\mathbb P^0$ and for every $\alpha \geq 1$, there exists $C>0$ such that
\begin{equation}\|e_0\|_{L^2 (\log_+L)^\alpha (\mathbb P^0)} \leq C.
\end{equation} 
\end{proposition}

\begin{proof}
For convenience, let us write $\tilde e_0^T := e_0^T(W_T^0+Z_T^0)$. Since $\|\tilde e_0^T\|_{L^2(\mathbb P^0)}=1$ for every $T$, the sequence $(\tilde e_0^T)_{T>0}$ is compact in the weak$-*$ topology of $L^2(\mathbb P^0)$. Let us write $\tilde e_0$ to denote a limit point. We wish to show that $\|\tilde e_0\|_{L^2(\mathbb P^0)}=1$, from which we can deduce that, for the associated subsequence,  $\|\tilde e_0^T-\tilde e_0\|_{L^2(\mathbb P^0)} \rightarrow 0$ (i.e.\ weak$-*$ plus convergence of norms implies strong convergence).  Upon taking another subsequence, we obtain the almost sure convergence. Let us also observe that $\tilde e_0 = e_0(W_\infty^0+Z_\infty^0)$, for some $e_0$,  in terms of measurability. In order to see this, note that each $e_0^T$ is measurable with respect to the $\sigma$-algebra generated by $(W^0+Z^0)_{t \geq T}$, and hence the intersection of these $\sigma$-algebras. We thus define the $e_0$ of the proposition to be equal to $\tilde e_0$. 

 Consider now a subsequence $(\tilde e^{T_n}_0)_n$ with limit point $\tilde e_0$ in the weak$-*$ topology. Recall the $\|\cdot\|_{L^2(\mathbb P^0)}$-norms are equal to $1$ along the whole subsequence. By the Cauchy-Schwarz inequality the sequence as random variables are tight and thus converge in law up to a further subsequence --- not relabelled ---  to a random variable $\hat e_0$, perhaps defined on a different probability space, whose law coincides with the law of $\tilde e_0$ under $\mathbb P^0$. By Skorokhod embedding, we may assume without loss of generality that $\tilde e_0^{T_n} \rightarrow \hat e_0$ almost surely and thus, thanks to the uniform $L^p$ boundedness of $(\tilde e_0^{T_n})$ of Lemma \ref{lemma: eigenvector bounds} for some $p>2$ and Vitali's convergence theorem, we have that $\tilde e_0^{T_n} \rightarrow \hat e_0$ in $L^2$. In particular, we have that the original limiting random variable we consider, $\tilde e_0$, satisfies $\|\tilde e_0\|_{L^2(\mathbb P^0)}=1$.

 Let us now turn to the prove that $\tilde e_0>0$ almost surely. Recall that for every $T<\infty$,
 \begin{equation}
 e_0^T(W^0_T+Z^0_T) = e^{E_0^T \tau } \int \mathcal A_\tau^T(W^0_T+Z^0_T, \tilde W^0_T+\tilde Z^0_T) e_0^T(\tilde W^0_T+\tilde Z^0_T) d\mathbb P^0(\tilde W^0+\tilde Z^0). 	
 \end{equation}
Up to another subsequence, we have that $(\tilde e_0^{T_n})$ converges almost surely to $\tilde e_0$. Taking limits on both sides of above, by Fatou's lemma, the convergence of the amplitudes in Theorem \ref{thm: bulk}, and Proposition \ref{proposition: cutoff spectral lower bound}, we have that there exists $C>0$ such that
\begin{equation}
e_0(W^0_\infty+Z^0_\infty) \geq e^{-C \tau} \int \mathcal A^\infty_\tau (W^0_\infty+Z^0_\infty, \tilde W^0_\infty + \tilde Z^0_\infty) e_0(\tilde W^0_\infty + \tilde Z^0_\infty) d\mathbb P(\tilde W^0 + \tilde Z^0) > 0. 
\end{equation}
Thus $e_0((W^0+Z^0)_\infty) > 0$ almost surely. 

Finally, we prove the $L^2(\log_+L)^\alpha$ bound. This follows immediately by Fatou's lemma together with the uniform bound on Orlicz $L^2 (\log_+ L)^\alpha$ norms of the approximations in Lemma \ref{lemma: eigenvector bounds}. 
\end{proof}

\subsection{Construction of the Hamiltonian}

In this subsection, we will construct the $\varphi^4_3$ Hamiltonian on $\mathbb T^2$ as an unbounded linear map $\mathbf H$ acting on a dense domain $D(\mathbf H) \subset \mathcal H:= L^2(\nu^0)$. To begin with, let us define the semigroup $(\mathbf P_\tau)_{\tau \geq  0}$ acting on $\mathcal H$ such that $\mathbf P_0 = {\rm Id}$ and for every $\tau > 0$ and for every bounded function $f:\mathcal H \rightarrow \mathbb R$,
\begin{equation}
\mathbf P_\tau f(\varphi):= \int f(\varphi') \mathcal A_\tau (\varphi,\varphi') \nu^0(d\varphi'), \qquad \forall \varphi \in \mathcal H.  
\end{equation}
Above, $\mathcal A_\tau$ is an abuse of notation to refer to the $\varphi^4_3$ amplitude on the cylinder $[0,\tau]\times \mathbb T^2$. We will construct $\mathbf H$ as the generator of $(\mathbf P_\tau)_{\tau \geq 0}$, i.e.\ $\mathbf P_\tau = \exp(-\tau\mathbf H)$ for every $\tau \geq 0$.

\begin{proposition} \label{prop: hille yosida}
There exists a self-adjoint linear map $\mathbf H$ acting on a dense domain $D(\mathbf H)\subset \mathcal H$ such that
\begin{equation}
\mathbf H f:= -\lim_{\tau\downarrow 0} \frac 1\tau  ( \mathbf P_\tau 	- {\rm Id}_{\mathcal H}) f, \quad \forall f \in D(\mathbf H),
\end{equation}
where the convergence is in norm.
\end{proposition}

\begin{definition} \label{defn:phi43hamiltonian}
The $\varphi^4_3$ Hamiltonian is the linear map $\mathbf H$. 	
\end{definition}

 In order to prove Proposition \ref{prop: hille yosida}, we will use the classical Hille-Yosida theorem. We will therefore need to show strong continuity of $(\mathbf P_\tau)_{\tau \geq 0}$ at $\tau=0$ and an exponential bound (in $\tau$) on their operator norms.

We will use the regularized semigroups considered in the preceding subsection in order to prove strong continuity of the semigroup $(\mathbf P_\tau)_{\tau \geq 0 }$. Recall that by Proposition \ref{proposition: cutoff hamiltonian}, the ground state $e_0^T$ of $\mathbf H_T$ is strictly positive almost surely. Therefore, we may conjugate the semigroup $(\mathbf P^T_\tau)_{\tau \geq 0}$ by $(e_0^T)^{-1}$. This gives rise to a so-called \emph{ground state-transformed semigroup} that has been considered in the context of one-dimensional Schr\"odinger operators, see \cite{LHB11} and references therein. Let $\overline \nu_T^0(d\varphi):= e_0^T(\varphi)^2 \nu_T^0(d\varphi)$ and denote $\overline{\mathcal H}_T:=L^2(\overline \nu^0_T)$. For every $\tau > 0$, define the map $\overline {\mathbf P}^T_\tau$ acting on bounded functions $f: \overline{\mathcal H}_T\rightarrow \mathbb R$ by 
\begin{equation}
    \overline{\mathbf P}^T_{\tau}(f)(\phi)=e^{-E_0^T \tau}\int \frac{\mathcal{A}_\tau^T(\phi,\phi')}{e^T_{0}(\phi)e^T_{0}(\phi')}f(\phi')  \overline \nu^0_T (d\phi'), \qquad \forall \varphi \in \overline{\mathcal H}_T. 
\end{equation} 
Furthermore, let $\overline{\mathbf P}_0 := {\rm Id}_{\overline{\mathcal H}_T}$. 

The semigroup $(\overline{\mathbf P}_\tau^T)_{\tau \geq 0}$ has a generator that is a positive linear map acting on $\overline{\mathcal H}_T$. One way to see this is that, for every $\tau \geq 0$, $\overline{\mathbf P}_\tau^T$ is obtained from $\mathbf P_t^T$ by conjugating with the multiplication map by $(e_0^T)^{-1}$ and shifting by $e^{-E_0^T \tau}$. Thus,
\begin{equation}
\overline{\mathbf P}_\tau = (e_0^T)^{-1} e^{-\tau(\mathbf H_T+E_0^T)} (e_0^T)^{-1}.	
\end{equation}
From this and the fact that $e_0^T$ is strictly positive almost everywhere, we have that the generator of $(\overline{\mathbf P}^T_\tau)_{\tau \geq 0}$ exists and is equal to $(e_0^T)^{-1} (\mathbf H_T+E_0^T)(e_0^T)^{-1}$. We deduce that its spectrum satisfies $\sigma(\overline{\mathbf H}_T) \subset [0,\infty)$. Thus, for every $\tau>0$ we have that the operator norm satisfies $\| \overline{\mathbf P}_\tau^T \|_{\overline{\mathcal H}_T}\leq 1$.

In the sequel, we will exploit that $(\overline{\mathbf P}^T_{\tau})_{\tau \geq 0}$ can be associated to a Markov process. This is formalized in the lemma below, which is a consequence of \cite[Section 3.10.2]{LHB11}. 

\begin{lemma} \label{lemma: gs markov}
For every $T \in [0,\infty)$, $(\overline{\mathbf P}^T_\tau)_{\tau \geq 0} $ is the transition semigroup of a Markov process $\tilde{X}_{T}$ with  continuous sample paths and stationary measure $\overline\nu^0_T$.
\end{lemma}

\begin{remark}
By an abuse of notation, we will assume that $\tilde X_T$ is defined on some probability space labelled $(\Omega, \mathcal F, \mathbb P)$. Given $\varphi \in H^{-1/2-\kappa}_z$, we will write $\mathbb P_\varphi$ to denote the process $\tilde X_T$ conditioned on $\tilde X_T(0,\cdot) = \varphi(\cdot)$. Finally, we will further abuse notation and write $\tilde X_T$ to denote the two-sided process obtained, conditionally on $\tilde X_T(0,\cdot)$, by concatenating two independent copies. 
\end{remark}

Below, we will consider bridge measures associated to the two-sided process $\tilde X_T$ by singularly conditioning the process on the two endpoints. We will only need the case of $[-L,L]$ for an arbitrary $L>0$. Let $\varphi_-,\varphi_+$ be sampled independently according to $\overline{\nu}^0_T$. Almost surely, we may define a stochastic process $X_T^{\varphi_-,\varphi_+} \in L^2([-L,L], H^{-1/2-\kappa}_z)$ whose law is characterized\footnote{We will abuse notation and write the probability space as $(\Omega, \mathcal F, \mathbb P)$.} by the following Laplace transform. For every $f \in L^2_\tau H^{1/2+\kappa}_z$, 
\begin{multline} \label{eq: Xtphi LT}
\mathbb{E}\left [\exp\left (-\langle f,X_T^{\phi_-,\phi_+} \rangle_{L^2_{\tau}H^{-1/2-\kappa}_z}\right )\right ]
\\:=\frac{1}{Z_T(\varphi_-,\varphi_+)}\mathbb{E}_{\mu}[\exp(-\langle f,\phi \rangle_{L^2_{\tau}(H_{z}^{-1/2-\epsilon})}-V_{T}(\phi+H(\phi_{-},\phi_{+})))].
\end{multline}
\begin{remark} \label{rem: Xtphi general LT}
By approximation, we may allow for the term $\langle f, X_T^{\varphi_-,\varphi_+} \rangle$ to be replaced by $f(X_T^{\varphi_-,\varphi_+})$ for continuous functions $f$ of linear growth. 	
\end{remark}

The following lemma asserts that the processes defined above are bridge processes derived from $\tilde X_T$. It is a consequence of the DLR property of the measures on the righthand side of \eqref{eq: Xtphi LT}. See \cite[Equation (4.1.17) and Proposition 4.1]{LHB11}. 

\begin{lemma} \label{lemma: gs bridge}
For every $T \in [0,\infty)$ and almost surely for every $\varphi_-,\varphi_+ \in H^{-1/2-\kappa}_z$, 
\begin{equation}
{\rm Law}_{\mathbb P}(\tilde{X}_T|_{[-L,L]}  \mid  \tilde{X}_T(-L,\cdot) =\phi_-, \tilde{X}_T(L,\cdot) =\phi_+)= {\rm Law}_{\mathbb P}(X_T^{\phi_-,\phi_+})
\end{equation}
\end{lemma}

Let us consider the set ${\rm Lip}_b:= \{ f: H^{-1/2-\kappa} \rightarrow \mathbb R : f \text{ is Lipschitz and bounded}\} \subset \overline{\mathcal H}_T$. We have that ${\rm Lip}_b$ is dense linear subspace of $\overline{\mathcal H}_T$ for every $T>0$. The following lemma contains the key technical estimate that we will need in order to prove strong continuity of the $\varphi^4_3$ semigroup.

\begin{lemma}\label{lem: strong cont technical}
For every $f \in {\rm Lip}_b$, we have that 
    \begin{equation}\lim_{\tau\rightarrow  0}\sup_{T}\|\overline{\mathbf P}^T_{\tau}f-f\|_{\overline{\mathcal H}_T}=0.	
    \end{equation}
\end{lemma}
\begin{proof}
Let $\tau, T>0$.  Recall the definition of $\overline{\mathbf P}_t^T$ and its associated Markov process defined by Lemma \ref{lemma: gs markov}. By the Cauchy-Schwarz inequality, and the tower property of conditional expectation, we have that
\begin{equation} \label{eq: overline p bound 1}
  \int (\overline{\mathbf P}^T_{\tau}f-f)^2 d \overline \nu^0_T \leq \mathbb E\left [ \left( f(\tilde X_T(\tau)) - f(\tilde X_T(0)) \right)^2 \right].
\end{equation}
Hence, by conditioning on the values of the process at $-L$ and $L$, and using Lemma \ref{lemma: gs bridge}, we have that 
\begin{equation}
\eqref{eq: overline p bound 1} \leq I_1 + I_2,	
\end{equation}
where
\begin{align}
I_1 =  \int \int \mathbbm {1}_{\|\Xi_T^\partial (\phi_-,\phi_+)\|_{\boldsymbol{\mathfrak B}}\leq K}\mathbb{E}\left[\left(f(X_T^{\phi_-,\phi_+}(\tau))-f(X_T^{\phi_-,\phi_+}(0))\right)^2\right]d\overline\nu^0_T d\overline\nu^0_T,
\\
I_2 =  \int \int \mathbbm {1}_{\|\Xi_T^\partial (\phi_-,\phi_+)\|_{\boldsymbol{\mathfrak B}}> K}\mathbb{E}\left[\left(f(X_T^{\phi_-,\phi_+}(\tau))-f(X_T^{\phi_-,\phi_+}(0))\right)^2\right]d\overline\nu^0_T d\overline\nu^0_T.
\end{align}

To estimate $I_1$, first note that by the fact that $f$ is Lipschitz continuous there exists $c_1>0$ such that, for every $\alpha \in (0,1)$,
\begin{equation}
I_1 \leq c_1 \tau^{2\alpha} \int \int \mathbbm {1}_{\|\Xi_T^\partial (\phi_-,\phi_+)\|_{\boldsymbol{\mathfrak B}} \leq K} \mathbb E \left[\|X^{\varphi_-,\varphi_+}_T\|_{\mathcal C^\alpha_\tau H^{-1/2-\kappa}_z}^2 \right] d\overline\nu^0_T d\overline \nu^0_T.  
\end{equation}
We claim that, provided $\alpha < 1/2$, there exists $c_2(K)>0$ such that
\begin{equation} \label{eq: overline p bound 2}
\int \int  \mathbbm {1}_{\|\Xi_T^\partial (\phi_-,\phi_+)\|_{\boldsymbol{\mathfrak B}} \leq K} \mathbb E \left[\|X^{\varphi_-,\varphi_+}_T\|_{\mathcal C^\alpha_\tau H^{-1/2-\kappa}_z}^2 \right] d\overline\nu^0_T d\overline \nu^0_T \leq c_2(K).
\end{equation}
Hence, uniformly in $T$,
\begin{equation}
I_1 \leq c_1 c_2(K) \tau^{2\alpha}. 	
\end{equation}

In order to establish the H\"older continuity estimate \eqref{eq: overline p bound 2}, we will apply the Bou\'e-Dupuis formula to the right-hand side of \eqref{eq: Xtphi LT} applied to 
\begin{equation}f(X_T^{\varphi_-,\varphi_+}) = \|X_T^{\varphi_-,\varphi_+}\|_{C^{1/2-\kappa}_\tau H^{-1/2-\kappa}_z},	
\end{equation}
see Remark \ref{rem: Xtphi general LT}. We may then repeat the analysis of the previous sections with the following observations to handle the norm term. Recall the intermediate ansatz on the drift.  By 1d Sobolev embedding, for every $\alpha < 1/2$, for every $\kappa > 0$ sufficiently small, there exists $\kappa'>0$ such that, for $c_3>0$,
    \begin{equation}
    \|Z_T(v)\|_{\mathcal C^\alpha_\tau H^{-1/2-\kappa}_z} \leq c_3 \|Z_T(v)\|_{H^{1-\kappa'}}.	
    \end{equation}
    This can then be estimated as usual.
    Furthermore, we have that for every $p \in [1,\infty)$, there exists $c_4>0$ such that
    \begin{equation}
    \mathbb E\left[\|\mathbb W_T^{[3]}\|_{\mathcal C^{1/2-\kappa}_\tau H^{-1/2-\kappa}_z}^p\right ] \leq \mathbb E[\|\mathbb W_T^{[3]}\|_{\mathcal C^{1/2-\kappa}_x}^p] \leq c_4. 
    \end{equation}
The rest of the argument follows as in Sections \ref{sec: renormalization of bulk amplitudes}-\ref{sec: convergence of bulk amplitudes} (without the need for the $\Gamma$-convergence argument).

In order to estimate $I_2$, since $f$ is bounded,  it suffices to estimate $\mathbb E_{\overline\nu^0_T\otimes \overline\nu^0_T}[\mathbbm {1}_{\|\Xi_T^\partial (\phi_-,\phi_+)\|_{\boldsymbol{\mathfrak B}}> K}]$. By Markov's inequality and changing the base measure, for every $\delta>0$ sufficiently small, there exists $c_5>0$ such that
\begin{equation}
\overline \nu^0_T \otimes \overline \nu^0_T \left( \|\Xi_T^\partial (\phi_-,\phi_+)\|_{\boldsymbol{\mathfrak B}}> K \right) \leq c_5 K^{-\delta} \mathbb E_{\nu^0_T\otimes \nu^0_T}[\|\Xi_T^\partial (\phi_-,\phi_+)\|_{\boldsymbol{\mathfrak B}}^\delta e_0^T(\varphi_-)^2 e_0^T(\varphi_+)^2].	
\end{equation}
By the bounds on the stochastic norms in Proposition \ref{prop: stochastic boundary v2} and the $(\Upsilon_i)_{i \geq 2}$ terms in Section \ref{sec:stochastic}, for $\alpha'>0$ sufficiently small we have that 
\begin{equation}\sup_T \mathbb E_{\nu^0_T \otimes \nu^0_T}[\|\Xi_T^{\partial}(\varphi_-,\varphi_+)\|_{\boldsymbol{\mathfrak B}}^{\alpha'}] < \infty.
\end{equation}
Thus, by H\"older's inequality, the above estimates, and the $L^q$ uniform in $T$ bounds on the ground states in Lemma \ref{lemma: eigenvector bounds}, there exists $c_6>0$ such that, uniformly in $T$,
 \begin{equation}
 I_2 \leq c_6 K^{-\delta}.	
 \end{equation}
 
 The lemma now follows from the estimates on $I_1$ and $I_2$ by taking $\tau \rightarrow 0$ and then $K \rightarrow \infty$. 
 \end{proof}

We now return to the problem of showing strong continuity of $(\mathbf P_\tau)_{\tau \geq 0}$. Once again, we will work with the simultaneous coupling. Let us define $\boldsymbol{\mathcal H}:=L^2(\mathbb P^0)$. We will isomorphically map the linear maps $\mathbf P_\tau$ and $\mathbf P_\tau^T$ on to this space as follows. Let $T \in [0,\infty]$. Re-define, for every bounded $f:\boldsymbol{\mathcal H}\rightarrow \mathbb R$,
\begin{equation}
\mathbf P_\tau^T f((W^0+Z^0)_T):= \int f((W^0+Z^0)_T) \mathcal A^T_\tau \left( (W^0+Z^0)_T, (\tilde W^0+\tilde Z^0)_T \right) \, d\mathbb P^0(\tilde W^0+\tilde Z^0).
\end{equation}
These isomorphisms preserve the corresponding operator norms. The results we obtain can then be translated back to the semigroup acting on $\mathcal H$ by the isomorphism. 

\begin{remark}
From now on, we only take the limit $T \rightarrow \infty$ along the subsequence for which Proposition \ref{prop: tightness eigenvectors} is valid. 	
\end{remark}

We begin with a preliminary estimate on the operator norm of $\mathbf P_\tau$. 

\begin{lemma} \label{lemma: semigroup convergence in coupling}
Let $\tau \geq 0$. For every $f \in {\rm Lip}_b$, 
\begin{equation}
\lim_{T\rightarrow \infty}\|\mathbf P_\tau (e_0 f) - \mathbf P_\tau^T (e_0^T f) \|_{\boldsymbol{\mathcal H}}	 = 0.
\end{equation}
\end{lemma}

\begin{proof}
The case $\tau = 0$ is a consequence of Proposition \ref{prop: tightness eigenvectors}. Let $\tau >0$. For compacity of notation, let us write $\boldsymbol \varphi = (W^0+Z^0)$. Let us define the integral
\begin{equation}
I_T(\boldsymbol\varphi):=  \int \left( e_0(\boldsymbol\varphi') f (\boldsymbol\varphi' ) \mathcal A_\tau (\boldsymbol\varphi, \boldsymbol\varphi'  ) - e_0^T(\boldsymbol\varphi_T')  f(\boldsymbol\varphi_T') \mathcal A_\tau^T(\boldsymbol\varphi_T,\boldsymbol\varphi_T')\right)^q d\mathbb P^0( \boldsymbol{\varphi}'). 
\end{equation}
To prove the lemma, it is equivalent to show that
\begin{equation} \label{eq: it convergence}
\lim_{T \rightarrow \infty}\int I_T(\boldsymbol\varphi)^2 d\mathbb P^0(\boldsymbol\varphi) = 0. 
\end{equation}
We will do this by a Vitali's convergence theorem argument.

To begin with, let us show that $I_T(\boldsymbol{\varphi}) \rightarrow 0$ almost surely. We will do this in two parts. First, we show that that the integrand of $I_T(\boldsymbol{\varphi)}$, denoted $\iota_T^{\boldsymbol{\varphi}}(\boldsymbol{\varphi}')$, converges to $0$ almost surely.  The amplitudes in the integrand converge thanks to Theorem \ref{thm: bulk}, which is applicable since since the enhancement of $\boldsymbol\varphi'_T$ converges to the enhancement of $ \boldsymbol\varphi'$ by Theorem \ref{thm: enhancement converge}. Furthermore, since the limit $T\rightarrow \infty$ is taken along the subsequence for which Proposition \ref{prop: tightness eigenvectors} holds, we have the almost sure convergence of $e_0^T(\boldsymbol{\varphi}'_T)$ to $e_0(\boldsymbol{\varphi}')$. Hence, $\iota_T(\boldsymbol{\varphi}') \rightarrow 0$ since $f$ is continuous and $\boldsymbol\varphi'_T\rightarrow \boldsymbol\varphi'$.

We now turn to showing $I_T(\boldsymbol{\varphi})$ converges to $0$ almost surely. In order to do this, let us compute
\begin{equation} \label{eq: iota orlicz}
\int |\iota_T^{\boldsymbol{\varphi}}(\boldsymbol{\varphi}')| \log_+ |\iota_T^{\boldsymbol{\varphi}}(\boldsymbol{\varphi}')| d\mathbb P^0(\boldsymbol{\varphi')}.
\end{equation}
Note that for $a,b \geq$, $a+b \leq 2 ab$ and hence we have that 
\begin{equation}
	\log_+(|x-y|) \leq \log 2 + \log_+|x| + \log_+ |y|. 
\end{equation}
Thus, we may bound \eqref{eq: iota orlicz} by a sum of terms which, considering only the $T=\infty$ terms (since the $T<\infty$ can be treated symmetrically), consists of  the integrals:
\begin{align}
& \int |e_0(\boldsymbol{\varphi}') f(\boldsymbol{\varphi}') \mathcal A_\tau(\boldsymbol{\varphi},\boldsymbol{\varphi}')| \, \log 2 d\mathbb P^0(\boldsymbol{\varphi}')\\
&\int |e_0(\boldsymbol{\varphi}') f(\boldsymbol{\varphi}') \mathcal A_\tau(\boldsymbol{\varphi},\boldsymbol{\varphi}')| \log_+ (|e_0(\boldsymbol{\varphi}')|)d\mathbb P^0(\boldsymbol{\varphi}') \\
&\int |e_0(\boldsymbol{\varphi}') f(\boldsymbol{\varphi}') \mathcal A_\tau(\boldsymbol{\varphi},\boldsymbol{\varphi}')| \log_+ (|f(\boldsymbol{\varphi}')|)d\mathbb P^0(\boldsymbol{\varphi}') \\
&\int |e_0(\boldsymbol{\varphi}') f(\boldsymbol{\varphi}') \mathcal A_\tau(\boldsymbol{\varphi},\boldsymbol{\varphi}')| \log_+ (|\mathcal A_\tau(\boldsymbol{\varphi},\boldsymbol{\varphi}')(\boldsymbol{\varphi}')|)d\mathbb P^0(\boldsymbol{\varphi}').
\end{align}
We apply the Cauchy-Schwarz inequality to all of the integrals above and use the uniform boundedness of the Orlicz norms on $e_0^T$ and $e_0$ of Lemma \ref{lemma: eigenvector bounds} and Proposition \ref{prop: tightness eigenvectors}, and the gluing property of Theorem \ref{thm: main}. The most tricky to treat is the last term, containing log of the amplitude, so we do the calculation explicitly there. By the Cauchy-Schwarz inequality and bounds on the amplitudes, there exists $\alpha'>0$ such that \begin{equation}
(\dots) \leq C \|f\|_{L^\infty} \|e_0\|_{L^2} \left( \int \mathcal A_\tau(\boldsymbol{\varphi}, \boldsymbol{\varphi'})A_\tau(\boldsymbol{\varphi'}, \boldsymbol{\varphi}) \| \Xi^\partial (\boldsymbol{\varphi},\boldsymbol{\varphi'})\|_{\boldsymbol{\mathfrak B}}^{\alpha'} d\mathbb P^0(\boldsymbol{\varphi}')  \right)^{1/2}.
\end{equation} 
The righthand side, in particular the amplitude expectation $E_A(\boldsymbol{\varphi})$, can be estimated as in the proof of Lemma \ref{lem: large field}. All in all, we obtain that
\begin{align}
\eqref{eq: iota orlicz} &\leq C'\|f\|_{L^\infty}(1+ \|\log_+|f|\|_{L^\infty})\max_{S \in \{T,\infty\}} \mathcal A_{2\tau}^S (\boldsymbol{\varphi}, \boldsymbol{\varphi})^{1/2} ( 1 + E_A(\boldsymbol{\varphi})^{1/2}). 
\end{align}
Thus the family $\iota_T^{\boldsymbol{\varphi}}(\boldsymbol{\varphi')}$ is uniformly integrable and hence $I_T(\boldsymbol{\varphi})$ converges to $0$ almost surely.

Finally, we turn to establishing \eqref{eq: it convergence}. By arguing as above, we have that
\begin{equation}
	I_T(\boldsymbol{\varphi}) \leq C\|f\|_{L^\infty} \max_{S \in \{T,\infty\}} \mathcal A^S_{2\tau}(\boldsymbol{\varphi}, \boldsymbol{\varphi})^{1/2}. 
\end{equation} 
Hence,
\begin{align}
\int &|I_T(\boldsymbol{\varphi})|^2 \log_+|I_T(\boldsymbol\varphi)| d\mathbb P^0(\boldsymbol{\varphi})
\\&\leq C\|f\|_{L^\infty} (1+ \| \log_+ |f| \|_{L^\infty}) 
\\
&\qquad \times \max_{S \in \{T,\infty\}} \int \mathcal A_{2\tau}^S(\boldsymbol{\varphi}, \boldsymbol{\varphi}) (1+\log_+ \mathcal A_{2\tau}^S(\boldsymbol{\varphi}, \boldsymbol{\varphi}) d\mathbb P^0(\boldsymbol{\varphi}).
\end{align}
By the bounds on the amplitudes in Theorem \ref{thm: bulk}, we have that for some $\alpha'>0$,
\begin{align}
\int \mathcal A_{2\tau}^S(\boldsymbol{\varphi}, \boldsymbol{\varphi}) (1+\log_+ \mathcal A_{2\tau}^S(\boldsymbol{\varphi}, \boldsymbol{\varphi}) d\mathbb P^0(\boldsymbol{\varphi}) \leq \int \mathcal A_\tau (\boldsymbol{\varphi},\boldsymbol{\varphi})^2 \|\Xi^\partial_S(\boldsymbol{\varphi},\boldsymbol{\varphi})\|_{\boldsymbol{\mathfrak B}}^{\alpha'}  d\mathbb P^0(\boldsymbol{\varphi}).
\end{align}
This can be estimated as in Lemma \ref{lem: large field}. Thus, by Vitali's convergence theorem, we have \eqref{eq: it convergence}.

\end{proof}

A simplification of the proof (redoing the estimates without the eigenvector) yields the convergence of  $\|\mathbf P_\tau^T\|_{\mathcal L(\mathcal H_T)}$ to $\|\mathbf P_\tau \|_{\mathcal L(\mathcal H)}$. Thus we get the required exponential growth bound on the operator norms required to apply the Hille-Yosida theorem. We formalize this in the following proposition.
 
\begin{proposition}\label{proposition: semigroup norm bounds}
There exists $C>0$ such that
\begin{equation}
\|\mathbf P_\tau \|_{\mathcal L(\mathcal H)} \leq e^{C \tau }.	
\end{equation}
\end{proposition}

\begin{proof}
By Lemma \ref{lemma: semigroup convergence in coupling} and the fact that the isomorphisms of the semigroups on to $\boldsymbol{\mathcal H}$ preserves operator norms, we have that
\begin{equation}
\lim_{T \rightarrow \infty} \| \mathbf P_\tau^T \|_{\mathcal L(\mathcal H_T)} = \|\mathbf P_\tau \|_{\mathcal L(\mathcal H)}. 
\end{equation}
Hence, by the fact that $\|\mathbf P_\tau^T \|_{\mathcal L(\mathcal H_T)} \leq e^{-E_0^T \tau}$ for every $\tau \geq 0$, and the uniform bounds of Proposition \ref{proposition: cutoff spectral lower bound}, there exists $C>0$ such that
\begin{equation}
\| \mathbf P_\tau \|_{\mathcal L(\mathcal H)} \leq e^{C\tau}. 
\end{equation}
\end{proof}

We now prove strong continuity of $(\mathbf P_\tau)_{\tau \geq 0}$.

\begin{proposition} \label{proposition: semigroup strong continuity}
For every $f \in \mathcal H$, 
\begin{equation} \label{eq: strong continuity}
\lim_{\tau \rightarrow 0} \| \mathbf P_\tau f- f\|_{\mathcal H} 	= 0. 
\end{equation}
\end{proposition}

\begin{proof}
Recall the construction of $e_0$ in Proposition \ref{prop: tightness eigenvectors}. Since $e_0>0$ almost surely, we may define the measure $\overline \nu^0$ by the density $\overline\nu^0(d\varphi):= e_0(\varphi)^2\nu^0(d\varphi)$ and let $\overline{\mathcal H}:= L^2(\overline\nu^0)$. Again, since $e_0 > 0$, the map $\mathcal I_\infty: \overline{\mathcal H} \rightarrow \mathcal H$, $f \mapsto e_0 f$ is an isometric bijection. Recall the definition of ${\rm Lip}_b$ above. Since ${\rm Lip}_b \subset \overline{\mathcal H}$ is dense, we have that the set
\begin{equation}
B:= \{ e_0 f: f\in {\rm Lip}_b \} \subset \mathcal H	
\end{equation}
is dense in $\mathcal H$. Thus it is sufficient to estimate
\begin{equation}
\|\mathbf P_\tau (e_0 g) - e_0g \|_{\mathcal H},	
\end{equation}
for $g \in {\rm Lip}_b$. By Lemma \ref{lemma: semigroup convergence in coupling},
\begin{equation} 
\|\mathbf P_\tau(e_0 g) - e_0 g\|_{\mathcal H} = \lim_{T \rightarrow \infty} \| \mathbf P^T_\tau (e_0^T g) - e_0^Tg \|_{\mathcal H_T}. 	
\end{equation}
Thus, since 
\begin{equation}
\|\mathbf P_\tau^T(e_0^Tg)-e_0^Tg\|_{\mathcal H_T} = \| \overline{\mathbf P}^T_\tau(g)-g\|_{\overline {\mathcal H}_T},	
\end{equation}
the proposition follows from the uniform convergence to $0$ on ${\rm Lip}_b$ of the right-hand side as stated in Lemma \ref{lem: strong cont technical}.
\end{proof}

We now prove Proposition \ref{prop: hille yosida} and thereby complete the construction of the $\varphi^4_3$  Hamiltonian $\mathbf H$. 
\begin{proof}[Proof of Proposition \ref{prop: hille yosida}]
	
The proposition follows by the Hille-Yosida theorem to the semigroup $(\mathbf P_\tau)_{\tau \geq 0}$, which is applicable thanks to Propositions \ref{proposition: semigroup norm bounds} and \ref{proposition: semigroup strong continuity}.	
	
\end{proof}

\subsection{Spectral properties of $\mathbf H$} \label{subsec: spectral properties of H}

We will now prove that the spectrum of $\mathbf H$, denoted by $\sigma(\mathbf H)$, is discrete. Recall that by the Rayleigh-Ritz criterion,
\begin{equation}
\inf \sigma(\mathbf H) = \frac{\langle \mathbf H \varphi, \varphi \rangle}{\|\varphi\|^2}.
\end{equation}
By construction, $\sigma(\mathbf H) \subset \mathbb R$.

\begin{proposition} \label{proposition: lower bound spectrum H}
The spectrum of $\mathbf H$ is bounded from below, i.e.\
\begin{equation}
\inf \sigma(\mathbf H) >-\infty.	
\end{equation}
 	
\end{proposition}

\begin{proof}
Let $\tau \geq 1$. By applying the Cauchy-Schwarz inequality and the gluing of amplitudes in Theorem \ref{thm: main}, for every $u \in \mathcal H$,
\begin{align}
|\langle \mathbf P_\tau u, u \rangle_{\mathcal H}| 
&\leq \|u\|_{\mathcal H} \left(\int \left( \int u(\varphi') \mathcal A_\tau(\varphi,\varphi') \nu^0(d\varphi') \right)^2 \nu^0(d\varphi) \right)^{1/2}
\\
&=\|u\|_{\mathcal H} \left(\int  \int \int u(\varphi') u(\varphi'') \mathcal A_\tau(\varphi,\varphi') \mathcal A_\tau(\varphi,\varphi'') \nu^0(d\varphi') \nu^0(d\varphi'')  \nu^0(d\varphi) \right)^{1/2}
\\
&= \|u\|_{\mathcal H} \left( \int \int \mathcal A_{2\tau}(\varphi',\varphi'') u(\varphi') u(\varphi'') \nu^0 (d\varphi') \nu^0 (d\varphi'') \right)^{1/2}.
\end{align}
Hence, again by the Cauchy-Schwarz inequality and the gluing of amplitudes in Theorem \ref{thm: main}, and by the bounds on the periodic amplitudes (c.f. the main theorem of \cite{BG20}), there exists $E_0 \in \mathbb R$ such that
\begin{equation}
|\langle \mathbf P_\tau u, u \rangle_{\mathcal H}|\leq \|u\|_{\mathcal H} \left( (\mathcal A_{4\tau}^{\rm per})^{1/2} \left( \int u(\varphi')^2 u(\varphi'')^2 \nu(d\varphi) \nu(d\varphi') \right)^{1/2}  \right)^{1/2}\leq e^{E_0 t}\|u\|^2_{\mathcal H}. 
\end{equation}

In particular, $\|e^{-t(\mathbf H+E_0)}\|_{\text{op}}\leq 1$. We claim that $(\mathbf H+E_0) \geq 0$. Assume not.  Let $P_{\mathbf H}$ denote the projection-valued measure associated to $\mathbf H$ coming from the spectral theorem. For every $a<0$ let us consider $u_a \in \mathcal H$ such that
\begin{equation}
\langle u_a, P_{\mathbf H}((-\infty,a)) u_a \rangle > 0,	
\end{equation}
and write $\mu^a$ to denote the finite Borel measure such that, for every Borel set $\Omega\subset \mathbb R$, 
\begin{equation}
\langle u_a, P_{\mathbf H}(\Omega) u_a \rangle = \int_{\Omega} \mu^a(d\lambda). 
\end{equation}
Note that such a non-zero projection exists by the assumption $\inf \sigma(\mathbf H+E_0)=-\infty$.
By functional calculus,
\begin{equation}
\langle e^{-t(\mathbf H+E_0)}u_a, u_a\rangle = \int_{\lambda < -E_0} e^{-t(\lambda+E_0)}\mu^a(d\lambda) + \int_{\lambda \geq  -E_0}e^{-t(\lambda+E_0)}\mu^a(d\lambda).  	
\end{equation}
The second term is bounded uniformly in $t$. Let us analyze the first term for $a<-E_0$. Note that $(-\infty,a)$ is of positive $\mu^a$-measure by construction. Thus, by monotone convergence  
\begin{equation}
\lim_{t\rightarrow  \infty }	\int_{\lambda < -E_0} e^{-t(\lambda+E_0)}\mu^a(d\lambda) \geq \lim_{t\rightarrow \infty} \int_{(-\infty,a)} e^{-t(\lambda+E_0)}\mu^a(d\lambda) = + \infty. 
\end{equation}
This gives a contradiction and hence $\mathbf H+E_0 \geq 0$ and hence $\inf \sigma(\mathbf H) \geq -E_0$. 
\end{proof}

Let us now prove that $\sigma(\mathbf H)$ is discrete. We follow the strategy of \cite{GGV24}.

\begin{proposition} \label{proposition: spectrum H discrete}
The spectrum $\sigma(\mathbf H)$ is discrete and consists of eigenvalues of finite multiplicity with associated eigenvectors that form an orthonormal basis of $\mathcal H$. 
\end{proposition}

\begin{proof}
For every $\tau>0$,  $\mathbf P_\tau$ is Hilbert-Schmidt. To see this, first note that $P_\tau$ is an integral operator with kernel $\mathcal A_\tau$. It is sufficient to show that $\| \mathcal A_\tau\|_{L^2(\nu^0\otimes \nu^0)}$ is finite.  By the gluing property, for every $\tau > 0$, there exists $C(\tau) \geq 0$ such that
\begin{equation}
\int \mathcal A_\tau(\varphi,\varphi')\mathcal A_\tau(\varphi,\varphi') \nu(d\varphi) \nu(d\varphi') = \mathcal A_{2\tau}^{\rm per} \leq e^{C(\tau)}. 	
\end{equation}

For every $\alpha \in \mathbb C \setminus [-E_0,\infty)$, let us define the resolvent $\mathbf R_\alpha:=(\mathbf H-\alpha)^{-1}$. By Proposition \ref{proposition: lower bound spectrum H}, $\inf \sigma(\mathbf H) > -\infty$, by functional calculus we have that $\mathbf P_\tau \mathbf R_\alpha \rightarrow \mathbf R_\alpha$ in $\mathcal L(\mathcal H)$. Since $\mathbf P_\tau \mathbf R_\alpha$ is compact, we obtain $\mathbf R_\alpha$ is compact. 

Let $\alpha_0 < -E_0$ and consider the map
\begin{equation}
B(\lambda) := (\lambda-\alpha_0) \mathbf R_{\alpha_0}. 	
\end{equation}
Note that $B(\lambda)$ is compact for every $\lambda \in \mathbb C$ and $\lim_{\lambda \rightarrow \alpha_0}(\lambda-\alpha_0)^{-1}(B(\lambda)-B(\alpha_0)) = \mathbf R_{\alpha_0}$ exists. Therefore, by the analytic Fredholm theorem, $({\rm Id}_\mathcal H - B(\lambda))^{-1}$ exists outside of a discrete set of poles in $\mathbb C$. Now, by the resolvent identity, for every $\alpha \in \mathbb C \setminus [-E_0,\infty)$,
\begin{equation}
\mathbf R_\alpha = \mathbf R_{\alpha_0} + \mathbf R_{\alpha_0}\,  (\alpha-\alpha_0){\rm Id}_{\mathcal H} \, \mathbf R_\alpha. 
\end{equation}
Hence, we obtain that the resolvent admits a meromorphic extension outside of a discrete set on $\mathbb C$ (with poles of finite multiplicity). Thus $\sigma(\mathbf H)$ is discrete. Furthermore, since $\mathbf R_{\alpha_0}$ is compact, it has an orthonormal basis of eigenvectors. Thus, $\mathbf H$ admits an orthonormal basis eigenvectors.
\end{proof}

We now establish a Perron-Froebenius theorem on the lowest eigenvalue and associated ground state. By abuse of notation, let us write $E_0$ for the lowest eigenvalue.  
\begin{proposition}
The eigenvalue $E_0$ is almost surely simple and its associated normalized eigenvector, which we denote by $e_0$, can be chosen so that $e_0 >0$ almost surely with respect to $\nu^0$. 
\end{proposition}

\begin{proof}
By \cite[Proposition XIII.44]{RS78}, it is sufficient to show that $\mathbf P_\tau $ is positivity improving for every $\tau >0$. This means that for every $f \in \mathcal H$ such that $f\geq 0$ and $f$ is not identically $0$ almost surely, then we have that $\mathbf P_\tau f > 0$ almost surely. This follows from the fact that $\mathcal A_\tau$ is positive $\nu^0\otimes\nu^0$-almost everywhere.  	
\end{proof}

\begin{remark}
Although we did not prove this, it is reasonable to expect that the random variable $e_0$ constructed in Proposition \ref{prop: tightness eigenvectors} coincides with $e_0$ above (hence the abuse of notation). 	
\end{remark}

Finally, we end with uniform $L^q$ bounds on the eigenvectors of $\mathbf H$. The proof follows from similar arguments as in Lemma \ref{lemma: eigenvector bounds} and thus we omit it. 
\begin{proposition} \label{prop: eigenvector bounds limiting}
There exists $q>2$ such that, for every $i \in \mathbb N$, there exists $c(q,i)>0$ such that
\begin{equation}
\| e_i\|_{L^q(\nu^0)} \leq c(q,i),
\end{equation}
where $e_i$ is a normalized eigenvector associated to eigenvalue $E_i$. 
\end{proposition}

\subsection{Proof of  Theorem \ref{thm: hamiltonian}} \label{subsec: proof of Theorem Hamiltonian}

We now combine the results established in this section in order to prove Theorem \ref{thm: hamiltonian}. 

\begin{proof}[Proof of theorem \ref{thm: hamiltonian}]

The construction of the $\varphi^4_3$ Hamiltonian is a consequence of Proposition \ref{prop: hille yosida}, see Definition \ref{defn:phi43hamiltonian}. The spectral properties and eigenvector bounds are a consequence of all of the results in Section \ref{subsec: spectral properties of H}.
	
\end{proof}

\appendix

\section{Besov spaces, paraproducts, and resonant products} \label{appendix:besov}

In this appendix, we recall some basic definitions and properties of Besov spaces and paradifferential calculus from \cite[Sections 2.7 and 2.8]{BCD11}. 

\subsection{Besov spaces}

We define Besov spaces on $\mathbb T^d$. Let $B(x,r)$ denote the ball centred at $x \in \mathbb R^d$ of radius $r > 0$ and let $A$ denote the annulus $B(0, \frac 43) \setminus B(0, \frac 38)$. Let $\tilde \Delta, \Delta \in C^\infty_c(\mathbb R^d;[0,1])$ be radially symmetric and satisfy 
\begin{itemize}
\item $\mathrm{supp} \tilde \chi \subset B(0, \frac 43)$ and $\mathrm{supp} \chi \subset A$; 
\item $\sum_{k \geq -1} \chi_k = 1$, where $\chi_{-1} = \tilde \chi$ and $\chi_k(\cdot) = \chi(2^{-k}\cdot)$ for $k \in \mathbb N \cup \{0\}$. 
\end{itemize}
Identify $\Delta_k$ with its Fourier multiplier. 	

For $s \in \mathbb R$, $p,q \in [1,\infty]$, we define the Besov spaces $B^s_{p,q}(\mathbb T^d)$ to be the completion of $C^\infty(\mathbb T^d)$ with respect to the norm
\begin{equation}
    \| f \|_{B^s_{p,q}(\mathbb T^d)} = \Big\| \Big( 2^{ks} \|\Delta_k f \|_{L^p(\mathbb T^d)} \Big)_{k \geq -1} \Big\|_{l^q}
\end{equation}
where $l^q$ is the usual space of $q$-summable sequences, interpreted as a supremum when $q=\infty$.

The following analytic estimates are standard.
\begin{proposition}
The following estimates hold.
\begin{enumerate}
\item[(i)] \textbf{Duality.} Let $s \in \mathbb R$ and $p_1,p_2,q_1,q_2 \in [1,\infty]$ such that $\frac 1{p_1} + \frac 1{p_2} = \frac 1{q_1} + \frac 1{q_2} = 1$. Then
	\begin{equation}
	\left| \int_{\mathbb T^d} f g dx  \right| \leq
	\| f \|_{B^{-s}_{p_1,q_1}} \|g \|_{B^s_{p_2,q_2}}, \qquad \forall f,g \in C^\infty(\mathbb T^d). 
	\end{equation}
\item[(ii)] \textbf{Fractional Leibniz.} Let $s \in \mathbb R$, $p,p_1,p_2,p_3,p_4,q \in [1,\infty]$ satisfy $\frac 1p = \frac 1{p_1} + \frac 1{p_2} = \frac 1{p_3} + \frac 1{p_4}$. Then there exists $C>0$ such that
	\begin{equation}
	\| fg \|_{B^s_{p,q}} \leq C\| f \|_{B^s_{p_1,q}} \| g \|_{L^{p_2}} + \| f \|_{L^{p_3}} \|g \|_{B^s_{p_4,q}}, \qquad \forall f,g \in C^\infty(\mathbb T^d).
	\end{equation}
\item[(iii)] \textbf{Interpolation.} Let $s,s_1,s_2 \in \mathbb R$ such that $s_1 < s < s_2$, $p,p_1,p_2,q,q_1,q_2 \in [1,\infty]$ and $\theta \in (0,1)$ satisfy
	\begin{equation}
	s= \theta s_1 + (1-\theta)s_2, \quad \frac 1p = \frac \theta {p_1} + \frac{1-\theta}{p_2}, \quad \frac 1q=\frac \theta {q_1} + \frac{1-\theta}{q_2}.
	\end{equation}
	Then there exists $C>0$ such that 
	\begin{equation}
	\| f \|_{B^s_{p,q}}\leq C\| f \|_{B^{s_1}_{p_1,q_1}}^\theta \| f \|_{B^{s_2}_{p_2,q_2}}^{1-\theta}.
	\end{equation}
\item[(iv)] \textit{Bernstein inequalities.} Let $R>0$ and let $B_{\mathbb Z^d}(R) = \{ n \in \mathbb Z^d : |n| \leq R \}$. Let $s_1, s_2 \in \mathbb R$ such that $s_1 < s_2$, $p,q \in [1,\infty]$. There exists $C>0$ such that 
\begin{equation}
\| f \|_{B^{s_2}_{p,q}}\leq C R^{s_2 - s_1} \| f \|_{B^{s_1}_{p,q}}, \qquad \forall f,g \in C^\infty(\mathbb T^d)
\end{equation}
and
\begin{equation}
\| g \|_{B^{s_1}_{p,q}}\leq CR^{s_1 - s_2}\| g \|_{B^{s_2}_{p,q}},	
\end{equation}
for $f,g \in C^\infty(\mathbb T^d)$ such that $\mathrm{supp} (\mathcal F f) \subset B_{\mathbb Z^d}(R)$ and $\mathrm{supp} (\mathcal F g) \subset \mathbb Z^d \setminus B_f(R)$. 
\end{enumerate}	
\end{proposition}

\subsection{Para and resonant products}

We now introduce para and resonant products following \cite[Section 2.8]{BCD11}. Let $f,g \in C^\infty(\mathbb T^d)$. Define the paraproduct
\begin{equation}
f \succ g=\sum_{l < k-1} \Delta_k f \Delta_l g	
\end{equation}
and the resonant product
\begin{equation}
f \circ g
=
\sum_{|k-l| \leq 1} \Delta_k f \Delta_l g.	
\end{equation}
We also write $f \prec g$ interchangeably for $g \succ f$. We thus have Bony's decomposition 
\begin{equation}
fg=f \prec g + f \circ g + f \succ g.	
\end{equation}

The following para and resonant product estimates are standard. 
\begin{proposition}
The following estimates hold.
\begin{enumerate}
	\item[(i)]\textbf{Paraproduct estimate.} Let $s \in \mathbb R$ and $p,p_1,p_2,q \in [1,\infty]$ be such that $\frac 1p = \frac 1 {p_1} + \frac 1{p_2}$. There exists $C>0$ such that
\begin{equation} 
	\|f \succ g\|_{B^{s}_{p,q}}\leq C\| f \|_{B^s_{p_1,q}} \| g \|_{L^{p_2}}, \qquad \forall f,g \in C^\infty(\mathbb T^d).
\end{equation}
	\item[(ii)] \textbf{Resonant product estimate.} Let $s_1,s_2 \in \mathbb R$ such that $s=s_1 + s_2 > 0$. Let $p,p_1,p_2,q \in [1,\infty]$ satisfy $\frac 1p = \frac 1{p_1} + \frac 1{p_2}$. There exists $C>0$ such that
	\begin{equation}
	\| f \circ g \|_{B^s_{p,q}} \leq C\| f \|_{B^{s_1}_{p_1,\infty}} \| g \|_{B^{s_2}_{p_2,q}}, \qquad \forall f,g \in C^\infty(\mathbb T^d). 	
	\end{equation}
\end{enumerate}	
\end{proposition}

\section{Bulk covariance estimates} \label{sec: Dirichlet covariance estimates}

\subsection{Interpolation with cutoff scale $T$}
\begin{proposition} \label{prop: M covariance interp T}
Let $\alpha \in [0,1)$. Then there exists $C>0$ such that, for every $T>0$, 
\begin{equation}
|C^M_T(x,y)|\leq C \frac{\langle T\rangle ^\alpha}{|x-y|^{1-\alpha}}, \qquad \forall x,y \in M. 
\end{equation}
\end{proposition}

In the case $\alpha > 0$, Proposition \ref{prop: M covariance interp T} follows from Proposition \ref{prop: M derivative covariance interp} and the fundamental theorem of calculus. The case $\alpha = 0$ uses the fact that $C^M_T(x,y) = \int_0^T \partial_t \rho_t^2 \ast C^M_\infty(x,y)dt$ and standard bounds on convolutions and the usual Dirichlet Green function $C^M_\infty$.

\subsection{Interpolation of time-derivative with cutoff scale $T$}

\begin{proposition}\label{prop: M derivative covariance interp}
Let $\alpha \in [0,1)$. Then there exists $C>0$ such that, for every $t>0$,
\begin{equation}
|\dot{C}^M_t(x,y)| \leq C \frac{\langle t\rangle ^{\alpha-1}}{|x-y|^{1-\alpha}}, \qquad \forall x,y\in M.
\end{equation}
Furthermore, in the case $\alpha \in (-1,0)$, there exists $C>0$ such that for every $t>0$,
\begin{equation}
|\dot{C}^M_t(x,y)| \leq C \frac{\langle t\rangle ^{\alpha-1} \log\left( 1 + \frac{|x-y|^2}{(x_1-y_1)^2} \right)}{|x-y|^{1-\alpha}}, \qquad \forall x,y\in M.
\end{equation}	
\end{proposition}

\begin{proof}

First of all, note that 
\begin{equation}
\dot C_t^M(x,y) = \int_{\mathbb T^2} (\delta_{u_1=y_1} \otimes \partial_t \rho_t^2)(u-y) C_\infty^M(u,x)	du.
\end{equation}
Here we abuse notation, denoting by $\rho$ the inverse Fourier Transform fo $\rho$. This is consistent with denoting operators and kernels in the same way. 
Note that $\delta_{u_1=y_1}\otimes \partial_t \rho_t^2(u-y) = \delta_{u_1=y_1} \langle t \rangle^{-1} \langle t \rangle^2 \rho'(\langle t\rangle((u_2,u_3 ) - (y_2,y_3)))$. By abuse of notation, let us write this as $\langle t \rangle\chi(\langle t \rangle(u-y))$ with (vector in the $(e_2,e_3)$ coordinates) primitive $\tilde\chi(\langle t\rangle (u-y))$ ---  we will use the primitive when considering the case $\alpha < 0$ below. 

Let us first analyze when $\alpha > 0$. By the estimate on the Green function (see \cite{CPR10}),
\begin{align}
|\dot C_t^M(x,y)| &\leq c_1 \langle t \rangle^{-1} \int_{\mathbb T^2} \langle t \rangle^2 |\chi(\langle t \rangle(u-y))| \frac{1}{|u-x|} du	
\\
&=\langle t \rangle^{-1+2-\beta} \int_{\mathbb T^2} \langle t \rangle^\beta|u-y|^\beta |\chi(\langle t \rangle(u-y))| \frac{1}{|u-y|^\beta|u-x|} du.
\end{align}
Choosing $1<\beta<2$ and computing the convolution, we thus have that
\begin{equation}
|\dot C_t^M(x,y)| \leq c_2 \frac{\langle t \rangle^{1-\beta}}{|x-y|^{\beta - 1}}.
\end{equation}
Choosing $\beta = 2-\alpha$ then yields the desired estimate in the case $\alpha \in (0,1)$. 
 
We now consider $\alpha = 0$. We split $\dot C_t^M(x,y) = I_1 + I_2$ according to whether $|u-y| \geq |x-y|/2$ or $|u-y| \leq |x-y|/2$, respectively. In the first case, we can proceed above to obtain
\begin{align}
|I_1| \leq c_3 \frac{\langle t \rangle^{-1}}{|x-y|^\varepsilon} 	\int_{|u-y|\leq |x-y|/2} \frac{1}{|u-y|^{2-\varepsilon}|u-x|} \leq c_4 \frac{\langle t \rangle^{-1}}{|x-y|}.
\end{align}
On the other hand, when $|u-y| \leq |x-y|/2$, we have that $|u-x| \geq |x-y|/2$. Hence
\begin{equation}
|I_2| \leq c_5 \frac{\langle t \rangle^{-1}}{|x-y|} \int \langle t \rangle^2 |\chi(\langle t\rangle (u-y))| du \leq c_6 \frac{\langle t \rangle^{-1}}{|x-y|}.	
\end{equation}

We will now analyze the case when $\alpha < 0$ in which we will obtain a log correction. By integration by parts, and recalling the primitive $\tilde \chi$, we have that  
\begin{equation}
\dot C_t^M(x,y) = \int_{\mathbb T^2} \tilde\chi(\langle t \rangle (u-y)) \cdot  \nabla_{(u_2,u_3)} C^M_\infty(u,x)  du.
\end{equation}
Let $\beta \in \mathbb R$ be arbitrary. Then by the gradient estimate on the Green function (see for instance \cite{CPR10}),
\begin{align}
|\dot C_t^M(x,y)| &\leq c_1 \langle t \rangle ^{-\beta} \int_{\mathbb T} \delta_{u_1=y_1} \langle t\rangle ^\beta |u-y|^\beta |\tilde\chi(\langle t \rangle (u-y))| \frac{1}{|u-y|^\beta |u-x|^2} du
\\
&\leq c_2 \langle t\rangle^{-\beta} \int_{\mathbb T^2} \frac{1}{((u_2-y_2)^2+(u_3-y_3)^2)^{\beta/2}} 
\\
&\hspace{40mm} \times \frac{1}{((y_1-x_1)^2+(u_2-x_2)^2 + (u_3-x_3)^2)^2}du 
\end{align}
provided $0<\beta < 2$. We claim that the convolution above is finite and in fact bounded by
\begin{equation}
c_3 \frac{\langle t \rangle^{-\beta} \log\left( 1+ \frac{|x-y|^2}{(x_1-y_1)^2}\right)}{|x-y|^\beta }.	
\end{equation}
 Setting $\beta = 1-\alpha$ then yields the desired estimate with the range $\alpha \in (-1,0)$ (and in fact a suboptimal estimate in the full range $\alpha \in (-1,1)$). 

We now turn to the convolution and for ease of notation set $y=0$ (the other cases follow similarly). Then we need to estimate
\begin{equation}
I:= \int_{\mathbb T^2} \frac{1}{(u_2^2+u_3^2)^{\beta/2}(x_1^2 + (u_2-x_2)^2 + (u_3-x_3)^2} du.	
\end{equation}
Let us split it into the integration domains:
\begin{align}
\mathcal A_1&:= \{ x_1^2 + (u_2-x_2)^2 + (u_3-x_3)^2 \leq |x|^2/4 \}, \,\mathcal A_2:= \mathbb T^2 \setminus \mathcal A_1. 	
\end{align}
On $\mathcal A_1$, we have that $c|x| \leq |(u_2,u_3)| \leq C |x|$. Thus we obtain that
\begin{equation}
I(\mathcal A_1) \leq \frac{c_3}{|x|^\beta} \int_0^{|x|} \frac{r}{(x_1^2+r^2)}dr	= \frac{c_3}{2|x|^\beta} (\log (x_1^2+|x|^2) - \log(x_1^2))
\end{equation}
On $\mathcal A_2=\mathcal A_1^c$, we have that
\begin{equation}
I(\mathcal A_2) \leq \frac{c_4}{|x|^{\varepsilon}} \int \frac{1}{|(u_2,u_3)|^\beta |(0,u_2,u_3)-x|^{2-\varepsilon}} du \leq \frac{c_5}{|x|^{\beta}}.
\end{equation}

\end{proof}

\subsection{Interpolation with distance to boundary}

\begin{proposition}
\label{prop: M covariance interp bdry}	
Let $\alpha_1,\alpha_2\in[0,1)$. Then there exists $C>0$ such that, for every $T>0$,
\begin{equation}
|C^M_T(x,y)| \leq C \frac{d(x,\partial M)^{\alpha_1}d(y,\partial M)^{\alpha_2}}{|x-y|^{1+\alpha_1+\alpha_2}}, \qquad \forall x,y \in M. 	
\end{equation}
\end{proposition}

\begin{proof}
By symmetry and interpolation, it is sufficient to prove that for $\alpha \in [0,1)$,
\begin{equation}
|C^M_T(x,y)| \leq C \frac{d(x,\partial M)^\alpha}{|x-y|^{1+\alpha}}.	
\end{equation}
We will use that for the case $T=\infty$,
\begin{equation}
|C^M_\infty(x,y)| \leq C \frac{\min (d(x,\partial M)^\alpha, d(y,\partial M)^\alpha)}{|x-y|^{1+\alpha}},	
\end{equation}
see for instance \cite{CPR10}. 

Note that
\begin{equation}
C_T^M(x,y) = \int_{\mathbb T^2} \delta_{u_1=y_1}\otimes \langle T\rangle^2 \rho^2(\langle T \rangle (u-y)) C^M_T(u,x) du.
\end{equation}
Let us split this integral into $I_1$ and $I_2$ according to whether $|u-y| \geq |x-y|/2$ or $|u-y| \leq |x-y|/2$, respectively.

For $I_1$, using the bound above,
\begin{align}
|I_1| &\leq c_1 d(x,\partial M)^\alpha \int \delta_{u_1=y_1}\otimes \langle T\rangle^2 \rho^2(\langle T \rangle (u-y)) |u-y|^\beta \frac{1}{|u-y|^\beta |u-x|^{1+\alpha}} du
\\
&\leq c_2 \langle T \rangle^{2-\beta} d(x,\partial M)^\alpha \int \frac{1}{|u-y|^\beta |u-x|^{1+\alpha}} du 
\\ &\leq c_3 \frac{\langle T \rangle^{2-\beta} d(x,\partial M)^\alpha}{|x-y|^{\epsilon}} \int \frac{1}{|u-y|^{\beta-\varepsilon}|u-x|^{1+\alpha}} du,
\end{align}
where we choose $\beta = 2$. Hence,
\begin{equation}
|I_1| \leq c_4 	\frac{d(x,\partial M)^\alpha}{|x-y|^{1+\alpha}}. 
\end{equation}

Now, in the case $|u-y| \leq |x-y|/2$, we have that $|u-x| \geq |x-y|/2$. Hence by the bound on $C^M_\infty(x,y)$, we have that
\begin{equation}
|I_2| \leq c_1 \frac{d(x,\partial M)^\alpha}{|x-y|^{1+\alpha}} \int \delta_{u_1=y_1}\otimes \langle T \rangle^2 \rho^2( \langle T \rangle (u-y)) du \leq c_2 	\frac{d(x,\partial M)^\alpha}{|x-y|^{1+\alpha}}.
\end{equation}

Combining the estimates for $I_1$ and $I_2$ yields the desired result.

\end{proof}

\subsection{Properties of $J_t$}\label{appendix: regularizing jt}

\begin{lemma}[Commutator estimate]\label{commutator-estimate}
For every $s<0$, every $p \in [1,\infty)$, Then for $1/q>1/p+2\epsilon$ and $\epsilon<1/4$, for every $f,g \in C^\infty$,  
\begin{equation}
	\|J_t(f \succ g) - J_t f \succ g\|_{B^{s-\varepsilon'}_{p,p}} \leq C\langle t \rangle^{-1/2-\varepsilon'}\|f\|_{\mathcal C^{s-\varepsilon}} \|g\|_{B^{1/4}_{q,q}}.
\end{equation}
\end{lemma}

\begin{proof}
Let us write $j_t:=J_t^{\rm per}$, the corresponding map defined as in the periodic case. The commutator estimate for $j_t$ follows, so therefore we have to estimate the terms
\begin{equation}
	T_1:= (J_t-j_t)(f \succ g), \qquad 
	T_2:=((J_t-j_t)f)\succ g.
\end{equation}
We will analyze $T_1$, the desired estimate on $T_2$ follows similarly.

We use the heat kernel decomposition:
\begin{equation}
J_t-j_t:= \mathcal F^{-1}\left( \sqrt{\rho_t^2} \right) \ast \left( \int_0^\infty s^{-1/2} e^{s (\Delta-m^2)} ds - \int_0^\infty s^{-1/2} e^{s(\Delta^{\rm per}-m^2)} ds \right),
\end{equation}
where we recall $\ast$ is convolution in the $z$-direction.
The first factor satisfies the bound 
\begin{equation} \|\mathcal F^{-1}\left( \sqrt{\rho_t^2} \right) \ast f\|_{B^{s-\epsilon}_{p,p}}\leq \langle t\rangle^{-1/2-\epsilon} \|f\|_{B^{s}_{p,p}}.	
\end{equation}

And we need to estimate the operator on the right. Let us define this operator to be $\Gamma$. The kernel of $\Gamma$ viewed as an operator on $M^{\rm per}$ is given by
\begin{equation}
	\Gamma(x,y):= \int_0^\infty s^{-1/2} \left( e^{s(\Delta-m^2)}(x,y) -e^{s(\Delta^{\rm per}-m^2)}(x,y) \right) ds.  
\end{equation}
We represent the heat kernels by their Poisson summations. Let $(-\Delta^\infty+m^2)$ be the massive Laplacian on $\mathbb R^3$. Then (using the method of images for the second one), we have that
\begin{align}
	e^{s(\Delta^{\rm per}-m^2)}(x,y) &= \sum_{n \in \mathbb Z^3}e^{s(\Delta^\infty-m^2)}(x+n,y),
	\\
	e^{s(\Delta-m^2)}(x,y)&= \sum_{n \in \mathbb Z^3} (-1)^{n_1} e^{s(\Delta^\infty-m^2)}(x+n,y).
\end{align}
Hence,
\begin{equation} 
e^{s(\Delta-m^2)}(x,y) - e^{s(\Delta^{\rm per}-m^2)}(x,y)= - 2\sum_{\substack{n \in \mathbb Z^3 \\ n_1 \in 2\mathbb Z+1}}  e^{s(\Delta^\infty-m^2)}(x+n,y).
\end{equation}
We now claim that for $\beta$ a multindex with $|\beta|\leq 2$ for any $\alpha<2+|\beta|$
\begin{equation}\label{eq:heat-difference}
|\partial_{x,y}^{\beta} \Gamma(x,y)| \leq \frac{{d(x,\partial M)^\alpha}}{|x-y|^{2+|\beta|-\alpha}}
\end{equation}

We now split $\Gamma = \Gamma^<+\Gamma^>$ according to the $s$ integration near $0$ and at $\infty$, respectively. Let us first estimate $\Gamma^>$. By the above, it suffices to estimate
\begin{equation}
	\Gamma^>(x,y) = -2\sum_{\substack{n \in \mathbb Z^3 \\ n_1 \in 2\mathbb Z+1}}\int_{s_0}^\infty s^{-1/2} \frac{1}{(2\pi s)^{3/2} }e^{- \frac{|x+n-y|^2}{2s}}e^{-sm^2} ds.
\end{equation}
Thanks to the exponential decay in $s$ due to the $m^2$ term, $\Gamma^>$ and its derivatives are uniformly bounded. 

We now turn to $\Gamma^<$. We do the $\partial^{\beta_1}_x$ derivative, the other follows similarly. From above, we have that
\begin{align}
	|\partial^{\beta_1}_x \Gamma^<(x,y)| &\leq c_1 \sum_{\substack{n \in \mathbb Z^3 \\ n_1 \in 2\mathbb Z+1}}\int_{0}^{s_0}\frac{|x+n-y|}{s^3} e^{- \frac{|x+n-y|^2}{2s}} ds.
\end{align}
We split the sum according to when $n_1 \in \{-1,1\}$ and otherwise. For the first, we have that by a change of variables 
\begin{multline}
	\sum_{\substack{n \in \mathbb Z^3 \\ n_1=-1,1}}\int_{0}^{s_0}\frac{|x+n-y|}{s^3} e^{- \frac{|x+n-y|^2}{2s}} ds \\ \leq \sum_{\substack{n \in \mathbb Z^3 \\ n_1=-1,1}}\frac{c_2}{|x+n-y|^3} \leq c_3 \frac{\min (d(x,\partial M)^{-\delta}, d(y,\partial M)^{-\delta})}{|x-y|^{3-\delta}}. 
\end{multline}
For the second, we use that $|x-n-y| \geq c_4 |n_1|$ uniformly over $x$ and $y$. From this, we have that 
\begin{align}
	\sum_{\substack{n \in \mathbb Z^3 \\ |n_1|\geq 2}}\int_{0}^{s_0}\frac{|x+n-y|}{s^3} e^{- \frac{|x+n-y|^2}{2s}} ds &\leq \sum_{\substack{n \in \mathbb Z^3 \\ |n_1|\geq 2}}\int_{0}^{s_0}\frac{|x+n-y|}{s^3} e^{- \frac{|x+n-y|^2}{4s}} e^{-\frac{c_4|n_1|}{4s}} ds 
	\\
	&\leq  \sum_{\substack{n \in \mathbb Z^3 \\ |n_1|\geq 2}}\frac{c_5 e^{-c_4|n_1|}}{|x+n-y|^3}  \leq c_6. 
\end{align}
This shows \eqref{eq:heat-difference}. Now by Young's inequality \eqref{eq:heat-difference} implies for $|\beta|\leq 2$
\[
\|d(x,\partial M)^{\alpha} \left(\partial^{\beta} \int \Gamma(x,y)f(y)\mathrm{d}y \right)\|_{L^{p}} \leq \|f\|_{L^{p}}
\]
Provided $\alpha>1$, which implies that $f \mapsto \int \int \Gamma(x,y)f(y) \mathrm{d}y$ is bounded from $L^p$ to the weighted Sobolev space $W^{2,p}(d^{1+\delta}(x,\partial M))$, 
and similarly one may show boundedness 
\[W^{-1,p}~\mapsto~W^{1,p}(d^{1+\delta}(x,\partial M)).\]
In the same way one may show that it is also bounded $W^{s,p} \mapsto W^{s+1,p}$ (without weight) for $s \in \mathbb{R}$. Now by interpolation (see \cite{ADD23}) we obtain $\Gamma$ is bounded 
\[ W^{s,p} \mapsto W^{s+\epsilon}(d^{2\epsilon,p}(x,\partial M))\] which in turn implies boundedness  $W^{s,p} \mapsto W^{s+\epsilon,q}$ for $1/q>1/p+2\epsilon$ by Hölder's inequality. From this our statement follows by Standard Besov embeddings.
\end{proof}

The following lemma follows from above commutator proof. 


\begin{lemma}
\begin{equation}
\|J_t f\|_{B^{s+1-\varepsilon'}_{p,p}} \leq C \langle t \rangle^{-1/2-\varepsilon'} \|f\|_{B^s_{p,p}}
\end{equation}	
\end{lemma}
\begin{proof}
  This follows from \eqref{eq:heat-difference} and standard arguments. 
\end{proof}

\subsection{Regularizing properties of $(C_T^M)^{1/2}$}

\label{subsec: regularizing cthalf}

As can be seen by using the basis of the Dirichlet Laplacian (product of sines), $C^M_\infty$ is a positive operator and we denote its square root by $(C^M_\infty)^{1/2}$. We may represent this by the formula,
\begin{equation}
	(C_\infty^M)^{1/2} = \int_0^\infty t^{-1/2} e^{-t (-\Delta+m^2)} dt.
\end{equation}

The following lemma follows from \cite{CPR10}. Note that the asymmetric bound is as in their heat kernel bound.
\begin{lemma} 
For every $\alpha_1,\alpha_2 \in (0,1)$, there exists $C>0$ such that
\begin{equation}
	|(C_\infty^M)^{1/2}(x,y)| \leq C\frac{d(x,\partial M)^{\alpha_1}d(y,\partial M)^{\alpha_2}}{|x-y|^{2+\alpha_1+\alpha_2}}, \qquad \forall x,y \in M.
\end{equation}	
\end{lemma}

\section{Harmonic extension and Poisson kernel estimates} \label{appendix: harmonic extension}

\subsection{Regularity versus boundary blow-up} \label{appendix: harmonic extension estimates}

Let $\overline H$ denote the $m$-Harmonic extension on $\mathbb{R}\times \mathbb{T}^2$ with associated Poisson kernel $\overline P_0$. For each $\tau \geq 0$, let $\widehat{\overline P}_0(\tau,\cdot)$ denote its Fourier transform. Explicitly, we have 
\begin{equation}
\widehat{\overline P}_0(\tau,n)\label{eq:poisson-fourier}
=
e^{-|\tau|(4\pi^2|n|^2+m^2)^{1/2}}, \qquad \forall \tau \in \mathbb R, \, \forall n \in \mathbb Z^2.	
\end{equation}
Let $\overline H^*$ denote the adjoint of $\overline H$. Its action on smooth functions is given by
\begin{align}
\widehat{\overline H^*} g(n)
=
\int_\mathbb{R} e^{-|\tau|(4\pi^2|n|^2+m^2)^{1/2}} \hat g(\tau,n) d\tau, \qquad \forall g\in C^\infty(\mathbb R\times \mathbb T^2), \, \forall n \in \mathbb Z^2. 	
\end{align}

Let $g \in L^1_\tau B^s_{p,q,z}$. We claim that $\overline H^* g \in B^s_{p,p,z}$. To see this, note that
\begin{align}
\int_{\mathbb R} \|\widehat{\overline H^*} g(\tau,\cdot)\|_{B^s_{p,p,z}}d\tau
&\leq
C \int_{\mathbb R} \| g(\tau,\cdot) \|_{B^s_{p,p,z}} d\tau.
\end{align}
Hence $\overline H^*: L^1_\tau B^s_{p,q,z} \rightarrow B^s_{p,q,z}$. 

We now have the following lemma 
\begin{lemma}
  Let $\alpha \in \mathbb R$ and $\beta\geq0$, then
  \begin{equation} \label{tool: H reg decay}
  \| \overline H\varphi(\tau,\cdot)\|_{B^{\alpha +\beta}_{p,q,z}}
  \leq
  C(\beta) \tau^{-\beta}\|\varphi\|_{B^\alpha_{p,q,z}}.	
\end{equation}
And similarly 
\begin{equation}
  \|H(\phi_-,\phi_+)(\tau,\cdot)\|_{B^{\alpha +\beta}_{p,q,z}}\leq C(\beta,L) d(\tau,\cdot)^{-\beta} \left(\|\phi_-\|_{B^{\alpha}_{p,q,z}}+\|\phi_+\|_{B^{\alpha}_{p,q,z}}\right).
\end{equation}
\end{lemma}

\begin{proof}
To show the first statement observe that by \eqref{eq:poisson-fourier} we have that $\overline{H}(\phi)(\tau,\cdot)=M_{\tau}\phi$, where $M_{\tau}$ is an $S^{\beta}$ multiplier in the sense of \cite{BCD11}, with multiplier norm being bounded by $\tau^{-\beta}$. For the second statement we observe that (for a cylinder $[0,L]\times \mathbb{T}^2$)
\begin{equation}
  H(\phi_-,\phi_+)(\tau,\cdot)=H\left(\overline{H}(\phi_-)(L,\cdot),\overline{H}(\phi_+)(-L)\right)(\tau,\cdot)+\overline{H}(\phi_-)(\tau,\cdot)+\overline{H}(\phi_+)(\tau-L,\cdot).
\end{equation}
The first term satisfies
\begin{align*}
  \|H\left(\overline{H}(\phi_-)(L,\cdot),\overline{H}(\phi_+)(-L)\right)&\|_{L^\infty_{\tau}B^{\alpha+\beta}_{p,q,z}} 
  \\ \leq& C\left( \|\overline{H}(1,\cdot)(\phi_{-})\|_{B^{\alpha+\beta}_{p,q,z}}+\|\overline{H}(0,\cdot)(\phi_{+})\|_{B^{\alpha+\beta}_{p,q,z}})\right)\\
                        \leq&C(L,\beta)\left(\|\phi_-\|_{B^{\alpha}_{p,q,z}}+\|\phi_+\|_{B^{\alpha}_{p,q,z}} \right),
\end{align*}
where in the last line we have applied \eqref{tool: H reg decay}. Now we can conclude by triangle inequality.

\end{proof}

\subsection{Gradient estimate}

\begin{lemma} \label{lem: gradient Cb}
  For every $\delta > 0$, there exists $C>0$ such that 
  \begin{equation} |\nabla_y C^B_\infty (x, y)| \leq C \frac{1}{d (y, \partial M)^{2 + \delta}} \frac{1}{d
     (x, \partial M)^{\delta}}, \qquad \forall x,y \in M.
  \end{equation}
\end{lemma}

\begin{proof}
We drop $\infty$ from the notation. By definition we have
  \[ \nabla_y C^B (x, y) = \int_{\partial M} \int_{\partial M} C^{M} (\tilde{x},
     \tilde{y}) P (x, \tilde{x}) \nabla_y P (y, \tilde{y}) \mathd \tilde{x}
     \mathd \tilde{y}. \]
  Now from \cite{K06} we have
  \begin{eqnarray*}
    | \nabla_y C^B (x, y) | & \leq & \int_{\partial M} \int_{\partial M} | C^{M}
    (\tilde{x}, \tilde{y}) P (x, \tilde{x}) \nabla_y P (y, \tilde{y}) | \mathd
    \tilde{x} \mathd \tilde{y}\\
    & \leq & C\int_{\partial M} \int_{\partial M} \frac{1}{| \tilde{x} - \tilde{y} |} \frac{d
    (x, {\partial M})}{| x - \tilde{x} |^3} \frac{d (y, {\partial M})}{| y - \tilde{y} |^4} \mathd
    \tilde{x} \mathd \tilde{y}\\
    & \leq &C' \int_{\partial M} \int_{\partial M} \frac{1}{| \tilde{x} - \tilde{y} |} \frac{1}{|
    x - \tilde{x} |^2} \frac{1}{| y - \tilde{y} |^3} \mathd \tilde{x} \mathd
    \tilde{y}\\
    & \leq & C''\frac{1}{d {(y, {\partial M})^{2 + \varepsilon}} } \frac{1}{d (x,
    {\partial M})^{\varepsilon}} \int_{\partial M} \int_{\partial M} \frac{1}{| \tilde{x} - \tilde{y} |}
    \frac{1}{| x - \tilde{x} |} \frac{1}{| y - \tilde{y} |} \mathd \tilde{x}
    \mathd \tilde{y}\\
    & \leq & C'''\frac{1}{d {(y, {\partial M})^{2 + \varepsilon}} } \frac{1}{d (x,
    {\partial M})^{\varepsilon}}.
  \end{eqnarray*}
  
\end{proof}

\bibliographystyle{plain}
\bibliography{refs.bib}

\end{document}